\documentclass[12pt]{amsart}

\usepackage{amssymb}
\usepackage[all]{xy}

\newtheorem{theorem}{Theorem}[section]
\newtheorem{lemma}[theorem]{Lemma}

\newtheorem{cor}[theorem]{Corollary}
\theoremstyle{definition}
\newtheorem{definition}[theorem]{Definition}

\theoremstyle{remark}
\newtheorem{remark}[theorem]{Remark}
\numberwithin{equation}{section}
\theoremstyle{conjecture}
\newtheorem{conjecture}[theorem]{Conjecture}

\theoremstyle{problem}

\newcommand\C{\mathbb{C}}

\newcommand\Z{\mathbb{Z}}

\newcommand\R{\mathbb{R}}
\newcommand\N{\mathbb{N}}
\newcommand\T{\mathbb{T}}
\newcommand\cA{\mathcal{A}}

\newcommand\cC{\mathcal{C}}
\newcommand\cD{\mathcal{D}}
\newcommand\cE{\mathcal{E}}

\newcommand\cG{\mathcal{G}}

\newcommand\cO{\mathcal{O}}

\newcommand\cZ{\mathcal{Z}}

\newcommand\Aut{\operatorname{Aut}}

\newcommand\End{\operatorname{End}}
\newcommand\Tube{\operatorname{Tube}}
\newcommand\id{\mathrm{id}}
\newcommand\Ad{\mathrm{Ad}}
\newcommand\Hom{\operatorname{Hom}}

\newcommand{\Irr}{\operatorname{Irr}}
\newcommand{\Inv}{\operatorname{Inv}}

\newcommand\inpr[2]{\langle{#1,#2}\rangle}

\newcommand{\tmu}{\widetilde{\mu}}
\newcommand{\tpi}{\widetilde{\pi}}

\newcommand{\tpsi}{\widetilde{\psi}}

\newcommand{\brho}{\overline{\rho}}

\newcommand{\trho}{\tilde{\rho}}
\newcommand{\tvarphi}{\tilde{\varphi}}

\title[Infinite families of modular data]
{Infinite families of potential modular data related to quadratic categories} 

\author{Pinhas Grossman}
\address{School of Mathematics and Statistics, University of New South Wales,
Sydney NSW 2052, Australia}
\email{p.grossman@unsw.edu.au}

\author{Masaki Izumi}
\address{Department of Mathematics\\ Graduate School of Science\\
Kyoto University\\ Sakyo-ku, Kyoto 606-8502\\ Japan}
\email{izumi@math.kyoto-u.ac.jp}

\subjclass[2010]{ 
Primary 46L37; Secondary 18D10}
\keywords{modular data, subfactors, fusion categories}

\thanks{Supported in part by JSPS KAKENHI Grant Number JP15H03623 and ARC grant DP170103265.}

\begin{document} 

\begin{abstract}
We present several infinite families of potential modular data motivated by examples of Drinfeld centers of quadratic categories. In each case, the input is a pair of involutive metric groups with Gauss sums differing by a sign, along with some conditions on the fixed points of the involutions and the relative sizes of the groups. From this input we construct $S$ and $T$ matrices which satisfy the modular relations and whose Verlinde coefficients are non-negative integers. We also check certain restrictions coming from Frobenius-Schur indicators.

These families generalize Evans and Gannon's conjectures for the modular data associated to generalized Haagerup and near-group categories for odd groups, and include the modular data of the Drinfeld centers of almost all known quadratic categories. In addition to the subfamilies conjecturally realized by centers of quadratic categories, these families include many examples of potential modular data which do not correspond to known types of modular tensor categories.
\end{abstract}

\maketitle

\section{Introduction} 
In this paper we construct several infinite families of potential modular data, motivated by examples of quadratic categories which appeared in the classification of small-index subfactors. Modular data are numerical invariants of modular tensor categories, expressed as a pair of matrices $S$ and $T$ which are the images of the canonical generators $$ \left(  \begin{array}{cc} 0 & -1 \\ 1 & 0 \end{array} \right ) \text{ and } \left(  \begin{array}{cc} 1 & 1 \\ 0 & 1 \end{array} \right )$$ in a 
projective unitary representation of the modular group $SL_2(\Z) $ which is associated to the category. The matrix $S$ is symmetric, the matrix $T$ is diagonal with finite order, and $S^2$ is a permutation matrix commuting with $T$. The $S$ and $T$ matrices are indexed by the simple objects of the category, with the Verlinde coefficients 
\begin{equation}\label{verlinde}
N^k_{ij}=\sum_{r} \limits \displaystyle\frac{S_{i,r}S_{j,r}S_{\overline{k},r}}{S_{0,r}}
\end{equation}
giving the fusion rules $$\text{dim}(\text{Hom}(X_i\otimes X_j, X_k))=N^k_{ij}$$
(see \cite{MR954762,MR1153682,MR1797619}).

Modular data has proven to be a useful invariant of modular tensor categories, with the first examples of distinct modular categories sharing the same modular data discovered only recently \cite{1708.02796}. It is therefore natural to consider construction/classification of modular data as a first step towards construction/classification of modular tensor categories. The properties enjoyed by modular data - in particular non-negative integrality of the Verlinde coefficients - place severe restrictions on the matrices involved, and the modular data in turn encode a wealth of information about any categories which may realize them.

Modular tensor categories first appeared in conformal field theory \cite{MR1153682}, and play an important role in quantum topology \cite{MR1186845}.
The category of representations of a rational vertex operator algebra is a modular tensor category \cite{MR2140309}, and it is a major open problem whether every modular tensor category can be realized this way. Modular tensor categories also arise as categories of representations of quantum groups at roots of unity, and it is also a major open problem whether such categories generate the Witt group of unitary modular tensor categories \cite{MR3039775}.

These two problems reflect the paucity of known examples and constructions of modular tensor categories. Besides quantum groups, one source of examples is the Drinfeld center construction. If $\cC $ is a spherical fusion category, its Drinfeld center $\cZ(\cC) $ is a modular tensor category (by definition belonging to the trivial Witt class). 
A number of examples of fusion categories not arising from representations of groups or quantum groups have been discovered through the classification of small-index subfactors (see \cite{MR1317352,MR3166042,1509.00038}). With one exception (the Extended Haagerup subfactor \cite{MR2979509}), these examples are all related to quadratic categories.

A quadratic category is a fusion category $ \cC$ whose set of simple objects has a unique non-trivial orbit under the tensor product action of the group of invertible objects $\text{Inv}(\cC) $. (Such categories are also sometimes called \textit{generalized near-group categories} \cite{MR3078486}.) Given a quadratic category, one can consider the stabilizer subgroup of $\text{Inv}(\cC) $ for a given non-invertible simple object $X$. If this stabilizer subgroup is all of $\text{Inv}(\cC) $, then $X$ is the only non-invertible simple object, and the category is called a near-group category \cite{MR1997336}. At the other extreme, if the stabilizer subgroup is trivial, then the number of non-invertible simple objects is equal to  $|\text{Inv}(\cC) |$. Examples of such categories are the generalized Haagerup categories introduced in \cite{MR1832764}.

A general method for constructing various types of quadratic categories from endomorphisms of Cuntz algebras was developed by the second named author in \cite{MR1228532,MR1832764,MR3635673,MR3827808}. This construction was used to describe the tube algebra and thereby compute the modular data of the Drinfeld center for a number of examples, including the Haagerup subfactor \cite{MR1832764}.

Evans and Gannon observed that the modular data for the Drinfeld center of the Haagerup categories, as well for the centers of other generalized Haagerup categories associated to odd groups, can be viewed as composed of two distinct blocks ``grafted'' together \cite{MR2837122}. They gave a definition of grafting, noting that the concept could be ``massively generalized." In a subsequent paper, they considered modular data for near-group categories, and conjectured a remarkably simple formula for the modular data for a certain type of near-group category when the associated group is odd \cite{MR3167494}. 

Their formula for odd near-group modular data is expressed entirely in terms a pair of metric groups. A metric group is a finite abelian group equipped with a non-degenerate quadratic form. The category of metric groups is equivalent to the category of non-degenerate braided pointed fusion categories \cite[Section 8.4]{MR3242743}. By \cite[Theorem 12.9]{MR3635673}, generalized Haagerup categories for odd groups are de-equivariantizations of near-group categories, and their modular data could also be determined from this fornula (if the conjecture is correct). 

In this paper, we study a number of generalizations of the Evans-Gannon conjectures motivated by examples of generalized Haagerup categories for even groups and their (de)-equivariantizations, which are discussed in the accompanying paper \cite{GI19_1}. A new ingredient is the introduction of involutions on the metric groups. For even groups, one natural involution is the inverse operation, whose fixed points form the subgroup of order two elements; but other involutions are of interest as well.

The families of potential modular data which we construct from the pairs of metric groups are infinite, but we do not address the difficult question of whether such modular data actually arise from modular tensor categories, beyond known examples. For certain choices of input data (i.e. particular types of groups equipped with particular quadratic forms), we conjecture that the resulting modular data are realized by Drinfeld centers of infinite families of quadratic fusion categories. But it is not known whether the corresponding infinite families of quadratic categories exist. There are also many other choices of input data which do not correspond to any  known modular tensor categories. 

The paper is organized as follows. After this introduction and a section of preliminaries on modular data and metric groups, there are five chapters, each dealing with a different family of potential modular data. In each case, the starting point is a pair of (involutive) metric groups with Gauss sums differing by a sign, and some conditions on the involutive fixed points. From this input, potential modular data are constructed and the modular relations are checked.  Then the Verlinde coefficients are computed, and non-negative integrality gives constraints on the difference in size of the groups. We also check consistency with certain conditions for Frobenius-Schur indicators of modular data, which in some cases give further restrictions on the admissible groups and quadratic forms. 

Then we consider examples of this potential modular data for small groups and various quadratic forms and compare to known examples. In some cases, we can formulate conjectures for the realization of the modular data based on the known examples.

Briefly, the five families are as follows.

\begin{enumerate}
\item The first family (Section 3) generalizes the Evans-Gannon odd near-group conjecture to an arbitrary finite abelian group. Here the input is a pair of involutive metric groups $G $ and $\Gamma $ which are assumed to have isomorphic involutive fixed point subgroups. For the Drinfeld center of a near group category with $A=\text{Inv}(\cC) $ and multiplicity $|A| $, $G$ is conjectured to be $A  \times A$, where the involution is the flip, with $\Gamma $ having order $ |A|(|A|+4)$. But there are potentially other examples where $G$ is not of the form $A \times A $ but rather some other group; such examples would not necessarily correspond to centers of anything (this phenomenon is already present in the odd case). 

\item For the second family (Section 4), the involution on each metric group is the inverse map. It is assumed that the fixed point subgroups (consisting of the order two elements) are both $\Z_2 $, and there are a couple of other assumptions. Here the integrality of coefficients requires that $|\Gamma|=|G|+2$. We conjecture that examples of modular categories realizing this family are given by Drinfeld centers of $\Z_2 $-de-equivariantizations of near-group categories for $\Z_{2^{n+1}}  \times A$ (with one known instance). 

\item  For the third family (Section 5), it is assumed that the involutive fixed point subgroups both have order $4$, but that the quadratic forms only agree on a common order $2$ subgroup (and some other assumptions). Here the integrality of coefficients requires that $|\Gamma|=|G|+4$. There are several cases to consider here, according to whether the nontrivial invertible object is a boson or a fermion, and some other conditions. We conjecture that examples of modular tensor categories realizing this family are given by Drinfeld centers of generalized Haagerup categories for $A$ with $|A_2|=2$, and $ G =A \times \hat{A} $ (with two known instances, and numerical evidence for other instances); or more generally by Drinfeld centers of quadratic categories with two non-invertible simple objects which are stabilized by an odd group of invertible objects. 

\item For the fourth family (Section 6), the group of order two elements of $G$ has size $4$, and $ |\Gamma|=|G|+1$. This family unifies the modular data of the Drinfeld centers of the Asaeda-Haagerup categories (which come from a $\Z_2 $-de-equivariantization of a generalized Haagerup category for $\Z_4 \times \Z_2 $ \cite{MR3859276}), the generalized Haagerup category for $ \Z_2 \times \Z_2$, and the $\Z_2$-de-equivariantizations of generalized Haagerup categories for $\Z_4 $ and $\Z_8 $. These four types of modular data correspond to different forms of the even part of $G$ and different quadratic forms, and can all be generalized to infinite subfamilies.

\item The fifth family (Section 7) comes from the fourth family by switching the roles of $G$ and $\Gamma $. We conjecture that examples of this family can be realized by Drinfeld centers of near-group categories with multiplicity twice the order of the group.

\end{enumerate}

In \cite{MR3641612}, a construction called zesting is introduced whereby a spin-modular category (a modular category containing a fermion) can be twisted to give $8$ different modular closures of its supermodular subcategory (out of a total of $16$). A formula for the zested modular data is given there. 

 For a group of the form $\Z_2 \times A$, with $|A|$ odd, we have conjectural modular data for the centers of both generalized Haagerup categories and near-group categories which contain fermions. In both cases the metric group $G$ used to construct the modular data has even part $\Z_2 \times \Z_2 $. It turns out that in each case, the eight zested modular data belong to the same family where $\Z_2 \times \Z_2 $ is possibly replaced by $\Z_4 $ and the quadratic forms vary (see Remarks \ref{ze2} and \ref{ze1}.)
 
 The verification of the modular relations and the computations of the Verlinde coefficients and Frobenius-Schur indicators for the various families are lengthy and tedious, so we defer them to an online appendix, which is included in \texttt{arxiv} version of the paper. There is a separate section of the appendix for the calculations in each of Sections 3-6 of the paper, which correspond to the first four families (the fifth family is related to the fourth family so new calculations are not needed.)
 The appendix also contains a computation of modular data for Drinfeld centers of $\Z_2 $-de-equivariantizations of $\Z_4 $-near-group categories, verifying an instance of Conjecture \ref{con4}.

We also include with the \texttt{arxiv} source the Mathematica notebook \texttt{ifpmd.nb} which constructs the modular data for one example from each family, verifies the modular relations, and computes the Verlinde coefficients and Frobenius-Schur indicators. The notebook is annotated so that the reader can modify the input (which consists of specifying the pair of involutive metric groups and some associated data, depending on the family) to look at modular data for other examples.

\textbf{Acknowledgements.} We would like to thank the American Institute of Mathematics for its generous hospitality during the SQuaRE Classifying Fusion Categories and the Isaac Newton Institute for its generous hospitality during the semester program Operator algebras: Subfactors and their Applications. We would like to thank Eric Rowell for pointing out that all sixteen rank $10$ modular data related to the center of the near-group category for $\Z_2 $ with multiplicity $2$ can be realized though zesting and complex conjugation (see Section \ref{near1}), and we would like to thank Marcel Bischoff for first pointing out the realization of the $c=-1$ case. We would like to thank Terry Gannon for showing us his further generalizations of grafting in \cite{G18}. We would like to thank Andrew Schopieray for helpful comments on the preprint.

\section{Preliminaries}
\subsection{Modular data}\label{md}
Throughout the paper, we use the notation $$\T=\{z\in \C;\;|z|=1\}.$$
For $n\in \N$, we use the notation $\zeta_n=e^{\frac{2\pi i}{n}}$. 

Let $\cC$ be a unitary modular tensor category with $J$ the set of isomorphism classes of simple objects in $\cC$ (for the definition and basic properties of modular tensor categories, see \cite{MR1797619} or \cite{MR1990929}). 
The modular data of $(S,T)$ of $\cC$ are a pair of $J$ by $J$ matrices. 
Since there are several different normalizations for $(S,T)$ in the literature 
(see \cite{MR3641612} and \cite{MR2164398} for example), we first fix ours.  
We assume that $S$ is a symmetric unitary matrix satisfying $S_{0,j}>0$ for every $j\in J$ and 
$S^2=C$ with $C_{ij}=\delta_{i,\overline{j}}$, and 
$T$ is a unitary diagonal matrix of finite order with $T_{0,0}=1$ and $CT=TC$, 
where $\overline{j}$ is the dual object of $j\in J$, and $0$ is the unit object.  
We are mainly concerned with the following known constraints for the modular data $(S,T)$ (see \cite{MR3641612}, \cite{MR2164398}, and \cite{MR2313527} for example).
\begin{enumerate}
\item  $(ST)^3=cC$, where $c\in \T$ is given by 
$$c=\frac{1}{S_{00}}\sum_{i\in J}S_{0i}^2T_{ii}.$$
\item  For $i,j,k\in J$, let 
$$N_{ijk}=\sum_{a\in J}\frac{S_{ai}S_{aj}S_{a,k}}{S_{0,a}}.$$ 
Then $N_{ij}^k=N_{ij\overline{k}}$ is the fusion coefficient $\dim \Hom(k,i\otimes j)$. 
In particular $N_{ijk}$ is a non-negative integer. \\
\item For $k\in J$ and $n\in \N$, let 
$$\nu_n(k)=\sum_{i,j\in J}N_{ij}^k S_{0i}S_{0j}(\frac{T_{jj}}{T_{ii}})^n.$$
Then $\nu_n(k)$ is the $n$-th Frobenius-Schur indicator of $k$. 
\end{enumerate}
In particular, 
\begin{itemize}
\item[(FS2)] We have $\nu_2(k)=0$ if $k\neq\overline{k}$, and $\nu_2(k)\in \{1,-1\}$ if $k=\overline{k}$. 
\item[(FS3)] The third indicator $\nu_3(k)$ is a sum of $N_{kkk}$ numbers belonging to $\{1,\zeta_3,\zeta_3^2\}$.   
\end{itemize}

In this paper we propose a number of potential modular data: pairs of unitary matrices $(S,T)$
satisfying the above conditions, without knowing whether the underlying modular tensor categories exist. 
More precisely, we take the above formulae (2) and (3) as the definitions of $N_{ijk}$ and $\nu_{n}(k)$ 
for a given $(S,T)$ satisfying (1), and check whether these numbers satisfy the conditions in (2) and (3). 
Although there are other number theoretic constraints known for $(S,T)$, we do not discuss them as 
they suit for case-by-case analysis and it is relatively easy to apply them to our examples. 

\subsection{Metric groups}

Every finite  abelian group $G$ is canonically decomposed as 
$G=G_e\times G_o$, where $G_e$ is a 2-group and $G_o$ is an odd group. 
We set $$G_2=\{g\in G;\; 2g=0\},$$ which is a subgroup of $G_e$.

For a quadratic form $q$ on a finite abelian group $G$, we adopt multiplicative notation. 
Namely, a quadratic form $q$ on $G$ is a map $q:G\to \T=\{z\in \C;\; |z|=1\}$  
satisfying $q(-g)=q(g)$ for all $g\in G$ 
such that the map $\inpr{\cdot}{\cdot}_q:G\times G\to \T$ defined by $$\inpr{g}{h}_q=q(g+h)q(g)^{-1}q(h)^{-1}$$ 
is a bicharacter. 
We say the $q$ is non-degenerate if the associated bicharacter is non-degenerate. 
We often suppress $q$ in $\inpr{\cdot}{\cdot}_q$ if there is no possibility of confusion. 

We define the Gauss sum for non-degenerate $q$ and $k\in \Z$ by 
$$\cG(q,k)=\frac{1}{\sqrt{|G|}}\sum_{g\in G}q(g)^k\in \C,$$
and we denote $\cG(q)=\cG(q,1)\in \T$.  
For a given non-degenerate symmetric bicharacter $\inpr{\cdot}{\cdot}$, we can always find a quadratic form $q$ 
giving $\inpr{\cdot}{\cdot}$ as above, and if moreover $G$ is an odd group, it is unique and of the form $q(g)=\inpr{g}{g}^{1/2}$. 
For the classification of non-degenerate quadratic forms, the reader is referred to \cite{MR3418743}. 
We often use the following facts. 
We have $\cG(q)\in \{1,-1\}$ if $|G|\equiv 1 \mod 4$ and 
$\cG(q)\in \{i,-i\}$ if $|G|\equiv 3 \mod 4$. 
For odd $r$ and integer $n\geq 1$, we have 
$$\frac{1}{\sqrt{2^n}}\sum_{x=0}^{2^n-1}\zeta_{2^{n+1}}^{rx^2}=(-1)^{n\frac{r^2-1}{8}}\zeta_8^r.$$

We say that a pair $(G,q)$ is a metric group if $G$ is a finite abelian group and $q$ is a non-degenerate quadratic 
form on $G$. 
For such a pair, we define unitary matrices 
$$S^{(G,q)}_{g,g'}=\frac{1}{\sqrt{|G|}}\overline{\inpr{g}{g'}_q},\quad T^{(G,q)}_{g,g'}=\delta_{g,g'}q(g),\quad C^{(G,q)}_{g,g'}=\delta_{g+g',0}.$$
Then they satisfy the relations 
$$({S^{(G,q)}})^2=C^{(G,q)}, \quad (S^{(G,q)}T^{(G,q)})^3=\cG(q)C^{(G,q)}, \quad C^{(G,q)}T^{(G,q)}=T^{(G,q)}C^{(G,q)}.$$
It is known that modular data of a pointed modular tensor category is always of the form 
$(S^{(G,q)},T^{(G,q)})$ (see \cite[Section 8.4]{MR3242743}). 
Every metric group $(G,q)$ is canonically decomposed as $(G,q)=(G_e,q_e)\times (G_o,q_o)$.

\begin{definition}
We say that $(G,q,\theta)$ is an involutive metric group if $(G,q)$ is a metric group 
and $\theta\in \Aut(G)$ is an involution (an order 2 automorphism of $G$) preserving $q$. 
For an involution $\theta\in \Aut(G)$, we set $$G^\theta=\{g\in G;\; \theta(g)=g\}.$$
\end{definition}

We recall the following easy but useful lemma. 

\begin{lemma}\label{split} Assume that $(G,q)$ is a metric group and $H$ is a subgroup of $G$. 
If the restriction of $q$ to $H$ is non-degenerate, 
we have $G=H\times H^\perp$ and $q=q|_H\times q|_{H^\perp}$, 
where 
$$H^\perp=\{g\in G;\; \forall h\in H,\;\inpr{g}{h}_q=1\}.$$
If moreover $\theta\in \Aut(G)$ is an involution preserving $q$ and $\theta(H)=H$, 
then we also have $\theta=\theta|_H\times \theta|_{H^\perp}$.  
\end{lemma}

\subsection{When $G^{\theta}=\Z_2\times \Z_2 $}
For later use we classify involutive metric groups whose fixed point subgroups are $\Z_2\times \Z_2$.  

\begin{lemma}\label{prelist} Let $(A,q,\theta)$ be an involutive metric group with a 2-group $A$. 
Assume $A^\theta=\{0,a\}\cong \Z_2$. \\
$(1)$ If $q(a)=\pm i$, we have $A=\Z_2$, $q(x)=i^{\pm x^2}$, and $\theta=-1$.\\
$(2)$ If $q(a)=-1$, one of the following holds: 
\begin{itemize}
\item $A=\Z_4$, $q(x)=\zeta_8^{rx^2}$ with odd $r$, and $\theta=-1$.
\item $A=\Z_2\times \Z_2$, $q$ is a flip invariant non-degenerate quadratic form, and $\theta(x,y)=(y,x)$. 
\end{itemize}
$(3)$ If $q(a)=1$, we have $A=\Z_{2^n}$ with $n\geq 3$, $q(x)=\zeta_{2^{n+1}}^{rx^2}$ with odd $r$, and $\theta=-1$. 
\end{lemma}

\begin{proof} If $q(a)=\pm i$, the restriction of $q$ to $B=\{0,a\}$ is non-degenerate. 
Thus we have the factorization $A=B\times B^\perp$ as an involutive metric group. 
If $B^\perp$ is not trivial, the involution $\theta$ has a non-trivial fixed point as $B^\perp$ is an even group, which is a contradiction. 
Thus (1) holds. 

Let $\varphi:A\to A^\theta$ be a group homomorphism given by $\varphi(x)=x+\theta(x)$, and let $\varphi_0$ be the restriction of $\varphi$ to $A_2$. 
Then $\ker \varphi_0=A^\theta$, and $|A_2|/|A^\theta|=|\varphi_0(A)|\leq 2$, which shows $|A_2|\leq 4$. 
Thus the 2-rank of $A$ is at most 2. 

If the 2-rank of $A$ is 1, the group $A$ is cyclic, say $\Z_{2^n}$, and we have $q(x)=\zeta_{2^{n+1}}^{rx^2}$ with odd $r$ and $a=2^{n-1}$. 
Since $\theta$ preserves $q$, there exists an integer $s$ with $s^2-1\equiv 0 \mod 2^{n+1}$ with $\theta(x)=sx$. 
Since $A^\theta=\{0,2^{n-1}\}$, the number $(s-1)/2$ is odd and $s+1\equiv 0\mod 2^n$, which implies $\theta=-1$. 

Assume that the 2-rank of $A$ is 2. 
Then $|A_2|/|A^\theta|=2$, and $\varphi_0$ is a surjection.  
Thus there exists $b\in A_2$ satisfying $b+\varphi(b)=a$. 
If $q(a)=-1$, we have $A_2=\{0,a,b,\theta(b)\}$, and $\{q(a),q(b),q(\theta(a))\}=\{-1,q(b),q(b)\}$. 
This shows that $q$ restricted to $A_2$ is non-degenerate, and we get $A=A_2$ as before. 
Assume $q(a)=1$ now.  
Since $|A|=2|\ker \varphi|$ and $\ker \varphi \cap \{0,b\}=\{0\}$, we have the decomposition $A=\ker \varphi\times \{0,b\}$. 
Let $x\in \ker \varphi$. 
Then we have $q(x+b)=q(x)q(b)\inpr{x}{b}$, and on the other hand, we have 
\begin{align*}
\lefteqn{q(\theta(x+b))=q(-x+a+b)} \\
 &= q(-x+a)q(b)\inpr{-x+a}{b}=q(x)q(a)q(b)\inpr{-x}{a}\inpr{-x+a}{b}\\
 &=q(x)q(b)\inpr{x}{a}\inpr{x}{b}\inpr{a}{b}.
\end{align*}
We have 
$$\inpr{a}{b}=\frac{q(a+b)}{q(a)q(b)}=\frac{q(\theta(b))}{q(b)}=1.$$
Since $q(\theta(x+b))=q(x+b)$, we get $\inpr{x}{a}=1$, and we also have $\inpr{a}{b}=1$. 
Since $A$ is generated by $\ker \varphi$ and $b$ this means that $q$ is degenerate, which is a contradiction. 
\end{proof}

\begin{theorem}\label{list} Let $(A,q,\theta)$ be an involutive metric group with a 2-group $A$. 
Assume $A^\theta=\{0,a_0,a_1,a_2\}\cong \Z_2\times \Z_2$. 
Then the following hold up to group automorphism.\\
$(1)$ If $q(a_0)=-1$, $q(a_1)=q(a_2)=1$, we have $A=\Z_2^2$, $q(x,y)=(-1)^{xy}$, and $\theta=-1$. \\
$(2)$ If $q(a_0)=q(a_1)=q(a_2)=-1$, we have $A=\Z_2^2$, $q(x,y)=(-1)^{x^2+xy+y^2}$, and $\theta=-1$.\\ 
$(3)$ If $q(a_0)=-1$, $q(a_1)=q(a_2)=\pm i$, we have $A=\Z_2^2$, $q(x,y)=i^{\pm (x^2+y^2)}$, and $\theta=-1$.\\  
$(4)$ If $q(a_0)=1$, $q(a_1)=i$, $q(a_2)=-i$, we have $A=\Z_2^2$, $q(x,y)=i^{x^2-y^2}$, and $\theta=-1$.\\ 
$(5)$ If $q(a_0)=-1$, $q(a_1)=i$, and $q(a_2)=-i$, one of the following holds: 
\begin{itemize}
\item $A=\Z_2\times \Z_4$, $q(x,y)=\zeta_4^{r_1x^2}\zeta_8^{r_2y^2}$ with odd $r_1,r_2$, and $\theta=-1$. 
\item $A=\Z_2\times \Z_2\times \Z_2$, $q(x,y,z)=\zeta_4^{r_1x^2}q'(y,z)$ with odd $r$, and $\theta(x,y,z)=(x,z,y)$, where 
$q'$ is a flip invariant non-degenerate quadratic form on $\Z_2\times \Z_2$. 
\end{itemize}
$(6)$ If $q(a_0)=1$, $q(a_1)=-1$, $q(a_2)=-1$, one of the following holds: 
\begin{itemize}
\item $A=\Z_4\times \Z_{2^m}$ with $m\geq 2$, $q(x,y)=\zeta_8^{r_1x^2}\zeta_{2^{m+1}}^{r_2y^2}$ with odd $r_1,r_2$, and $\theta=-1$ for $n=2$, 
and $\theta=-1$ or $\theta(x,y)=(-x+2y,2^{n-1}x+(-1+2^{n-1})y)$ for $n\geq 3$.   
\item $A=\Z_2\times \Z_2\times \Z_{2^m}$ with $m\geq 2$, $q(x,y,z)=q'(x,y)\zeta_{2^{m+1}}^{rz^2}$ with odd $r$, and $\theta(x,y,z)=(y,x,-z)$, 
where $q'$ is a flip invariant non-degenerate quadratic form on $\Z_2\times \Z_2$. 
\item $A=\Z_2^4$, $q(x_1,x_2,x_3,x_4)=q'(x_1,x_3) q''(x_2,x_4)$, where $q'$ and $q''$ are flip invariant non-degenerate 
quadratic forms on $\Z_2\times \Z_2$, and $\theta(x_1,x_2,x_3,x_4)=(x_3,x_4,x_1,x_2)$. 

\end{itemize}
$(7)$ If $q(a_0)=1$, $q(a_1)=q(a_2)=\pm i$, we have $A=\Z_2\times \Z_{2^n}$ with $n\geq 3$, $q(x,y)=i^{\pm x^2}\zeta_{2^{n+1}}^{ry^2}$, and $\theta=-1$.\\ 
$(8)$ If $q(a_0)=q(a_1)=q(a_2)=1$, one of the following holds: 
\begin{itemize}
\item $A=\Z_{2^n}^2$ with $n\geq 2$, $q(x,y)=\zeta_{2^n}^{xy}$ or $q(x,y)=\zeta_{2^n}^{x^2+xy+y^2}$, and $\theta=-1$ for $n=2$, and $\theta=-1$ or 
$\theta=-1+2^{n-1}$ for $n\geq 3$. 
\item $A=\Z_{2^m}\times \Z_{2^n}$ with $3\leq m\leq n$, $q(x,y)=\zeta_{2^{m+1}}^{r_1x^2}\zeta_{2^{n+1}}^{r_2y^2}$ with odd $r_1,r_2$, and $\theta=-1$ 
or $\theta(x,y)=(-x+2^{m-1}y,2^{n-1}x-y)$. 
\item $A=\Z_2^4$, $q(x_1,x_2,x_3,x_4)=(-1)^{x_1x_4+x_2x_3}$ or $q(x_1,x_2,x_3,x_4)=(-1)^{x_1x_4+x_2x_3}i^{(x_1+x_3)^2}$, 
and $\theta(x_1,x_2,x_3,x_4)=(x_3,x_4,x_1,x_2)$. 
\end{itemize}
\end{theorem}

\begin{proof}
In cases (1)-(4), the restriction of $q$ to $\{0,a_0,a_1,a_2\}$ is non-degenerate, and we get the statement. 
If $a_1=\pm i$, the restriction of $q$ to $B=\{0,a_1\}$ is non-degenerate, and we get $A=B\times B^\perp$ 
as an involutive metric group. 
Thus the restriction of $\theta$ to $B^\perp$ has exactly one non-trivial fixed point, and the statement follows from the previous lemma. 
The remaining cases are (6) and (8). 

Let $\varphi:A\to A^{\theta}$ be a group homomorphism given by $\varphi(x)=x+\theta(x)$, and let $\varphi_0$ be the restriction of $\varphi$ 
to $A_2$. 
Then $\ker \varphi_0=A^{\theta}$ and $|A_2|/|A^\theta|=|\varphi_0(A_2)|\leq 4$. 
Thus $|A_2|\leq 16$, and the 2-rank of $A$ is at most 4. 
When it is 4, the map $\varphi_0$ is surjection, and when it is $3$, we get $|\varphi_0(A_2)|=2$. 
If it is 2, we have $A_2=A^\theta$, and the possible structure of the metric group $(A,q)$ is determined from the classification of 
metric groups. 

(6) 
Assume that the 2-rank of $A$ is 2. 
Then we may assume that $A=\Z_4\times \Z_{2^n}$ with $n\geq 2$ and $q(x,y)=\zeta_8^{r_1x^2}\zeta_{2^{n+1}}^{r_2y^2}$ with odd $r_1,r_2$. 
Since $\varphi(A)\subset A_2$, the involution $\theta$ is of the form $\theta(x,y)=(\varepsilon x+2sy,2^{n-1}tx+(-1+2^{n-1}u)y)$ 
with $\varepsilon\in \{1,-1\}$, $s,t,u\in \{0,1\}$. 
Since $\theta$ preserves $q$ and $A^\theta=A_2$, we get the statement. 

Assume that the 2-rank of $A$ is 4. Then there exists $b\in A_2$ satisfying $b+\theta(b)=a_1$. 
Let $B=\{0,a_1,b,\theta(b)\}\cong \Z_2\times \Z_2$. 
Since $\{q(a_1),q(b),q(b)\}=\{-1,q(b),q(b)\}$, the restriction of $q$ to $B$ is non-degenerate, and we get the factorization 
$A=B\times B^\perp$ as an involutive metric group. 
Now the statement follows from the previous lemma.

Now assume that the 2-rank of $A$ is 3. 
If $\varphi_0(A_2)=\{0,a_1\}$ or $\{0,a_2\}$, the statement follows from the same argument as above, 
and we assume $\varphi_0(A)=\{0,a_0\}$. 
If there exists $d\in \ker\varphi$ satisfying $2d=a_1$, the restriction of $q$ to $B=\langle d\rangle\cong \Z_4$ 
is non-degenerate as $q(2d)=-1$, and $B$ is globally preserved by $\theta$. 
Thus $A=B\times B^\perp$ as an involutive metric group, and we get the statement. 
Therefore we assume that neither $a_1$ nor $a_2$ has a square root in $\ker\varphi$. 
We choose $b\in A_2$ satisfying $b+\theta(b)=a_0$. 
Then $A_2$ is generated by $a_0$, $a_1$, and $b$. 
For $\varphi(A)$, we have two possibilities: $\{0,a_0\}$ and $\{0,a_0,a_1,a_2\}$. 

First we assume $\varphi(A)=\{0,a_0\}$. 
Then $|A|=2|\ker\varphi|$. 
Since $\ker\varphi\cap \{0,b\}=\{0\}$, we get $A=\ker \varphi\times \{0,b\}$. 
Let $x\in \ker \varphi$. 
Then $q(x+b)=\inpr{x}{b}q(x)q(b)$. 
On the other hand, 
\begin{align*}
\lefteqn{q(\theta(x+b))=q(-x+a_0+b)} \\
 &=\inpr{-x+a_0}{b}q(-x+a_0)q(b) \\
 &=\inpr{x}{b}\inpr{a_0}{b}\inpr{-x}{a_0}q(a_0)q(b) \\
 &=\inpr{x}{b}\inpr{a_0}{b}\inpr{x}{a_0}q(b)
\end{align*}
We can show that $\inpr{a_0}{b}=1$ as before, and since $\theta$ preserves $q$, we get $\inpr{x}{a_0}=1$. 
Since $A$ is generated by $\ker\varphi$ and $b$, we see that $q$ is degenerate, which is a contradiction. 
Thus $\varphi(A)=\{0,a_0\}$ does not occur. 

Assume $\varphi(A)=A^{\theta}$. 
Then there exists $c\in A\setminus A_2$ satisfying $c+\theta(c)=a_1$. 
Note that we have $2c\in \ker\varphi\setminus\{0\}$, and $A$ is generated by $\ker\varphi$, $b$, and $c$. 
Let $\psi:A\to \ker \varphi$ be a group homomorphism given by $\psi(x)=x-\theta(x)$. 
Then $\ker\psi=A^\theta$. 
Thus $|\psi(A)|=|A|/4$, and we also have $|\ker\varphi|=|A|/|\varphi(A)|=|A|/4$. 
Thus $\psi$ is a surjection from $A$ onto $\ker\varphi$. 
We have $\psi(\ker\varphi)=2\ker\varphi$, $\psi(b)=a_0$, and $\psi(c)=2c-a_1$. 
Since $|\ker \varphi \cap A_2|=|A^\theta|=4$, the 2-rank of $\ker \varphi$ is 2, and 
we get $\ker \varphi=\langle a_0\rangle\times \langle 2c+a_1\rangle$.  
We express $a_1$ as $ma_0+n(2c+a_1)$ with integers $m,n$. 
Then $a_2=(m+1)a_0+n(2c+a_1)$. 
Since neither $a_1$ nor $a_2$ has a square root in $\ker\varphi$, the number $n$ is odd, and we also have 
$\ker\varphi=\langle a_0\rangle\times \langle a_1\rangle\cong \Z_2\times \Z_2$. 
Thus we have $2c+a_1=a_2$ as $2c\neq 0$, and $2c=a_0$. 
This means that we have $A=\langle a_1\rangle\times \langle b\rangle\times \langle c\rangle\cong \Z_2\times \Z_2\times \Z_4$ with 
$\theta(a_1)=a_1$, $\theta(b)=b+a_0=b+2c$, and $\theta(c)=a_1-c$. 
We express $xa_1+yb+zc$ as $(x,y,z)$. 
Then since $q(a_1)=-1$, we have 
$q(x,y,z)=(-1)^xi^{r_1y^2}\zeta_8^{r_2z^2}(-1)^{sxy+txz+uyz}$ with $r_1\in \{0,1\}$, $r_2\in \{0,\pm 1,\pm 2,\pm 3\}$, and $s,t,u\in \{0,1\}$. 
Since $q(a_0)=1$ and $q(a_2)=-1$, we have $1=q(0,0,2)=(-1)^{r_2}$, and $-1=q(1,0,2)=-(-1)^{r_2}$, which shows that $r_2$ is even. 
However this implies $q(x,y,z+2)=q(x,y,z)$, and $q$ is degenerate, which is a contradiction. 

(8) If the 2-rank of $A$ is 2, we can show the statement as in (6). 

Assume the 2-rank of $A$ is 4. 
Then we have $\varphi_0(A_2)=A^\theta$. 
Thus we can choose $b_0,b_1\in A_2$ satisfying $b_0+\theta(b_0)=a_0$ and $b_1+\theta(b_1)=a_1$. 
Then $b_2=b_0+b_1$ satisfies $b_2+\theta(b_2)=a_2$. 
Since $|\ker \varphi|=|A|/4$ and $\ker\varphi\cap \{0,b_0,b_1,b_2\}=\{0\}$, we have a factorization 
$A=\ker\varphi\times \{0,b_0,b_1,b_2\}$. 
Since $\ker\varphi \cap A_2=A^\theta$, the 2-rank of $\ker \varphi$ is 2. 
Thus we may assume that $\ker \varphi =\Z_{2^m}\times \Z_{2^n}$ with $1\leq m\leq n$ and $a_0=(2^{m-1},0)$ and $a_1=(0,2^{n-1})$. 
Identifying $\{0,b_0,b_1,b_2\}$ with $\langle b_0\rangle\times \langle b_1\rangle$, we are in the following situation: 
$A=\Z_{2^m}\times \Z_{2^n}\times \Z_2\times \Z_2$ and
$$\theta(x_1,x_2,x_3,x_4)=(-x_1+2^{m-1}x_3,-x_2+2^{n-1}x_4,x_3,x_4).$$
If $n\geq2$, we have $(0,2^{n-1},0,1)\in A^\theta$, which is a contradiction. 
Thus $m=n=2$. 
Since $q(1,0,0,0)=q(0,1,0,0)=q(1,1,0,0)=1$, we have
$$q(x_1,x_2,x_3,x_4)=i^{r_1x_3^2+r_2x_4^2}(-1)^{s_{13}x_1x_3+s_{14}x_1x_4+s_{23}x_2x_3+s_{24}x_2x_4+s_{34}x_3x_4},$$
with $r_1,r_2\in \{0,\pm 1,2\}$ and $s_{13},s_{14},s_{23},s_{24},s_{34}\in \{0,1\}$. 
Since $q$ is preserved by $\theta$, we get $s_{13}=s_{24}=0$ and $s_{14}=s_{23}$, and then
$$q(x_1,x_2,x_3,x_4)=i^{r_1x_3^2+r_2x_4^2}(-1)^{s x_3x_4} (-1)^{t(x_1x_4+x_2x_3)},$$
with $s,t\in \{0,1\}$. 
The Gauss sum is 
\begin{align*}
 &\cG(q)=\frac{1}{4}\sum_{x_3,x_4}i^{r_1x_3^2+r_2x_4^2}(-1)^{s x_3x_4}(1+(-1)^{tx_3})(1+(-1)^{tx_4}) \\
 &=\left\{
\begin{array}{ll}
1+i^{r_1}+i^{r_2}+i^{r_1+r_2}(-1)^s , &\quad t=0 \\
1 , &\quad t=1
\end{array}
\right.,
\end{align*}
which shows that $q$ is non-degenerate if and only if $t=1$. 
The centralizer of $\theta$ in $\Aut(A)=GL(4,2)$ is 
$$C_{\Aut(A)}(\theta)=\{\left(
\begin{array}{cc}
X &Y  \\
0 &X 
\end{array}
\right)\in GL(4,2);\; X\in GL(2,2)\},$$
and there are exactly two $C_{\Aut(A)}(\theta)$-orbits in the non-degenerate quadratic forms preserved by $\theta$. 
The two orbits are represented by $q(x_1,x_2,x_3,x_4)=(-1)^{x_1x_4+x_2x_3}$ and $q(x_1,x_2,x_3,x_4)=i^{x_3^2}(-1)^{x_1x_4+x_2x_3}$. 
Expressing $q$ and $\theta$ with respect to the  basis 
$$\{(0,0,1,0),(0,0,0,1),(1,0,1,0),(0,1,0,1)\},$$
we get the statement. 

Assume that the 2-rank of $A$ is 3. 
Since $|\varphi_0(A_2)|=2$, we may assume that there exists $b\in A_2$ satisfying $b+\theta(b)=a_0$. 
We can show that $|\varphi(A)|=2$ does not occur as in (6), and we assume $\varphi(A)=A^\theta$. 
Then there exists $c\in A\setminus A_2$ satisfying $c+\theta(c)=a_1$, $2c\in \ker \varphi$, and $A$ 
is generated by $\ker\varphi$, $b$, and $c$.  
As before we have $\psi(A)=\ker \varphi$, and $\ker \varphi$ is generated by 
$\psi(\ker\varphi)=2\ker\varphi$, $\psi(b)=a_0$, and $\psi(c)=2c+a_1$. 
Since $\ker\varphi\cap A_2=A^\theta$, the 2-rank of $\ker \varphi$ is 2, and we get 
$\ker\varphi=\langle  b\rangle\times \langle 2c+a_2\rangle\cong \Z_2\times \Z_{2^m}$. 
We can show that $m=1$ does not occur as in (6), and we assume $m\geq 2$.
Then we have either $a_1=2^{m-1}(2c+a_1)=2^mc$ or $a_1=a_0+2^{m-1}(2c+a_1)=a_0+2^mc$. 
In any case, we have $a_1\in \langle a_0,c\rangle$, and 
$$A=\langle a_0\rangle\times \langle b\rangle\times \langle c\rangle\cong \Z_2\times \Z_2\times \Z_{2^{m+1}}.$$
Since $q(1,0,0)=q(0,0,2^m)=(1,0,2^m)=1$, we get 
$$q(x,y,z)=i^{r_2y^2}\zeta_{2^{m+2}}^{r_3z^2}(-1)^{sx_1x_2+tx_1x_3+ux_2x_3},$$
with $r_2\in \{0,\pm 1,2\}$, $r_3\in \{0,\pm 1\pm 2,\pm 3\}$, and $s,t,u\in \{0,1\}$. 

Assume $a_1=2^mc$ first. 
Then since $\theta(a_0)=a_0$, $\theta(b)=b+a_0$, and $\theta(c)=-c+a_1=(2^m-1)c$, we have 
$\theta(x,y,z)=(x+y,y,(2^m-1)z)$. 
Since $\theta$ preserves $q$, we see that $r_3$ is even and $q(x,y,z)=q(x,y,z+2^m)$. 
This shows that $q$ is degenerate, and this case does not occur. 
Assume $a_1=a_0+2^mc$ now. 
Then we have $\theta(x,y,z)=(x+y+z,y,(2^m-1)z)$. 
Again we can see that $r_3$ is even and $q$ is degenerate. 
Thus this case does not occur either.  
\end{proof}
\subsection{Quadratic categories}
Let $P$ and $Q$ be finite abelian groups and let $m$ be a natural number. 
We say that a fusion category $\cC$ is a quadratic category of type $(P,Q,m)$ 
if the set of equivalence classes of simple objects in $\cC$ are represented by   
$$\{\alpha_{p}\otimes\beta_q\}_{p\in P,\;q\in Q}\cup \{\alpha_p\otimes\rho\}_{p\in P},$$
and they satisfy the following fusion rules, 
$$[\alpha_p][\alpha_{p'}]=[\alpha_{p+p'}],\quad [\beta_q][\beta_{q'}]=[\beta_{q+q'}],\quad 
[\alpha_p][\beta_q]=[\beta_q][\alpha_p],$$
$$[\alpha_p][\rho]=[\rho][\alpha_{-p}],\quad [\beta_q][\rho]=[\rho],$$
$$[\rho^2]=\sum_{q\in Q}[\beta_q]+m|Q|\sum_{p\in P} [\alpha_p][\rho].$$
Thanks to $[\beta_q][\rho]=[\rho]$, we can see that $\{\beta_q\}_{q\in Q}$ has a trivial associator. 
Irrational near-group categories are of type $(\{0\},Q,m)$ and generalized Haagerup categories are of type 
$(P,\{0\},1)$. 

We can also consider a non-self-dual version. 
Namely, we say that a fusion category $\cC$ is a quadratic category of type $(P,Q,m)$ with non-self-dual $\rho$ 
if the set of equivalence classes of simple objects in $\cC$ is as above, 
but satisfying the following fusion rules: 
$$[\alpha_p][\alpha_{p'}]=[\alpha_{p+p'}],\quad [\beta_q][\beta_{q'}]=[\beta_{q+q'}],\quad 
[\alpha_p][\beta_q]=[\beta_q][\alpha_p],$$
$$[\alpha_p][\rho]=[\rho][\alpha_{-p}],\quad [\beta_q][\rho]=[\rho],$$
$$[\rho\overline{\rho}]=\sum_{q\in Q}[\beta_q]+m|Q|\sum_{p\in P} [\alpha_p][\rho];$$
and there exists $p_0\in P$ such that 
$$[\brho]=[\alpha_{p_0}][\rho].$$
The $\Z_2$-de-equivariantizaiton of a near-group category for $\Z_{4m}$ with multiplicity $4m$ is a quadratic category 
of type $(\Z_2,\Z_m,1)$ with non-self-dual $\rho$ (\cite[Theorem 12.5]{MR3635673}). 

\section{Near-group family}

\subsection{General formulae} 
The modular data $(S,T)$ of the Drinfeld center $\cZ(\cC)$ of a near-group category $ \cC$ with a finite abelian group $A$ 
and multiplicity $|A|$ (i.e. a quadratic category of type $(\{0\},A,1)$ was computed 
in \cite{MR1832764} and \cite{MR3167494}. 
Evans-Gannon \cite[Conjecture 2]{MR3167494} conjectured an explicit formula of $(S,T)$ 
for odd $A$, and they verified the conjecture for $|A|\leq 13$. 
Their formula is surprising simple and it involves only two metric groups $(A,q)$ and $(A',q')$ satisfying 
$|A'|=|A|+4$ and $\cG(q')=-\cG(q)$. 
It is immediate to describe the first metric group $(A,q)$ in terms of $\cC$, 
while the second metric group $(A',q')$ is rather mysterious, and it is captured only through 
heavy computation. 
In this section, we generalize their formula to general finite abelian groups. 
Our family of potential modular data also include examples which are not associated to Drinfeld centers of fusion categories. 

Throughout this section, we assume that $(G,q_1,\theta_1)$ and $(\Gamma,q_2,\theta_2)$ are involutive metric groups 
satisfying $c:=\cG(q_1)=-\cG(q_2)$ and   
$$(U,q_0):=(G^{\theta_1},q_1|_{G^{\theta_1}})\cong (\Gamma^{\theta_2},q_2|_{\Gamma^{\theta_2}}).$$
We fix an isomorphism of $G^{\theta_1}$ and $\Gamma^{\theta_2}$ as above, and regard 
$U$ as a common subgroup of $G$ and $\Gamma$.

We set $a=1/\sqrt{|G|}$, $b=1/\sqrt{|\Gamma|}$. 
We choose and fix subsets $G_*\subset G$ and $\Gamma_*\subset \Gamma$ satisfying 
$$G=U\sqcup G_*\sqcup \theta_1(G_*),\quad \Gamma=U\sqcup \Gamma_*\sqcup \theta_2(\Gamma_*),$$
and set
$$J=(U\times \{0,\pi\})\sqcup G_*\sqcup \Gamma_*.$$
We use letters $u,u',u'',v$ for elements of $U$, $g,g',g'',h$ for elements of $G$ and $\gamma,\gamma',\gamma''
,\xi$ for elements of $\Gamma$. 
We introduce an involution of $J$ by setting $\overline{(u,0)}=(-u,0)$, $\overline{(u,\pi)}=(-u,\pi)$, and
$$\overline{g}=\left\{
\begin{array}{ll}
-g , &\quad \textrm{if $-g\in G_*$}  \\
\theta_1(-g) , &\quad \textrm{otherwise}
\end{array}
\right.,
$$
$$\overline{\gamma}=\left\{
\begin{array}{ll}
-\gamma , &\quad \textrm{if $-\gamma\in \Gamma_*$}  \\
\theta_2(-\gamma) , &\quad \textrm{otherwise}
\end{array}
\right..
$$
We often suppress the subscript $i$ in $\theta_i$ and $q_i$ in $\inpr{\cdot}{\cdot}_{q_i}$ 
if there is no possibility of confusion.

\begin{definition}\label{ST3}
Let $S$, $T$, and $C$ be $J$ by $J$ matrices defined by  
\begin{align*}
S=
&\left(
\begin{array}{cccc}
\frac{a-b}{2}\overline{\inpr{u}{u'}}& \frac{a+b}{2}\overline{\inpr{u}{u'}}&a\overline{\inpr{u}{g'}} &b\overline{\inpr{u}{\gamma'}}  \\
\frac{a+b}{2}\overline{\inpr{u}{u'}}& \frac{a-b}{2}\overline{\inpr{u}{u'}}&a\overline{\inpr{u}{g'}} &-b\overline{\inpr{u}{\gamma'}}  \\
a\overline{\inpr{g}{u'}}            &a\overline{\inpr{g}{u'}}             &a(\overline{\inpr{g}{g'}+\inpr{\theta(g)}{g'}})&0  \\
b\overline{\inpr{\gamma}{u'}}       &-b\overline{\inpr{\gamma}{u'}}       &0&-b(\overline{\inpr{\gamma}{\gamma'}+\inpr{\theta(\gamma)}{\gamma'}}) 
\end{array}
\right),
\end{align*}
$$T=\mathrm{Diag}(q_0(u),q_0(u),q_1(g),q_2(\gamma)),$$
and $C_{x,y}=\delta_{x,\overline{y}}$ (where the four blocks of $S$ and $T$ are indexed by $U \times \{0\}$, $U \times \{\pi \}$, $G_*$, and $\Gamma_*$, respectively.) 
\end{definition}

Lemma \ref{split} immediately implies the following. 

\begin{lemma}\label{factorization} 
Let the notation be as in Definition \ref{ST3}. 
Assume that $L$ is a subgroup of $U$ such that the restriction of $q_0$ to $L$ is non-degenerate, 
and let 
$$G'=\{g\in G;\; \forall l\in L\; \inpr{g}{l}=1\},$$
$$\Gamma'=\{\gamma\in \Gamma;\; \forall l\in L\; \inpr{\gamma}{l}=1\},$$
$q_1'=q_1|_{G'}$, $\theta_1'=\theta_1|_{G'}$, $q_2'=q_2|_{\Gamma'}$, $\theta_2'=\theta_2|_{\Gamma'}$, and $q_0'=q_0|_L$. 
Then $(G',q_1',\theta_1')$ and $(\Gamma',q_2',\theta_2')$ also satisfy the assumptions of this section.
If we let $(S',T')$ denote the unitary matrices defined in Definition \ref{ST3} for $G'$ and $\Gamma'$, 
we have  
$$(S,T)=(S^{(L,q_0')}\otimes S',T^{(L,q_0')}\otimes T').$$
\end{lemma}

\begin{remark}\label{even-odd}
Recall that $G$ (and hence $q_1$ and $\theta_1$ too) is canonically decomposed as 
$G_e\times G_o$, where $G_e$ is a 2-group and $G_o$ is an odd group. 
Moreover, we can diagonalize $\theta_1$ on $G_o$. 
Thus we have the canonical decomposition 
$$(G,q_1,\theta_1)=(G_e,q_{1,e},\theta_{1,e})\times (G_{o,+},q_{1,o,+},1)\times (G_{o,-},q_{1,o,-},-1),$$
$$(\Gamma,q_2,\theta_2)=(\Gamma_e,q_{2,e},\theta_{2,e})\times (\Gamma_{o,+},q_{2,o,+},1)\times 
(\Gamma_{o,-},q_{2,o,-},-1).$$
Under the identification $(U,q_0)=(G^{\theta},q_1|_{G^{\theta_1}})=(\Gamma^{\theta_2},q_2|_{\Gamma^{\theta_2}}),$
we get 
$$(U,q_0)=(G_e^{\theta_{1,e}},q_1|_{G_e^{\theta_{1,e}}})\times (G_{o,+},q_{1,o+})
=(\Gamma_e^{\theta_{2,e}},q_2|_{\Gamma_2^{\theta_{2,e}}})\times (\Gamma_{o,+},q_{2,o+}).$$
Thus we can always apply Lemma \ref{factorization} to $L=G_{o,+}=\Gamma_{o,+}$. 
Therefore for the purposes of constructing new modular data, we may assume that $G_{o,+}=\Gamma_{o,+}=\{0\}$. 
\end{remark}

Direct computation shows the following.
\begin{lemma}\label{relations3} Let the notation be as in Definition \ref{ST3}. 
Then $S$, $T$, and $C$ are unitary matrices satisfying $S^2=C$, $(ST)^3=cC$, and $T C=CT$. 
\end{lemma}

\begin{lemma}\label{Verlinde3} Let the notation be as Definition \ref{ST3}. 
Then we have 
$$N_{(u,0),(u',0),(u'',0)}=N_{(u,0),(u',\pi),(u'',\pi)}=\delta_{u+u'+u'',0},$$
$$N_{(u,0),(u',0),(u'',\pi)}=N_{(u,0),(u',0),g}=N_{(u,0),(u',0),\gamma}=N_{(u,0),(u',\pi),g}
=N_{(u,0),(u',\pi),\gamma}=0,$$
$$N_{(u,0),g,g'}=\delta_{u+g+g',0}+\delta_{u+\theta(g)+g',0},\quad 
N_{(u,0),g,\gamma}=0,\quad  
N_{(u,0),\gamma,\gamma'}=\delta_{u+\gamma+\gamma',0}+\delta_{u+\theta(\gamma)+\gamma',0},$$
$$N_{(u,\pi),(u',\pi),(u'',\pi)}
=\frac{4|U|}{|\Gamma|-|G|}\frac{1}{|U|}\sum_{v\in U}\inpr{u+u'+u''}{v},$$ 
$$N_{(u,\pi),(u',\pi),g}
=\frac{4|U|}{|\Gamma|-|G|}\frac{1}{|U|}\sum_{v\in U}\inpr{u+u'+g}{v},$$
$$N_{(u,\pi),(u',\pi),\gamma}
=\frac{4|U|}{|\Gamma|-|G|}\frac{1}{|U|}\sum_{v\in U}\inpr{u+u'+\gamma}{v},$$
$$N_{(u,\pi),g,g'}=
\frac{4|U|}{|\Gamma|-|G|}\frac{1}{|U|}\sum_{v\in U}\inpr{u+g+g'}{v} +\delta_{u+g+g',0}+\delta_{u+\theta(g)+g',0},$$
$$N_{(u,\pi),g,\gamma}=\frac{4|U|}{|\Gamma|-|G|}\frac{1}{|U|}\sum_{v\in U}\inpr{u}{v}\inpr{g}{v}\inpr{\gamma}{v},$$
$$N_{(u,\pi),\gamma,\gamma'}
 =\frac{4|U|}{|\Gamma|-|G|}\frac{1}{|U|}\sum_{v\in U}\inpr{u+\gamma+\gamma'}{v}
 -\delta_{u+\gamma+\gamma',0}-\delta_{u+\theta(\gamma)+\gamma',0},$$
\begin{align*}
N_{g,g',g''}&= \frac{4|U|}{|\Gamma|-|G|}\frac{1}{|U|}\sum_{v\in U}\inpr{g+g'+g''}{v} \\
 &+\delta_{g+g'+g'',0}+\delta_{\theta(g)+g'+g'',0}
 +\delta_{g+\theta(g')+g'',0}+\delta_{g+g'+\theta(g''),0},
\end{align*}
$$N_{g,g',\gamma}=\frac{4|U|}{|\Gamma|-|G|}\frac{1}{|U|}\sum_{v\in U}\inpr{g+g'}{v}\inpr{\gamma}{v},$$
$$N_{g,\gamma,\gamma'}=\frac{4|U|}{|\Gamma|-|G|}\frac{1}{|U|}\sum_{v\in U}\inpr{g}{v}\inpr{\gamma+\gamma'}{v},$$
\begin{align*}
 N_{\gamma,\gamma',\gamma''} &=\frac{4|U|}{|\Gamma|-|G|}\frac{1}{|U|}\sum_{v\in U}\inpr{\gamma+\gamma'+\gamma''}{v} \\
 &-\delta_{\gamma+\gamma'+\gamma'',0}-\delta_{\theta(\gamma)+\gamma'+\gamma'',0}
 -\delta_{\gamma+\theta(\gamma')+\gamma'',0}-\delta_{\gamma+\gamma'+\theta(\gamma''),0}.
\end{align*}
\end{lemma}

The above computation shows that in order for $(S,T)$ to be modular data, 
the number $4|U|/(|\Gamma|-|G|)$ has to be a positive integer. 
On the other hand, since $U$ is a common subgroup of $G$ and $\Gamma$, the number 
$(|\Gamma|-|G|)/|U|$ is a positive integer too. 
Thus we have only three possibilities: $|\Gamma|=|G|+|U|$, $|\Gamma|=|G|+2|U|$, and $|\Gamma|=|G|+4|U|$. 

\begin{theorem}\label{STtheorem3} Let the notation be as in Definition \ref{ST3}. 
All the fusion coefficients in Lemma \ref{Verlinde3} are non-negative integers if and only of  
one of the following occurs: 
\begin{itemize}
\item $|\Gamma|=|G|+4|U|$,
\item $|\Gamma|=|G|+2|U|$, and $|G|/|U|$ is an even number. 
\end{itemize}
\end{theorem}

\begin{proof} 
We first assume $|\Gamma|=|G|+|U|$. 
Then either $|\Gamma|/|U|$ or $|G|/|U|$ is an odd number. We assume that $|G|/|U|$ is odd as the other case can be treated in the same way. 
Let 
$$(G,q_1,\theta_1)=(G_e,q_{1,e},\theta_{1,e})\times (G_{o,+},q_{1,o+},1)\times (G_{o,-},q_{1,o-},-1)$$
be the decomposition as in Remark \ref{even-odd}. 
Then $|G|/|U|$ being odd implies that $\theta_{1,e}$ is trivial, and 
$$(U,q_0)=(G_{1,e},q_{1,e})\times (G_{o,+},q_{1,o+}).$$  
In particular $q_0$ is non-degenerate, and Lemma \ref{factorization} implies that 
we have a factorization 
$$(\Gamma,\theta_2)=(U,q_0)\times (\Gamma',\theta_2'),$$
with $\theta_2'$ a fixed-point-free involution. 
However $|\Gamma'|=|G|/|U|+1$ implies that $\Gamma'$ is an even group, which never allows a 
fixed-point-free involution, and we get contradiction. 

Assume now that $|\Gamma|=|G|+2|U|$, and $|G|/|U|$ is odd. 
Then $|\Gamma|/|U|$ is odd too, and we have the decompositions 
$$(G,q_1,\theta_1)=(G_e,q_{1,e},\theta_{1,e})\times (G_{o,+},q_{1,o+},1)\times (G_{o,-},q_{1,o-},-1),$$
$$(\Gamma,q_2,\theta_2)=(\Gamma_e,q_{2,e},\theta_{2,e})\times (\Gamma_{o,+},q_{1,o+},1)\times (\Gamma_{o,-},q_{1,o-},-1), $$
with trivial $\theta_{1,e}$ and $\theta_{2,e}$. 
Thus 
$$(G,q_1)=(U,q_0)\times (G_{o,-},q_{1,o-}),$$
$$(\Gamma,q_2)=(U,q_0)\times (\Gamma_{o,-},q_{2,o-}),$$
and $\cG(q_{1,o-})=-\cG(q_{2,o-})$. 
However, this is impossible as $G_{o,-}$ and $\Gamma_{o,-}$ are odd groups with 
$|\Gamma_{o,-}|=|G_{o,-}|+2$. 
\end{proof}

Recall that a near-group category for a finite abelian group 
$A$ with multiplicity $|A|$ comes with a non-degenerate symmetric bicharacter, 
say $\langle\langle\cdot, \cdot\rangle\rangle$ 
(we avoid the symbol $\inpr{\cdot}{\cdot}$ used in \cite{MR1832764}, \cite{MR3635673} to prevent possible confusion). 
We generalize Evans-Gannon's conjecture. 
It is easy to see from Lemma \ref{factorization} and Remark \ref{even-odd} that the following statement actually 
generalizes \cite[Conjecture 2]{MR3167494} from the odd group case to the general case.

\begin{conjecture}\label{near-group} The modular data of the Drinfeld center of a near-group category for a finite abelian group $A$ 
with multiplicity $|A|$ and a non-degenerate symmetric bicharacter $\langle\langle\cdot, \cdot\rangle\rangle$  
is given by $(S,T)$ of Definition \ref{ST3} 
with $G=A\times A$, $q_1(a_1,a_2)=\langle\langle a_1,a_2\rangle\rangle$, $\theta_1(a_1,a_2)=(a_2,a_1)$, and
$$|\Gamma|=|G|+4|U|=|A|(|A|+4).$$  
\end{conjecture}

The conjecture is true for odd groups $A$ with $|A|\leq 13$, as well as $A=\Z_2$ and $A=\Z_4$. 

We compute the Frobenius-Schur indicators now. 
For $u\in U$ and $m\in \Z$, let 
$$\cG(q_1,u,m)=\frac{1}{\sqrt{|G|}}\sum_{g\in G}\inpr{u}{g}q_1(g)^m,\quad \cG(q_2,u,m)=\frac{1}{\sqrt{|\Gamma|}}\sum_{\gamma\in \Gamma}\inpr{u}{\gamma}q_2(\gamma)^m,$$
and
$$\cG(q_1,q_2,v,m)=\frac{1}{\sqrt{|G|}}\sum_{g\in G}\inpr{v}{g}q_1(g)^m
+\frac{1}{\sqrt{|\Gamma|}}\sum_{\gamma\in \Gamma}\inpr{v}{\gamma}q_2(\gamma)^m.$$

\begin{lemma}\label{FS3} Let the notation be as above. 
Then 
$$\nu_m((u,0))=\delta_{mu,0}q_0(u)^m,$$
$$\nu_m((u,\pi))=\frac{1}{|\Gamma|-|G|}\sum_{v\in U}\inpr{u}{v}|\cG(q_1,q_2,v,m)|^2,$$
$$\nu_m(g)=\frac{1}{|\Gamma|-|G|}\sum_{v\in U}\inpr{g}{v}|\cG(q_1,q_2,v,m)|^2+\delta_{mg,0}q_1(g)^m,$$
$$\nu_m(\gamma)=\frac{1}{|\Gamma|-|G|}\sum_{v\in U}\inpr{\gamma}{v}|\cG(q_1,q_2,v,m)|^2-\delta_{m\gamma,0}q_1(\gamma)^m.$$
\end{lemma}

Note that if $G$ is an odd group and $(S,T)$ comes from a modular tensor category, 
we have $|\Gamma|=|G|+4$, and then
$$\nu_m(\pi)=\frac{1}{4}|\cG(q_1,m)+\cG(q_2,m)|^2,$$ 
$$\nu_m(g)=\frac{1}{4}|\cG(q_1,m)+\cG(q_2,m)|^2+\delta_{mg,0}q_1(g)^m,$$
$$\nu_m(\gamma)=\frac{1}{4}|\cG(q_1,m)+\cG(q_2,m)|^2-\delta_{m\gamma,0}q_1(\gamma)^m.$$

\begin{lemma}\label{restriction}
Assume $G$ is an odd group with $\theta_1=-1$ and $|\Gamma|=|G|+4$. 
\begin{itemize}
\item[(1)] We have $\nu_2(x)=1$ for any $x\in J$, and this does not give any restriction for the pair $(G,q_1)$ and $(\Gamma,q_2)$. 
\item[(2)] If the 3-rank of $G$ and that of $\Gamma$ are both less than or equal to 1, the values of $\nu_3(x)$ are 
consistent with $N_{x,x,x}$ for any $x\in J$ . 
\item[(3)] If either the 3-rank of $G$ or that of $\Gamma$ is greater than or equal to 3, the value of $\nu_3(\pi)$ is 
not consistent with $N_{\pi,\pi,\pi}=1$, and the pair $(S,T)$ never comes from a modular tensor category. 
\end{itemize}
\end{lemma}

\begin{proof} (1)
We have 
\begin{align*}
\lefteqn{\nu_2(\pi)=|\frac{\cG(q_1,2)+\cG(q_2,2)}{2}|^2} \\
 &=|\frac{\cG(q_1)(-1)^{\frac{|G|^2-1}{8}}+\cG(q_2)(-1)^{\frac{(|G|+4)^2-1}{8}}}{2}|^2 \\
 &=|\frac{\cG(q_1)-\cG(q_2)}{2}|^2=|\cG(q_1)|^2=1. 
\end{align*}
We can similarly show $\nu_2(x)=1$ for the other $x\in J$ from this computation.  

(2) 
If either $G$ or $\Gamma$ has no 3-component (i.e. $|G|\equiv 1 \mod 3$), 
then using the Jacobi symbol and the quadratic reciprocity law, we get 
\begin{align*}
\lefteqn{\nu_3(\pi)=|\frac{\cG(q_1,3)+\cG(q_2,3)}{2}|^2} \\
 &=|\frac{\cG(q_1)(\frac{3}{|G|})+\cG(q_2)(\frac{3}{|\Gamma|})}{2}|^2
 =|\frac{\cG(q_1)(-1)^{\frac{|G|-1}{2}}(\frac{|G|}{3})+\cG(q_2)(-1)^{\frac{|G|+3}{2}}(\frac{|G|+4}{3})}{2}|^2\\
 &=|\frac{\cG(q_1)-\cG(q_2)}{2}|^2=|\cG(q_1)|^2=1.
\end{align*}
If $G$ has 3-rank 1, we have either $\cG(q_1,3)\in \{\sqrt{3},-\sqrt{3}\}$ 
and $\cG(q_2,3)\in \{i,-i\}$, or $\cG(q_1,3)\in \{\sqrt{3}i,-\sqrt{3}i\}$ and $\cG(q_2,3)\in \{1,-1\}$.In either case we have $\nu_3(\pi)=1$. 
We get the same conclusion if $\Gamma$ has 3-rank 1. 
Consistency of $\nu_3(x)$ and $N_{x,x,x}$ for other $x$ follows from the same computation. 

(3) If $G$ has 3-rank larger than or equal to 3, 
we get $|\cG(q_1,3)|\geq 3\sqrt{3}$ and $|\cG(q_2,3)|=1$. 
Then $\nu_3(\pi)> 1$, which is not consistent with $N_{\pi,\pi,\pi}=1$. 
We get the same conclusion if $\Gamma$ has 3-rank larger than equal to 3. 
\end{proof}

\begin{remark}\label{3rank2} 
The most interesting case is of course when either the 3-rank of $G$ or that of $\Gamma$ is 2. 
In this case, our formulae for $\nu_3(\pi)$ and $N_{\pi,\pi,\pi}=1$ give a restriction of quadratic forms. 
For example, assume $G=\Z_5$ and $\Gamma=\Z_3\times \Z_3$. 
Then 
$$\nu_3(\pi)=|\frac{\cG(q_1,3)+3}{2}|^2=|\frac{-\cG(q_1)+3}{2}|^2=|\frac{\cG(q_2)+3}{2}|^2.$$
Thus $\cG(q_2)=1$ is forbidden, and only $\cG(q_2)=-1$ (and consequently $\cG(q_1)=1$) is consistent with 
$N_{\pi,\pi,\pi}=1$. 
This explains why there are only 3 solutions for the polynomial equations for the near-group category 
for $\Z_5$ with multiplicity 5 (see \cite[Table 2]{MR3167494}). In fact there are two solutions corresponding to $\Gamma=\Z_9$  
with $q_2(x)=\zeta_9^{x^2}, \zeta_9^{-x^2}$, and one solution corresponding to $\Gamma=\Z_3\times \Z_3$ with $q_2(x,y)=\zeta_3^{x^2+y^2}$. 
In a similar way, for $G=\Z_3\times \Z_3$ and $\Gamma=\Z_{13}$, only $\cG(q_1)=1$ is possible, 
which gives the modular data of the Drinfeld center of the Haagerup category. 
\end{remark}

\begin{cor} If Conjecture \ref{near-group} is true, there exists no generalized Haagerup category 
with an odd abelian group $A$ whose 3-rank is larger than or equal to 2. 
\end{cor}

\begin{proof} If $\cC$ is such a category, $\cC\boxtimes \mathrm{Vec}_A$ is Morita equivalent to 
a near-group category for $A\times A$ with multiplicity $|A|^2$ (see \cite[Theorem 12.9]{MR3635673}), which contradicts Lemma \ref{restriction},(3). 
\end{proof}

We get back to the general case. 
Recall that if $m$ is an odd number and $(A,q)$ is a metric group with $|A|=2^n$, we have 
$$\frac{\cG(q,m)}{\cG(q)^m}=(-1)^{n\frac{m^2-1}{8}}.$$

\begin{lemma}\label{nu3computation3} Let the notation be as above. 
Let $m\in \Z$ and let $u\in U$. 
\begin{itemize}
\item[(1)] For odd $m$, we have
$$|\cG(q_1,q_2,u,m)| =|\cG(q_1,q_2,0,m)|=|(-1)^{n_1\frac{m^2-1}{8}}\frac{\cG(q_{1,o},m)}{\cG(q_{1,o})^m}-(-1)^{n_2\frac{m^2-1}{8}}\frac{\cG(q_{2,o},m)}{\cG(q_{2,o})^m}|,$$
where $|G_e|=2^{n_1}$ and $|\Gamma_e|=2^{n_2}$. 
If moreover $(m,|G|)=(m,|\Gamma|)=1$, we have
$$|\cG(q_1,q_2,u,m)| 
=|(-1)^{n_1\frac{m^2-1}{8}}(\frac{|G_o|}{m})
 -(-1)^{n_2\frac{m^2-1}{8}}(\frac{|\Gamma_o|}{m})|,$$
where $(\frac{|G_o|}{m})$ is the Jacobi symbol. 
\item[(2)] For $m=2$, we have
$$|\cG(q_1,q_2,u,2)|=|\frac{\cG(q_{1,e},u,2)}{\cG(q_{1,e})}(-1)^{\frac{|G_o|^2-1}{8}}-\frac{\cG(q_{2,e},u,2)}{\cG(q_{2,e})}(-1)^{\frac{|\Gamma_o|^2-1}{8}}|.$$
\end{itemize}
\end{lemma}

\begin{proof} (1) Since $U$ is a 2-group, for any $u\in U$ there exists a unique element $u'\in U$ satisfying $mu'=u$. 
Thus
\begin{align*}
\lefteqn{|\cG(q_1,q_2,u,m)|} \\
 &=|\frac{1}{\sqrt{|G_e|}}\sum_{g\in G_e}\inpr{u}{g}_{q_{1,e}}q_{1,e}(g)^m\cG(q_{1,o},m)
 +\frac{1}{\sqrt{|\Gamma_e|}}\sum_{\gamma\in \Gamma_e}\inpr{u}{\gamma}_{q_{2,e}}q_{2,e}(\gamma)^m \cG(q_{2,o},m)| \\
 &=|\frac{1}{\sqrt{|G_e|}}\sum_{g\in G_e}\inpr{u'}{g}_{q_{1,e}^m}q_{1,e}^m(g)\cG(q_{1,o},m)
 +\frac{1}{\sqrt{|\Gamma_e|}}\sum_{\gamma\in \Gamma_e}\inpr{u'}{\gamma}_{q_{2,e}^m}q_{2,e}^m(\gamma) \cG(q_{2,o},m)| \\
 &=|\frac{1}{\sqrt{|G_e|}}\sum_{g\in G_e}\frac{q_{1,e}^m(g+u')}{q_{1,e}^m(u')}\cG(q_{1,o},m)
 +\frac{1}{\sqrt{|\Gamma_e|}}\sum_{\gamma\in \Gamma_e}\frac{q_{2,e}^m(\gamma+u')}{q_{2,e}^m(u')} \cG(q_{2,o},m)| \\
 &=|\cG(q_1,m)+\cG(q_2,m)|.
\end{align*}
Since $\cG(q_1)=-\cG(q_2)$, this is equal to 
$$|\frac{\cG(q_1,m)}{\cG(q_1)^m}-\frac{\cG(q_2,m)}{\cG(q_2)^m}|
=|\frac{\cG(q_{1,e},m)}{\cG(q_{1,e})^m}\frac{\cG(q_{1,o},m)}{\cG(q_{1,o})^m}-\frac{\cG(q_{2,e},m)}{\cG(q_{2,e})^m}\frac{\cG(q_{2,o},m)}{\cG(q_{2,o})^m}|,$$
which shows the first part. 

If moreover $(m,|G|)=(m,|\Gamma|)=1$, the quadratic reciprocity law implies 
$$\frac{\cG(q_{1,o},m)}{\cG(q_{1,o})^m}=\frac{(-1)^{\frac{m-1}{2}\frac{|G_o|-1}{2}}}{\cG(q_{1,o})^{m-1}}(\frac{|G_o|}{m})=
(\frac{(-1)^{\frac{|G_o|-1}{2}}}{\cG(q_{1,o})^2})^{\frac{m-1}{2}}(\frac{|G_o|}{m})=(\frac{|G_o|}{m}),$$
and we get the second part. 

(2) Since $\cG(q_1)=-\cG(q_2)$, we get 
\begin{align*}
\lefteqn{|\cG(q_1,u,2)+\cG(q_2,u,2)|} \\
 &=|\frac{\cG(q_{1,e},u,2)\cG(q_{1,o},2)}{\cG(q_{1,e})\cG(q_{1,o})}-\frac{\cG(q_{2,e},u,2)\cG(q_{2,o},2)}{\cG(q_{2,e})\cG(q_{2,o})}| \\
 &=|\frac{\cG(q_{1,e},u,2)}{\cG(q_{1,e})}(-1)^{\frac{|G_o|^2-1}{8}}-\frac{\cG(q_{2,e},u,2)}{\cG(q_{2,e})}(-1)^{\frac{|\Gamma_o|^2-1}{8}}|.
\end{align*}
\end{proof}
\subsection{Examples with $|\Gamma|=|G|+4|U|$}
We give examples of involutive metric groups $(G,q_1,\theta_1)$ and $(\Gamma,q_2,\theta_2)$ satisfying the assumptions of 
Definition \ref{ST3} with$|\Gamma|=|G|+4|U|$ such that (FS2) and (FS3) from Section \ref{md} are satisfied. 
Taking Lemma \ref{factorization} and Remark \ref{even-odd} into account, 
we assume 
$$(G,q_1,\theta_1)=(G_e,q_{1,e},\theta_{1,e})\times (G_o,q_{1,o},-1),$$
$$(\Gamma,q_2,\theta_2)=(\Gamma_e,q_{2,e},\theta_{2,e})\times (\Gamma_o,\theta_{2,o},-1).$$
We also assume that $q_0$ is degenerate on any non-trivial subgroup of $U$. 

Lemma \ref{Verlinde3} implies that
$$N_{(u,\pi),(u,\pi),(u,\pi)}=\frac{1}{|U|}\sum_{v\in U}\inpr{u}{v},$$
$$N_{g,g,g}=\frac{1}{|U|}\sum_{v\in U}\inpr{g}{v}+\delta_{3g,0}+3\delta_{2g+\theta(g),0},$$
$$N_{\gamma,\gamma,\gamma}=\frac{1}{|U|}\sum_{v\in U}\inpr{\gamma}{v}-\delta_{3\gamma,0}-3\delta_{2\gamma+\theta(\gamma),0}.$$
On the other hand, Lemma \ref{FS3} and \ref{nu3computation3} show that
$$\nu_3(u,\pi)=\frac{|\cG(q_1,q_2,0,3)|^2}{4}\frac{1}{|U|}\sum_{v\in U}\inpr{u}{v},$$
$$\nu_3(g)=\frac{|\cG(q_1,q_2,0,3)|^2}{4}\frac{1}{|U|}\sum_{v\in U}\inpr{g}{v}+\delta_{3g,0}q_1(g)^3,$$
$$\nu_3(\gamma)=\frac{|\cG(q_1,q_2,0,3)|^2}{4}\frac{1}{|U|}\sum_{v\in U}\inpr{\gamma}{v}-\delta_{3\gamma,0}q_2(\gamma)^3.$$
Thus the condition (FS3) is equivalent to $|\cG(q_1,q_2,0,3)|=2$. 
Direct computation using Lemma \ref{nu3computation3} shows that restriction coming from (FS3) in the following examples 
is very similar to that discussed in Lemma \ref{restriction} and Remark \ref{3rank2}, and we will not mention it in what follows.

In the following subsections we consider different forms of $G_e$ and $\theta_{1,e} $. 

\subsubsection{Assume $G_e=\Z_2\times \Z_2$ with $\theta_{1,e}(x,y)=(y,x)$.}  \label{near1}
In this case, we have $U\cong \Z_2$ 
and $|\Gamma|=4|G_o|+8=4(|G_o|+2)$, and hence $|\Gamma_e|=4$, $|\Gamma_o|=|G_o|+2$. 
Thus $\cG(q_{2,o})/\cG(q_{1,o})\in \{i,-i\}$, and so $\cG(q_{2,e})/\cG(q_{1,e})\in \{i,-i\}$ too. 
This implies $\Gamma_{e}=\Z_2\times \Z_2$. 
Since $\theta_{2,e}$ is non-trivial, we may assume $\theta_{2,e}(x,y)=(y,x)$. 
Note that the only flip-invariant quadratic forms of $\Z_2\times \Z_2$ are $q(x,y)=(-1)^{xy}$, $(-1)^{x^2+xy+y^2}$, and  
$i^{\pm(x^2+y^2)}$. We consider these quadratic forms separately.\\

\paragraph{Assume $q_{1,e}(x,y)=(-1)^{xy}$}
In this case we have $\cG(q_{1,e})=1$ and $q_{2,e}(x,y)=i^{\pm (x^2+y^2)}$. 
To fulfill $\cG(q_2)=-\cG(q_1)$, we need $\cG(q_{2,o})=\pm i \cG(q_{1,o})$. 
Direct computation shows that (FS2) gives no restriction.
In the simplest example $G=\Z_2\times \Z_2$, $\Gamma=\Z_2\times \Z_2\times \Z_3$, 
the pair $(S,T)$ is the modular data of the Drinfeld center 
of the near-group category for $\Z_2$ with multiplicity 2 (see \cite{MR1832764}). 
More generally, we conjecture that the modular data of the Drinfeld center of a near-group category for $\Z_2\times G_o$ 
with multiplicity $2|G_o|$ is given by 
$$(S^{(G_o,\overline{q_{1,o}})}\otimes S, T^{(G_o,\overline{q_{1,o}})}\otimes T).$$

\paragraph{Assume $q_{1,e}(x,y)=(-1)^{x^2+xy+y^2}$} 
In this case we have $\cG(q_{1,e})=-1$ and $q_{2,e}(x,y)=i^{\pm (x^2+y^2)}$. 
To fulfill $\cG(q_2)=-\cG(q_1)$, we need $\cG(q_{2,o})=\mp i \cG(q_{1,o})$. 
Direct computation shows that (FS2) gives no restriction. \\

\paragraph{Assume $q_{1,e}(x,y)=i^{\pm(x^2+y^2)}$} 
In this case $\cG(q_{1,e})=\pm i$ and $q_{2,e}(x,y)$ is either $(-1)^{xy}$ or 
$(-1)^{x^2+xy+y^2}$. 
We have $\cG(q_{2,o})=\mp i\cG(q_{1,o})$ for $q_{2,e}(x,y)=(-1)^{xy}$, and 
$\cG(q_{2,o})=\pm i\cG(q_{1,o})$ for $q_{2,e}(x,y)=(-1)^{x^2+xy+y^2}$. 
Direct computation shows that (FS2) gives no restriction. 
\subsubsection{Assume $G_e=\Z_4$.} 
Since $\theta_{1,e}$ is non-trivial, we have $\theta_{1,e}(x)=-x$. 
In this case we have $U\cong \Z_2$, 
and $|\Gamma|=4|G_o|+8=4(|G_o|+2)$, and hence $|\Gamma_e|=4$, $|\Gamma_o|=|G_o|+2$. 
Thus $\cG(q_{2,o})/\cG(q_{1,o})\in \{i,-i\}$, and so $\cG(q_{2,e})/\cG(1,e)\in \{i,-i\}$ too. 
This implies that $\Gamma_{e}=\Z_4$. 
Since $\theta_{2,e}$ is non-trivial, we have $\theta_{2,e}(x)=-x$. 
Non-degenerate quadratic forms on $\Z_4$ are of the form $q(x)=\zeta_8^{rx^2}$ with odd $r$, 
and $\cG(q)=\zeta_8^r$. 
Possible combinations are $q_{1,e}(x)=\zeta_8^{r_1x^2}$, $q_{2,e}(x)=\zeta_8^{r_2x^2}$, and $\cG(q_{2,o})=-\zeta_8^{r_1-r_2} \cG(q_{1,o})$, with $r_1-r_2\equiv 2 \mod 4$. 
Direct computation shows that (FS2) gives no restriction.

\begin{remark} \label{ze2} In the above examples with $|G_e|=4$, we fix two odd metric groups $(G_o,q_{1,o})$ and $(\Gamma_o,q_{2,o})$. 
Then we can show that all the above potential modular data $(S,T)$ can be obtained from that for $q_{1,e}(x,y)=(-1)^{xy}$ 
by applying the zesting construction introduced in \cite[Theorem 3.15]{MR3641612}. 
\end{remark}
\subsubsection{Assume $G_e=\Z_4\times \Z_4$ and $\theta_{1,e}(x,y)=(y,x)$.} 
In this case we have $U\cong \Z_4$ and $|\Gamma|=2^4(|G_o|+1)$. 
Thus there exists $n\geq 3$ satisfying $|\Gamma_e|=2^{n+2}$, and we have $|G_o|=2^{n-2}|\Gamma_o|-1$. 
We assume that $\Gamma_e=\Z_{2^{n+1}}\times \Z_2$, $q_{2,e}(x,y)=\zeta_{2^{n+2}}^{rx^2}\zeta_4^{r(1-2^{n-2})y^2}$, and 
$$\theta_{2,e}(x,y)=((2^{n-1}-1)x+2^ny,x+y),$$
with $r\in \{1,-1,3,-3\}$.  
Then $\theta_{2,e}$ preserves $q_{2,e}$ and $\Gamma_e^{\theta_{2,e}}=<(2^{n-1},1)>\cong \Z_4$. 
We have $$q_{2,e}(2^{n-1},1)=\zeta_{2^{n+2}}^{r2^{2n-2}}\zeta_4^{r(1-2^{n-2})}=\zeta_4^r.$$
We have 
$$\cG(q_{2,e})=(-1)^{n\frac{r^2-1}{8}}(-1)^{\frac{2^{n-3}(2^{n-3}-1)}{2}}i^{r(1-2^{n-3})}.$$
To fulfill the condition $(G^{\theta_1},q_1|_{G^{\theta_1}})\cong (\Gamma^{\theta_2},q_2|_{\Gamma^{\theta_2}})$, we have 3 possibilities: $q_{1,e}(x,y)=\zeta_4^{rxy}$, $\zeta_4^{-r(x^2+xy+y^2)}$, or $\zeta_8^{r(x^2+y^2)}$. We consider these quadratic forms separately.\\

\paragraph{Assume $q_{1,e}(x,y)=\zeta_4^{rxy}$}
In this case $\cG(q_{1,e})=1$ and 
$$\cG(q_{1,o})=-(-1)^{n\frac{r^2-1}{8}}(-1)^{\frac{2^{n-3}(2^{n-3}-1)}{2}}i^{r(1-2^{n-3})}\cG(q_{2,o}).$$
For $n=3$, we have $\cG(q_{1,o})/\cG(q_{2,o})\in \{1,-1\}$ and $|G_o|\equiv |\Gamma_o|\mod 4$. 
This together with $|G_o|=2|\Gamma_o|-1$ implies that $|G_o|\equiv 1\mod 8$. 
For $n\geq 4$, we have $\cG(q_{1,o})/\cG(q_{2,o})\in \{i,-i\}$ and $|G_o|\equiv |\Gamma_o|-2 \mod 4$. 
This together with $|G_o|=2^{n-2}|\Gamma_o|-1$ implies that $|G_o|\equiv 3\mod 4$ and $|\Gamma_o|\equiv 1\mod 4$.  

Direct computation shows that (FS2) gives no restriction.  

For the smallest example $G=\Z_4\times \Z_4$, $\Gamma=\Z_{16}\times \Z_2$, 
there are two possibilities: $r=\pm 3$. These give the modular data for the Drinfeld centers of the two complex-conjugate near-group categories for $ \Z_4$ with multiplicity $4$. This can be verified using Theorem 6.8 of \cite{MR1832764}, after solving Equations (6.18)-(6.21) for $\Z_4 $ using the structure constants found in Example A.2 there.

More generally, we conjecture that the pair 
$$(S^{(G_o,\overline{q_{1,o}})}\otimes S,T^{(G_o,\overline{q_{1,o}})}\otimes T),$$
is the modular data of the Drinfeld center of a near-group category for 
$\Z_4\times G_o$ with multiplicity $4|G_o|$. \\

\paragraph{Assume $q_{1,e}(x,y)=\zeta_4^{-r(x^2+xy+y^2)}$} In this case $\cG(q_{1,e})=1$, and the computation is very similar to the previous case. \\

\paragraph{Assume $q_{1,e}(x,y)=\zeta_8^{r(x^2+y^2)}$}
In this case $\cG(q_{1,e})=i^r$ and 
$$\cG(q_{2,o})=-(-1)^{n\frac{r^2-1}{8}}(-1)^{\frac{2^{n-3}(2^{n-3}-1)}{2}}i^{r2^{n-3}}\cG(q_{1,o}).$$
For $n=3$, we have $\cG(q_{1,e})/\cG(q_{2,e})\in \{i,-i\}$ and $|G_o|\equiv |\Gamma_e|-2 \mod 4$. 
This together with $|G_o|=2|\Gamma_e|-1$ shows that $|G_o|\equiv 5\mod 8$ and $|\Gamma_o|\equiv 3\mod 8$. 
For $n\geq 4$, we have  $\cG(q_{1,e})/\cG(q_{2,e})\in \{1,-1\}$ and $|G_o|\equiv |\Gamma_e| \mod 4$. 
This together with $|G_o|=2|\Gamma_e|-1$ shows that $|G_o|\equiv |\Gamma_e|\equiv 3\mod 4$. 

Direct computation shows that (FS2) gives no restriction.  

For the smallest example $G=\Z_4\times \Z_4\times \Z_5$, $\Gamma=\Z_{16}\times \Z_2\times \Z_3$, 
there are 8 possibilities: two choices of $q_{1,o}$ and four choices of $r$. We have $\cG(q_{2,o})=-(-1)^{\frac{r^2-1}{8}}i^r\cG(q_{1,o}).$  
\subsubsection{Assume $G_e=\Z_2^4$ and $\theta_{1,e}(x_1,x_2,y_1,y_2)=(y_1,y_2,x_1,x_2)$.} 
In this case we have $U\cong \Z_2\times \Z_2$ and $|\Gamma|=2^4(|G_o|+1)$. Again we consider different quadratic forms.\\

\paragraph{Assume $q_{1,e}(x_1,x_2,y_1,y_2)=(-1)^{x_1y_1+x_2y_2}$} We choose $n\geq 3$ with 
$|\Gamma_e|=2^{n+2}$ and $|G_o|=2^{n-2}|\Gamma_o|-1$. 
Since $\cG(q_{2,e})^4=\cG(q_{1,e})^4=1$, it follows from Theorem \ref{list} that $\Gamma_e=\Z_{2^n}\times \Z_4$ is the only possibility. We have $q_{2,e}(x,y)=\zeta_{2^{n+1}}^{rx^2}\zeta_8^{sy^2}$ with $r,s\in\{1,-1,3,-3\}$, 
and $\theta_{2,e}(x,y)=(-x,-y)$ or $((-1+2^{n-1})x+2^{n-1}y,2^{n-1}x-y)$. 
We denote the latter by $\tau(x,y)$. 

We have $q_2(2^{n-1},0)=1$, $q_2(0,2)=-1$, $q_2(2^{n-1},2)=-1$, and $(G^{\theta_1},q_1)\cong (\Gamma^{\theta_2},q_2)$. 
We have $\cG(q_{1,e})=1$ and $\cG(q_{2,e})=(-1)^{n\frac{r^2-1}{8}}\zeta_8^{r+s}$, and so 
$$\cG(q_{1,o})+(-1)^{n\frac{r^2-1}{8}}\zeta_8^{r+s}\cG(q_{2,o})=0.$$
The group automorphism $(x,y)\to (x+2^{n-2}y,\pm x+y)$ of $\Z_{2^n}\times \Z_4$ transforms $(r,s)$ to $(r+2^{n-2}s,2^{n-2}r+s)$. 

Direct computation shows that $|\cG(q_1,q_2,k,2)|^2=4$ for $k\neq 0$, and 
$$|\cG(q_1,q_2,0,2)|^2=|4-2(-1)^{\frac{r^2-1}{8}}(-1)^{\frac{s^2-1}{8}}(-1)^{\frac{|G_o|^2-1}{8}}(-1)^{\frac{|\Gamma_o|^2-1}{8}}|^2,$$
which is compatible with $\nu_2(0,\pi)=\pm 1$ if and only if 
$$(-1)^{\frac{r^2-1}{8}}(-1)^{\frac{s^2-1}{8}}(-1)^{\frac{|G_o|^2-1}{8}}(-1)^{\frac{|\Gamma_o|^2-1}{8}}=1.$$
Assuming this condition, we have 
$$\nu_2(\gamma)=\frac{1}{|U|}\sum_{u\in U}\inpr{\gamma}{u}-\delta_{2\gamma,0}q_2(\gamma)^2.$$
This shows that (FS2) is satisfied only for $\theta_{2,e}=\tau$. 

For the smallest example $G=\Z_2^4$, $\Gamma=\Z_{8}\times \Z_4$, we have either $(r,s)=(3,-3)$ or $(-3,3)$, 
which are transformed to each other by a group automorphism. 

We conjecture that the pair 
$$(S^{(G_o,\overline{q_{1,o}})}\otimes S,T^{(G_o,\overline{q_{1,o}})}\otimes T)$$
with $\theta_{2,e}=\tau$ is the modular data of the Drinfeld center of a near-group category for $\Z_2\times \Z_2\times G_o$ with multiplicity 
$4|G_o|$. \\

\paragraph{Assume $q_{1,e}(x_1,x_2,y_1,y_2)=(-1)^{x_1y_2+x_2y_1}$} Note that we have $q_0(k)=1$ in this case. 
Theorem \ref{list} shows that only $\Gamma_e=\Z_{2^{m}}\times \Z_{2^{n}}$ with $3\leq m\leq n$ is possible for $\Gamma_e$, 
and we have
$$\Gamma_e^{\theta_{2,e}}=\{0, (2^{m-1},0),(0,2^{n-1}),(2^{m-1},2^{n-1})\}.$$
Since $|G_o|=2^{m+n-4}|\Gamma_o|-1$, we have $|G_o|\equiv 3 \mod 4$. 
Thus $c=\cG(q_{1,o})=-\cG(q_{2,e})\cG(q_{2,o})$ implies $\cG(q_{2,e})^2\cG(q_{2,o})^2=-1$. 
We have the following three possibilities for $(q_{2,e},\theta_{2,e})$:
\begin{itemize}
\item $q_{2,e}(x,y)=\zeta_{2^{m+1}}^{rx^2}\zeta_{2^{n+1}}^{sy^2}$ with $r,s\in \{1,-1,3,-3\}$, and 
$\theta_{2,e}(x,y)=(-x,-y)$ or $\theta_{2,e}(x,y)=(-x+2^{m-1}y,2^{n-1}x-y)$. 
We denote the latter by $\tau'(x,y)$. 
\item $m=n$, $q_{2,e}(x,y)=\zeta_{2^m}^{xy}$, and $\theta=-1$ or $\theta=-1+2^{m-1}$. 
\item $m=n$, $q_{2,e}(x,y)=\zeta_{2^m}^{x^2+xy+y^2}$, and $\theta=-1$ or $\theta=-1+2^{m-1}$.
\end{itemize}
If $|\Gamma_o|\equiv 1\mod 4$, only the first of these is possible, and we have
$$\cG(q_{1,o})+(-1)^{m\frac{r^2-1}{8}}(-1)^{n\frac{s^2-1}{8}}\zeta_8^{r+s}\cG(q_{2,o})=0.$$

Direct computation shows that $|\cG(q_1,q_2,l,2)|^2=4$ for $l\neq 0$, and 
$$|\cG(q_1,q_2,0,2)|^2
=|4-2(-1)^{\frac{r^2-1}{8}}(-1)^{\frac{s^2-1}{8}}(-1)^{\frac{|G_o|^2-1}{8}}(-1)^{\frac{|\Gamma_o|^2-1}{8}}|^2,$$
which is possible only if 
$$(-1)^{\frac{r^2-1}{8}}(-1)^{\frac{s^2-1}{8}}(-1)^{\frac{|G_o|^2-1}{8}}(-1)^{\frac{|\Gamma_o|^2-1}{8}}=1.$$
Assuming this condition, we see that (FS2) is satisfied only for $\theta_{2,e}=\tau'$. 

For the smallest example $G=\Z_2^4\times \Z_3$, we have $m=n=3$ and $\Gamma_o=\{0\}$.  
Up to group automorphism, there are two possibilities: $(\cG(q_{1,o}),r,s)=\pm (i,3,-1).$

We conjecture that the pair 
$$(S^{(G_o,\overline{q_{1,o}})}\otimes S,T^{(G_o,\overline{q_{1,o}})}\otimes T)$$
with $\theta_{2,e}=\tau'$ is the modular data of the Drinfeld center of a near-group category for $\Z_2\times \Z_2\times G_o$ with multiplicity 
$4|G_o|$. \\

\paragraph{Other quadratic forms}There are various other $\theta_{1,e}$-invariant non-degenerate quadratic forms on $\Z_2^4$ (see Theorem \ref{list}). 
\subsubsection{Assume $G_e=\Z_{2^n}\times \Z_{2^n}$ with $n\geq 3$ and $\theta_{1,e}(x,y)=(y,x)$.} 
In this case we have $U\cong \Z_{2^n}$ and $|\Gamma|=2^{2n}+2^{n+2}|G_o|$, and hence 
$|\Gamma_e|=2^{n+2}$ and $|\Gamma_o|=|G_o|+2^{n-2}$. 
Let $\Gamma_e=\Z_{2^{n+1}}\times \Z_2$, let 
$$q_{2,e}(x,y)=\zeta_{2^{n+2}}^{sx^2}\zeta_4^{-(1+2^{n-2})sy^2}$$
with odd $s$, and let $\theta_{2,e}(x,y)=((1+2^{n-1})x+2^n y,x+y)$. 
Then $\theta_{2,e}$ is an involution preserving $q_{2,e}$, and
$$\Gamma_e^{\theta_{2,e}}=<(2,1)>\cong \Z_{2^n}.$$
We have $q_{2,e}(2,1)=\zeta_{2^n}^{s(1-2^{n-2}-2^{2n-4})}$ and
$$\cG(q_{2,e})=i^{-2^{n-3}s}(-1)^{\frac{2^{n-3}(2^{n-3}-1)}{2}}(-1)^{n\frac{s^2-1}{8}}.$$
Thus $\cG(q_{2,e})\in \{i,-i\}$ for $n=3$, and $\cG(q_{2,e})\in \{1,-1\}$ for $n\geq 4$. 
Again we consider different quadratic forms.\\

\paragraph{Assume $q_{1,e}(x,y)=\zeta_{2^n}^{rxy}$ with $r\in \{1,-1,3,-3\}$}  \label{caseabove} 
We set 
$$s=(1+2^{n-2})r+2^nt,$$
with $t\in \{0,1,2,3\}$. 
Then 
$$s(1-2^{n-2}-2^{2n-4})\equiv r\mod 2^n,$$
and $q_{1,e}(l,l)=q_{2,e}(2l,l)$.  
Since the automorphism $(x,y)\mapsto ((1+2^n)x,y)$ of $\Z_{2^{n+1}}\times \Z_2$ commutes with $\theta_{2,e}$ and 
transforms $t$ to $t+2$, we need to consider only $t=0$ and $t=1$.  
We have 
$$\cG(q_{1,o})=-i^{-2^{n-3}s}(-1)^{\frac{2^{n-3}(2^{n-3}-1)}{2}}(-1)^{n\frac{s^2-1}{8}}\cG(q_{2,o}).$$
Since $|\Gamma_e|=2^{n-2}|G_o|+1$, and $|\Gamma_o|\equiv |G_o|+2\mod 4$ for $n=3$ and $|\Gamma_o|\equiv |G_o|$ for $n\geq 4$, 
we have $|G_o|\equiv 1\mod 4$, and $|\Gamma_o|\equiv 3 \mod 4$ for $n=3$ and $|\Gamma_o|\equiv 1\mod 4$ for $n\geq 4$. 

Direct computation shows that (FS2) gives no restriction. 

We give a few examples. 
\begin{itemize}
\item[(i)] $G_o=\{0\}$ with $n=3$. 
In this case we have $\Gamma_o=\Z_3$. 
There are four possibilities up to group automorphism. We have $\cG(q_{2,o})=-(-1)^{\frac{r^2-1}{8}}i^r$. 
\item[(ii)] $G_o=\{0\}$ with $n=4$. In this case we have $\Gamma_o=\Z_5$. 
There are four possibilities, and we have $\cG(q_{2,o})=-1$. 
\end{itemize}

We conjecture that the pair 
$$(S^{(G_o,\overline{q_{1,o}})}\otimes S,T^{(G_o,\overline{q_{1,o}})}\otimes T).$$
is the modular data of the Drinfeld center of a near-group category for $\Z_{2^n}\times G_o$ 
with multiplicity $2^n|G_o|$.\\ 

\paragraph{Assume $q_{1,e}(x,y)=\zeta_{2^n}^{r(x^2+xy+y^2)}$}This case can be treated in a similar way as the previous case. \\

\paragraph {Assume $q_{1,e}(x,y)=\zeta_{2^{n+1}}^{r(x^2+y^2)}$ with $r\in \{1,-1,3,-3\}$}
We set 
$$s=(1+2^{n-2})r+2^nt,$$
with $t\in \{0,1,2,3\}$. 
Then 
$$s(1-2^{n-2}-2^{2n-4})\equiv r\mod 2^n,$$
and $q_{1,e}(l,l)=q_{2,e}(2l,l)$.  
Since the automorphism $(x,y)\mapsto ((1+2^n)x,y)$ of $\Z_{2^{n+1}}\times \Z_2$ commutes with $\theta_{2,e}$ and 
transforms $t$ to $t+2$, again we only need to consider $t=0$ and $t=1$.  
Since $\cG(q_{1,e})=i^r$, we have 
$$\cG(q_{1,o})
=-i^{-(1+2^{n-3}+2^{2n-5})r}(-1)^{\frac{2^{n-3}(2^{n-3}-1)}{2}}(-1)^{n\frac{(1+2^{n-2})^2r^2-1}{8}}\cG(q_{2,o}).$$
Thus $|\Gamma_o|\equiv |G_o|\mod 4$ for $n=3$ and $|\Gamma_o|\equiv |G_o|+2\mod 4$ for $n\geq 4$. 
Now $|\Gamma_o|=2^{n-2}|G_o|+1$ implies that $|G_o|\equiv 3\mod 4$, and $|\Gamma_o|\equiv 3\mod 4$ for $n=3$ and 
$|\Gamma_o|\equiv 1\mod 4$ for $n\geq 4$. 
The rest of computation is similar to \ref{caseabove} above.
\subsubsection{Assume $G_e=\Z_{2^n}$ with $n\geq 3$, and $\theta_{1,e}(x)=-x$.} 
In this case $U\cong \Z_2$, and $|\Gamma|=2^n|G_o|+2^3=2^3(2^{n-3}|G_o|+1)$. 
Let $q_{1,e}(x)=\zeta_{2^{n+1}}^{rx^2}$ with $r\in \{1,-1,3,-3\}$. 
We have $\cG(q_{1,e})=\zeta_8^r(-1)^{n\frac{r^2-1}{8}}$, and  $q_1(2^{n-1})=\zeta_{2^{n+1}}^{r2^{2n-2}}=1$. 
We consider the cases $n=3$ and $n>3 $ separately.\\

\paragraph{Assume n=3} In this case $G_e=\Z_8$, $\Gamma_e=\Z_{2^m}$, and $|G_o|=2^{m-3}|\Gamma_o|-1$ with $m\geq 4$. 
Assume $q_{2,e}(x)=\zeta_{2^{m+1}}^{sx^2}$ with $s\in \{1,-1,3,-3\}$ and $\theta_2(x)=-x$. 
Then $\cG(q_{1,e})=(-1)^{\frac{r^2-1}{8}}\zeta_8^r$ and $\cG(q_{2,e})=(-1)^{m\frac{s^2-1}{8}}\zeta_8^s$. 
Thus 
$$\cG(q_{2,o})=-(-1)^{\frac{r^2-1}{8}}(-1)^{m\frac{s^2-1}{8}}\zeta_8^{r-s}\cG(q_{1,o}).$$ 
Direct computation shows that (FS2) gives no restriction. 

For the smallest example $G=\Z_8$, $\Gamma=\Z_{16}$, 
there are four possibilities for $(r,s)$: $(1,-3)$, $(-1,3)$, $(3,3)$, and $(-3,-3)$.  \\

\paragraph{Assume $n > 3$} In this case $|\Gamma_e|=8$ and $|\Gamma_o|=2^{n-3}|G_o|+1$.  
Assume $\Gamma_e=\Z_8$, $q_{2,e}(x)=\zeta_{16}^{sx^2}$, and $\theta_{2,e}(x)=-x$. 
Then $\cG(q_1)=(-1)^{n\frac{r^2-1}{8}}\zeta_8^r\cG(q_{1,o})$ and $\cG(q_2)=(-1)^{\frac{s^2-1}{8}}\zeta_8^s\cG(q_{2,o})$, 
and hence 
$$ (-1)^{n\frac{r^2-1}{8}}\zeta_8^r\cG(q_{1,o})+(-1)^{\frac{s^2-1}{8}}\zeta_8^s\cG(q_{2,o})=0.$$
Direct computation shows that (FS2) gives no restriction. 

For the smallest example $G=\Z_{16}$, $\Gamma=\Z_8\times \Z_3$, 
there are four possibilities, we have $r-s\equiv 2 \mod 4$, and $\cG(q_{2,o})=(-1)^{\frac{s^2-1}{8}}\zeta_8^{r-s}$.

\section{Near-group over $\Z_2$ family}
\subsection{General formulae} 
Throughout this section we assume that $(G,q_1)$ and $(\Gamma,q_2)$ are metric groups satisfying 
$G_2\cong \Gamma_2\cong \Z_2$ and $c:=\cG(q_1)=-\cG(q_2)$.   
We denote $G_2=\{0,g_0\}$, $\Gamma_2=\{0,\gamma_0\}$, and assume that $\inpr{g_0}{g_0}=-\inpr{\gamma_0}{\gamma_0}$ 
and $c^2=q_1(g_0)^{-1}q_2(\gamma_0)$. 
We set 
$$s=c\frac{q_1(g_0)^{-1}+q_2(\gamma_0)^{-1}}{\sqrt{2}}.$$ 
Since $q_1(g_0)^2=\inpr{g_0}{g_0}$ and $q_2(\gamma_0)^2=\inpr{\gamma_0}{\gamma_0}$, we get
$$s^2=c^2\frac{\inpr{g_0}{g_0}+2q_1(g_0)^{-1}q_2(\gamma_0)^{-1}+\inpr{\gamma_0}{\gamma_0}}{2}=q_1(g_0)^{-2}=\inpr{g_0}{g_0}.$$

We set $a=1/\sqrt{|G|}$, $b=1/\sqrt{|\Gamma|}$. 
We choose and fix subsets $G_*\subset G$ and $\Gamma_*\subset \Gamma$ satisfying 
$$G=G_2\sqcup G_*\sqcup -G_*,\quad \Gamma=\Gamma_2\sqcup \Gamma_*\sqcup -\Gamma_*.$$
We set 
$$J=\{0,\pi\}\sqcup (\{g_0\}\times \{1,-1\})\sqcup G_*\sqcup (\{\gamma_0\}\times \{1,-1\})\sqcup \Gamma_*.$$
We denote $J_1=\{g_0\}\times \{1,-1\}$ and $J_2=\{\gamma_0\}\times \{1,-1\}$. 
We use letters, $g,g',g'',h\cdots$ for elements of $G$, and $\gamma,\gamma',\gamma'',\xi\cdots$ for elements of $\Gamma$. 

We introduce an involution of $J$ by setting $\overline{(g_0,\varepsilon)}=(g_0,\inpr{g_0}{g_0}\varepsilon)$, 
$\overline{(\gamma_0,\varepsilon)}=(\gamma_0,\inpr{g_0}{g_0}\varepsilon)$, and leaving the other indices fixed. 

\begin{definition}\label{ST4} Under the above assumptions, we define $J$ by $J$ matrices $S$, $T$, and $C$ by  
\begin{align*}
S=&\tiny\left(
\begin{array}{cccccc}
\frac{a-b}{2}&\frac{a+b}{2}&\frac{a}{2}&a&\frac{b}{2}&b  \\
\frac{a+b}{2}&\frac{a-b}{2}&\frac{a}{2}&a&-\frac{b}{2}&-b\\
\frac{a}{2}&\frac{a}{2}&(\frac{a}{2}+\frac{\varepsilon\varepsilon's}{2\sqrt{2}})\inpr{g_0}{g_0}&a\inpr{g_0}{g'}&\frac{\varepsilon\varepsilon's}{2\sqrt{2}}&0  \\
a&a&a\inpr{g}{g_0}&a(\inpr{g}{g'}+\overline{\inpr{g}{g'}})&0&0  \\
\frac{b}{2}&-\frac{b}{2}&\frac{\varepsilon\varepsilon' s}{2\sqrt{2}}&0&
-(\frac{b}{2}-\frac{\varepsilon\varepsilon's}{2\sqrt{2}})\inpr{\gamma_0}{\gamma_0} &-b\inpr{\gamma_0}{\gamma'}  \\
b&-b&0&0&-b\inpr{\gamma}{\gamma_0} &-b(\inpr{\gamma}{\gamma'}+\overline{\inpr{\gamma}{\gamma'}}) 
\end{array}
\right),
\end{align*}
$$T=\mathrm{Diag}(1,1,q_1(g_0),q_1(g),q_2(\gamma_0),q_2(\gamma)),$$
and $C_{j,j'}=\delta_{\overline{j},j'}=\delta_{j,\overline{j'}}$ (where the six blocks in $S$ and $T$ are indexed by $\{0\}$,$\{\pi \}$,$J_1 $, $G_*$, $J_2$, and $\Gamma_*$, respectively).  
\end{definition}

Direct computation shows the following.
\begin{lemma}\label{relations4} Let the notation be as in Definition \ref{ST4}. 
Then $S$, $T$, and $C$ are unitary matrices satisfying $S^2=C$, $(ST)^3=cC$ and $T C=CT$. 
\end{lemma}

\begin{lemma}\label{Verlinde4} Let the notation be as in Definition \ref{ST4}. Then we have 
$$N_{\pi,\pi,\pi}=\frac{4}{|\Gamma|-|G|},\; 
N_{\pi,\pi,(g_0,\varepsilon)}=\frac{2}{|\Gamma|-|G|},\;
N_{\pi,\pi,g}=\frac{4}{|\Gamma|-|G|},$$
$$N_{\pi,\pi,(\gamma_0,\varepsilon)}=\frac{2}{|\Gamma|-|G|},\;
N_{\pi,\pi,\gamma}=\frac{4}{|\Gamma|-|G|},$$
$$N_{\pi,(g_0,\varepsilon),(g_0,\varepsilon')}=\frac{1}{|\Gamma|-|G|}+\frac{1}{2},\;
N_{\pi,(g_0,\varepsilon),g}=\frac{2}{|\Gamma|-|G|},$$
$$N_{\pi,(g_0,\varepsilon),(\gamma_0,\varepsilon')}
=\frac{1}{|\Gamma|-|G|}+\frac{\varepsilon\varepsilon'}{2}, \:
N_{\pi,(g_0,\varepsilon),\gamma}=\frac{2}{|\Gamma|-|G|},$$
$$N_{\pi,g,g'}=\frac{4}{|\Gamma|-|G|}+\delta_{g,g'},\;
N_{\pi,g,(\gamma_0,\varepsilon)}=\frac{2}{|\Gamma|-|G|},\;
N_{\pi,g,\gamma}=\frac{4}{|\Gamma|-|G|},$$
$$N_{\pi,(\gamma_0,\varepsilon),(\gamma_0,\varepsilon')}=\frac{1}{|\Gamma|-|G|}-\frac{1}{2},\;
N_{\pi,(\gamma_0,\varepsilon),\gamma}=\frac{2}{|\Gamma|-|G|},\; 
N_{\pi,\gamma,\gamma'} =\frac{4}{|\Gamma|-|G|}-\delta_{\gamma,\gamma'},$$
$$N_{(g_0,\varepsilon),(g_0,\varepsilon'),(g_0',\varepsilon'')}
=\frac{1}{2(|\Gamma|-|G|)}
 +\frac{\varepsilon\varepsilon'+\varepsilon'\varepsilon''+\varepsilon''\varepsilon}{4},$$
$$N_{(g_0,\varepsilon),(g_0,\varepsilon'),g} 
=\frac{1}{|\Gamma|-|G|}+\frac{\varepsilon\varepsilon'\inpr{g_0+g}{g_0}}{2},$$
$$N_{(g_0,\varepsilon),(g_0,\varepsilon'),(\gamma_0,\varepsilon'')}
 =\frac{1}{2(|\Gamma|-|G|)}
 +\frac{\varepsilon''(\varepsilon+\varepsilon')\inpr{g_0}{g_0}+\varepsilon \varepsilon'}{4},$$
$$N_{(g_0,\varepsilon),(g_0,\varepsilon'),\gamma}
 =\frac{1}{|\Gamma|-|G|}-\frac{\varepsilon\varepsilon' \inpr{g_0}{g_0}\inpr{\gamma_0}{\gamma}}{2},$$
$$N_{(g_0,\varepsilon),g,g'}
 =\frac{2}{|\Gamma|-|G|}+\delta_{g_0+g+g',0}+\delta_{g_0+g-g',0},$$
$$N_{(g_0,\varepsilon),g,(\gamma_0,\varepsilon')}
=\frac{1}{|\Gamma|-|G|}+\frac{\varepsilon\varepsilon'\inpr{g}{g_0}}{2},\;
N_{(g_0,\varepsilon),g,\gamma}=\frac{2}{|\Gamma|-|G|},$$
$$N_{(g_0,\varepsilon''),(\gamma_0,\varepsilon),(\gamma_0,\varepsilon')}
=\frac{1}{2(|\Gamma|-|G|)}+\frac{\varepsilon\varepsilon'-\varepsilon''(\varepsilon+\varepsilon')\inpr{g_0}{g_0}}{4},$$
$$N_{(g_0,\varepsilon),\gamma,\gamma'}=\frac{2}{|\Gamma|-|G|},$$
$$N_{(g_0,\varepsilon),(\gamma_0,\varepsilon'),\gamma}
  =\frac{1}{|\Gamma|-|G|}+\frac{\varepsilon\varepsilon'\inpr{\gamma}{\gamma_0}}{2},$$
$$N_{g,g',g''}=\frac{4}{|\Gamma|-|G|} 
 +\delta_{g+g'+g'',0}+\delta_{-g+g'+g'',0}
 +\delta_{g-g'+g'',0}+\delta_{g+g'-g'',0},$$
$$N_{g,g',(\gamma_0,\varepsilon)}
 =\frac{2}{|\Gamma|-|G|},\;
N_{g,g',\gamma}=\frac{4}{|\Gamma|-|G|},$$
$$N_{g,(\gamma_0,\varepsilon),(\gamma_0,\varepsilon')}
 =\frac{1}{|\Gamma|-|G|}+\frac{\varepsilon\varepsilon'\inpr{g_0+g}{g_0} }{2},$$
 $$N_{g,(\gamma_0,\varepsilon),\gamma}=\frac{2}{|\Gamma|-|G|},\;
 N_{g,\gamma,\gamma'}=\frac{4}{|\Gamma|-|G|},$$
$$N_{(\gamma_0,\varepsilon),(\gamma_0,\varepsilon'),(\gamma_0,\varepsilon'')}
=\frac{1}{2(|\Gamma|-|G|)}+\frac{\varepsilon\varepsilon'+\varepsilon'\varepsilon''+\varepsilon''\varepsilon}{4},$$
$$N_{(\gamma_0,\varepsilon),(\gamma_0,\varepsilon'),\gamma}
=\frac{1}{|\Gamma|-|G|}+\frac{\varepsilon\varepsilon'\inpr{\gamma_0+\gamma}{\gamma_0}}{2},$$
$$N_{(\gamma_0,\varepsilon),\gamma,\gamma'}
=\frac{2}{|\Gamma|-|G|}-\delta_{\gamma_0+\gamma+\gamma',0}-\delta_{\gamma_0+\gamma-\gamma',0},$$
$$N_{\gamma,\gamma',\gamma''}
=\frac{4}{|\Gamma|-|G|}
 -\delta_{\gamma+\gamma'+\gamma'',0}-\delta_{-\gamma+\gamma'+\gamma'',0}
 -\delta_{\gamma-\gamma'+\gamma'',0}-\delta_{\gamma+\gamma'-\gamma'',0}.$$
\end{lemma}

\begin{theorem} Let the notation be as in Definition \ref{ST4}. 
Then all the fusion coefficients $N_{ijk}$ are non-negative integers if and only if $|\Gamma|-|G|=2$. 
\end{theorem}

We assume that $|\Gamma|=|G|+2$ in the rest of this section. 

\begin{conjecture}\label{con4} Let $A$ be an odd abelian group and let $G=\Z_{2^n}\times A$ with $n \geq 1 $. 
Then the pair
$$(S^{(G,\overline{q_1})}\otimes S,T^{(G,\overline{q_1})}\otimes T)$$
is the modular data of the Drinfeld center of the $\Z_2$-de-equivariantization of a near-group category for 
$\Z_{2^{n+1}}\times A$ with multiplicity $2^{n+1}|A|$. 
\end{conjecture} 

The conjecture is true for $n=1$ and trivial $A$ (see the online appendix). 
More generally, we have the following. 

\begin{conjecture} Let $\cC$ be a quadratic category of type $(\Z_2,Q,1)$ with non-self-dual $\rho$ and with the $\text{Vec}_{\Z_2}$ subcategory having 
non-trivial associator. 
Then the modular data of the Drinfeld center $\cZ(\cC)$ is given by 
$$(S^{G,\overline{q_1}}\otimes S,T^{(G,\overline{q_1})}\otimes T)$$
with $G=\Z_2\times Q$. 
\end{conjecture}

Note that we have 
$$N_{\pi,\pi,\pi}=2, \quad N_{(g_0,\varepsilon),(g_0,\varepsilon),(g_0,\varepsilon)}=1,\quad 
N_{g,g,g}=2+\delta_{3g,0},$$
$$N_{(\gamma_0,\varepsilon),(\gamma_0,\varepsilon),(\gamma_0,\varepsilon)}=1,\quad  
N_{\gamma,\gamma,\gamma}=2-\delta_{3\gamma,0}.$$ 
Direct computation shows the following.
\begin{lemma}\label{FS4} We have
$$\nu_m(\pi)=\frac{|\cG(q_1,m)+\cG(q_2,m)|^2}{2},$$
$$\nu_m((g_0,\varepsilon))=\frac{|\cG(q_1,m)+\cG(q_2,m)|^2}{4}+\frac{\delta_{mg_0,0}q_1(g_0)^m}{2},$$
$$\nu_m(g)=\frac{|\cG(q_1,m)+\cG(q_2,m)|^2}{2}+\delta_{mg,0}q_1(g)^m,$$
$$\nu_m((\gamma_0,\varepsilon))=\frac{|\cG(q_1,m)+\cG(q_2,m)|^2}{4}-\frac{\delta_{m\gamma_0,0}q_1(\gamma_0)^m}{2},$$
$$\nu_m(\gamma)=\frac{|\cG(q_1,m)+\cG(q_2,m)|^2}{2}-\delta_{m\gamma,0}q_2(\gamma)^m.$$
In particular, (FS2) and (FS3) are equivalent to the following conditions, respectively:  
\begin{equation}\label{C1}
|\cG(q_1,2)+\cG(q_2,2)|=\sqrt{2},
\end{equation}
\begin{equation}\label{C2}
|\cG(q_1,3)+\cG(q_2,3)|=2.
\end{equation}
\end{lemma}

\subsection{Examples}
We give a list of metric groups $(G,q_1)$, $(\Gamma,q_2)$ satisfying $G_2=\{0,g_0\}\cong \Gamma_2=\{0,\gamma_0\}\cong \Z_2$, 
$|\Gamma|=|G|+2$, $c:=\cG(q_1)=-\cG(q_2)$, $c^2=q_1(g_0)^{-1}q_2(\gamma_0)$, 
$\inpr{g_0}{g_0}=-\inpr{\gamma_0}{\gamma_0}$, and Eq.(\ref{C1})-(\ref{C2}). 

A similar computation as in the proof of Lemma \ref{nu3computation3} shows that 
Eq.(\ref{C1}) is equivalent to 
$$|\frac{\cG(q_{1,e},2)}{\cG(q_{1,e})}(-1)^{\frac{|G_o|^2-1}{8}}-\frac{\cG(q_{2,e},2)}{\cG(q_{2,e})}(-1)^{\frac{|\Gamma_o|^2-1}{8}}|=\sqrt{2},$$
and Eq.(\ref{C2}) is equivalent to 
$$|(-1)^{n_1}\frac{\cG(q_{1,o},3)}{\cG(q_{1,o})^3}-(-1)^{n_2}\frac{\cG(q_{2,o},3)}{\cG(q_{2,o})^3}|=2,$$
where $|G_e|=2^{n_1}$ and $|\Gamma_e|=2^{n_2}$. 
Moreover, if neither $G$ nor $\Gamma$ has a 3-component, Eq.(\ref{C2}) is further equivalent to 
$$|(-1)^{n_1}(\frac{|G_o|}{3})-(-1)^{n_2}(\frac{|\Gamma_o|}{3})|=2.$$
Direct computation shows that the restriction coming from (FS3) in the following examples is very similar to 
that discussed in Lemma \ref{restriction} and Remark \ref{3rank2}, and we will not mention it in what follows. 

We consider separately the cases $G_e=\Z_2 $ and $G_e=\Z_{2^n} $ with $n > 1 $.

\subsubsection{Assume $G_e=\Z_2$}
In this case there exists $n\geq 2$ satisfying $\Gamma_e=\Z_{2^n}$ and $|G_o|=2^{n-1}|\Gamma_o|-1$. 
We have $n=2$ for $|G_o|\equiv 1 \mod 4$, and $n\geq 3$ for $|G_o|\equiv 3\mod 4$. 
We may assume $q_{1,e}(x)=i^{r_1x^2}$ with $r_1\in \{1,-1\}$ and $q_{2,e}(x)=\zeta_{2^{n+1}}^{r_2x^2}$ with $r_2\in \{1,-1,3,-3\}$. 
Then $\cG(q_{1,e})=\zeta_8^{r_1}$, $\cG(q_{2,e})=\zeta_8^{r_2}(-1)^{n\frac{r_2^2-1}{8}}$, and
$$c=\zeta_8^{r_1}\cG(q_{1,o})=-\zeta_8^{r_2}(-1)^{n\frac{r_2^2-1}{8}}\cG(q_{2,o}).$$ 
Since $\cG(q_{1,e},2)=0$ and $\cG(q_{2,e},2)=\sqrt{2}\zeta_8^{r_2}(-1)^{(n-1)\frac{r_2^2-1}{8}}$, 
Eq.(\ref{C1}) always holds. 

For the smallest example $G=\Z_2$, $\Gamma=\Z_4$, there are two possibilities: $(r_1,r_2)=(\pm 1,\mp 3)$. 
These two modular data  
$$(S^{(\Z_2,\overline{q_1})}\otimes S,T^{(\Z_2,\overline{q_1})}\otimes T),$$
are the modular data of the Drinfeld center of the two categories constructed in \cite{MR3635673}, \cite{MR3306607} 
(see the online appendix for a proof of this).

\subsubsection{Assume $G_e=\Z_{2^n}$ with $n\geq 2$.}
In this case we have $\Gamma_e=\Z_2$ and $|\Gamma_o|=2^{n-1}|G_o|+1$. 
Thus $|\Gamma_o|\equiv 3\mod 4$ for $n=2$, and $|\Gamma_o|\equiv 1\mod 4$ for $n\geq 3$. 
We may assume $q_{1,e}(x)=\zeta_{2^{n+1}}^{r_1x^2}$ with $r_1\in \{1,-1,3,-3\}$ and $q_{2,e}(x)=i^{r_2x^2}$. 
Then $\cG(q_{1,e})=\zeta_8^{r_1}(-1)^{n\frac{r_1^2-1}{8}}$, $\cG(q_{2,e})=\zeta_8^{r_2}$, and
$$c=\zeta_8^{r_1}(-1)^{n\frac{r_1^2-1}{8}}\cG(q_{1,o})=-\zeta_8^{r_2}\cG(q_{2,o}).$$
As above, Eq.(\ref{C1}) always holds. 
For the smallest example $G=\Z_4$, $\Gamma=\Z_6$, 
there are four possibilities determined by $\cG(q_{2,o})=-\zeta_8^{r_1-r_2}$ and $r_1-r_2\equiv 2\mod 4$. 
\section{Even generalized Haagerup family}
\subsection{General formulae}
Throughout this section we assume that $(G,q_1,\theta_1)$ and $(\Gamma,q_2,\theta_1)$ are involutive metric groups satisfying the following conditions: 
$$c:=\cG(q_1)=-\cG(q_2), \quad  
| G^{\theta_1}|=|\Gamma^{\theta_2}|=4,$$ 
and there exist order 2 elements $k_0\in G^{\theta_1}\setminus \{0\}$, $\sigma_0\in \Gamma^{\theta_2}\setminus\{0\}$ with $q_1(k_0)=q_2(\sigma_0)$. 
We identify $\{0,k_0\}$ and $\{0,\sigma_0\}$, and denote them by $U=\{0,u_0\}$, which is regarded as a common subgroup 
of $G^{\theta_1}$ and $\Gamma^{\theta_2}$. 
We denote by $q_0$ the restriction of $q_1$ (and $q_2$ too) to $U$. 
We denote by $\inpr{\cdot}{\cdot}_i$ the bicharacter associated with $q_i$, 
and we often suppress subscript $i$ in $\inpr{\cdot}{\cdot}_i$ and $\theta_i$ if there is no possibility of confusion. 

We denote $K_*:=\Gamma_2\setminus \{0,k_0\}=\{k_1,k_2\}$ and $\Sigma_*:=\Gamma_2\setminus\{0,\sigma_0\}=\{\sigma_1,\sigma_2\}$. 
We assume that one of the following holds:\\
\begin{itemize}
\item[(A1)] $q_1(k_1)=q_1(k_2)$ and $q_2(\sigma_1)=-q_2(\sigma_2)$. In this case, we have 
$$\inpr{k_0}{k_1}=\inpr{k_0}{k_2}=q_0(u_0)^{-1} \text{ and } \inpr{\sigma_0}{\sigma_1}=\inpr{\sigma_0}{\sigma_2}=-q_0(u_0)^{-1}.$$ 
We set $s=c/q_1(k_1)$, and assume $s^4=1$.\\  
\item[(A2)] $q_1(k_1)=-q_1(k_2)$ and $q_2(\sigma_1)=q_2(\sigma_2)$. In this case, we have 
$$\inpr{k_0}{k_1}=\inpr{k_0}{k_2}=-q_0(u_0)^{-1} \text{ and }\inpr{\sigma_0}{\sigma_1}=\inpr{\sigma_0}{\sigma_2}=q_0(u_0)^{-1}.$$ 
We set $s=c/q_2(\sigma_1)$, and assume $s^4=1$.\\
\end{itemize}
In either case, we have 
\begin{equation}\label{E1}
\inpr{u_0}{u_0}=1,
\end{equation}
\begin{equation}\label{E2}
\inpr{k_1}{k_1}=\inpr{k_2}{k_2},\quad \inpr{\sigma_1}{\sigma_1}=\inpr{\sigma_2}{\sigma_2},
\end{equation}
\begin{equation}\label{E3}
q_0(u_0)\in \{1,-1\}.
\end{equation}
\begin{equation}\label{E4}
s(q_1(k_1)+q_1(k_2)+q_2(\sigma_1)+q_2(\sigma_2))=2c.
\end{equation}
Indeed, since $k_0+k_1=k_2$, Eq.(\ref{E1}) is an easy consequence of the assumptions (A1) and (A2) above. 
Since $\inpr{g}{-g}=q_1(0)/q_1(g)^2$, we have $q_1(g)^2=\inpr{g}{g}$, and so Eq.(\ref{E2}) and Eq.(\ref{E3}) hold. 
Eq.(\ref{E4}) is obvious. 

Later, we will assume 
\begin{equation}\label{E5}
(1+\inpr{u_0}{k})(1-s^2\inpr{k}{k})=0,\quad \forall k\in K_*,
\end{equation}
and 
\begin{equation}\label{E6}
(1+\inpr{u_0}{\sigma})(1+s^2\inpr{\sigma}{\sigma})=0,\quad \forall \sigma\in \Sigma_*,
\end{equation}
too, but for the moment we do not assume these. 

We set $a=1/\sqrt{|G|}$, $b=1/\sqrt{|\Gamma|}$. 
We choose and fix subsets $G_*\subset G$ and $\Gamma_*\subset \Gamma$ satisfying 
$$G=G_2\sqcup G_*\sqcup \theta_1(G_*),\quad \Gamma=\Gamma_2\sqcup \Gamma_*\sqcup \theta_2(\Gamma_*),$$
and set
$$J=(U\times\{0,\pi\})\sqcup (K_*\times \{1,-1\})\sqcup G_*\sqcup (\Sigma_*\times \{1,-1\})\sqcup \Gamma_*.$$
We denote $J_1=K_*\times \{1,-1\}$ and $J_2=\Sigma_*\times \{1,-1\}$. 
We use letters $u,u',u'',v,\cdots$ for elements in $U$, $k,k',k'',l\cdots$ for elements of $K_*$, 
$g,g',g'',h\cdots$ for elements of $G$, $\sigma,\sigma'\sigma'',\tau,\cdots$ for elements in $\Sigma_*$, 
and $\gamma,\gamma',\gamma'',\xi\cdots$ for elements of $\Gamma$. 

We choose a function $f:K_*\times \Sigma_*\to \{s,-s\}$ satisfying 
$f(k+k_0,\sigma)=f(k,\sigma)\overline{\inpr{u_0}{\sigma}}$, $f(k,\sigma+\sigma_0)=f(k,\sigma)\overline{\inpr{k}{u_0}}$. 
Such a choice is possible thanks to $q_1(k_0)=q_2(\sigma_0)$. 
(We may assume $f(k_1,\sigma_1)=s$ if it is convenient.) 
We have $f(k,\sigma)^2=s^2$ for any choice of $f$. 
We introduce an involution of $J$ by setting $\overline{(u,0)}=(u,0)$, $\overline{(u,\pi)}=(u,\pi)$, 
$\overline{(k,\varepsilon)}=(-k,s^2\varepsilon)$, 
$\overline{(\sigma,\varepsilon)}=(-\sigma,s^2\varepsilon)$, and
$$\overline{g}=\left\{
\begin{array}{ll}
-g , &\quad \textrm{if $-g\in G_*$}  \\
\theta_1(-g) , &\quad \textrm{otherwise}
\end{array}
\right.,
$$
$$\overline{\gamma}=\left\{
\begin{array}{ll}
-\gamma , &\quad \textrm{if $-\gamma\in \Gamma_*$}  \\
\theta_2(-\gamma) , &\quad \textrm{otherwise}
\end{array}
\right..
$$

\begin{definition}\label{ST5} Under the above assumptions, we define $J$ by $J$ matrices $S$, $T$, and $C$ by  
\begin{align*}
S=&\tiny\left(
\begin{array}{cccccc}
\frac{a-b}{2}\overline{\inpr{u}{u'}}&\frac{a+b}{2}\overline{\inpr{u}{u'}}&\frac{a}{2}\overline{\inpr{u}{k'}}&a\overline{\inpr{u}{g'}}&\frac{b}{2}\overline{\inpr{u}{\sigma'}}&b\overline{\inpr{u}{\gamma'}}  \\
\frac{a+b}{2}\overline{\inpr{u}{u'}}&\frac{a-b}{2}\overline{\inpr{u}{u'}}&\frac{a}{2}\overline{\inpr{u}{k'}}&a\overline{\inpr{u}{g'}}&-\frac{b}{2}\overline{\inpr{u}{\sigma'}}&-b\overline{\inpr{u}{\gamma'}}\\
\frac{a}{2}\overline{\inpr{k}{u'}}&\frac{a}{2}\overline{\inpr{k}{u'}}&(\frac{a}{2}+\frac{\varepsilon\varepsilon's}{4})\overline{\inpr{k}{k'}}&a\overline{\inpr{k}{g'}}&\frac{\varepsilon\varepsilon'f(k,\sigma')}{4}&0  \\
a\overline{\inpr{g}{u'}}&a\overline{\inpr{g}{u'}}&a\overline{\inpr{g}{k'}}&a(\overline{\inpr{g}{g'}}+\overline{\inpr{g}{\theta(g')}})&0&0  \\
\frac{b}{2}\overline{\inpr{\sigma}{u'}}&-\frac{b}{2}\overline{\inpr{\sigma}{u'}}&\frac{\varepsilon\varepsilon'f(k',\sigma)}{4}&0&
-(\frac{b}{2}+\frac{-\varepsilon\varepsilon's}{4})\overline{\inpr{\sigma}{\sigma'}} &-b\overline{\inpr{\sigma}{\gamma'}}  \\
b\overline{\inpr{\gamma}{u'}} &-b\overline{\inpr{\gamma}{u'}}&0&0&-b\overline{\inpr{\gamma}{\sigma'}} &-b(\overline{\inpr{\gamma}{\gamma'}}+\overline{\inpr{\gamma}{\theta(\gamma')}}) 
\end{array}
\right),
\end{align*}
$$T=\mathrm{Diag}(q_0(u),q_1(k),q_1(g),q_2(\sigma),q_2(\gamma)),$$
and $C_{j,j'}=\delta_{\overline{j},j'}=\delta_{j,\overline{j'}}$ (where the six blocks in $S$ and $T$ are indexed by $U \times \{ 0\} $, $U \times \{ \pi \} $, $J_1$, $G_* $, $J_2$, and $\Gamma_* $,  respectively. ) 
\end{definition}

\begin{remark} As an alternative to assumptions (A1) and (A2) above, we could consider a third possibility corresponding to $\inpr{u_0}{u_0}=-1$ together with an appropriate condition. 
In this case the quadratic form $q_0$ is non-degenerate and $S$ and $T$ factorize as 
$S=S^{(U,q_0)}\otimes S'$ and $T=T^{(U,q_0)}\otimes T'$. 
The pair $(S',T')$ arising in this way is none other than the one that we discussed in the previous section. It is more convenient to treat this case separately. 
\end{remark}

\begin{lemma} \label{eqns} Let the notation be as above. Then the following equations hold. 
\begin{itemize}
\item[$(1)$]
$$\sum_{l\in K_*}\inpr{k'-k}{l}+\sum_{\sigma\in \Sigma_*}f(k,\sigma)\overline{f(k',\sigma)}=4\delta_{k,k'}.$$
\item[$(2)$]
$$\sum_{\tau\in \Sigma_*}\inpr{\sigma'-\sigma}{\tau}+\sum_{k\in K_*}f(k,\sigma)\overline{f(k,\sigma')}=4\delta_{\sigma,\sigma'},$$
\item[$(3)$]
$$s\sum_{l\in K_*}\overline{\inpr{k}{l}}\overline{f(l,\sigma)}+\overline{s}\sum_{\tau\in \Sigma_*}\inpr{\sigma}{\tau}f(k,\tau)=0.$$
\item[$(4)$]
$$s^2\sum_{l\in K_*}\overline{\inpr{k+k'}{l}}q_1(l)
 +\sum_{\sigma\in \Sigma_*}f(k,\sigma)f(k',\sigma) q_2(\sigma)=2cs\overline{\inpr{k}{k'}}\overline{q_1(k)q_1(k')}.$$
\item[$(5)$]
$$s(\sum_{l\in K_*}f(l,\sigma)\overline{\inpr{k}{l}}q_1(l)
 +\sum_{\tau\in \Sigma_*}f(k,\tau)\overline{\inpr{\sigma}{\tau}}q_2(\tau))=2cf(k,\sigma)\overline{q_1(k)q_2(\sigma)}.$$
\item[$(6)$]
$$\sum_{l\in K_*}f(l,\sigma)f(l,\sigma')q_1(l)
 +s^2\sum_{\tau \in \Sigma_*}\overline{\inpr{\sigma+\sigma'}{\tau}}q_2(\tau)=2cs\overline{\inpr{\sigma}{\sigma'}}\overline{q_2(\sigma)q_2(\sigma')}.$$
\end{itemize}
\end{lemma}

\begin{proof} (1) Let $u=k'-k\in U$. Then 
$$\sum_{l\in K_*}\inpr{k'-k}{l}+\sum_{\sigma\in \Sigma_*}f(k,\sigma)\overline{f(k',\sigma)}
=\sum_{l\in K_*}\inpr{u}{l}+\sum_{\sigma\in \Sigma_*}\inpr{u}{\sigma}=4\delta_{u,0}.$$
Eq. (2) can be shown in the same way. 

(3) Note that we have $s\overline{f(k,\sigma)}=\overline{s}f(k,\sigma)$. 
\begin{align*}
\lefteqn{s\sum_{l\in K_*}\overline{\inpr{k}{l}}\overline{f(l,\sigma)}+\overline{s}\sum_{\tau\in \Sigma_*}\inpr{\sigma}{\tau}f(k,\tau)} \\
 &=s\sum_{u\in U}\overline{\inpr{k}{k+u}}\overline{f(k+u,\sigma)}+\overline{s}\sum_{v\in U}\inpr{\sigma}{\sigma+v}f(k,\sigma+v) \\
 &=s\sum_{u\in U}\overline{\inpr{k}{k+u}} \overline{f(k,\sigma)} \inpr{u}{\sigma} +\overline{s}\sum_{v\in U}\inpr{\sigma}{\sigma+v}f(k,\sigma)\overline{\inpr{k}{v}}  \\
 &=s\overline{f(k,\sigma)}(\overline{\inpr{k}{k}}\sum_{u\in U}\overline{\inpr{k}{u}}\inpr{u}{\sigma}+\inpr{\sigma}{\sigma}\sum_{v\in U}\inpr{\sigma}{v}\overline{\inpr{k}{v}})\\
 &=s\overline{f(k,\sigma)}(\overline{\inpr{k}{k}}+\inpr{\sigma}{\sigma})(1+\inpr{k}{u_0}\inpr{\sigma}{u_0})=0.
\end{align*}

(4) Let $u=k'-k$. 
\begin{align*}
\lefteqn{s^2\sum_{l\in K_*}\overline{\inpr{k+k'}{l}}q_1(l)
 +\sum_{\sigma\in \Sigma_*}f(k,\sigma)f(k',\sigma) q_2(\sigma)} \\
 &=s^2\sum_{l\in K_*}q_1(k+k'-l)\overline{q_1(k+k')}
 +\sum_{\sigma\in \Sigma_*}f(k,\sigma)^2\overline{\inpr{u}{\sigma}} q_2(\sigma) \\
 &=s^2(\overline{\inpr{k}{k'}q_1(k)q_1(k')}\sum_{l\in K_*}q_1(l)+\sum_{\sigma\in \Sigma_*}q_2(u-\sigma)\overline{q_2(u)}) \\
 &=\overline{\inpr{k}{k'}q_1(k)q_1(k')}s^2(\sum_{l\in K_*}q_1(l)+\inpr{k}{2k'}\sum_{\sigma\in \Sigma_*}q_2(\sigma))\\
 &=\overline{\inpr{k}{k'}q_1(k)q_1(k')}s^2(q_1(k_1)+q_2(k_2)+\inpr{k}{2k'}(q_2(\sigma_1)+q_2(\sigma_2))).
\end{align*}
If $G^{\theta_1}\subset G_2$, the statement follows from Eq.(\ref{E4}). 
If not, we have $G^{\theta_1}\cong \Z_4$, and $k_1=-k_2$. 
Hence $q_1(k_1)=q_1(k_2)$ and $q_2(\sigma_1)+q_2(\sigma_2)=0$, 
and the statement holds too. Eq. (6) can be shown in the same way.  

(5) We have:
\begin{align*}
\lefteqn{s(\sum_{l\in K_*}f(l,\sigma)\overline{\inpr{k}{l}}q_1(l)
 +\sum_{\tau\in \Sigma_*}f(k,\tau)\overline{\inpr{\sigma}{\tau}}q_2(\tau))} \\
 &=s(\sum_{l\in K_*}f(l,\sigma)q_1(l-k)\overline{q_1(k)}
 +\sum_{\tau\in \Sigma_*}f(k,\tau)    q_2(\sigma-\tau)\overline{q_2(\sigma)})\\
 &=s(\overline{q_1(k)}\sum_{u\in U}f(k+u,\sigma)q_1(u)
 +\overline{q_2(\sigma)}\sum_{v\in U}f(k,\sigma+v)q_2(v)) \\
 &=sf(k,\sigma)(\overline{q_1(k)}\sum_{u\in U} \inpr{u}{-\sigma}  q_0(u)
 +\overline{q_2(\sigma)}\sum_{v\in U}\inpr{-k}{v}q_0(v))\\ 
 &=sf(k,\sigma)(\overline{q_1(k)}\sum_{u\in U}   q_2(u-\sigma)\overline{q_2(\sigma)}
 +\overline{q_2(\sigma)}\sum_{v\in U}  q_1(v-k)\overline{q_1(k)})\\
 &=sf(k,\sigma)\overline{q_1(k)q_2(\sigma)}(q_1(k_1)+q_1(k_2)+q_2(\sigma_1)+q_2(\sigma_2))\\
 &=2cf(k,\sigma)\overline{q_1(k)q_2(\sigma)}. 
\end{align*}
\end{proof}

Using Lemma \ref{eqns}, we can show the following by direct computation. 

\begin{lemma}\label{relations5} Let the notation be as in Definition \ref{ST5}. 
Then $S$, $T$, and $C$ are unitary matrices satisfying $S^2=C$, $(ST)^3=cC$ and $T C=CT$. 
\end{lemma}

\begin{lemma}\label{Verlinde5} Let the notation be as in Definition \ref{ST5}. Then we have
$$N_{(u,0),(u',0),(u'',0)}=N_{(u,0),(u',\pi),(u'',\pi)}=\delta_{u+u'+u'',0},$$
$$N_{(u,0),(u',0),(u'',\pi)}=N_{(u,0),(u',0),(k,\varepsilon)}=N_{(u,0),(u',0),g}=N_{(u,0),(u',0),(\sigma,0)}
=N_{(u,0),(u',0),\gamma}=0,$$
$$N_{(u,0),(u',\pi,),(k,\varepsilon)}=N_{(u,0),(u',\pi,),g}=N_{(u,0),(u',\pi,),(\sigma,\varepsilon)}
=N_{(u,0),(u',\pi),\gamma}=0,$$
$$N_{(u,0),(k,\varepsilon),(k',\varepsilon')}=\delta_{u+k+k',0}\delta_{\varepsilon,\varepsilon'},\quad N_{(u,0),g,g'}=\delta_{u+g+g',0}+\delta_{u+g-g',0},$$
$$N_{(u,0),(\sigma,\varepsilon),(\sigma',\varepsilon')}=\delta_{u+\sigma+\sigma',0}\delta_{\varepsilon,\varepsilon'},
\quad N_{(u,0),\gamma,\gamma'}=\delta_{u+\gamma+\gamma',0}+\delta_{u+\gamma-\gamma',0},$$
$$N_{(u,0),(k,\varepsilon),g}=N_{(u,0),(k,\varepsilon),(\sigma,\varepsilon')}=N_{(u,0),(k,\epsilon),\gamma}
=N_{(u,0),g,(\sigma,\varepsilon)}=N_{(u,0),g,\gamma}=N_{(u,0),(\sigma,\varepsilon),\gamma}=0,$$ 
$$N_{(u,\pi),(u',\pi),(u'',\pi)}
=\frac{8}{|\Gamma|-|G|},$$ 
$$N_{(u,\pi),(u',\pi),(k,\varepsilon)}
=\frac{2}{|\Gamma|-|G|}(1+\inpr{k}{u_0}),$$ 
$$N_{(u,\pi),(u',\pi),g}
=\frac{4}{|\Gamma|-|G|}(1+\inpr{g}{u_0}),$$
$$N_{(u,\pi),(u',\pi),(\sigma,\varepsilon)}
=\frac{2}{|\Gamma|-|G|}(1+\inpr{\sigma}{u_0}),$$ 
$$N_{(u,\pi),(u',\pi),\gamma}
=\frac{4}{|\Gamma|-|G|}(1+\inpr{\gamma}{u_0}),$$
\begin{align*}
\lefteqn{N_{(u,\pi),(k,\varepsilon),(k,\varepsilon')}} \\
 &=\frac{2}{|\Gamma|-|G|}+\frac{\delta_{u+k+k',0}}{2}+
 \frac{\varepsilon\varepsilon's^2}{4}(\inpr{u+k+k'}{k_1}-\inpr{u+k+k'}{\sigma_1}\inpr{2k}{\sigma_1}).
\end{align*}
$$N_{(u,\pi),(k,\varepsilon),g} =\frac{2}{|\Gamma|-|G|}(1+\inpr{k+g}{u_0}),$$
$$N_{(u,\pi),(k,\varepsilon),(\sigma,\varepsilon')}=0,$$
$$N_{(u,\pi),(k,\varepsilon),\gamma}
=\frac{2}{|\Gamma|-|G|}(1+\inpr{k}{u_0}\inpr{\gamma}{u_0}),$$
$$N_{(u,\pi),g,g'}=
\frac{4}{|\Gamma|-|G|}(1+\inpr{g+g'}{u_0}) +\delta_{u+g+g',0}+\delta_{u+g+\theta(g'),0},$$
$$N_{(u,\pi),g,(\sigma,\varepsilon)}=\frac{2}{|\Gamma|-|G|}(1+\inpr{g}{u_0}\inpr{\sigma}{u_0}),$$
$$N_{(u,\pi),g,\gamma}=\frac{4}{|\Gamma|-|G|}(1+\inpr{g}{u_0}\inpr{\gamma}{u_0}),$$
\begin{align*}
\lefteqn{N_{(u,\pi),(\sigma,\varepsilon),(\sigma',\varepsilon')}} \\
 &=\frac{2}{|\Gamma|-|G|}-\frac{\delta_{u+\sigma+\sigma',0}}{2}
 +\frac{\varepsilon\varepsilon' s^2}{4}(\inpr{u+\sigma+\sigma'}{k_1}\inpr{2\sigma}{k_1}-\inpr{u+\sigma+\sigma'}{\sigma}),
\end{align*}
$$N_{(u,\pi),(\sigma,\varepsilon),\gamma}
=\frac{2}{|\Gamma|-|G|}(1+\inpr{\sigma+\gamma}{u_0}),$$
$$N_{(u,\pi),\gamma,\gamma'}
 =\frac{4}{|\Gamma|-|G|}(1+\inpr{\gamma+\gamma'}{u_0})
 -\delta_{u+\gamma+\gamma',0}-\delta_{u+\gamma+\theta(\gamma'),0},$$
\begin{align*}
\lefteqn{N_{(k,\varepsilon),(k',\varepsilon'),(k'',\varepsilon'')}} \\
&=(1+\inpr{u_0}{k+k'+k''})(\frac{1}{2(|\Gamma|-|G|)}
 +\frac{s^2(\varepsilon\varepsilon'+\varepsilon'\varepsilon''+\varepsilon''\varepsilon)}{8}\overline{\inpr{k+k'+k''}{k+k'+k''}}),\\
\end{align*}
$$
N_{(k,\varepsilon),(k',\varepsilon'),g}
=(\frac{1}{|\Gamma|-|G|}+\frac{s^2\varepsilon\varepsilon'\overline{\inpr{k}{k+k'+g}}}{4})(1+\inpr{g}{u_0}),$$
\begin{align*}
\lefteqn{N_{(k,\varepsilon),(k',\varepsilon'),(\sigma,\varepsilon'')}} \\
 &=(1+\inpr{\sigma}{u_0})(\frac{1}{2(|\Gamma|-|G|)}+\frac{\varepsilon''(\varepsilon+ \varepsilon')s}{8}\inpr{k+k'}{k}f(k,\sigma)
 -\frac{s^2\varepsilon\varepsilon'}{8}\inpr{k-k'}{\sigma}\overline{\inpr{\sigma}{\sigma}}),
\end{align*}
$$N_{(k,\varepsilon),(k',\varepsilon),\gamma}
 =(1+\inpr{\gamma}{u_0})(\frac{1}{|\Gamma|-|G|}-\frac{\varepsilon\varepsilon's^2}{4}\inpr{k-k'}{\sigma_1}\overline{\inpr{\gamma}{\sigma_1}}),$$
$$N_{(k,\varepsilon),g,g'}=\frac{2}{|\Gamma|-|G|}(1+\inpr{k+g+g'}{u_0})+\delta_{k+g+g',0}+\delta_{k+g+\theta(g'),0},$$ 
$$N_{(k,\varepsilon),g,(\sigma,\varepsilon')}
=(\frac{1}{|\Gamma|-|G|}+\frac{\varepsilon\varepsilon' s f(k,\sigma)\overline{\inpr{k+g}{k}}}{4})(1-\inpr{g}{u_0}),$$
$$N_{(k,\varepsilon),g,\gamma}
=\frac{2}{|\Gamma|-|G|}(1+\inpr{k+g}{u_0}\inpr{\gamma}{u_0}),$$
\begin{align*}
\lefteqn{N_{(k,\varepsilon''),(\sigma,\varepsilon),(\sigma',\varepsilon')}} \\
 &=(\frac{1}{2(|\Gamma|-|G|)}+\frac{\varepsilon\varepsilon's^2\inpr{\sigma-\sigma'}{k}\overline{\inpr{k}{k}}}{8}  
 -\frac{s\varepsilon''(\varepsilon+\varepsilon')f(k,\sigma)\overline{\inpr{\sigma+\sigma'}{\sigma}}}{8}) (1+\inpr{k}{u_0}),
\end{align*}
$$N_{(k,\varepsilon),(\sigma,\varepsilon'),\gamma}
 =(\frac{1}{|\Gamma|-|G|}-\frac{\varepsilon \varepsilon' s f(k,\sigma)\overline{\inpr{\sigma+\gamma}{\sigma}}}{4})(1-\inpr{\gamma}{u_0}),$$
$$N_{(k,\varepsilon),\gamma,\gamma'}
 =\frac{2}{|\Gamma|-|G|}(1+\inpr{k}{u_0}\inpr{\gamma+\gamma'}{u_0}),$$
\begin{align*}
N_{g,g',g''}&= \frac{4}{|\Gamma|-|G|}(1+\inpr{g+g'+g''}{u_0}) \\
 &+\delta_{g+g'+g'',0}+\delta_{\theta(g)+g'+g'',0}
 +\delta_{g+\theta(g')+g'',0}+\delta_{g+g'+\theta(g''),0},
\end{align*}
$$N_{g,g',(\sigma,\varepsilon)}
 =\frac{2}{|\Gamma|-|G|}(1+\inpr{g+g'}{u_0}\inpr{\sigma}{u_0}),$$
$$N_{g,g',\gamma}=\frac{4}{|\Gamma|-|G|}(1+\inpr{g+g'}{u_0}\inpr{\gamma}{u_0}),$$
$$N_{g,(\sigma,\varepsilon),(\sigma',\varepsilon')}
=(\frac{1}{|\Gamma|-|G|}+\frac{\varepsilon\varepsilon's^2\overline{\inpr{g}{k_1}}\inpr{k_1}{\sigma-\sigma'} }{4})(1+\inpr{g}{u_0}).$$
$$N_{g,(\sigma,\varepsilon),\gamma}=\frac{2}{|\Gamma|-|G|}(1+\inpr{g}{u_0}\inpr{\sigma+\gamma}{u_0}),$$
$$N_{g,\gamma,\gamma'}=\frac{4}{|\Gamma|-|G|}(1+ \inpr{g}{u_0}\inpr{\gamma+\gamma'}{u_0}),$$
\begin{align*}
\lefteqn{N_{(\sigma,\varepsilon),(\sigma',\varepsilon'),(\sigma'',\varepsilon'')}} \\
 &=(1+\inpr{\sigma+\sigma'+\sigma''}{u_0})(\frac{1}{2(|\Gamma|-|G|)}- \frac{s^2(\varepsilon\varepsilon'+\varepsilon'\varepsilon''+\varepsilon''\varepsilon)}{8}
 \overline{\inpr{\sigma+\sigma'+\sigma''}{\sigma+\sigma'+\sigma''}}),
\end{align*}
$$N_{(\sigma,\varepsilon),(\sigma',\varepsilon'),\gamma}
 =(\frac{1}{|\Gamma|-|G|}-\frac{\varepsilon\varepsilon's^2\overline{\inpr{\sigma+\sigma'+\gamma}{\sigma}}}{4})(1+\inpr{\sigma+\sigma'+\gamma}{u_0}),$$
$$N_{(\sigma,\varepsilon),\gamma,\gamma'}
=\frac{2}{|\Gamma|-|G|}(1+\inpr{\sigma+\gamma+\gamma'}{u_0})-\delta_{\sigma+\gamma+\gamma',0}-\delta_{\sigma+\gamma+\theta(\gamma'),0},$$
\begin{align*}
N_{\gamma,\gamma',\gamma''} &=\frac{4}{|\Gamma|-|G|}(1+\inpr{\gamma+\gamma'+\gamma''}{u_0}) \\
 &-\delta_{\gamma+\gamma'+\gamma'',0}-\delta_{\theta(\gamma)+\gamma'+\gamma'',0}
 -\delta_{\gamma+\theta(\gamma')+\gamma'',0}-\delta_{\gamma+\gamma'+\theta(\gamma''),0}.
\end{align*}
\end{lemma}

From the above computations (in particular from the formulas for
$N_{(k,\varepsilon),(k',\varepsilon'),(k",\varepsilon'')}$ and $N_{(\sigma,\varepsilon),(\sigma',\varepsilon'),(\sigma",\varepsilon'')}$), we find the following.

\begin{theorem}\label{STtheorem5} Under the assumptions of Definition \ref{ST5}, all the fusion coefficients are non-negative integers if and only if the following five conditions are satisfied.

\begin{itemize}
\item[$(1)$] $|\Gamma|=|G|+4$.
\item[$(2)$] $(1+\inpr{u_0}{k})(1-s^2\inpr{k}{k})=0$ for all $k\in K_*$.
\item[$(3)$] $(1+\inpr{u_0}{\sigma})(1+s^2\inpr{\sigma}{\sigma})=0$ for all $\sigma\in \Sigma_*$.
\item[$(4)$] If $G^{\theta_1}\cong \Z_4$, we have $q_0(u_0)=-1$.
\item[$(5)$] If $\Gamma^{\theta_2}\cong \Z_4$, we have $q_0(u_0)=-1$.
\end{itemize}
Under these conditions, the fusion coefficients are 
$$N_{(u,0),(u',0),(u'',0)}=N_{(u,0),(u',\pi),(u'',\pi)}=\delta_{u+u'+u'',0},$$
$$N_{(u,0),(u',0),(u'',\pi)}=N_{(u,0),(u',0),(k,\varepsilon)}=N_{(u,0),(u',0),g}=N_{(u,0),(u',0),(\sigma,0)}
=N_{(u,0),(u',0),\gamma}=0,$$
$$N_{(u,0),(u',\pi,),(k,\varepsilon)}=N_{(u,0),(u',\pi,),g}=N_{(u,0),(u',\pi,),(\sigma,\varepsilon)}
=N_{(u,0),(u',\pi),\gamma}=0,$$
$$N_{(u,0),(k,\varepsilon),(k',\varepsilon')}=\delta_{u+k+k',0}\delta_{\varepsilon,\varepsilon'},\quad N_{(u,0),g,g'}=\delta_{u+g+g',0}+\delta_{u+g-g',0},$$
$$N_{(u,0),(\sigma,\varepsilon),(\sigma',\varepsilon')}=\delta_{u+\sigma+\sigma',0}\delta_{\varepsilon,\varepsilon'},
\quad N_{(u,0),\gamma,\gamma'}=\delta_{u+\gamma+\gamma',0}+\delta_{u+\gamma-\gamma',0},$$
$$N_{(u,0),(k,\varepsilon),g}=N_{(u,0),(k,\varepsilon),(\sigma,\varepsilon')}=N_{(u,0),(k,\epsilon),\gamma}
=N_{(u,0),g,(\sigma,\varepsilon)}=N_{(u,0),g,\gamma}=N_{(u,0),(\sigma,\varepsilon),\gamma}=0,$$ 
$$N_{(u,\pi),(u',\pi),(u'',\pi)}=2, \quad 
N_{(u,\pi),(u',\pi),(k,\varepsilon)}
=\frac{1+\inpr{k}{u_0}}{2},\quad  
N_{(u,\pi),(u',\pi),g}=1+\inpr{g}{u_0},$$
$$N_{(u,\pi),(u',\pi),(\sigma,\varepsilon)}
=\frac{1+\inpr{\sigma}{u_0}}{2},\quad 
N_{(u,\pi),(u',\pi),\gamma}
=1+\inpr{\gamma}{u_0},$$
$$N_{(u,\pi),(k,\varepsilon),(k,\varepsilon')}
 =\frac{1+\delta_{u+k+k',0}}{2}+
 \frac{\varepsilon\varepsilon's^2}{4}(\inpr{u+k+k'}{k_1}-\inpr{u+k+k'}{\sigma_1}),$$
$$N_{(u,\pi),(k,\varepsilon),g} =\frac{1+\inpr{k+g}{u_0}}{2},\quad 
N_{(u,\pi),(k,\varepsilon),(\sigma,\varepsilon')}=0,\quad 
N_{(u,\pi),(k,\varepsilon),\gamma}=\frac{1+\inpr{k}{u_0}\inpr{\gamma}{u_0}}{2},$$
$$N_{(u,\pi),g,g'}=1+\inpr{g+g'}{u_0} +\delta_{u+g+g',0}+\delta_{u+g+\theta(g'),0},$$
$$N_{(u,\pi),g,(\sigma,\varepsilon)}=\frac{1+\inpr{g}{u_0}\inpr{\sigma}{u_0}}{2},\quad 
N_{(u,\pi),g,\gamma}=1+\inpr{g}{u_0}\inpr{\gamma}{u_0},$$
$$N_{(u,\pi),(\sigma,\varepsilon),(\sigma',\varepsilon')}=\frac{1-\delta_{u+\sigma+\sigma',0}}{2}
 +\frac{\varepsilon\varepsilon' s^2}{4}(\inpr{u+\sigma+\sigma'}{k_1}-\inpr{u+\sigma+\sigma'}{\sigma},$$
$$N_{(u,\pi),(\sigma,\varepsilon),\gamma}=\frac{1+\inpr{\sigma+\gamma}{u_0}}{2},\quad 
N_{(u,\pi),\gamma,\gamma'}
 =1+\inpr{\gamma+\gamma'}{u_0} -\delta_{u+\gamma+\gamma',0}-\delta_{u+\gamma+\theta(\gamma'),0},$$
$$N_{(k,\varepsilon),(k',\varepsilon'),(k'',\varepsilon'')}
=\frac{(1+\inpr{u_0}{k+k'+k''})(1+\varepsilon\varepsilon'+\varepsilon'\varepsilon''+\varepsilon''\varepsilon)}{8},$$
$$
N_{(k,\varepsilon),(k',\varepsilon'),g}
=\frac{(1+s^2\varepsilon\varepsilon'\overline{\inpr{k}{k+k'+g}})(1+\inpr{g}{u_0})}{4},$$
$$N_{(k,\varepsilon),(k',\varepsilon'),(\sigma,\varepsilon'')}
=\frac{(1+\inpr{\sigma}{u_0})(1+\varepsilon\varepsilon'+\varepsilon''(\varepsilon+ \varepsilon')\inpr{k+k'}{k}sf(k,\sigma))}{8},$$
$$N_{(k,\varepsilon),(k',\varepsilon),\gamma}
 =\frac{(1+\inpr{\gamma}{u_0})(1-\varepsilon\varepsilon's^2\inpr{k-k'}{\sigma_1}\overline{\inpr{\gamma}{\sigma_1}})}{4},$$
$$N_{(k,\varepsilon),g,g'}=\frac{1+\inpr{k+g+g'}{u_0}}{2}+\delta_{k+g+g',0}+\delta_{k+g+\theta(g'),0},$$ 
$$N_{(k,\varepsilon),g,(\sigma,\varepsilon')}
=\frac{(1+\varepsilon\varepsilon' s f(k,\sigma)\overline{\inpr{k+g}{k}})(1-\inpr{g}{u_0})}{4},$$
$$N_{(k,\varepsilon),g,\gamma}=\frac{1+\inpr{k+g}{u_0}\inpr{\gamma}{u_0}}{2},$$
$$N_{(k,\varepsilon''),(\sigma,\varepsilon),(\sigma',\varepsilon')}
=\frac{(1+\varepsilon\varepsilon'   -\varepsilon''(\varepsilon+\varepsilon')sf(k,\sigma)\overline{\inpr{\sigma+\sigma'}{\sigma}}) (1+\inpr{k}{u_0})}{8},$$
$$N_{(k,\varepsilon),(\sigma,\varepsilon'),\gamma}
 =\frac{(1-\varepsilon \varepsilon' s f(k,\sigma)\overline{\inpr{\sigma+\gamma}{\sigma}})(1-\inpr{\gamma}{u_0})}{4},$$
$$N_{(k,\varepsilon),\gamma,\gamma'} =\frac{1+\inpr{k}{u_0}\inpr{\gamma+\gamma'}{u_0}}{2},$$
$$N_{g,g',g''}=1+\inpr{g+g'+g''}{u_0}+\delta_{g+g'+g'',0}+\delta_{\theta(g)+g'+g'',0}
 +\delta_{g+\theta(g')+g'',0}+\delta_{g+g'+\theta(g''),0},$$
$$N_{g,g',(\sigma,\varepsilon)} =\frac{1+\inpr{g+g'}{u_0}\inpr{\sigma}{u_0}}{2},\quad
N_{g,g',\gamma}=1+\inpr{g+g'}{u_0}\inpr{\gamma}{u_0},$$
$$N_{g,(\sigma,\varepsilon),(\sigma',\varepsilon')}
=\frac{(1+\varepsilon\varepsilon's^2\overline{\inpr{g}{k_1}}\inpr{k_1}{\sigma-\sigma'} )(1+\inpr{g}{u_0})}{4},$$
$$N_{g,(\sigma,\varepsilon),\gamma}=\frac{1+\inpr{g}{u_0}\inpr{\sigma+\gamma}{u_0}}{2},\quad 
N_{g,\gamma,\gamma'}=1+ \inpr{g}{u_0}\inpr{\gamma+\gamma'}{u_0},$$
$$N_{(\sigma,\varepsilon),(\sigma',\varepsilon'),(\sigma'',\varepsilon'')}
=\frac{(1+\inpr{\sigma+\sigma'+\sigma''}{u_0})(1+ \varepsilon\varepsilon'+\varepsilon'\varepsilon''+\varepsilon''\varepsilon))}{8},$$
$$N_{(\sigma,\varepsilon),(\sigma',\varepsilon'),\gamma}
 =\frac{(1-\varepsilon\varepsilon's^2\overline{\inpr{\sigma+\sigma'+\gamma}{\sigma}})(1+\inpr{\sigma+\sigma'+\gamma}{u_0})}{4},$$
$$N_{(\sigma,\varepsilon),\gamma,\gamma'}
=\frac{1+\inpr{\sigma+\gamma+\gamma'}{u_0}}{2}-\delta_{\sigma+\gamma+\gamma',0}-\delta_{\sigma+\gamma+\theta(\gamma'),0},$$
$$N_{\gamma,\gamma',\gamma''}=1+\inpr{\gamma+\gamma'+\gamma''}{u_0} 
 -\delta_{\gamma+\gamma'+\gamma'',0}-\delta_{\theta(\gamma)+\gamma'+\gamma'',0}
 -\delta_{\gamma+\theta(\gamma')+\gamma'',0}-\delta_{\gamma+\gamma'+\theta(\gamma''),0}.$$
\end{theorem}

\begin{proof} Note that $|G|$ and $|\Gamma|$ are multiples of 4. 
Since $\inpr{k}{u_0}\inpr{\sigma}{u_0}=-1$, either $N_{(\pi,u),(\pi,u'),(k,\varepsilon)}$ or $N_{(\pi,u),(\pi,u'),(\sigma,\varepsilon)}$ is 
$\frac{4}{|\Gamma|-|G|}$, and so (1) is necessary for the fusion coefficients to be non-negative integers. 
We assume (1) for the rest of the proof. 

Necessity of (2) and (3) follows from the formulas for 
$N_{(k,\varepsilon),(k,\varepsilon),(-k,-\varepsilon)}$ and  $N_{(\sigma,\varepsilon),(\sigma,\varepsilon),(-\sigma,-\varepsilon)}$ 
respectively. 
For (4), we have 
$$N_{(u,\pi),(k,\varepsilon),(k',\varepsilon')}=\frac{1+\delta_{u+k+k'}}{2}+\frac{\varepsilon\varepsilon's^2}{4}(\inpr{u+k+k'}{k_1}-\inpr{u+k+k'}{\sigma_1}\inpr{2k}{\sigma}),$$
which is $1+\frac{\varepsilon\varepsilon's^2}{4}(1-\inpr{2k}{\sigma})$ if $u+k+k'=0$, and is 
$$\frac{1}{2}+\frac{\varepsilon\varepsilon's^2}{4}(\inpr{u_0}{k_1}-\inpr{u_0}{\sigma_1}\inpr{2k}{\sigma})
=\frac{1}{2}+\frac{\varepsilon\varepsilon's^2\inpr{u_0}{k_1}}{4}(1+\inpr{2k}{\sigma}),$$
if $u+k+k'=u_0$. 
In the both cases $N_{(u,\pi),(k,\varepsilon),(k',\varepsilon')}$ is a non-negative integer if and only if 
$\inpr{2k}{\sigma}=1$. 
If $G^{\theta_1}\cong \Z_2\times \Z_2$, this does not give any restriction. 
Assume $G^{\theta_1}\cong\Z_4$. 
Then we have $2k=u_0$.  
Since $k_1=-k_2$, we have $q_1(k_1)=q_2(k_2)$, and we are in the case (A1). 
Thus $\inpr{2k}{\sigma}=\inpr{u_0}{\sigma}=-1/q_0(u_0)$. 
This shows that  $\inpr{2k}{\sigma}=1$ if and only if $q_0(u_0)=-1$, and we get (4). 
Necessity of (5) follows from the formula for
$N_{(u,\pi),(\sigma,\varepsilon),(\sigma,\varepsilon')}$.  

Assume (1)-(5) on the other hand. 
Then it is easy to show that the fusion coefficients are non-negative integers except for 
$$N_{(k,\varepsilon),(k',\varepsilon'),(\sigma,\varepsilon'')},
N_{(k,\varepsilon),g,(\sigma,\varepsilon')},
N_{(k,\varepsilon),(\sigma,\varepsilon'),(\sigma',\varepsilon'')},
N_{(k,\varepsilon),(\sigma,\varepsilon'),\gamma}.$$
We show that the first two are non-negative integers, and the last two can be treated in the same way.  
We have 
$$N_{(k,\varepsilon),(k',\varepsilon'),(\sigma,\varepsilon'')}=(1+\inpr{u_0}{\sigma})\frac{1+\varepsilon''(\varepsilon+\varepsilon')sf(k,\sigma)\inpr{k+k'}{k}-\varepsilon \varepsilon'\inpr{k'-k}{\sigma}s^2\overline{\inpr{\sigma}{\sigma}}}{8}.$$
If $\inpr{u_0}{\sigma}=-1$, there is nothing to show, so we assume $\inpr{u_0}{\sigma}=1$. 
Then $\inpr{k'-k}{\sigma}=1$ and (3) implies that $s^2\overline{\inpr{\sigma}{\sigma}}=-1$. 
Thus 
\begin{align*}
&N_{(k,\varepsilon),(k',\varepsilon'),(\sigma,\varepsilon'')}=\frac{1+\varepsilon \varepsilon'+\varepsilon''(\varepsilon+\varepsilon')sf(k,\sigma)\inpr{k+k'}{k}}{4} \\&=\frac{(1+\varepsilon\varepsilon')(1+\varepsilon\varepsilon''sf(k,\sigma)\inpr{k+k'}{k})}{4}, 
\end{align*}
which is either $0$ or $1$. 

Since
$$N_{(k,\varepsilon),g,(\sigma,\varepsilon')}=\frac{(1-\inpr{g}{u_0})(1+\varepsilon\varepsilon'sf(k,\sigma)\overline{\inpr{k+g}{k}})}{4},$$
to show that this is a non-negative integer, it suffices to show that $\inpr{g}{u_0}=-1$ implies that $\inpr{k+g}{k}\in \{1,-1\}$, or equivalently $\inpr{k+g}{k}^2=1$. 
If $G^{\theta_1}\cong \Z_2\times \Z_2$, there is nothing to show. 
Assume $G^{\theta_1}\cong \Z_4$. 
Since $2k=u_0$ and $q_1(u_0)=-1$, the restriction of $q_1$ to $G^{\theta_1}$ is non-degenerate, and we get $\inpr{k}{k}^2=-1$. 
Thus $$\inpr{k+g}{k}^2=\inpr{k}{k}^2\inpr{g}{u_0}=1.$$ 
\end{proof}

In the rest of this section, we assume that the conditions in Theorem \ref{STtheorem5} are satisfied. 

\begin{lemma}\label{FS5} We have the following formulas for the Frobenius-Schur indicators:
$$\nu_m((u,0))=\delta_{mu,0}q_0(u)^m,$$
$$\nu_m((u,\pi))=\frac{1}{4}\sum_{v\in U}\inpr{u}{v}|\cG(q_1,q_2,v,m)|^2,$$
$$\nu_m((k,\varepsilon))=\frac{1}{8}\sum_{v\in U}\inpr{k}{v}|\cG(q_1,q_2,v,m)|^2+\frac{\delta_{mk,0}q_1(k)^m}{2},$$
$$\nu_m(g)=\frac{1}{4}\sum_{v\in U}\inpr{g}{v}|\cG(q_1,q_2,v,m)|^2+\delta_{mg,0}q_1(g)^m,$$
$$\nu_m((\sigma,\varepsilon))=\frac{1}{8}\sum_{v\in U}\inpr{\sigma}{v}|\cG(q_1,q_2,v,m)|^2-\frac{\delta_{m\sigma,0}q_2(\sigma)^m}{2},$$
$$\nu_m(\gamma)=\frac{1}{4}\sum_{v\in U}\inpr{\gamma}{v}|\cG(q_1,q_2,v,m)|^2-\delta_{m\gamma,0}q_1(\gamma)^m,$$
where 
$$\cG(q_1,q_2,v,m)=\frac{1}{\sqrt{|G|}}\sum_{g\in G}\inpr{v}{g}q_1(g)^m
+\frac{1}{\sqrt{|\Gamma|}}\sum_{\gamma\in \Gamma}\inpr{v}{\gamma}q_2(\gamma)^m.$$
In particular, (FS2) implies 
$$|\cG(q_1,q_2,0,2)|^2+|\cG(q_1,q_2,u_0,2)|^2=4.$$
\end{lemma}

\begin{proof} Direct computation gives the formulae for $\nu_m(x)$. 
Since $(0,\pi)$ is self-dual, we get the last formula. 
\end{proof}

Recall that we have 
$$N_{(u,\pi),(u,\pi),(u,\pi)}=2,\quad N_{(k,\varepsilon),(k,\varepsilon),(k,\varepsilon)}=\frac{1+\inpr{k}{u_0}}{2},$$
$$N_{g,g,g}=1+\inpr{g}{u_0}+\delta_{3g,0},\quad 
N_{(\sigma,\varepsilon),(\sigma,\varepsilon),(\sigma,\varepsilon)}=\frac{1+\inpr{\sigma}{u_0}}{2},$$
$$N_{\gamma,\gamma,\gamma}=1+\inpr{\gamma}{u_0}-\delta_{3\gamma,0}.$$

We can show the following lemma as in the proof of Lemma \ref{nu3computation3}. 

\begin{lemma}\label{nu3} Let the notation be as above. \\
$(1)$ If $m$ is an odd number,  then
$$|\cG(q_1,q_2,u,m)|=|\cG(q_1,m)+\cG(q_2,m)|=|\frac{\cG(q_1,m)}{\cG(q_1)^m}-\frac{\cG(q_2,m)}{\cG(q_2)^m}|.$$
$(2)$ The condition (FS3) is equivalent to  
$$|(-1)^{n_1}\frac{\cG(q_{1,o},3)}{\cG(q_{1,o})^3}-(-1)^{n_2}\frac{\cG(q_{2,o},3)}{\cG(q_{2,o})^3}|=2,$$
where $|G_e|=2^{n_1}$ and $|\Gamma_e|=2^{n_2}$. 
If neither $G$ nor $\Gamma$ has a 3-component, this is further equivalent to 
$$(-1)^{n_1+n_2}(\frac{|G_o|}{3})(\frac{|\Gamma_o|}{3})=-1.$$ 
$(3)$ We have
$$|\cG(q_1,q_2,u,2)|=|\frac{\cG(q_{1,e},u,2)}{\cG(q_{1,e})}(-1)^{\frac{|G_o|^2-1}{8}}-\frac{\cG(q_{2,e},u,2)}{\cG(q_{2,e})}(-1)^{\frac{|\Gamma_o|^2-1}{8}}|.$$
\end{lemma}

\subsection{Examples}
We give a list of examples satisfying the conditions in Definition \ref{ST5}, Theorem \ref{Verlinde5}, 
(FS2), and (FS3). 
Note that the restriction of $\theta_1$ (resp. $\theta_2$) to $G_o$ (resp. $\Gamma_o$) is always multiplying by $-1$. 
Direct computation using Lemma \ref{nu3} shows that the constraint coming from (FS3) in the following examples is very similar to 
that discussed in Lemma \ref{restriction} and Remark \ref{3rank2}, and we will not mention it in what follows. 

Recall that we have either $\inpr{u_0}{k_1}=\inpr{u_0}{k_2}=-1$ or $\inpr{u_0}{\sigma_1}=\inpr{u_0}{\sigma_2}=-1$. 

\begin{lemma}\label{nondegenerate} Let the notation be as above. \\
$(1)$ Assume $\inpr{u_0}{k_1}=\inpr{u_0}{k_2}=-1$. 
Then we have either $G_e=\Z_2\times \Z_2$ or $G_e=\Z_4$, and $|\Gamma_e|$ is a multiple of 8. 
The 2-rank of $\Gamma$ is either 2 or 4 for $G_e=\Z_2\times \Z_2$ and it is 3 for $G_e=\Z_4$. \\
$(2)$ Assume $\inpr{u_0}{\sigma_1}=\inpr{u_0}{\sigma_2}=-1$. 
Then we have either $\Gamma_e=\Z_2\times \Z_2$ or $\Gamma_e=\Z_4$, and $|G_e|$ is a multiple of 8. 
The 2-rank of $G$ is either 2 or 4 for $\Gamma_e=\Z_2\times \Z_2$ and it is 3 for $\Gamma_e=\Z_4$. \\
\end{lemma}

\begin{proof} (1) 
If $\inpr{u_0}{k_1}=\inpr{u_0}{k_2}=-1$, since we have $\inpr{k_1}{k_1}=\inpr{k_2}{k_2}$, 
the restriction of $q_1$ to $G^{\theta_1}$ is non-degenerate, and we have the factorization $G=G^{\theta_1}\times {G^{\theta_1}}^\perp$ 
as an involutive metric group. 
If ${G^{\theta_1}}^\perp$ were an even group, the restriction of $\theta_1$ to it would have a non-trivial fixed point, 
which is a contradiction. 
Thus we get $G_e=G^{\theta_1}$, and $G_e$ is either $\Z_2\times \Z_2$ or $\Z_4$. 
Since $|\Gamma|=4(|G_o|+1)$ in this case, the order of $\Gamma_e$ is a multiple of 8. 
Since $\cG(q_1)=-\cG(q_2)$, we have $\cG(q_{1,e})^4=\cG(q_{2,e})^4$, 
and the difference of the 2-rank of $G$ and that of $\Gamma$ is even. 
Part (2) can be shown in the same way. 
\end{proof}

\textbf{Division into cases:} Since there are a number of different cases to consider, we briefly outline the division. We consider four different general cases according to whether (A1) or (A2) holds, and whether $u_0$ is a fermion, i.e. $q_0(u_0)=-1$, or 
a boson, i.e. $q_0(u_0)=1$. These four cases are treated in subsections 5.2.1-5.2.4. Each of these four cases is further subdivided into two different subcases according to values of the quadratic forms on the fixed point subgroups of $G$ and $\Gamma $. These subdivisions are treated in paragraphs, e.g. 5.2.1.1 and 5.2.1.2. Finally, each of these subcases is further subdivided according to the forms of the even groups $G_e $ and $ \Gamma_e$ and the associated quadratic forms $q_{1,e} $ and $q_{2,e} $, as in previous sections. These further subcases are treated in subparagraphs, e.g. 5.2.1.1.1. 

\subsubsection{Assume (A1) and $q_0(u_0)=-1$.}  
We necessarily have $\inpr{u_0}{k_1}=\inpr{u_0}{k_2}=-1$ and $\inpr{u_0}{\sigma_1}=\inpr{u_0}{\sigma_2}=1$. 
Thus $|G_e|=4$ and the restriction of $\theta_1$ to $G_e$ is trivial. 
Since $s=c/q_1(k)$, Theorem \ref{STtheorem5} implies $s^2\inpr{\sigma}{\sigma}=-1$ for $\sigma \in \Sigma_*$, and so 
$c^2=-\inpr{k}{k}\inpr{\sigma}{\sigma}$ for $k\in K_*$ and $\sigma\in \Sigma_*$. 
Since 
$$\{q_1(k_0),q_1(k_1),q_1(k_2)\}=\{-1,q_1(k_1),q_1(k_1)\},$$
we have the following possibilities: 
\begin{itemize} 
\item If $G_e=\Z_2\times \Z_2$, we may assume $k_0=(1,1)$, $k_1=(1,0)$, and $k_2=(0,1)$. 
We have four possibilities for $q_{1,e}(x,y)$: 
(1) $(-1)^{xy}$, (2) $(-1)^{x^2+xy+y^2}$, (3) $i^{x^2+y^2}$, and (4) $i^{-(x^2+y^2)}$. 
\item If $G_e=\Z_4$, we may assume $k_1=1$, $k_0=2$, $k_2=3$, and $q_{1,e}(x)=\zeta_8^{rx^2}$ with $r\in \{1,-1,3,-3\}$. 
\end{itemize}

The set $\{q_1(\sigma_0),q_1(\sigma_1),q_1(\sigma_2)\}$ is either $\{-1,i,-i\}$, or $\{-1,1,-1\}$, and 
since $q_2(\sigma_1)\neq q_2(\sigma_2)$, the group $\Gamma^{\theta_2}$ is isomorphic to $\Z_2\times \Z_2$. 
Thus Theorem \ref{list} shows that the possibilities for $ \Gamma_e$ are as follows: 
\begin{itemize} 
\item $|G_o|\equiv 1 \mod 4$ and $\Gamma_e=\Z_2\times \Z_4$ or $\Gamma_e=\Z_2\times \Z_2\times \Z_2$. 
\item $|G_o|\equiv 3 \mod 4$ and $\Gamma_e=\Z_4\times \Z_{2^m}$ or $\Gamma_e=\Z_2\times \Z_2\times \Z_{2^m}$ with $m\geq 2$ 
or $\Z_2^4$. 
\end{itemize}
We now consider the two possibilities for $q_2(\Gamma^\theta) $ separately.\\

\paragraph{$q_2(\sigma_1)=i$ and $q_2(\sigma_2)=-i$} In this case, we have $s^2=1$, $c^2=\inpr{k}{k}$ for $k\in K_*$. 
We have $|G_o|\equiv 1 \mod 4$ and $|G_o|=2|\Gamma_o|-1$. 
More precisely, $|\Gamma_o|\equiv 1\mod 4$ if and only if $|G|\equiv 1\mod 8$, and $|\Gamma_o|\equiv 3\mod 4$ if and only if $|G_o|\equiv 5\mod 8$.  
We have either $G_e=\Z_2\times \Z_2$ and $\Gamma_e=\Z_2\times \Z_4$, or $G_e=\Z_4$ and $\Gamma_e=\Z_2^3$. 

If $\Gamma_e=\Z_2\times \Z_4$, we may assume $\sigma_0=(0,2)$, $\sigma_1=(1,0)$, $\sigma_2=(1,2)$, $q_{2,e}(x,y)=i^x\zeta_8^{ry^2}$ with 
$r\in \{1,-1,3,-3\}$, and $\theta=-1$ up to group automorphism. 
If $\Gamma_e=\Z_2^3$, we may assume $\sigma_0=(1,1,0)$, $\sigma_1=(0,0,1)$, $\sigma_2=(1,1,1)$, 
$q(x,y,z)=q'(x,y)i^{z^2}$, and $\theta(x,y,z)=(y,x,z)$, where $q'(x,y)$ is one of the following: $(-1)^{xy}$, $(-1)^{x^2+xy+y^2}$, 
$i^{\pm (x^2+y^2)}$.  

Since $\cG(q_{2,e},0,2)=\cG(q_{2,e},\sigma_0,2)=0,$ direct computation shows that (FS2) gives no restriction. 

We now consider the different possibilities for $G_e $ separately.\\

\subparagraph{Assume $G_e=\Z_2\times \Z_2$ and $q_{1,e}(x,y)=(-1)^{xy}$}
In this case we have $\cG(q_{1,e})=1$, $c^2=1$, and $c=\cG(q_{1,o})=-\zeta_8^{1+r}\cG(q_{2,o})$. 
We conjecture that the modular data of the Drinfeld center of a generalized Haagerup category for $\Z_2\times A$ with odd $A$ is 
$(S,T)$ with $G_e=\Z_2\times \Z_2$, $q_{1,e}(x,y)=(-1)^{xy}$, $G_o=A\times \hat{A}$, and $q_{1,o}(p,\chi)=\chi(p)$. 
There are at least two possibilities of $(\Gamma,q_2)$ as we will see in (i) below: 
$q_{2,e}(x,y)=i^{x^2}\zeta_8^{-y^2}$ with $\cG(q_{2,o})=-1$ and $q_{2,e}(x,y)=i^{x^2}\zeta_8^{3y^2}$ with $\cG(q_{2,o})=1$ 
(with the exception of $A=\{0\}$, where the first case does not occur).  
The conjecture is true for $A=\{0\}$, and numerical evidence is given for $\Z_3,\Z_5$ (see  \cite[Section 4.2]{GI19_1}). 
More generally, we conjecture that the pair 
$$(S^{(G_o,\overline{q_{1,o}})}\otimes S,T^{(G_o,\overline{q_{1,o}})}\otimes T)$$
is the modular data of the Drinfeld center $\cZ(\cC)$ of a quadratic category $\cC$ of type $(\Z_2,G_o,1)$ with self-dual $\rho$ 
such that $\alpha_1$ lifts to a fermion in $\cZ(\cC)$ .  
We can consider the following two possibilities:

(i) Assume $|G_o|\equiv 1 \mod 8$. In this case we have $|\Gamma_o|\equiv 1 \mod 4$, and $r \in \{-1,3 \}$. 
The smallest example is $G=\Z_2 \times \Z_2$ and $\Gamma=\Z_2\times \Z_4$, and the only possibility is $r=3$. This indeed comes from 
the Drinfeld center of the unique generalized Haagerup category for $\Z_2$, 
or equivalently the even part of the $A_7$ subfactor (see \cite[Example 4.2]{GI19_1}). 

(ii) Assume $|G_o|\equiv 5\mod 8$. 
In this case we have $|\Gamma_o|\equiv 3 \mod 4$, and $r \in \{1,-3 \}$.  
Thus $\cG(q_{1,o})/\cG(q_{2,o})\in \{i,-i\}$ and so $\cG(q_{2,e})=\{i,-i\}$. 
The smallest example is $G_o=\Z_5$ and $\Gamma_o=\Z_3$, and there are four possibilities. \\

\subparagraph {Assume $G_e=\Z_2\times \Z_2$ and $q_{1,e}(x,y)=(-1)^{x^2+xy+y^2}$} 
Then we have $\cG(q_{1,e})=-1$, and $c^2=1$, and $c=-\cG(q_{1,o})=-\zeta_8^{1+r}\cG(q_{2,o})$. 
We have $r \in \{-1,3\}$ for $|G_o|\equiv 1 \mod 8$, and $r\in \{1,-3 \}$ for $|G_o|\equiv 5\mod 8$.\\

\subparagraph{ Assume $G_e=\Z_2\times \Z_2$ and $q_{1,e}(x,y)=i^{x^2+y^2}$}
Then we have $\cG(q_{1,e})=i$,  $c^2=-1$, and $c=i\cG(q_{1,o})=-\zeta_8^{1+r}\cG(q_{2,o})$. 
We have $r=1,-3$ for $|G_o|\equiv 1 \mod 8$, and  $r=-1,3$ for $|G_o|\equiv 5\mod 8$. \\

\subparagraph{Assume $G_e=\Z_2\times \Z_2$ and $q_{1,e}(x,y)=i^{-(x^2+y^2)}$} This is the complex conjugate of the previous case. \\

\subparagraph{Assume $G_e=\Z_4$ and $q_{1,e}(x)=\zeta_8^{rx^2}$ with $r\in \{1,-1,3,-3\}$}
Then we have $\cG(q_{1,e})=\zeta_8^r$, $c^2=i^r$, and $c=\zeta_8^r\cG(q_{1,o})=-\cG(q')\zeta_8\cG(q_{2,o})$. 
We have $i^{r-1}=\cG(q')^2$ for $|G_o|\equiv 1 \mod 8$, and  $i^{r-1}=-\cG(q')^2$ for $|G_o|\equiv 5\mod 8$. 

\begin{remark} \label{ze1} In the above examples with $|G_e|=4$, fix two odd metric groups $(G_o,q_{1,o})$ and $(\Gamma_o,q_{2,o})$. 
Then we can show that all the above potential modular data $(S,T)$ can be obtained from that for $q_{1,e}(x,y)=(-1)^{xy}$ 
by applying the zesting construction introduced in \cite[Theorem 3.15]{MR3641612}. 
\end{remark}

\paragraph{Assume $q_2(\sigma_1)=1$ and $q_2(\sigma_2)=-1$}
In this case we have $s^2=-1$, $c^2=-\inpr{k}{k}$ for $k\in K_*$, and there is an $n\geq 2$ such that
$|\Gamma_e|=2^{n+2}$ and $|G_o|=2^n|\Gamma_o|-1$. 
In particular, we have $|G_o|\equiv 3 \mod 8$ for $n=2$, i.e. $|\Gamma_e|=16$, and $|G_o|\equiv 7 \mod 8$ for $n\geq 3$. 

\begin{lemma}\label{nu2I} Under the assumptions of this section, the following hold. \\
$(1)$ If $G_e=\Z_2^2$, then (FS2) implies that the 2-rank of $\Gamma$ is 2 and $\theta_2=-1$. 
Under this additional assumption, we have the following. 
\begin{itemize}
\item If either $q_{1,e}(x,y)=(-1)^{xy}$ or $q_{1,e}(x,y)=(-1)^{x^2+xy+y^2}$, then (FS2) is equivalent to 
$$|\cG(q_{2,e},\sigma_0,2)|=2\text{ and }\frac{\cG(q_{2,e},0,2)}{\cG(q_{2,e})}=\frac{2}{\cG(q_{1,e})}(-1)^{\frac{|G_o|^2-1}{8}+\frac{|\Gamma_o|^2-1}{8}}.$$
\item If $q_{1,e}(x,y)=i^{\pm (x^2+y^2)}$, then (FS2) is equivalent to $$|\cG(q_{2,e},0,2)|=2\text{ and }\frac{\cG(q_{2,e},\sigma_0,2)}{\cG(q_{2,e})}=2(\mp i)(-1)^{\frac{|G_o|^2-1}{8}+\frac{|\Gamma_o|^2-1}{8}}.$$
\end{itemize}
$(2)$ If $G_e=\Z_4$, then (FS2) implies $|\cG(q_1,q_2,0,2)|=|\cG(q_1,q_2,u_0,2)|=\sqrt{2}$. 
Under this additional assumption, (FS2) is equivalent to the restriction coming from 
$$\nu_2(\gamma)=\frac{1+\inpr{\gamma}{u_0}}{2}-\delta_{2\gamma,0}q_2(\gamma)^2.$$ 
\end{lemma}

\begin{proof} 
Since $s^2=-1$, we have $\nu_2(k,\varepsilon)=0$. 
Since $\inpr{k}{u_0}=-1$ for $k\in K_*$, Lemma \ref{FS5} implies that
$$|\cG(q_1,q_2,0,2)|^2=2(1-q_{1,e}(k)^2\delta_{2k,0}) \text{ and } |\cG(q_1,q_2,u_0,2)|^2=2(1+q_{1,e}(k)^2\delta_{2k,0}).$$

(1) Assume $G_e=\Z_2^2$. 
If either $q_{1,e}(x,y)=(-1)^{xy}$ or $(-1)^{x^2+xy+y^2}$, we have $q_1(k)^2=1$, which implies  
$$0=|\cG(q_1,q_2,0,2)|=|\frac{2}{\cG(q_{1,e})}(-1)^{\frac{|G_o|^2-1}{8}}-\frac{\cG(q_{2,e},0,2)}{\cG(q_{2,e})}(-1)^{\frac{|\Gamma_o|^2-1}{8}}|,$$
$$2=|\cG(q_1,q_2,u_0,2)|=|\cG(q_{2,e},\sigma_0,2)|.$$
Assume that this holds. 
Since $|\cG(q_{2,e},0,2)|=|\cG(q_{2,e},\sigma_0,2)|=2$, the case $\Gamma_e=\Z_2^4$ in the list of Theorem \ref{list} never occurs, 
and the 2-rank of $\Gamma_e$ is 2. 
Thus $\Gamma^{\theta_2}=\Gamma_2$. 
Lemma \ref{FS5} shows that $\nu_2(k,\varepsilon)=0$, $\nu_2(g)=\inpr{k_0}{g}$,  and
$\nu_2(\gamma)=\inpr{\gamma}{\sigma_0}-\delta_{2\gamma,0}q_2(\gamma)^2$. 
If $\theta_{2,e}\neq -1$, there would exist a non-self-dual $\gamma\in \Gamma_*$ which satisfies $2\gamma\neq 0$. 
However, this implies $\nu_2(\gamma)=0$, which is  contradiction. 
Therefore we get (1) in the case of $q_{1,e}(x,y)=(-1)^{xy}$, $(-1)^{x^2+xy+y^2}$.

If $G_e=\Z_2^2$ and $q_{1,e}(x,y)=i^{\pm (x^2+y^2)}$, we get 
$$2=|\cG(q_1,q_2,0,2)|=|\cG(q_{2,e},0,2)|,$$
$$0=|\cG(q_1,q_2,u_0,2)|=|2(\mp i)(-1)^{\frac{|G_o|^2-1}{8}}-\frac{\cG(q_{2,e},u_0,2)}{\cG(q_{2,e})}(-1)^{\frac{|\Gamma_o|^2-1}{8}}|.$$
In the same way as above, we can show the statement. 

(2) Since $2k\neq 0$ and $q_1(k)^2=\pm i$, we get $2=|\cG(q_1,q_2,0,2)|=|\cG(q_1,q_2,u_0,2)|$. 
Under this condition, Lemma \ref{FS5} implies $$\nu_2(g)=\frac{1+\inpr{g}{u_0}}{2}, \ \nu_2(k,\varepsilon)=0,\text{ and } \nu_2(\gamma)=\frac{1+\inpr{\gamma}{u_0}}{2}-\delta_{2\gamma,0}q_2(\gamma)^2,$$ which shows the statement. 
\end{proof}

Theorem \ref{list}, Lemma \ref{nondegenerate}, and Lemma \ref{nu2I} show that we have only the following possibilities for  $G_e$, which we treat separately.\\

\subparagraph{Assume $G_e=\Z_2\times \Z_2$ and $q_{1,e}(x,y)=(-1)^{xy}$}
We may assume $\Gamma_e=\Z_4\times \Z_{2^n}$ with $n\geq 2$, $q_{2,e}(x,y)=\zeta_8^{r_1x^2}\zeta_{2^{n+1}}^{r_2y^2}$ 
with $r_1,r_2\in \{1,-1,3,-3\}$, and $\theta_{2,e}=-1$. 
We may also assume $\sigma_0=(2,0)$ up to involutive metric group isomorphism.   
We have $c^2=-1$ and $$c=\cG(q_{1,o})=-(-1)^{n\frac{r_2^2-1}{8}}\zeta_8^{r_1+r_2}\cG(q_{2,o}).$$ 
Then (FS2) is equivalent to $(-1)^{\frac{|G_o|^2-1+|\Gamma_o|^2-1+r_1^2-1+r_2^2-1}{8}}=1$.

The smallest example is $G=\Z_2\times \Z_2\times \Z_3$ and $\Gamma=\Z_4\times \Z_4$, and  
there are four possibilities for $(r_1,r_2)$: $(3,-1)$, $(-1,3)$,$(1,-3)$,and $(-3,1)$, which give two metric group 
isomorphism classes (consider the transformation $(x,y)\mapsto (x+2y,2x+y)$).  

We conjecture that the pair 
$$(S^{(G_o,\overline{q_{1,o}})}\otimes S,T^{(G_o,\overline{q_{1,o}})}\otimes T)$$
is the modular data of the Drinfeld center $\cZ(\cC)$ of a quadratic category $\cC$ of type $(\Z_2,G_o,1)$ with non-self-dual $\rho$ 
such that $\alpha_1$ lifts to a fermion in $\cZ(\cC)$.  
We also conjecture that there are exactly two such fusion categories $\cC$ for $G_o=\Z_3$. 

\paragraph{Other quadratic forms for $G_e=\Z_2\times \Z_2$} We can treat the other quadratic forms for $G_e=\Z_2\times \Z_2$, namely $q_{1,e}(x,y)=(-1)^{x^2+xy+y^2}$, $q_{1,e}(x,y)=i^{x^2+y^2}$, and $q_{1,e}(x,y)=i^{-x^2-y^2}$, in a similar way. \\

\paragraph{Assume $G_e=\Z_4$, $\Gamma_e=\Z_2\times \Z_2\times \Z_{2^n}$ with $n\geq 2$, $q_{1,e}(x)=\zeta_8^{r_1x^2}$, 
$q_{2,e}(x,y,z)=q'(x,y)\zeta_{2^{n+1}}^{r_2z^2}$ 
with $r_1, r_2\in \{1,-1,3,-3\}$ and a flip invariant non-degenerate quadratic form $q'$ on $\Z_2\times \Z_2$, and 
$\theta_{2,e}(x,y,z)=(y,x,-z)$}
We have $\sigma_0=(1,1,0)$ or $(0,0,2)$ for $n=2$, and $\sigma_0=(1,1,0)$ or $(1,1,2)$ for $n\geq 3$. 
Then $c^2=-i^{r_1}$, and $$c=\zeta_8^{r_1}\cG(q_{1,o})=-(-1)^{n\frac{r_2^2-1}{8}}\zeta_8^{r_2}\cG(q')\cG(q_{2,o}).$$

Lemma \ref{nu2I} implies that $$|\cG(q_1,q_2,0,2)|^2=|\cG(q_1,q_2,u_0,2)|^2=2,$$ and 
$$\nu_2(\gamma)=\frac{1+\inpr{\gamma}{u_0}}{2}-\delta_{2\gamma,0}q_2(\gamma)^2.$$ 
We choose $\gamma_0\in \Gamma_0\setminus \{0\}$, and set $\gamma_1=(1,0,0,\gamma_0)$. 
Then $\gamma_1\in \Gamma_*$ and it is not self-dual. 
Thus $0=\nu_2(\gamma)$, and $\inpr{(1,0,0)}{\sigma_0}=-1$. 
Let $\gamma_2=(0,0,1,0)\in \Gamma_*$. 
Then $\theta_2(\gamma_2)=-\gamma_2$, and $\gamma_2$ is self-dual. 
Since $2\gamma_2\neq 0$, we get $$\nu_2(\gamma_2)=\frac{1+\inpr{(0,0,1)}{\sigma_0}}{2}\text{ and }\inpr{(0,0,1)}{\sigma_0}=1.$$
Thus we get $\sigma_0=(1,1,0)$.  

The relation
$$\sqrt{2}=|\cG(q_1,q_2,0,2)|=|\sqrt{2}(-1)^{\frac{r_1^2-1}{8}+\frac{|G_o|^2-1}{8}}-\sqrt{2}\frac{\cG(q',2)}{\cG(q')}(-1)^{\frac{r_2^2-1}{8}+\frac{|\Gamma_o|^2-1}{8}}|$$
implies that either $\cG(q',2)=0$ or 
$$\cG(q',2)=2 \cG(q')(-1)^{\frac{r_1^2-1}{8}+\frac{r_2^2-1}{8}+\frac{|G_o|^2-1}{8}+\frac{|\Gamma_o|^2-1}{8}},$$ 
and the relation 
$$\sqrt{2}=|\cG(q_1,q_2,u_0,2)|=|\sqrt{2}i^{-r_1}(-1)^{\frac{r_1^2-1}{8}+\frac{|G_o|^2-1}{8}}-\sqrt{2}\frac{\cG(q',(1,1),2)}{\cG(q')}(-1)^{\frac{r_2^2-1}{8}+\frac{|\Gamma_o|^2-1}{8}}|$$
implies that either $\cG(q',(1,1),2)=0$ or 
$$\cG(q',(1,1),2)=2i^{-r_1}\cG(q')(-1)^{\frac{r_1^2-1}{8}+\frac{r_2^2-1}{8}+\frac{|G_o|^2-1}{8}+\frac{|\Gamma_o|^2-1}{8}}.$$

In summary, we get the following conditions: 
\begin{itemize}
\item $c^2=-i^{r_1}$, $c=\zeta_8^{r_1}\cG(q_{1,o})=-(-1)^{n\frac{r_2^2-1}{8}}\zeta_8^{r_2}\cG(q')\cG(q_{2,o})$. 
\item $\cG(q')=(-1)^{\frac{r_1^2-1}{8}+\frac{r_2^2-1}{8}+\frac{|G_e|^2-1}{8}+\frac{|\Gamma_e|^2-1}{8}}$ if $q'(x,y)=(-1)^{xy}$ or $q'(x,y)=(-1)^{x^2+xy+y^2}$.  
\item $\cG(q')=i^{r_1}(-1)^{\frac{r_1^2-1}{8}+\frac{r_2^2-1}{8}+\frac{|G_e|^2-1}{8}+\frac{|\Gamma_e|^2-1}{8}}$ if $q'(x,y)=i^{\pm(x^2+y^2)}$. 
\end{itemize}

The smallest example is $G_o=\Z_3$, and we have $n=2$ and $\Gamma_o=\{0\}$ in this case. 
There are 16 possibilities, and 8 isomorphism classes of involutive metric groups 
(consider the transformation $(x,y,z)\mapsto (x+z,y+z,2(x+y)+z)$.) 

\subsubsection{Assume (A2) and $q_0(u_0)=-1$.} 
We necessarily have $\inpr{u_0}{k}=1$ for $k\in K_*$ and $\inpr{u_0}{\sigma}=-1$ for $\sigma\in \Sigma_*$. 
Thus one of the following holds:
\begin{itemize} 
\item $\Gamma_e=\Z_2\times \Z_2$, $q_{2,e}(x,y)$ is $(-1)^{xy}$, $(-1)^{x^2+xy+y^2}$,  
$i^{x^2+y^2}$, or $i^{-x^2-y^2}$, and $\theta_2=-1$.  
We may assume $\sigma_0=(1,1)$, $\sigma_1=(1,0)$, and $\sigma_2=(0,1)$. 
\item $\Gamma_e=\Z_4$, $q_{2,e}(x)=\zeta_8^{rx^2}$ with $r\in \{1,-1,3,-3\}$, and $\theta_{2,e}=1$. 
We may assume $\sigma_0=2$, $\sigma_1=1$, and $\sigma_2=3$. 
\end{itemize} 
Since $s=c/q_2(\sigma)$, Theorem \ref{STtheorem5} implies $s^2=\inpr{k}{k}$ and $c^2=\inpr{k}{k}\inpr{\sigma}{\sigma}$.
The order $|G_e|$ is a multiple of 8 and the set $\{q_1(k_0),q_1(k_1),q_1(k_2)\}$ is either $\{-1,i,-i\}$ or $\{-1,-1,1\}$.  As before we treat these two cases separately.\\

\paragraph{Assume $q_1(k_1)=i $ and $q_1(k_2)=-i $}
In this case we have $s^2=-1$ and $c^2=-\inpr{\sigma}{\sigma}$.  
Since $$-\inpr{\sigma}{\sigma}=c^2=\cG(q_{2.e})^2\cG(q_{2.o})^2,$$ we have $\cG(q_{2.o})^2=-1$ and $|\Gamma_o|\equiv 3\mod 4$. 
Since $\cG(q_{1,e})^4=\cG(q_{2,4})^4$, we have only the following combinations: 
\begin{itemize} 
\item $G_e=\Z_4\times \Z_2$ and $\Gamma_e=\Z_2\times \Z_2$. 
\item $G_e=\Z_2^3$ and $\Gamma_e=\Z_4$. 
\end{itemize}
Since $$\cG(q_{1,e},0,2)=\cG(q_{1,e},k_0,2)=0,$$ we can show that (FS2) does not give any restriction. 
The smallest examples are $$G=\Z_4\times \Z_2 \text{ and }\Gamma=\Z_2\times \Z_2\times \Z_3$$ and $$G=\Z_2^3 \text{ and } \Gamma=\Z_4\times \Z_3.$$

\paragraph{Assume $q_1(k_1)=-1 $ and $q_1(k_2)=1$}  
In this case we have $s^2=1$ and $c^2=\inpr{\sigma}{\sigma}$.

\begin{lemma}\label{nu2II} Under the assumptions of this section, the following hold. \\
$(1)$ If $\Gamma_e=\Z_2^2$, then (FS2) implies that the 2-rank of $G$ is 2 and $\theta_1=-1$. 
Under this additional assumption, we have the following. 
\begin{itemize}
\item If either $q_{2,e}(x,y)=(-1)^{xy}$ or $(-1)^{x^2+xy+y^2}$, then (FS2) is equivalent to 
$$|\cG(q_{1,e},k_0,2)|=2 \text{ and }\frac{\cG(q_{1,e},0,2)}{\cG(q_{1,e})}=\frac{2}{\cG(q_{2,e})}(-1)^{\frac{|G_o|^2-1}{8}+\frac{|\Gamma_o|^2-1}{8}}.$$
\item If $q_{2,e}(x,y)=i^{\pm (x^2+y^2)}$, then (FS2) is equivalent to $$|\cG(q_{1,e},0,2)|=2\text{ and }\frac{\cG(q_{1,e},k_0,2)}{\cG(q_{1,e})}=2(\mp i)(-1)^{\frac{|G_o|^2-1}{8}+\frac{|\Gamma_o|^2-1}{8}}.$$ 
\end{itemize}
$(2)$ If $\Gamma_e=\Z_4$, (FS2) implies $|\cG(q_1,q_2,0,2)|=|\cG(q_1,q_2,u_0,2)|=\sqrt{2}$. 
Under this additional assumption, (FS2) is equivalent to the restriction coming from 
$$\nu_2(g)=\frac{1+\inpr{g}{u_0}}{2}+\delta_{2g,0}q_2(g)^2.$$
\end{lemma}

\begin{proof}
Since $s^2=1$, we have $\nu_2(\sigma,\varepsilon)=\pm 1$, and since $\inpr{\sigma}{u_0}=-1$ for $\sigma\in \Sigma_*$, Lemma \ref{FS5} implies  
$$|\cG(q_1,q_2,0,2)|^2=2(1+2\nu_2(\sigma,\varepsilon)+q_{2,e}(\sigma)^2\delta_{2\sigma,2}),$$
$$|\cG(q_1,q_2,u_0,2)|^2=2(1-2\nu_2(\sigma,\varepsilon)-q_{2,e}(\sigma)^2\sigma,2).$$
If $q_{2,e}(x,y)=(-1)^{xy}$ or $(-1)^{x^2+xy+y^2}$, we get 
$$|\cG(q_1,q_2,0,2)|^2=4(1+\nu_2(\sigma,\varepsilon)),\quad |\cG(q_1,q_2,u_0,2)|^2=-4\nu_2(\sigma,\varepsilon),$$
which implies $\nu_2(\sigma,\varepsilon)=-1$, and 
$$0=|\cG(q_1,q_2,0,2)|=|\frac{\cG(q_{1,e},0,2)}{\cG(q_{1,e})}(-1)^{\frac{|G_o|^2-1}{8}}-\frac{2}{\cG(q_{2,e})}(-1)^{\frac{|\Gamma_o|^2-1}{8}}|,$$
$$2=|\cG(q_1,q_2,u_0,2)|=|\cG(q_{1,e},k_0,2)|.$$
If $q_{2,e}(x,y)=i^{\pm (x^2+y^2)}$, we get 
$$|\cG(q_1,q_2,0,2)|^2=4\nu_2(\sigma,\varepsilon),\quad |\cG(q_1,q_2,u_0,2)|^2=4(2-\nu_2(\sigma,\varepsilon)),$$
which implies $\nu_2(\sigma,\varepsilon)=1$, and 
$$2=|\cG(q_1,q_2,0,2)|=|\cG(q_{1,e},0,2)|,$$
$$0=|\cG(q_1,q_2,u_0,2)|=|\frac{\cG(q_{1,e},k_0,2)}{\cG(q_{1,e})}(-1)^{\frac{|G_o|^2-1}{8}}-2(\mp i)(-1)^{\frac{|\Gamma_o|^2-1}{8}}|.$$

The rest of the proof is the same as the proof of Lemma \ref{nu2I}. 
\end{proof}

Theorem \ref{list} and Lemma \ref{nondegenerate} show that we have the following possibilities: 
\begin{itemize}
\item $G_e=\Z_4\times \Z_{2^n}$ with $n\geq 2$ and $\Gamma_e=\Z_2\times \Z_2$. 
\item $G_e=\Z_2\times \Z_2\times \Z_{2^n}$ with $n\geq 2$ and $\Gamma_e=\Z_4$. 
\end{itemize}

\subsubsection{(A1) and $q_0(u_0)=1$.} 
We necessarily have $\inpr{u_0}{k}=1$ for $k\in K_*$ and $\inpr{u_0}{\sigma}=-1$ for $\sigma\in \Sigma_*$, 
and hence $\Gamma_e=\Z_2\times \Z_2$. 
Since $s=c/q_1(k)$, Theorem \ref{STtheorem5} implies that $c^2=1$. 
Since 
$$\{q_2(\sigma_0),q_2(\sigma_1),q_2(\sigma_2)\}=\{1,q_2(\sigma_1),-q_2(\sigma_1)\},$$ 
we have the following two possibilities: 
\begin{itemize}
\item $q_{2,e}(x,y)=(-1)^{xy}$ with $\sigma_0=(1,0)$, $\sigma_1=(0,1)$, $\sigma_2=(1,1)$. 
\item $q_{2,e}(x,y)=i^{x^2-y^2}$ with $\sigma_0=(1,1)$, $\sigma_1=(1,0)$, $\sigma_2=(0,1)$. 
\end{itemize}
Either way, we have $\cG(q_{2,e})=1$. Also, $c^2=1$ implies that $|\Gamma_o|\equiv 1 \mod 4$. 
Thus $|G_e|$ is a multiple of 16.
 For $\{q_1(k_0),q_1(k_1),q_1(k_2)\}$ we have 4 possibilities: $\{1,1,1\}$, $\{1,-1,-1\}$, $\{1,i,i\}$, and $\{1,-i,-i\}$. 
As before, we treat these cases separately.\\

\paragraph{Assume $q_1(k_1)=q_1(k_2)= \pm i$  }
Then we have $s^2=-1$. 
We may assume  $G_e=\Z_2\times \Z_{2^m}$ with $m\geq 3$,
$q_{1,e}(x,y)=i^{r_1 x^2}\zeta_{2^{m+1}}^{r_2y^2} $ with  $ r_1\in \{1,-1\}$ and $ r_2\in \{1,-1,3,-3\}$, $\theta_1=-1$, and $k_1=(0,2^{n-1})$.
Hence $$2^{m-1}|G_{1,o}|+1=|\Gamma_{2,o}|\text{ and } c=\zeta_8^{r\pm 1}(-1)^{m\frac{r^2-1}{8}}\cG(q_{1,o})=-\cG(q_{2,o}).$$ 
Since $\cG(q_{1,e},2)=\cG(q_{1,e},k_0,2)=0$, we can show that (FS2) does not give any restriction. 
The smallest example is $G=\Z_2\times \Z_8$ and $\Gamma=\Z_2\times \Z_2\times \Z_5$. \\

\paragraph{Assume $q_1(k_1)=q_1(k_2)= \pm 1$ }
Then we have $s^2=1$, which implies that $(\sigma,\varepsilon)$ is self-dual, and hence $\nu_2(\sigma,\varepsilon)=\pm 1$. 
We can show the following Lemma as before. 

\begin{lemma}\label{nu2III} Under the above assumption, (FS2) implies the 2-rank of $G$ is 2 and $\theta_{1,e}=-1$. 
Under this additional assumption, we have the following: 
\begin{itemize}
\item If $q_{2,e}(x,y)=(-1)^{xy}$, (FS2) is equivalent to $$|\cG(q_{1,e},k_0,2)|=2 \text{ and }\frac{\cG(q_{1,e},0,2)}{\cG(q_{1,e})}=2(-1)^{\frac{|G_o|^2-1}{8}+\frac{|\Gamma_o|^2-1}{8}}.$$
\item If $q_{2,e}(x,y)=i^{x^2-y^2}$ (FS2) is equivalent to $$|\cG(q_{1,e},0,2)|=2\text{ and }\frac{\cG(q_{1,e},k_0,2)}{\cG(q_{1,e})}=2(-1)^{\frac{|G_o|^2-1}{8}+\frac{|\Gamma_o|^2-1}{8}}.$$
\end{itemize}
\end{lemma}

Among involutive metric groups in Theorem \ref{list}, the only cases with $\theta_{1,e}=-1$ up to 
isomorphism of involutive metric groups are the following: 
\begin{itemize}
\item $G_e=\Z_{2^m}\times \Z_{2^m}$ with $m\geq 2$, $q_{1,e}(x,y)=\zeta_{2^m}^{xy}$, and 
$k_0=(2^{m-1},2^{m-1})$ or $(2^{m-1},0)$.  

\item $G_e=\Z_{2^m}\times \Z_{2^m}$ with $m\geq 2$, $q_{1,e}(x,y)=\zeta_{2^m}^{x^2+xy+y^2}$, and  $k_0=(2^{m-1},2^{m-1})$ or $(2^{m-1},0)$. 
\item $G_e=\Z_{2^m}\times \Z_{2^n}$ with $n\geq m\geq 3$, $q_{1,e}(x,y)=\zeta_{2^{m+1}}^{r_1x^2}\zeta_{2^{n+1}}^{r_2y^2}$ with 
$r_1,r_2\in \{1,-1,3,-3\}$, and $k_0=(2^{m-1},2^{n-1})$ or $k_0=(0,2^{n-1})$. 

\item $G_e=\Z_4\times \Z_{2^n}$ with $n\geq 3$, $q_{1,e}(x,y)=\zeta_8^{r_1x^2}\zeta_{2^{n+1}}^{r_2y^2}$ with $r_1,r_2\in \{1,-1,3,-3\}$, 
and $k_0=(0,2^{n-1})$. 

\item $G_e=\Z_4\times \Z_4$, $q_{1,e}(x,y)=\zeta_8^{r_1x^2+r_2y^2}$ with $r_1,r_2\in \{1,-1,3,-3\}$, 
and $k_0=(2,2)$. 
\end{itemize}

We only look at the first case in detail. 
Assume $G_e=\Z_{2^m}\times \Z_{2^m}$ with $m\geq 2$ and $q_{1,e}(x,y)=\zeta_{2^m}^{xy}$. 
Then we have $2^{2m-2}|G_o|+1=|\Gamma_o|$, $\cG(q_{1,e})=1$, and $c=\cG(q_{1,o})=-\cG(q_{2,o}), \text{ with } c^2=1$. 
We have $$\cG(q_{1,e},0,2)=\cG(q_{1,e},(2^{m-1},0),2)=2 \text{ and }\cG(q_{1,e},(2^{m-1},2^{m-1}),2)=2(-1)^{2^{m-2}}.$$
So we get $|\cG(q_{1,e},2)|=|\cG(q_{2,e},k_0,2)|=2$. We consider different quadratic forms separately.\\

\subparagraph{Assume $q_{2,e}(x,y)=i^{x^2-y^2}$}
Lemma \ref{nu2III} implies 
$$\cG(q_{1,e},k_0,2)=2(-1)^{\frac{|G_o|^2-1}{8}+\frac{|\Gamma_o|^2-1}{8}}=2(-1)^{\frac{|G_o|^2-1}{8}}(-1)^{2^{2(m-2)}}.$$ 
On the other hand, for $k_0=(2^{m-1},2^{m-1})$, we have 
$\cG(q_{1,e},k_0,2)=2(-1)^{2^{m-2}}$, and we get $(-1)^{\frac{|G_o|^2-1}{8}}=1$. 
Thus (FS2) is satisfied if and only if $|G_o|\equiv 1\mod 8$.

For $\sigma_0=(2^{m-1},0)$, we have $\cG(q_{1,e},k_0,2)=2$, and 
$$|G_o|\equiv \left\{
\begin{array}{ll}
5 \mod 8, &\quad m=2 \\
1 \mod 8 , &\quad m\geq 3.
\end{array}
\right.
$$

We conjecture that the modular data of the Drinfeld center of a generalized Haagerup category for $\Z_{2^m}\times A$ 
with odd $A$ is given by $(S,T)$ in this case with $G_o=A\times \hat{A}$, $q_{1,o}(p,\chi)=\chi(p)$ and 
$k_0=(2^{m-1},2^{m-1})$ for non-trivial $\epsilon_{h}(g)$ and $k_0=(2^{m-1},0)$ for trivial $\epsilon_h(g)$ (see \cite[Example 4.3]{GI19_1}). 
The conjecture is true for $m=2$ and $A=\{0\}$. 
More generally, the pair 
$$(S^{(G_o,\overline{q_{1,o}})}\otimes S,T^{(G_o,\overline{q_{1,o}})}\otimes T)$$
may possibly arise from the Drinfeld center of a quadratic category of type $(\Z_{2^m},G_o,1)$. \\
\subparagraph{Assume $q_{2,e}(x,y)=(-1)^{xy}$}
Then Lemma \ref{nu2III} implies 
$$\cG(q_{1,e},2)=2(-1)^{\frac{|G_o|^2-1+|\Gamma_o|^2-1}{8}}=2(-1)^{\frac{|G_o|^2-1}{8}+2^{2(m-2)}},$$
and 
$$|G_o|\equiv \left\{
\begin{array}{ll}
5 \mod 8, &\quad m=2 \\
1 \mod 8 , &\quad m\geq 3.
\end{array}
\right.
$$

\subsubsection{(A2) and $q_0(u_0)=1$.} 
We necessarily have $\inpr{u_0}{k}=-1$ for $k\in K_*$ and $\inpr{u_0}{\sigma}=1$ for $\sigma\in \Sigma_*$. 
Thus $G_e=\Z_2\times \Z_2$. 
Theorem \ref{STtheorem5} implies $s^2\inpr{\sigma}{\sigma}=-1$, 
and since $s=c/q_2(\sigma)$, we get $c^2=-1$.  
Since
$$\{q_1(k_0),q_1(k_1),q_1(k_2)\}=\{1, q_1(k_1),-q_1(k_1)\}.$$
we have two possibilities for $q_{1,e}$: 
\begin{itemize}
\item $q_{1,e}(x,y)=(-1)^{xy}$  with $k_0=(1,0)$, $k_1=(0,1)$, $k_2=(1,1)$. 
\item $q_{1,e}(x,y)=i^{x^2-y^2}$ with $k_0=(1,1)$, $k_1=(1,0)$, $k_2=(0,1)$.   
\end{itemize}
Either way we have $\cG(q_{1,e})=1$. 
Since $c^2=-1$, we get $\cG(q_{1,o})^2=-1$ and $|G_o|\equiv 3 \mod 4$. 
Therefore $|\Gamma_e|$ is a multiple of 16. 

For $\{q_2(\sigma_0),q_2(\sigma_1),q_2(\sigma_2)\}$, we have the following possibilities: $\{1,\pm i,\pm i\}$, 
$\{1,-1,-1\}$, $\{1,1,1\}$. As before, we treat these separately.\\ 

\paragraph{Assume $q_2(\sigma_1)=q_2(\sigma_2)=\pm i$}
Then $s^2=1$. 
We may assume that $\Gamma_{e}=\Z_2\times \Z_{2^n}$ with $n\geq 2$, $q_{2,e}(x,y)=\zeta_4^{r_1x^2}\zeta_{2^{n+1}}^{r_2y^2} $ with $ r_1\in \{1,-1\}$ and $r_2\in \{1,-1,3,-3\}$, $\theta_2=-1$, and $\sigma_0=(0,2^{n-1})$. 
We have $$c=\cG(q_{1,o})=-(-1)^{n\frac{r_2^2-1}{8}}\zeta_8^{r_1+r_2}\cG(q_{2,o}).$$
Since $\cG(q_{2,e},2)=\cG(q_{2,e},\sigma_0,2)=0$, we can show that (FS2) does not give any restriction. 

The smallest example is $G=\Z_2\times \Z_2\times \Z_3$ and $\Gamma=\Z_2\times \Z_8$. 
There are four possibilities: $(r_1,r_2)=(1,1)$, $(-1,-1)$, $(1,-3)$, or $(-1,3)$; and 2 isomorphism classes 
(consider the transformation $(x,y)\mapsto (x+y,2x+y)$). 

We conjecture that for $q_{1,e}(x,y)=(-1)^{xy}$, the pair 
$$(S^{(G_o,\overline{q_{1,o}})}\otimes S,T^{(G_o,\overline{q_{1,o}})}\otimes T),$$
is given by the Drinfeld center $\cZ(\cC)$ of a quadratic category $\cC$ of type $(\Z_2,G_o,1)$ with self-dual $\rho$ 
such that $\alpha_1$ lifts to a boson in $\cZ(\cC)$.  
We also conjecture that there are exactly two such categories $\cC$ for $G_o=\Z_3$. \\

\paragraph{Assume $q_2(\sigma_1)=q_2(\sigma_2)=\pm i$} 
Then $s^2=-1$, and $\nu_2(k,\varepsilon)=0$. 
As before, we can show the following lemma.

\begin{lemma}\label{nu2III} Under the assumptions of this section, (FS2) implies that the 2-rank of $\Gamma$ is 2 and $\theta_{2,e}=-1$. 
Under this additional assumption, we have the following: 
\begin{itemize}
\item If $q_{1,e}(x,y)=(-1)^{xy}$, then (FS2) is equivalent to $$|\cG(q_{2,e},\sigma_0,2)|=2 \text{ and }\frac{\cG(q_{2,e},0,2)}{\cG(q_{2,e})}=2(-1)^{\frac{|G_o|^2-1}{8}+\frac{|\Gamma_o|^2-1}{8}}.$$
\item If $q_{1,e}(x,y)=i^{x^2-y^2}$ then (FS2) is equivalent to $$|\cG(q_{2,e},0,2)|=2 \text{ and }\frac{\cG(q_{2,e},\sigma_0,2)}{\cG(q_{2,e})}=2(-1)^{\frac{|G_o|^2-1}{8}+\frac{|\Gamma_o|^2-1}{8}}.$$
\end{itemize}
\end{lemma}

Among the involutive metric groups in Theorem \ref{list}, the only possible cases with $\theta_{2,e}=-1$ are the following, up to 
isomorphism: 
\begin{itemize}
\item $\Gamma_e=\Z_{2^m}\times \Z_{2^m}$ with $m\geq 2$, $q_{2,e}(x,y)=\zeta_{2^m}^{xy}$, and 
$\sigma_0=(2^{m-1},2^{m-1})$ or $(2^{m-1},0)$.  

\item $\Gamma_e=\Z_{2^m}\times \Z_{2^m}$ with $m\geq 2$, $q_{2,e}(x,y)=\zeta_{2^m}^{x^2+xy+y^2}$, and  $\sigma_0=(2^{m-1},2^{m-1})$ or $(2^{m-1},0)$. 
\item $\Gamma_e=\Z_{2^m}\times \Z_{2^n}$ with $n\geq m\geq 3$, $q_{2,e}(x,y)=\zeta_{2^{m+1}}^{r_1x^2}\zeta_{2^{n+1}}^{r_2y^2}$ with 
$r_1,r_2\in \{1,-1,3,-3\}$, and $\sigma_0=(2^{m-1},2^{n-1})$ or $(0,2^{n-1})$. 

\item $\Gamma_e=\Z_4\times \Z_{2^n}$ with $n\geq 3$, $q_{2,e}(x,y)=\zeta_8^{r_1x^2}\zeta_{2^{n+1}}^{r_2y^2}$ with $r_1,r_2\in \{1,-1,3,-3\}$, 
and $\sigma_0=(0,2^{n-1})$. 

\item $\Gamma_e=\Z_4\times \Z_4$, $q_{2,e}(x,y)=\zeta_8^{r_1x^2+r_2y^2}$ with $r_1,r_2\in \{1,-1,3,-3\}$, 
and $\sigma_0=(2,2)$. 
\end{itemize}

We conjecture that for $q_{1,e}(x,y)=(-1)^{xy}$, the pair 
$$(S^{(G_o,\overline{q_{1,o}})}\otimes S,T^{(G_o,\overline{q_{1,o}})}\otimes T)$$
is the modular data of the Drinfeld center $\cZ(\cC)$ of a quadratic category $\cC$ of type $(\Z_2,G_o,1)$ with non-self-dual $\rho$ 
such that $\alpha_1$ lifts to a boson in $\cZ(\cC)$. 
We also conjecture that there are exactly two such categories $\cC$ for $G_o=\Z_3$.  

The first test case is $G=\Z_2\times \Z_2\times \Z_3$, $\Gamma=\Z_4\times \Z_4$, $q_{1,e}(x,y)=(-1)^{xy}$, and $q_{2,e}(x,y)=\zeta_8^{r_1x^2+r_2y^2}$ 
with $(r_1,r_2)=(1,-3)$ or $(r_1,r_2)=(-1,3)$. 


\subsection{Conjectures}
We summarize the conjectures stated so far in this section. 

\begin{conjecture} Let $\cC$ be a generalized Haagerup category with an abelian group $A$ satisfying 
$A_2\cong \Z_2$. 
We denote $A_2=\{0,a_0\}$ and $(\hat{A})_2=\{0,\chi_0\}$. 
Then the modular data of the Drinfeld center $\cZ(\cC)$ is given by $(S,T)$ in Definition \ref{ST5} with $G=A\times \hat{A}$, 
$q_1(a,\chi)=\chi(a)$, $k_0=(a_0,\chi_0)$ for non-trivial $\epsilon_{h}(g)$ and $k_0=(a_0,0)$ for trivial $\epsilon_h(g)$. 
If moreover $A_e=\Z_{2^m}$ with $m\geq 2$, we have $\Gamma_2=\Z_2\times \Z_2$ and $q_{2,e}(x,y)=i^{x^2-y^2}$. 
\end{conjecture}

\begin{remark} If the above conjecture is true, it is easy to compute the Frobenius-Schur indicator as 
$$\nu_k(\rho)=\frac{\cG(q_1,k)+\cG(q_2,k)}{2}=\frac{|A_k|+\cG(q_{2,e},k)\cG(q_{2,o},k)}{2},$$
using \cite{MR2313527}. 
For $k=2$, we know $\nu_2(\rho)=1$, and the above formula would imply $\cG(q_{2,e},2)=0$. 
This shows that if $A_2=\Z_{2^m}$ with $m\geq 2$, then $q_{2,e}(x,y)=(-1)^{xy}$ is not consistent. 
For $k=3$, we can directly show $\nu_3(\rho)=1$. 
\end{remark}

Recall that a generalized Haagerup category $ \cC$ for an odd group $A$ can be equivariantized by the conjugation action of $\Inv(\cC) $ to produce a near group category $\cC^A$ (of quadratic type $(\{0\},A \times \hat{A},1) $), and the Drinfeld center splits as $\cZ(\cC^A)= \cZ(\cC)\boxtimes \cZ(\text{Vec}_{A})$. In a similar way, a generalized Haagerup category $ \cC$ for the group $\Z_2 \times A$ with $|A|$ odd can be equivariantized by the action of $A$. The resulting category $\cZ(\cC^A)$ is of quadratic type $(\Z_2,A \times \hat{A},1) $. Again the Drinfeld center splits as  $\cZ(\cC^A)=\cZ(\cC)\boxtimes \cZ(\text{Vec}_{A})$.
Therefore the above conjecture is a special case of the following one. 

\begin{conjecture} 
Let $A$ be a finite abelian group of odd order, 
and let $\cC$ be a quadratic category of type $(\Z_2,A,1)$ with the $\text{Vec}_{\Z_2} $ subcategory having trivial associator.  
Then the modular data of the Drinfeld center $\cZ(\cC)$ is given by 
$$(S^{(G_o,\overline{q_{1,o}})}\otimes S,T^{(G_o,\overline{q_{1,o}})}\otimes T),$$ 
where $(S,T)$ is as in Definition \ref{ST5} with $G_e=\Z_2\times \Z_2$, 
$G_o=A$, $q_{1,e}(x,y)=(-1)^{xy}$. 
Moreover, \\
$(1)$ If $\alpha_1$ lifts to a fermion in $\cZ(\cC)$ and $\rho$ is self-dual, we have $|A|\equiv 1 \mod 4$ and $k_0=(1,1)$. \\
$(2)$ If $\alpha_1$ lifts to a fermion in $\cZ(\cC)$ and $\rho$ is non-self-dual, we have $|A|\equiv 3 \mod 4$ and $k_0=(1,1)$.
For $A=\Z_3$, there are exactly two such categories. \\
$(3)$ If $\alpha_1$ lifts to a boson in $\cZ(\cC)$, we have $|A|\equiv 3 \mod 4$ and $k_0=(1,0)$. 
For $A=\Z_3$, there are exactly 2 such categories for self-dual $\rho$, and 2 for non-self-dual $\rho$. 
\end{conjecture}

\section{Asaeda-Haagerup family}
Throughout this section, we assume that $(G,q_1)$ and $(\Gamma,q_2)$ are metric groups satisfying 
$c:=\cG(q_1)=-\cG(q_2)$.   
We assume that $G$ is an even group and $\Gamma$ is an odd group. 
We denote $K=G_2$ for simplicity. 
We will eventually show that $K\cong \Z_2\times \Z_2$ is the only relevant case for our purposes, 
but for the moment we do not make this assumption. 

We set $a=1/\sqrt{|G|}$ and $b=1/\sqrt{|\Gamma|}$. 
We choose subsets $G_*\subset G$ and $\Gamma_*\subset \Gamma$ satisfying 
$$G=K\sqcup G_*\sqcup (-G_*),\quad \Gamma=\{0\}\sqcup \Gamma_*\sqcup (-\Gamma_*),$$
and set
$$J=\{0,\pi\}\sqcup J_1\sqcup G_*\sqcup \Gamma*,$$
where $J_1=(K\setminus\{0\})\times \{1,-1\}$. 
We use letters $k,k',k'',l$ for elements of $K$, $g,g',g'',h$ for elements of $G$ and $\gamma,\gamma',\gamma''
,\xi$ for elements of $\Gamma$. 
We introduce an involution of $J$ by setting $\overline{(k,\varepsilon)}=(k,c^2\inpr{k}{k}\varepsilon)$ for 
$(k,\varepsilon) \in J_1$, and leaving the other elements fixed. 
Note that since $\Gamma$ is odd, we have $\cG(q_2)^4=1$ and $c^2\in \{1,-1\}$. 

\begin{definition}\label{ST6} 
Let $S$, $T$, and $C$ be $J$ by $J$ matrices defined by  
\begin{align*}
S=&\left(
\begin{array}{ccccc}
\frac{a-b}{2} & \frac{a+b}{2}&\frac{a}{2} &a &b  \\
\frac{a+b}{2}& \frac{a-b}{2}&\frac{a}{2} &a &-b  \\
\frac{a}{2} &\frac{a}{2} &\frac{a\inpr{k}{k'}+c\varepsilon\varepsilon'q_1(k)\delta_{k,k'}}{2} &a\inpr{k}{g'} &0  \\
a &a &a\inpr{g}{k'}&a(\inpr{g}{g'}+\overline{\inpr{g}{g'}}) &0  \\
b &-b &0 &0 &-b(\inpr{\gamma}{\gamma'}+\overline{\inpr{\gamma}{\gamma'}}) 
\end{array}
\right),
\end{align*}
$$T=\mathrm{Diag}(1,1,q_1(k),q_1(g),q_2(\gamma)),$$
and $C_{x,y}=\delta_{x,\overline{y}}$ (where the block in $S$ and $T$ are indexed by $\{0\}$, $\{ \pi \} $, $J_1 $, $G_* $, and $\Gamma_* $.) 
\end{definition}

Direct computation gives the following results.
\begin{lemma}\label{relations6} Let the notation be as in Definition \ref{ST6}. 
Then $S$, $T$, and $C$ are unitary matrices satisfying $S^2=C$, $(ST)^3=cC$ and $T C=CT$. 
\end{lemma}

\begin{lemma}\label{Verlinde} Let the notation be as in Definition \ref{ST6}. Then we have
$$N_{\pi,\pi,\pi}=N_{\pi,\pi,\gamma}=N_{\pi,\pi,g}=N_{\pi,g,\gamma}=N_{\gamma,\gamma',g}=\frac{4}{|\Gamma|-|G|},$$
$$ N_{\pi,\pi,(k.\varepsilon)}= N_{\pi,\gamma,(k.\varepsilon)}= N_{\gamma,\gamma',(k.\varepsilon)}=\frac{2}{|\Gamma|-|G|},$$
$$N_{\pi,\gamma,\gamma'}=\frac{4}{|\Gamma|-|G|}-\delta_{\gamma,\gamma'},$$
$$N_{\pi,(k,\varepsilon),(k',\varepsilon')}=\delta_{(k,\varepsilon),\overline{(k',\varepsilon')}}+\frac{1}{|\Gamma|-|G|},
\quad N_{\gamma,(k,\varepsilon),(k',\varepsilon')}=\frac{1}{|\Gamma|-|G|},$$
$$N_{\pi,(k,\varepsilon),g}=N_{\gamma,(k,\varepsilon),g}=\frac{2}{|\Gamma|-|G|},$$
$$N_{\pi,g,g'}=\delta_{g,g'}+\frac{4}{|\Gamma|-|G|},
\quad N_{\gamma,g,g'}=\frac{4}{|\Gamma|-|G|},$$
$$N_{\gamma,\gamma',\gamma''}=\frac{4}{|\Gamma|-|G|}
-(\delta_{\gamma+\gamma'+\gamma'',0}+\delta_{\gamma+\gamma',\gamma''}
 +\delta_{\gamma'+\gamma'',\gamma}+\delta_{\gamma''+\gamma,\gamma'}),$$
\begin{align*}
\lefteqn{
N_{(k,\varepsilon),(k',\varepsilon'),(k'',\varepsilon'')}=\frac{1}{2}(\frac{1}{|\Gamma|-|G|}+\delta_{k+k'+k'',0})}\\
 &+c^2\frac{\varepsilon\varepsilon'\inpr{k}{k+k''}\delta_{k,k'}+\varepsilon'\varepsilon''\inpr{k'}{k'+k}\delta_{k',k''}+
\varepsilon''\varepsilon\inpr{k''}{k''+k'}\delta_{k'',k}}{2},
\end{align*}
$$N_{(k,\varepsilon),(k',\varepsilon'),g}
=\frac{1}{|\Gamma|-|G|}+c^2\varepsilon\varepsilon'\inpr{k}{k+g}\delta_{k,k'},$$
$$N_{(k,\varepsilon),g,g'}
 =\frac{2}{|\Gamma|-|G|}+\delta_{k+g+g',0}+\delta_{k+g-g',0}, $$
 $$N_{g,g',g''}
 =\frac{4}{|\Gamma|-|G|}+(\delta_{g+g'+g'',0}+\delta_{g+g',g''}+\delta_{g'+g'',g}+\delta_{g''+g,g'}).$$
\end{lemma}

\begin{theorem}\label{STtheorem6} Let the notation be as above. 
Then all the fusion coefficients $N_{ijk}$ are non-negative integers if and only if 
$|\Gamma|=|G|+1$ and $K\cong \Z_2\times \Z_2$. 
\end{theorem}

\begin{proof} 
The above computation shows that $|\Gamma|-|G|=1$ is necessary to have 
non-negative integer fusion coefficients. 
Assume this condition. 
Since the condition $\cG(q_1)=-\cG(q_2)$ is not fulfilled by $K=\Z_2$, we get $K\cong \Z_2^s$ with $s>1$. 
Let $k,k',k''\in K\setminus \{0\}$ be mutually distinct elements. 
Then we have 
$$N_{(k,\varepsilon),(k',\varepsilon'),(k'',\varepsilon'')}=\frac{1+\delta_{k+k'+k'',0}}{2}.$$ 
For $N_{(k,\varepsilon),(k',\varepsilon'),(k'',\varepsilon'')}$ to be a non-negative integer, 
the only possibility is $k+k'+k''=0$. 
This means that the only possible case is $K\cong \Z_2\times \Z_2$. 

On the other hand it is routine work to show that all the fusion coefficients $N_{ijk}$ are non-negative 
integers if $|\Gamma|=|G|+1$ and $K\cong \Z_2\times \Z_2$. 
\end{proof}

The explicit formula for $(S,T)$ in Definition \ref{ST6} stemmed out of an attempt to unify formulae 
appearing in the following three conjectures. 

\begin{conjecture} Let $\cC$ be a generalized Haagerup category with an abelian group $A$ satisfying 
$A_e\cong \Z_2\times \Z_2$.  
Then the modular data of the Drinfeld center $\cZ(\cC)$ is given by 
$$(S^{(G_e,q_{1,e})}\otimes S,T^{(G_{e},q_{1,e})} \otimes T)$$ 
with $G=\Z_2^2\times A_o\times \hat{A_o}$, $q_{1,e}(x,y)=(-1)^{x^2+xy+y^2}$, and $q_{1,o}(a,\chi)=\chi(a)$. 
\end{conjecture}
This conjecture is true for trivial $A_o=\{0\}$.

\begin{conjecture} Let $\cC$ be a generalized Haagerup category with an abelian group $A$ satisfying 
$A_e\cong \Z_{2^{n+1}}$.  
Then the modular data of the Drinfeld center $\cZ(\cC/\Z_2)$ of the $\Z_2$-de-equivariantization $\cC/\Z_2$ 
of $\cC$ is given by $(S,T)$ with $G=\Z_{2^n}^2\times A_o\times \hat{A_o}$, $q_{1,e}(x,y)=\zeta_{2^{n+1}}^{x^2-y^2}$, 
and $q_{1,o}(a,\chi)=\chi(a)$. 

\end{conjecture}
This conjecture is true for $A=\Z_4$ and $A=\Z_8$. 

\begin{conjecture} Let $\cC$ be a generalized Haagerup category with an abelian group $A$ satisfying 
$A_e\cong \Z_{2^n}\times \Z_2$ with $n\geq 2$.  
Then the modular data of the Drinfeld center $\cZ(\cC/\Z_2)$ of the $\Z_2$-de-equivariantization $\cC/\Z_2$ of $\cC$ 
is given by $(S,T)$ with $G=\Z_{2^n}^2\times A_o\times \hat{A_o}$, $q_{1,e}(x,y)=\zeta_{2^n}^{xy}$, 
and $q_{1,o}(a,\chi)=\chi(a)$. 
The modular data of the Drinfeld center $\cZ(\cC)$ is given by 
$$(S^{(\Z_2\times \Z_2,Q)} \otimes S,T^{(\Z_2\times \Z_2,Q)}\otimes T),$$ 
where $Q(x,y)=(-1)^{xy}$.

\end{conjecture}
This conjecture is true for $n=2$ and $A_o=\{0\}$.

We assume $|\Gamma|=|G|+1$ and $K\cong \Z_2\times \Z_2$ in the rest of this section. 
Note that we have 
$$N_{\pi,\pi,\pi}=4, \quad N_{(k,\varepsilon),(k,\varepsilon),(k,\varepsilon)}=2,\quad N_{g,g,g}=4+\delta_{3g,0},\quad 
N_{\gamma,\gamma,\gamma}=4-\delta_{3\gamma,0}.$$ 
Direct computation shows the following.
\begin{lemma}\label{FS} We have
$$\nu_m(\pi)=|\cG(q_1,m)+\cG(q_2,m)|^2,$$
$$\nu_m((k,\varepsilon))=\frac{|\cG(q_1,m)+\cG(q_2,m)|^2+\delta_{mk,0}q_1(k)^m}{2},$$
$$\nu_m(g)=|\cG(q_1,m)+\cG(q_2,m)|^2+\delta_{mg,0}q_1(g)^m,$$
$$\nu_m(\gamma)=|\cG(q_1,m)+\cG(q_2,m)|^2-\delta_{m\gamma,0}q_2(\gamma)^m.$$
In particular, (FS2) and (FS3) are equivalent to the following conditions respectively:  
\begin{equation}\label{D1}
|\cG(q_1,2)+\cG(q_2,2)|=1,
\end{equation}
\begin{equation}\label{D2}
|\cG(q_1,3)+\cG(q_2,3)|=2.
\end{equation}
\end{lemma}

\subsection{Examples}
We give a list of metric groups $(G,q_1)$, $(\Gamma,q_2)$ satisfying conditions in 
Theorem \ref{STtheorem6} and Lemma \ref{FS}. 
More precisely, the conditions are: $G_2\cong \Z_2\times \Z_2$, $|\Gamma|=|G|+1$, $\cG(q_1)=-\cG(q_2)$, and Eq.(\ref{D1})-(\ref{D2}). 
The condition $G_2\cong\Z_2\times \Z_2$ means that $G$ is of the form $\Z_{2^m}\times \Z_{2^n}\times G_o$ 
with $m,n>0$. 

As in Lemma \ref{nu3computation3}, 
Eq.(\ref{D1}) is equivalent to
\begin{equation}\label{D11}
|\frac{\cG(q_{1,e},2)}{\cG(q_{1,e})}-(-1)^{\frac{|G_o|^2-1}{8}}(-1)^{2^{m+n-2}}|=1,
\end{equation}
and Eq.(\ref{D2}) is equivalent to 
\begin{equation}\label{D21}
|(-1)^{m+n}\frac{\cG(q_{1,o},3)}{\cG(q_{1,o})^3}-\frac{\cG(q_2,3)}{\cG(q_2)^3}|=2.
\end{equation}
If neither $G$ nor $\Gamma$ has a 3-component (that is $|G_o|\equiv 1\mod 3$ and $|\Gamma|\equiv 2\mod 3$ for even $m+n$ 
or $|G_o|\equiv |\Gamma|\equiv 2 \mod 3$ for odd $m+n$), then we get  
\begin{align*}
\lefteqn{(-1)^{m+n}\frac{\cG(q_{1,o},3)}{\cG(q_{1,o})^3}-\frac{\cG(q_2,3)}{\cG(q_2)^3}} \\
 &=|(-1)^{m+n}\frac{(-1)^{\frac{|G_o|-1}{2}}(\frac{|G_o|}{3})}{\cG(q_{1,o})^2}
 -\frac{(-1)^{\frac{|\Gamma|-1}{2}}(\frac{|\Gamma|}{3})}{\cG(q_2)^2}|\\
 &=|(-1)^{m+n}(\frac{|G_o|}{3})-(\frac{|\Gamma|}{3})|=2.\\
\end{align*}
Direct computation with the above formulae shows that restriction coming from (FS3) in the following examples 
is very similar to that discussed in Lemma \ref{restriction} and Remark \ref{3rank2}, and we will not mention it in what follows 
except for one case.

Since 
$$|\Gamma|=2^{m+n}|G_o|+1\equiv 1 \mod 4,$$
we have $\cG(q_2)\in \{1,-1\}$, and $\cG(q_{1,e})\cG(q_{1,o})\in \{1,-1\}$. 
Thus we have either $\cG(q_{1,e})\in \{1,-1\}$ and $|G_o|\equiv 1 \mod 4$, 
or $\cG(q_{1,e})\in \{i,-i\}$ and $|G_o|\equiv 3 \mod 4$. 

We treat the different forms of $G_e $ separately.

\subsubsection{$G_e=\Z_2\times \Z_2$, $q_{1,e}(x,y)=(-1)^{x^2+xy+y^2}$.}

In this case we have $\cG(q_{1,e})=-1$ and $\cG(q_{1,e},2)=2$. 
Thus $|G_o|\equiv 1\mod 4$ and Eq.(\ref{D11}) implies $(-1)^{\frac{|G_o|^2-1}{8}}=1$. 
This is possible if and only if $|G_o|\equiv 1\mod 8$. 

The smallest example is $G=\Z_2\times \Z_2$ and $\Gamma=\Z_5$ with $\cG(q_2)=-1$.  
The pair 
$$(S^{(G,q_1)}\otimes S,T^{(G,q_1)}\otimes T)$$ 
is the modular data of the Drinfeld center of the $\Z_2\times \Z_2$ generalized Haagerup category 
(see \cite[Example 5.2]{GI19_1}). 

We conjecture that the modular data of the Drinfeld center of the generalized 
Haagerup category for $\Z_2\times \Z_2\times A$ with odd $A$ is 
$$(S^{(\Z_2\times \Z_2,q_{1,e})}\otimes S,T^{(\Z_2\times \Z_2,q_{1,e})}\otimes T)$$ 
with $G_o=A\times \hat{A}$ and $q_{1,o}(a,\chi)=\chi(a)$. 
More generally, the pair 
$$(S^{(G,\overline{q_1})}\otimes S,T^{(G,\overline{q_1})}\otimes T)$$ 
possibly arises from the Drinfeld center of a quadratic category of type $(\Z_2\times \Z_2,G_o,1)$ with the $\text{Vec}_{\Z_2 \times \Z_2} $ subcategory 
having trivial associator. 
\subsubsection{$G_e=\Z_2\times \Z_2$, $q_{1,e}(x,y)=(-1)^{xy}$.} 
In this case we have $\cG(q_{1,e})=1$ and $\cG(q_{1,e},2)=2$. 
Thus $|G_o|\equiv 1\mod 4$ and Eq.(\ref{D11}) implies $(-1)^{\frac{|G_o|^2-1}{8}}=-1$. 
This is possible if and only if $|G_o|\equiv 5\mod 8$. 

The smallest example is $G_o=\Z_5$ and $\Gamma=\Z_{21}$. 
There are four possible combinations of $q_{1,o}$ and $q_2$, up to group automorphism.

\subsubsection{$G_e=\Z_2\times \Z_2$, $q_{1,e}(x,y)=i^{x^2-y^2}$.}
In this case we have $\cG(q_{1,e})=1$ and $\cG(q_{1,e},2)=0$. 
Thus $|G_o|\equiv 1\mod 4$ and (FS2) gives no restriction. 

The smallest example is $G=\Z_2\times \Z_2$ and $\Gamma=\Z_5$ with $\cG(q_2)=-1$. 
The pair $(S,T)$ is the modular data of the Drinfeld center of the $\Z_2$-de-equivariantization of the generalized Haagerup 
category for $\Z_4$ (see \cite[Example 6.9]{GI19_1}). 

We conjecture that the modular data of the Drinfeld center of the $\Z_2$-de-equivariantization of a generalized 
Haagerup category for $\Z_4\times A$ with odd $A$ is $(S,T)$ 
with $G_o=A\times \hat{A}$ and $q_{1,o}(a,\chi)=\chi(a)$. 
More generally, the pair 
$$(S^{(G_o,\overline{q_{1,o}})}\otimes S,T^{(G_o,\overline{q_{1,o}})}\otimes T)$$ 
may possibly arise from the Drinfeld center of the $\Z_2$-de-equivariantization a quadratic category of type $(\Z_4,G_o,1)$. 

\subsubsection{$G_e=\Z_{2^n}\times \Z_{2^n}$ with $n\geq 2$ and $q_{1,e}(x,y)=\zeta_{2^n}^{x^2+xy+y^2}$.} 
In this case we have $\cG(q_{1,e})=(-1)^n$ and $\cG(q_{1,e},2)=2(-1)^{n-1}$. 
Thus $|G_o|\equiv 1 \mod 4$, and $(-1)^{\frac{|G_o|^2-1}{8}}=-1$. 
This is possible if and only if $|G_o|\equiv 5\mod 8$. 

The smallest example is $G_o=\Z_4\times \Z_4\times \Z_5$, and Eq.(\ref{D2}) implies that 
$\Gamma$ is either $\Z_{81}$, $\Z_3\times \Z_{27}$, or $\Z_9\times \Z_9$ with restricted possibilities 
of quadratic forms coming from (FS3).

\subsubsection{$G_e=\Z_{2^n}\times \Z_{2^n}$ with $n\geq 2$ and $q_{1,e}(x,y)=\zeta_{2^n}^{xy}$.}
In this case we have $\cG(q_{1,e})=1$ and $\cG(q_{1,e},2)=2$.  
Thus $|G_o|\equiv 1\mod 4$, and $(-1)^{\frac{|G_o|^2-1}{8}}=1$. 
This is possible if and only if $|G_o|\equiv 1\mod 8$. 

The smallest example is $G=\Z_4\times \Z_4$ and $\Gamma=\Z_{17}$ with $\cG(q_2)=-1$. 
The pair $(S,T)$ is the modular data 
of the Drinfeld center of the Asaeda-Haagerup category, namely the $\Z_2$ de-equivariantization 
of the $\Z_4\times \Z_2$ generalized Haagerup category   
(or a quadratic category of type $(\Z_4,\{0\},2)$ with the $\text{Vec}_{\Z_4} $ subcategory having trivial associator) 
(see \cite[Example 5.5]{GI19_1}). 

We conjecture that the modular data of the Drinfeld center $\cZ(\cC/\Z_2)$ of the $\Z_2$-de-equivariantization 
of a generalized Haagerup category for $\Z_{2^n}\times \Z_2\times A$ with odd $A$ 
(or a quadratic category $\cC$ of type $(\Z_{2^n}\times A,\{0\},2)$ with the $\text{Vec}_{\Z_{2^n}\times A} $ subcategory having trivial associator) is 
$(S,T)$ with $G_o=A\times \hat{A}$ and $q_2(a,\chi)=\chi(a)$. 
More generally, the pair 
$$(S^{(G_o,\overline{q_{1,o}})}\otimes S,T^{(G_o,\overline{q_{1,o}})}\otimes T)$$
may possibly arises from a quadratic category of type $(\Z_{2^n},G_o,2)$ with the $\text{Vec}_{\Z_{2^n}} $ subcategory  having trivial associator. 

\subsubsection{$G_e=\Z_{2^n}\times \Z_{2^n}$ with $n\geq 2$ and $q_{1,e}(x,y)=\zeta_{2^{n+1}}^{r(x^2-y^2)}$ with odd $r$.}
We may assume $r=1$, up to group automorphism. 
In this case we have $\cG(q_{1,e})=1$ and $\cG(q_{1,e},2)=2$. 
Thus $|G_o|\equiv 1 \mod 4$ and $(-1)^{\frac{|G_o|^2-1}{8}}=1$. 
This is possible if and only if $|G_o|\equiv 1\mod 8$.

The smallest example is $G=\Z_4\times \Z_4$ and $\Gamma=\Z_{17}$ with $\cG(q_2)=-1$.  
The pair $(S,T)$ is the modular data of the Drinfeld center of the $\Z_2$-de-equivariantization 
of the $\Z_8$ generalized Haagerup category   
(or a quadratic category of type $(\Z_4,\{0\},2)$ with  the $\text{Vec}_{\Z_{4}} $ subcategory  having non-trivial associator) - see \cite[ Example 6.10]{GI19_1}.

We conjecture that the modular data of the Drinfeld center $\cZ(\cC/\Z_2)$ of 
the $\Z_2$-de-equivariantization $\cC/\Z_2$ of a generalized Haagerup category $\cC$ 
for $\Z_{2^{n+1}}\times A$ with odd $A$ is 
$(S,T)$ with $G_o=A\times \hat{A}$ and $q_2(a,\chi)=\chi(a)$. 
More generally, the pair 
$$(S^{(G_o,\overline{q_{1,o}})}\otimes S,T^{(G_o,\overline{q_{1,o}})}\otimes T)$$
possibly arises from a quadratic category of type $(\Z_{2^{n}},G_o,2)$ with the $\text{Vec}_{\Z_{2^n}} $ subcategory having non-trivial associator.

\subsubsection{$G_e=\Z_2\times \Z_{2^n}$ with $n\geq 1$ and $q_{1,e}(x,y)=i^{rx^2}\zeta_{2^{n+1}}^{sy^2}$ with $r\in \{1,-1\}$ and $s\in \{1,-1,3,-3\}$.}
In this case we have $\cG(q_{1,e})=\zeta_8^{r+s}(-1)^{n\frac{s^2-1}{8}}$ and $\cG(q_{1,e},2)=0$. 
Thus $|G_o|\equiv 1\mod4$ if $r+s\equiv 0\mod 4$, and $|G_o|\equiv 3\mod 4$ if $r+s\equiv 2\mod 4$. 
Eq,(\ref{D11}) gives no restriction. 
The case with $n=1$ and $s=-r$ was already treated, and the first new example is 
$G_o=\Z_2\times\Z_2\times \Z_3$, and $\Gamma=\Z_{13}$. 
In this example $r=s=\pm1$, and there are two possibilities determined by 
the relation $i^r\cG(q_{1,o})=-\cG(q_2)$. 
\subsubsection{$G_e=\Z_{2^m}\times \Z_{2^n}$ with $2\leq m\leq n$ and $q_{1,e}(x,y)=\zeta_{2^{m+1}}^{rx^2}\zeta_{2^{n+1}}^{sy^2}$ 
with $r,s\in \{1,-1,3,-3\}$.} 
In this case we have $$\cG(q_{1,e})=\zeta_8^{r+s}(-1)^{m\frac{r^2-1}{8}+n\frac{s^2-1}{8}} 
\text{ and } \cG(q_{1,e},2)=2\zeta_8^{r+s}(-1)^{(m-1)\frac{r^2-1}{8}+(n-1)\frac{s^2-1}{8}},$$
and Eq.(\ref{D11}) implies that $(-1)^{\frac{r^2-1}{8}+\frac{s^2-1}{8}}=(-1)^{\frac{|G_o|^2-1}{8}}$. We get the following sub-cases:

(1) $|G_o|\equiv 1\mod 8$; $r+s=0$. 

(2) $|G_o|\equiv 5\mod 8$; $(r,s)=(1,3),(-1,-3),(3,1),(-3,-1)$. 

(3) $|G_o|\equiv 3\mod 8$; $(r,s)=(1,-3),(-1,3),(3,-1),(-3,1)$. 

(4) $|G_o|\equiv 7\mod 8$;  $(r,s)=(1,1), (-1,-1), (3,3),(-3,-3)$. 

\section{Switching the roles of $G$ and $\Gamma$}
As in the last section, we assume that $(G,q_1)$ and $(\Gamma,q_2)$ are metric group satisfying 
$c:=\cG(q_1)=-\cG(q_2)$.   
We assume that $G$ is an even group and $\Gamma$ is an odd group. 
We denote $K=G_2$ for simplicity. 

\begin{definition}\label{S'T'}
Let $S$ and $T$ be the unitary matrices defined in Definition \ref{ST6}, and let 
$R=\mathrm{Diag}(1,-1,-1_{J_1},-1_{G_*},-1_{\Gamma_*})$. 
We set $S'=-RS R$, $T'=T$, and $C'=C$; that is
\begin{align*}
S= &\left(
\begin{array}{ccccc}
\frac{-a+b}{2} & \frac{a+b}{2}&\frac{a}{2} &a &b  \\
\frac{a+b}{2}& \frac{-a+b}{2}&\frac{-a}{2} &-a &b  \\
\frac{a}{2} &\frac{-a}{2} &-\frac{a\inpr{k}{k'}+c\varepsilon\varepsilon'q_1(k)\delta_{k,k'}}{2} &-a\inpr{k}{g'} &0  \\
a &-a &-a\inpr{g}{k'} &-a(\inpr{g}{g'}+\overline{\inpr{g}{g'}}) &0  \\
b &b &0 &0 &b(\inpr{\gamma}{\gamma'}+\overline{\inpr{\gamma}{\gamma'}}) 
\end{array}
\right).
\end{align*} 
\end{definition}

Note that we have $S'_{0,x}>0$ if $|\Gamma|<|G|$. The following are consequences of the calculations in the previous section. 
\begin{cor} We have $S'^2=C'$, $(S'T')^3=-cC'$ and $T'C'=C'T'$. 
\end{cor}

\begin{proof} Since $R$ commutes with $C$ and $T$, we get $S'^2=RCR=C$ and 
$$S'T'S'=RS RT RS R=RSTS R=cR\overline{T}S\overline{T}R
=c\overline{T}RS R\overline{T}=-c\overline{T'}S'\overline{T'}.$$ 
\end{proof}

Let
$$N'_{u,v,w}=\sum_{x}\frac{S'_{u,x}S'_{v,x}S'_{w,x}}{S'_{0,x}}.$$

\begin{cor} We have
$$N'_{\pi,\pi,\pi}=N'_{\pi,\pi,\gamma}=N'_{\pi,\pi,g}=N'_{\pi,\gamma,g}=N'_{\gamma,\gamma',g}=\frac{4}{|G|-|\Gamma|},$$
$$ N'_{\pi,\pi,(k.\varepsilon)}= N'_{\pi,\gamma,(k.\varepsilon)}= N'_{\gamma,\gamma',(k.\varepsilon)}=\frac{2}{|G|-|\Gamma|},$$
$$N'_{\pi,\gamma,\gamma'}=\frac{4}{|G|-|\Gamma|}+\delta_{\gamma,\gamma'},$$
$$N'_{\pi,(k,\varepsilon),(k',\varepsilon')}=-\delta_{(k,\varepsilon),\overline{(k',\varepsilon')}}+\frac{1}{|G|-|\Gamma|},
\quad N'_{\gamma,(k,\varepsilon),(k',\varepsilon')}=\frac{1}{|G|-|\Gamma|},$$
$$N'_{\pi,(k,\varepsilon),g}=N'_{\gamma,(k,\varepsilon),g}=\frac{2}{|G|-|\Gamma|},$$
$$N'_{\pi,g,g'}=-\delta_{g,g'}+\frac{4}{|G|-|\Gamma|},
\quad N'_{\gamma,g,g'}=\frac{4}{|G|-|\Gamma|},$$
$$N'_{\gamma,\gamma',\gamma''}=\frac{4}{|G|-|\Gamma|}
+(\delta_{\gamma+\gamma'+\gamma'',0}+\delta_{\gamma+\gamma',\gamma''}
 +\delta_{\gamma'+\gamma'',\gamma}+\delta_{\gamma''+\gamma,\gamma'}),$$
\begin{align*}
\lefteqn{
N'_{(k,\varepsilon),(k',\varepsilon'),(k'',\varepsilon'')}=\frac{1}{2}(\frac{1}{|G|-|\Gamma|}-\delta_{k+k'+k'',0})}\\
 &-c^2\frac{\varepsilon\varepsilon'\inpr{k}{k+k''}\delta_{k,k'}+\varepsilon'\varepsilon''\inpr{k'}{k'+k}\delta_{k',k''}+
\varepsilon''\varepsilon\inpr{k''}{k''+k'}\delta_{k'',k}}{2},
\end{align*}
$$N'_{(k,\varepsilon),(k',\varepsilon'),g}
=\frac{1}{|G|-|\Gamma|}-c^2\varepsilon\varepsilon'\inpr{k}{k+g}\delta_{k,k'},$$
$$N'_{(k,\varepsilon),g,g'}
 =\frac{2}{|G|-|\Gamma|}-\delta_{k+g+g',0}-\delta_{k+g-g',0}, $$
$$N'_{g,g',g''}
 =\frac{4}{|G|-|\Gamma|}-(\delta_{g+g'+g'',0}+\delta_{g+g',g''}+\delta_{g'+g'',g}+\delta_{g''+g,g'}).$$
\end{cor}

\begin{proof} 
Note that we have $S'_{x,y}=-f(x)f(y)S_{x,y}$, where $f(0)=1$ and $f(j)=-1$ for the other $j\in J$. 
Thus we get 
$$N'_{u,v,w}=f(u)f(v)f(w)\sum_{x}\frac{S_{u,x}S_{v,x}S_{w,x}}{S_{0,x}}=f(u)f(v)f(w)N_{u,v,w}.$$
\end{proof}

As in the proof of Theorem \ref{STtheorem6}, we can show the following.
\begin{theorem}\label{STtheorem'} Let the notation be as in Definition \ref{S'T'}. 
Then all the fusion coefficients $N'_{ijk}$ are non-negative integers 
if and only if $|\Gamma|=|G|-1$ and $G_2\cong \Z_2\times \Z_2$. 
\end{theorem}

In the rest of this section, we assume $|\Gamma|=|G|-1$ and $G_2\cong \Z_2\times \Z_2$. 

Let 
$$\nu'_n(k)=\sum_{i,j\in J}N_{ij}^k S'_{0i}S'_{0j}(\frac{T'_{jj}}{T'_{ii}})^n.$$
As before, we have the following.
\begin{theorem}\label{FS'} Let the notation be as above. We have 
$$\nu'_m(\pi)=|\cG(q_1,m)+\cG(q_2,m)|^2,$$
$$\nu'_m((k,\varepsilon))=\frac{|\cG(q_1,m)+\cG(q_2,m)|^2-\delta_{mk,0}q_1(k)^m}{2},$$
$$\nu'_m(g)=|\cG(q_1,m)+\cG(q_2,m)|^2-\delta_{mg,0}q_1(g)^m,$$
$$\nu'_m(\gamma)=|\cG(q_1,m)+\cG(q_2,m)|^2+\delta_{m\gamma,0}q_2(\gamma)^m.$$
In particular, (FS2) and (FS3) are equivalent to the following respectively: 
$$|\cG(q_1,2)+\cG(q_2,2)|^2=1,$$
$$|\cG(q_1,3)+\cG(q_2,3)|^2=4.$$
\end{theorem}

\begin{conjecture} The modular data of the Drinfeld center of a near-group category with a finite abelian group 
$A$ with multiplicity $2|A|$ is given by 
$$(S^{(\Gamma,\overline{q_2})}\otimes S',T^{(\Gamma,\overline{q_2})}\otimes T')$$
with $\Gamma=A$. 
\end{conjecture}
\subsection{Examples}
We give a list of pairs of metric groups $(G,q_1)$ and $(\Gamma,q_2)$ satisfying the conditions in Theorem \ref{STtheorem'} and Theorem \ref{FS'}. 
More precisely, the conditions are: $G_2\cong \Z_2\times \Z_2$, $|\Gamma|=|G|-1$, $\cG(q_1)=-\cG(q_2)$, Eq.(\ref{D1}) and Eq.(\ref{D2}).
The condition $G_2\cong\Z_2\times \Z_2$ means that $G$ is of the form $\Z_{2^m}\times \Z_{2^n}\times G_o$ 
with $m,n>0$. 
As in the last section, Eq.(\ref{D1}) is equivalent to Eq.(\ref{D11}). 

Direct computation with the above formulae shows that the restriction coming from (FS3) through in the following examples 
is very similar to that discussed in Lemma \ref{restriction} and Remark \ref{3rank2}, and we will not mention it in what follows 
except for one case. 

Since 
$$|\Gamma|=2^{m+n}|G_o|-1\equiv 3 \mod 4,$$
we have $\cG(q_2)\in \{i,-i\}$, and $\cG(q_{1,e})\cG(q_{1,o})\in \{i,-i\}$. 
Thus we have either $\cG(q_{1,e})\in \{1,-1\}$ and $|G_o|\equiv 3 \mod 4$ 
or $\cG(q_{1,e})\in \{i,-i\}$ and $|G_o|\equiv 1 \mod 4$.
\subsubsection{$G_e=\Z_2\times \Z_2$, $q_{1,e}(x,y)=(-1)^{x^2+xy+y^2}$.}
In this case we have $\cG(q_{1,e})=-1$ and $\cG(q_{1,e},2)=2$. 
Thus $|G_o|\equiv 3\mod 4$ and Eq.(\ref{D11}) implies that $(-1)^{\frac{|G_o|^2-1}{8}}=1$. 
This is possible if and only if $|G_o|\equiv 7\mod 8$. 
The smallest example is $G_o=\Z_7$, and Eq.(\ref{D2}) allows only $\Gamma=\Z_{27}$ and $\Gamma=\Z_3\times \Z_9$ 
with restricted quadratic forms.  
\subsubsection{$G_e=\Z_2\times \Z_2$, $q_{1,e}(x,y)=(-1)^{xy}$.} 
In this case we have $\cG(q_{1,e})=1$ and $\cG(q_{1,e},2)=2$. 
Thus $|G_o|\equiv 3\mod 4$ and Eq.(\ref{D11}) implies that $(-1)^{\frac{|G_o|^2-1}{8}}=-1$. 
This is possible if and only if $|G_o|\equiv 3\mod 8$. 
The smallest example is $G_o=\Z_3$ and $\Gamma=\Z_{11}$. 
There are two possible combinations of $q_{1,o}$ and $q_2$ up to group automorphism.

\subsubsection{$G_e=\Z_2\times \Z_2$, $q_{1,e}(x,y)=i^{x^2-y^2}$.}
In this case we have $\cG(q_{1,e})=1$ and $\cG(q_{1,e},2)=0$. 
Thus $|G_o|\equiv 3\mod 4$ and Eq.(\ref{D11}) gives no restriction. 
The smallest example is $G_o=\Z_3$ and $\Gamma=\Z_{11}$, and there are two possible combinations of $q_{1,o}$ and $q_2$. 
\subsubsection{$G_e=\Z_{2^n}\times \Z_{2^n}$ with $n\geq 2$ and $q_{1,e}(x,y)=\zeta_{2^n}^{x^2+xy+y^2}$.} 
In this case we have $\cG(q_{1,e})=(-1)^n$ and $\cG(q_{1,e},2)=2(-1)^{n-1}$. 
Thus $|G_o|\equiv 3 \mod 4$, and $(-1)^{\frac{|G_o|^2-1}{8}}=-1$. 
This is possible if and only if $|G_o|\equiv 3\mod 8$. 
The smallest example is $n=2$, $G_o=\Z_3$, and $\Gamma=\Z_{47}$.

\subsubsection{$G_e=\Z_{2^n}\times \Z_{2^n}$ with $n\geq 2$ and $q_{1,e}(x,y)=\zeta_{2^n}^{xy}$.}
In this case we have $\cG(q_{1,e})=1$, $\cG(q_{1,e},2)=2$.  
Thus $|G_o|\equiv 3\mod 4$ and $(-1)^{\frac{|G_o|^2-1}{8}}=1$. 
This is possible if and only if $|G_o|\equiv 7\mod 8$. 
The smallest example is $n=2$, $G_o=\Z_7$, and $\Gamma =\Z_{111}$.  
\subsubsection{$G_e=\Z_{2^n}\times \Z_{2^n}$ with $n\geq 2$ and $q_{1,e}(x,y)=\zeta_{2^{n+1}}^{r(x^2-y^2)}$ with odd $r$.}
We may assume $r=1$, up to group automorphism. 
In this case, we have $\cG(q_{1,e})=1$ and $\cG(q_{1,e},2)=2$. 
Thus $|G_o|\equiv 3 \mod 4$, and $(-1)^{\frac{|G_o|^2-1}{8}}=1$. 
This is possible if and only if $|G_o|\equiv 7\mod 8$. 
The smallest example is $n=2$, $G_o=\Z_7$, and $\Gamma=\Z_{111}$.
\subsubsection{$G_e=\Z_2\times \Z_{2^n}$ with $n\geq 1$ and $q_{1,e}(x,y)=i^{rx^2}\zeta_{2^{n+1}}^{sy^2}$ with $r\in \{1,-1\}$ and $s\in \{1,-1,3,-3\}$.}
If $n=1$, we further assume $r+s\neq 0$. 
In this case we have $\cG(q_{1,e})=\zeta_8^{r+s}(-1)^{n\frac{s^2-1}{8}}$ and $\cG(q_{1,e},2)=0$. 
Thus $|G_o|\equiv 3\mod4$ if $r+s\equiv 0\mod 4$, and $|G_o|\equiv 1\mod 4$ if $r+s\equiv 2\mod 4$. 
Eq,(\ref{D11}) gives no restriction. 

We give a few examples. \\
(1) $n=1$, $G_o=\{0\}$, $\Gamma=\Z_3$, $r=s=\pm1$, and there are two possibilities determined by 
the relation $\cG(q_2)=-i^r$. \\
(2) $n=1$, $G_o=\Z_5$, $\Gamma=\Z_{19}$, $r=s=\pm1$, and there are two possibilities determined by 
the relation $\cG(q_2)=-i^r\cG(q_{1,o})$. \\
(3) $n=2$, $G_o=\{0\}$, and $\Gamma=\Z_7$. 
We have $\cG(q_2)=-\zeta_{8}^{s+t}$ with $r+s\equiv 2\mod 4$.  \\
(4) $n=3$, $G_o=\{0\}$, and $\Gamma=\Z_{15}$. We have 
$\cG(q_2)=-(-1)^{\frac{s^2-1}{8}}\zeta_8^{r+s}$ with $r+s\equiv 2\mod 4$. 
\subsubsection{$G_e=\Z_{2^m}\times \Z_{2^n}$ with $2\leq m\leq n$ and $q_{1,e}(x,y)=\zeta_{2^{m+1}}^{rx^2}\zeta_{2^{n+1}}^{sy^2}$ 
with $r,s\in \{1,-1,3,-3\}$.} 
If $m=n$, we further assume $r+s\neq 0$. 
In this case we have $\cG(q_{1,e})=\zeta_8^{r+s}(-1)^{m\frac{r^2-1}{8}+n\frac{s^2-1}{8}}$, 
$$\cG(q_{1,e},2)=2\zeta_8^{r+s}(-1)^{(m-1)\frac{r^2-1}{8}+(n-1)\frac{s^2-1}{8}},$$
and Eq.(\ref{D11}) implies that $(-1)^{\frac{r^2-1}{8}+\frac{s^2-1}{8}}=(-1)^{\frac{|G_o|^2-1}{8}}$. 

We get the following sub-cases. \\
(1) $|G_o|\equiv 1\mod 8$ and $r=s$. 
The smallest example is $n=2$, $G_o=\{0\}$, and $\Gamma=\Z_{15}$.\\ 
(2) $|G_o|\equiv 3\mod 8$ and $(r,s)=(1,3)$, $(-1,-3)$, $(3,1)$, or $(-3,-1)$. \\
(3) $|G_o|\equiv 5\mod 8$. and $(r,s)=(1,-3)$, $(-1,3)$, $(3,-1)$, or $(-3,1)$. \\
(4) $|G_o|\equiv 7\mod 8$ and  $r+s=0$. \\

\appendix

The first four sections of this (online only) appendix contain calculations for checking the modular relations and computing Verlinde coefficients and Frobenius-Schur indicators for the families of modular data in Sections 3-6, respectively. The last section of the appendix computes the modular data for the $\Z_2 $-de-equivariantizations of the $\Z_4 $-near-group categories referred to in Section 4.
\section{Calculations for Section 3}

\begin{proof}[Proof of Lemma 3.4]
\begin{align*}
\lefteqn{\sum_{x}S_{(u,0),x}\overline{S_{(u',0),x}}=\sum_{x}S_{(u,\pi),x}\overline{S_{(u',\pi),x}}} \\
 &=\frac{a^2+b^2}{2}\sum_{v\in U}\inpr{u'-u}{v}+a^2\sum_{g\in G_*}\inpr{u'-u}{g}
 +b^2\sum_{\gamma\in \Gamma_*}\inpr{u'-u}{\gamma} \\
 &=\frac{a^2}{2}\sum_{g\in G}\inpr{u'-u}{g}+ \frac{b^2}{2}\sum_{\gamma\in \Gamma}\inpr{u'-u}{\gamma} \\
 &=\delta_{u,u'}. 
\end{align*}
\begin{align*}
\lefteqn{\sum_{x}S_{(u,0),x}\overline{S_{(u',\pi),x}}} \\
 &=\frac{a^2-b^2}{2}\sum_{v\in U}\inpr{u'-u}{v}
+a^2\sum_{g\in G^*}\inpr{u'-u}{g}-b^2\sum_{\gamma\in \Gamma_*}\inpr{u'-u}{\gamma} \\
 &=\frac{a^2}{2}\sum_{g\in G}\inpr{u'-u}{g}-\frac{b^2}{2}\sum_{\gamma\in \Gamma}\inpr{u'-u}{\gamma}
 =\frac{\delta_{u,u'}}{2}-\frac{\delta_{u,u'}}{2}=0.\\
\end{align*}
\begin{align*}
\lefteqn{\sum_{x}S_{(u,0),x}\overline{S_{g,x}}=\sum_{x}S_{(u,\pi),x}\overline{S_{g,x}}} \\
 &=a^2\sum_{v\in U}\inpr{g-u}{v}+a^2\sum_{h\in G_*}(\inpr{g-u}{h}+\inpr{\theta(g)-u}{h})\\
 &=a^2\sum_{h\in G}\inpr{g-u}{h}=\delta_{g,u}=0. \\
\end{align*}
\begin{align*}
\lefteqn{\sum_{x}S_{(u,0),x}\overline{S_{\gamma,x}}=-\sum_{x}S_{(u,\pi),x}\overline{S_{\gamma,x}}} \\
 &=-b^2\sum_{v\in U}\inpr{\gamma-u}{v}-b^2\sum_{\xi\in \Gamma_*}(\inpr{\gamma-u}{\xi}+\inpr{\theta(\gamma)-u}{\xi})\\
 &=-b^2\sum_{\xi\in \Gamma}\inpr{\gamma-u}{\xi}=-\delta_{\gamma,u}=0. \\
\end{align*}
\begin{align*}
\lefteqn{\sum_{x}S_{g,x}\overline{S_{g',x}}} \\
 &=2a^2\sum_{v\in U}\inpr{g'-g}{v}
 +a^2\sum_{h\in G_*}(\inpr{-g}{h}+\inpr{-\theta(g)}{h}_1)(\inpr{g'}{h}+\inpr{\theta(g')}{h}) \\
 &=a^2(\sum_{h\in G}\inpr{g'-g}{v}+\sum_{h\in G}\inpr{g'-\theta(g)}{h}\\
 &=\delta_{g,g'}+\delta_{\theta(g),g'}=\delta_{g,g'}. 
\end{align*}
$\sum_{x}S_{g,x}\overline{S_{\gamma,x}}=0$ is obvious. 
\begin{align*}
\lefteqn{\sum_{x}S_{\gamma,x}\overline{S_{\gamma',x}}} \\
 &=2b^2\sum_{v\in U}\inpr{\gamma'-\gamma}{v}+b^2\sum_{\xi\in \Gamma_*}
 (\inpr{-\gamma}{\xi}+\inpr{-\theta(\gamma)}{\xi}) (\inpr{\gamma'}{\xi}+\inpr{\theta(\gamma')}{\xi})\\
 &=\delta_{\gamma,\gamma'}+\delta_{\theta(\gamma),\gamma'}=\delta_{\gamma,\gamma'}. 
\end{align*}
Thus $S$ is unitary. 
It is obvious that $T$ and $C$ are unitary.  
Since $\overline{S_{x,y}}=S_{\overline{x},y}=S_{x,\overline{y}}$ and $S$ is symmetric, we have $S^2=C$. 
Since $T_{x,x}=T_{\overline{x},\overline{x}}$, we have $CT=T C$. 

The only remaining condition is $(ST)^3=cC$, and it suffices to verify 
$$\sum_{x}S_{j,x}S_{j',x}T_{x,x}=cS_{j,j'}\overline{T_{j,j}T_{j',j'}}$$
for all $j,j'\in J$. 
\begin{align*}
\lefteqn{\sum_{x}S_{(u,0),x}S_{(u',0),x}T_{x,x}=\sum_{x}S_{(u,\pi),x}S_{(u',\pi),x}T_{x,x}} \\
 &=\frac{a^2+b^2}{2}\sum_{v\in U}\inpr{u+u'}{-v}q_0(v)
 +a^2\sum_{g\in G_*}\inpr{u+u'}{-g}q_1(g)
 +b^2\sum_{\gamma\in \Gamma_*}\inpr{u+u'}{-\gamma}q_2(\gamma) \\
 &=\frac{a^2}{2}\sum_{g\in G} q_1(u+u'-g)\overline{q_0(u+u)}
  +\frac{b^2}{2}\sum_{\gamma\in \Gamma}q_2(u+u'-\gamma)\overline{q_0(u+u')} \\
 &=(\frac{a}{2}\cG(q_1)+\frac{b}{2}\cG(q_2))\overline{\inpr{u}{u'}q_0(u)q_0(u')}\\
 &=cS_{(u,0),(u',0)}\overline{T_{(u,0),(u',0)}}=cS_{(u,\pi),(u',\pi)}\overline{T_{(u,\pi),(u',\pi)}}.
\end{align*}
\begin{align*}
\lefteqn{\sum_{x}S_{(u,0),x}S_{(u',\pi),x}T_{x,x}} \\
 &=\frac{a^2-b^2}{2}\sum_{v\in U}\inpr{u+u'}{-v}q_1(v)
 +a^2\sum_{g\in G_*}\inpr{u+u'}{-g}q_1(g)
 -b^2\sum_{\gamma\in \Gamma_*}\inpr{u+u'}{-\gamma}q_2(\gamma)\\
 &=\frac{a^2}{2}\sum_{g\in G} q_1(u+u'-g)\overline{q_0(u+u)}
  -\frac{b^2}{2}\sum_{\gamma\in \Gamma}q_2(u+u'-\gamma)\overline{q_0(u+u')} \\
 &=(\frac{a}{2}\cG(q_1)-\frac{b}{2}\cG(q_2))\overline{\inpr{u}{u'}q_0(u)q_0(u')}\\
 &=\frac{a\cG(q_1)}{2}-\frac{b\cG(q_2)}{2}
 =cS_{(u,0),(u',\pi)}\overline{T_{(u,0),(u,0)}T_{(u',\pi),(u',\pi)}}.
\end{align*}
\begin{align*}
\lefteqn{\sum_{x}S_{(u,0),x}S_{g,x}T_{x,x}=S_{(u,\pi),x}S_{g,x}T_{x,x}} \\
 &=a^2\sum_{v\in U}\inpr{u+g}{-v}q_1(v)+a^2\sum_{h\in G_*}(\inpr{u+g}{-h}+\inpr{u+\theta(g)}{-h})q_1(h)\\
 &=a^2\sum_{h\in G}q_1(u+g-h)\overline{q_1(u+g)}
 =a\cG(q_1)\overline{\inpr{u}{g}q_0(u)q_1(g)}\\
 &=cS_{(u,0),g}\overline{T_{(u,0),(u,0)}T_{g,g}}
 =cS_{(u,\pi),g}\overline{T_{(u,\pi),(u,\pi)}T_{g,g}}.
\end{align*}
\begin{align*}
\lefteqn{\sum_{x}S_{(u,0),x}S_{\gamma,x}T_{x,x}=-\sum_{x}S_{(u,\pi),x}S_{\gamma,x}T_{x,x}} \\
 &=-b^2\sum_{v\in U}\inpr{u+\gamma}{-v}q_0(v)-b^2\sum_{\xi\in \Gamma_*}(\inpr{u+\gamma}{-\xi}
 +\inpr{u+\theta(\gamma)}{-\xi})q_2(\xi)\\
 &=-b^2\sum_{\xi\in \Gamma}q_2(u+\gamma-\xi)\overline{q_2(u+\gamma)}
 =-b\cG(q_2)\overline{\inpr{u}{\gamma}q_0(u)q_2(\gamma)}\\
 &=cS_{(u,0),\gamma}\overline{T_{(u,0),(u,0)}T_{\gamma,\gamma}}
 =-cS_{(u,\pi),\gamma}\overline{T_{(u,\pi),(u,\pi)}T_{\gamma,\gamma}}.
\end{align*}
 \begin{align*}
\lefteqn{\sum_{x}S_{g,x}S_{g',x}T_{x,x}} \\
 &=2a^2\sum_{v\in U}\inpr{g+g'}{-v}_1q_1(v)
 +a^2\sum_{h\in G_*}(\inpr{g}{-h}+\inpr{\theta(g)}{-h})(\inpr{g'}{-h}+\inpr{\theta(g')}{-h})q_1(h)\\
 &=a^2\sum_{h\in G}(q_1(g+g'-h)\overline{q_1(g+g')}+q_1(\theta(g)+g'-h)\overline{q_1(\theta(g)+g')})\\
 &=ca\overline{(\inpr{g}{g'}q_1(g)q_1(g')+\inpr{\theta(g)}{g'}q_1(g)q_1(g')}\\
 &=cS_{g,g'}\overline{T_{g,g}T_{g',g'}}.
\end{align*}
$$\sum_{x}S_{g,x}S_{\gamma,x}T_{x,x}=0=cS_{g,\gamma}\overline{T_{g,g}T_{\gamma,\gamma}}.$$
 \begin{align*}
\lefteqn{\sum_{x}S_{\gamma,x}S_{\gamma',x}T_{x,x}} \\
 &=2b^2\sum_{v\in U}\inpr{\gamma+\gamma'}{-v} 
 +b^2\sum_{\xi\in \Gamma_*} (\inpr{\gamma}{-\xi}+\inpr{\theta(\gamma)}{-\xi})
 (\inpr{\gamma'}{-\xi}\inpr{\theta(\gamma')}{-\xi})q_2(\xi)\\
 &=b^2\sum_{\xi\in \Gamma}(q_2(\gamma+\gamma'-\xi)\overline{q_2(\gamma+\gamma')}
 +q_2(\theta(\gamma)+\gamma'-\xi)\overline{q_2(\theta(\gamma)+\gamma')})\\
 &=b\cG(q_2)(\overline{q_2(\gamma+\gamma')+q_2(\gamma-\gamma')})\\
 &=cS_{\gamma,\gamma'}\overline{T_{\gamma,\gamma}T_{\gamma',\gamma'}}.
\end{align*}
\end{proof}

\begin{proof}[Proof of Lemma 3.5] 
\begin{align*}
\lefteqn{N_{(u,0),(u',0),(u'',0)}=N_{(u,0),(u',\pi),(u'',\pi)}} \\
 &=(\frac{a-b}{2})^2\sum_{v\in U}\inpr{u+u'+u''}{-v}+(\frac{a+b}{2})^2\sum_{v\in U}\inpr{u+u'+u''}{-v} \\
 &+a^2\sum_{g\in G_*}\inpr{u+u'+u''}{-g}+b^2\sum_{\gamma\in \Gamma_*}\inpr{u+u'+u''}{-\gamma}\\
 &=\frac{a^2}{2}\sum_{g\in G}\inpr{u+u'+u''}{-g}+\frac{b^2}{2}\sum_{\gamma\in \Gamma}\inpr{u+u'+u''}{-\gamma}\\
 &=\delta_{u+u'+u'',0}.
\end{align*}
\begin{align*}
\lefteqn{N_{(u,0),(u',0),(u'',\pi)}} \\
 &=\frac{a^2-b^2}{4}\sum_{v\in U}\inpr{u+u'+u''}{-v}+\frac{a^2-b^2}{4}\sum_{v\in U}\inpr{u+u'+u''}{-v} \\
 &+a^2\sum_{g\in G_*}\inpr{u+u'+u''}{-g}-b^2\sum_{\gamma\in \Gamma_*}\inpr{u+u'+u''}{-\gamma}\\
 &=\frac{a^2}{2}\sum_{g\in G}\inpr{u+u'+u''}{g}-\frac{b^2}{2}\sum_{\gamma\in \Gamma}\inpr{u+u'+u''}{\gamma}=0.
\end{align*}
\begin{align*}
\lefteqn{N_{(u,0),(u',0),g}=N_{(u,0),(u',\pi),g}} \\
 &=\frac{a(a-b)}{2}\sum_{v\in U}\inpr{u+u'+g}{-v}+\frac{a(a+b)}{2}\sum_{v\in U}\inpr{u+u'+g}{-v} \\
 &+a^2\sum_{h\in G_*}\inpr{u+u'}{-h}(\inpr{g}{-h}+\inpr{\theta(g)}{-h})\\
 &=a^2\sum_{h\in G}\inpr{u+u'+g}{h}=\delta_{u+u'+g,0}=0.
\end{align*}
\begin{align*}
\lefteqn{N_{(u,0),(u',0),\gamma}=-N_{(u,0),(u',\pi),\gamma}} \\
 &=\frac{b(a-b)}{2}\sum_{v\in U}\inpr{u+u'+\gamma}{-v}-\frac{b(a+b)}{2}\sum_{v\in U}\inpr{u+u'+\gamma}{-v} \\
 &-b^2\sum_{\xi\in \Gamma_*}\inpr{u+u'}{-\xi}(\inpr{\gamma}{-\xi}+\inpr{\theta(\gamma)}{-\xi})\\
 &=-b^2\sum_{\xi\in \Gamma}\inpr{u+u'+\gamma}{\xi}=-\delta_{u+u'+\gamma,0}=0.
\end{align*}
\begin{align*}
\lefteqn{N_{(u,0),g,g'}} \\
 &=2a^2\sum_{v\in U}\inpr{u+g+g'}{-v}+a^2\sum_{h\in G_*}\inpr{u}{-h}(\inpr{g}{-h}+\inpr{\theta(g)}{-h})(\inpr{g'}{-h}+\inpr{\theta(g')}{-h}) \\
 &=a^2\sum_{h\in G}(\inpr{u+g+g'}{h}+\inpr{u+\theta(g)+g'}{h})\\
 &=\delta_{u+g+g',0}+\delta_{u+\theta(g)+g',0}.
\end{align*}
$$N_{(u,0),g,\gamma}=ab\sum_{v\in U}\inpr{u}{-v}\inpr{g}{-v}\inpr{\gamma}{-v}-ab\sum_{v\in U}\inpr{u}{-v}\inpr{g}{-v}\inpr{\gamma}{-v}=0.$$
\begin{align*}
\lefteqn{N_{(u,0),\gamma,\gamma'}} \\
 &=2b^2\sum_{v\in U}\inpr{u+\gamma+\gamma'}{-v}+b^2\sum_{\xi\in \Gamma_*}\inpr{u}{-\xi}(\inpr{\gamma}{-\xi}+\inpr{\theta(\gamma)}{-\xi})(\inpr{\gamma'}{-\xi}+\inpr{\theta(\gamma')}{-\xi}) \\
 &=b^2\sum_{\xi\in \Gamma}(\inpr{u+\gamma+\gamma'}{\xi}+\inpr{u+\theta(\gamma)+\gamma'}{\xi})\\
 &=\delta_{u+\gamma+\gamma',0}+\delta_{u+\theta(\gamma)+\gamma',0}.
\end{align*}
\begin{align*}
\lefteqn{N_{(u,\pi),(u',\pi),(u'',\pi)}} \\
 &=\frac{(a+b)^3}{4(a-b)}\sum_{v\in U}\inpr{u+u'+u''}{-v}+\frac{(a-b)^3}{4(a+b)} \sum_{v\in U}\inpr{u+u'+u''}{-v} \\
 &+a^2\sum_{g\in G_*}\inpr{u+u'+u''}{-g}-b^2\sum_{\gamma\in \Gamma_*}\inpr{u+u'+u''}{-\gamma}\\
 &=(\frac{(a+b)^4+(a-b)^4}{4(a^2-b^2)}-\frac{a^2-b^2}{2}\sum_{v\in U}\inpr{u+u'+u''}{v}\\
 &+\frac{a^2}{2}\sum_{g\in G}\inpr{u+u'+u''}{g}-\frac{b^2}{2}\sum_{\gamma\in \Gamma}\inpr{u+u'+u''}{\gamma}\\
 &=\frac{4a^2b^2}{a^2-b^2}\sum_{v\in U}\inpr{u+u'+u''}{v}+\frac{\delta_{u+u'+u'',0}}{2}-\frac{\delta_{u+u'+u'',0}}{2}\\
 &=\frac{4|U|}{|\Gamma|-|G|}\frac{1}{|U|}\sum_{v\in U}\inpr{u+u'+u''}{v}. 
\end{align*}
\begin{align*}
\lefteqn{N_{(u,\pi),(u',\pi),g}} \\
 & \frac{a(a+b)^2}{2(a-b)}\sum_{v\in U}\inpr{u+u'+g}{-v}+\frac{a(a-b)^2}{2(a+b)}\sum_{v\in U}\inpr{u+u'+g}{-v} \\
 &+a^2\sum_{h\in G_*}\inpr{u+u'}{-h}(\inpr{g}{-h}+\inpr{\theta(g)}{-h})\\
 &=(a\frac{(a+b)^3+(a-b)^3}{2(a^2-b^2)}-a^2)\sum_{v\in U}\inpr{u+u'+g}{v}+a^2\sum_{h\in G}\inpr{u+u'+g}{h} \\
 &=\frac{4a^2b^2}{a^2-b^2}\sum_{v\in U}\inpr{u+u'+g}{v}+\delta_{u+u'+g,0} \\
 &=\frac{4|U|}{|\Gamma|-|G|}\frac{1}{|U|}\sum_{v\in U}\inpr{u+u'+g}{v}.
\end{align*}
\begin{align*}
\lefteqn{N_{(u,\pi),(u',\pi),\gamma}} \\
 &=\frac{b(a+b)^2}{2(a-b)}\sum_{v\in U}\inpr{u+u'+\gamma}{-v}-\frac{b(a-b)^2}{2(a+b)}\sum_{v\in U}\inpr{u+u'+\gamma}{-v} \\
 &-b^2\sum_{\xi\in \Gamma_*}\inpr{u+u'}{-\xi}(\inpr{\gamma}{-\xi}+\inpr{\theta(\gamma)}{-\xi})\\
 &=(b\frac{(a+b)^3-(a-b)^3}{2(a^2-b^2)}+b^2)\sum_{v\in U}\inpr{u+u'+\gamma}{v}-b^2\sum_{\xi\in \Gamma}\inpr{u+u'+\gamma}{\xi} \\
 &=\frac{4a^2b^2}{a^2-b^2}\sum_{v\in U}\inpr{u+u'+\gamma}{v}-\delta_{u+u'+\gamma,0} \\
 &=\frac{4|U|}{|\Gamma|-|G|}\frac{1}{|U|}\sum_{v\in U}\inpr{u+u'+\gamma}{v}.\\
\end{align*}
\begin{align*}
\lefteqn{N_{(u,\pi),g,g'}} \\
 &=a^2\frac{a+b}{a-b}\sum_{v\in U}\inpr{u+g+g'}{-v}+a^2\frac{a-b}{a+b}\sum_{v\in U}\inpr{u+g+g'}{-v}\\
 &+a^2\sum_{h\in G_*}\inpr{u}{-h}(\inpr{g}{-h}+\inpr{\theta(g)}{-h})(\inpr{g'}{-h}+\inpr{\theta(g')}{-h}) \\
 &=a^2(\frac{(a+b)^2+(a-b)^2}{a^2-b^2}-2)\sum_{v\in U}\inpr{u+g+g'}{v}
 +a^2\sum_{h\in G}(\inpr{u+g+g'}{h}+\inpr{u+\theta(g)+g'}{h})\\
 &=\frac{4a^2b^2}{a^2-b^2}\sum_{v\in U}\inpr{u+g+g'}{v} +\delta_{u+g+g',0}+\delta_{u+\theta(g)+g',0}\\
 &=\frac{4|U|}{|\Gamma|-|G|}\frac{1}{|U|}\sum_{v\in U}\inpr{u+g+g'}{v} +\delta_{u+g+g',0}+\delta_{u+\theta(g)+g',0}.
\end{align*}
\begin{align*}
\lefteqn{N_{(u,\pi),g,\gamma}} \\
 &=\frac{ab(a+b)}{a-b}\sum_{v\in U}\inpr{u}{-v}\inpr{g}{-v}\inpr{\gamma}{-v}
-\frac{ab(a-b)}{a+b}\sum_{v\in U}\inpr{u}{-v}\inpr{g}{-v}\inpr{\gamma}{-v} \\
 &=ab\frac{(a+b)^2-(a-b)^2}{a^2-b^2}\sum_{v\in U}\inpr{u}{v}\inpr{g}{v}\inpr{\gamma}{v}\\
 &=\frac{4a^2b^2}{a^2-b^2}\sum_{v\in U}\inpr{u}{v}\inpr{g}{v}\inpr{\gamma}{v}\\
 &=\frac{4|U|}{|\Gamma|-|G|}\frac{1}{|U|}\sum_{v\in U}\inpr{u}{v}\inpr{g}{v}\inpr{\gamma}{v}.
\end{align*}
\begin{align*}
\lefteqn{N_{(u,\pi),\gamma,\gamma'}} \\
 &=b^2\frac{a+b}{a-b}\sum_{v\in U}\inpr{u+\gamma+\gamma'}{-v}+b^2\frac{a-b}{a+b}\sum_{v\in U}\inpr{u+\gamma+\gamma'}{-v}\\
 &-b^2\sum_{\xi\in \Gamma_*}\inpr{u}{-\xi}(\inpr{\gamma}{-\xi}+\inpr{\theta(\gamma)}{-\xi})(\inpr{\gamma'}{-\xi}+\inpr{\theta(\gamma')}{-\xi}) \\
 &=b^2(\frac{(a+b)^2+(a-b)^2}{a^2-b^2}+2)\sum_{v\in U}\inpr{u+\gamma+\gamma'}{v}
 -b^2\sum_{\xi\in \Gamma}(\inpr{u+\gamma+\gamma'}{\xi}+\inpr{u+\theta(\gamma)+\gamma'}{\xi})\\
 &=\frac{4a^2b^2}{a^2-b^2}\sum_{v\in U}\inpr{u+\gamma+\gamma'}{v}-\delta_{u+\gamma+\gamma',0}-\delta_{u+\theta(\gamma)+\gamma',0}\\
 &=\frac{4|U|}{|\Gamma|-|G|}\frac{1}{|U|}\sum_{v\in U}\inpr{u+\gamma+\gamma'}{v}
 -\delta_{u+\gamma+\gamma',0}-\delta_{u+\theta(\gamma)+\gamma',0}.
\end{align*}
\begin{align*}
\lefteqn{N_{g,g',g''}} \\
 &=\frac{2a^3}{a-b}\sum_{v\in U}\inpr{g+g'+g''}{-v}+ \frac{2a^3}{a+b}\sum_{v\in U}\inpr{g+g'+g''}{-v}\\
 &+a^2\sum_{h\in G_*}(\inpr{g}{-h}+\inpr{\theta(g)}{-h})(\inpr{g'}{-h}+\inpr{\theta(g')}{-h})(\inpr{g''}{-h}+\inpr{\theta(g'')}{-h}) \\
 &=(\frac{4a^4}{a^2-b^2}-4a^2)\sum_{v\in U}\inpr{g+g'+g''}{-v}\\
 &+a^2\sum_{h\in G}(\inpr{g+g'+g''}{h}+\inpr{\theta(g)+g'+g''}{h}+\inpr{g+\theta(g')+g''}{h}+\inpr{g+g'+\theta(g'')}{h})\\
 &=\frac{4a^2b^2}{a^2-b^2}\sum_{v\in U}\inpr{g+g'+g''}{v}+\delta_{g+g'+g'',0}+\delta_{\theta(g)+g'+g'',0}
 +\delta_{g+\theta(g')+g'',0}+\delta_{g+g'+\theta(g''),0}\\
 &=\frac{4|U|}{|\Gamma|-|G|}\frac{1}{|U|}\sum_{v\in U}\inpr{g+g'+g''}{v}\\
 &+\delta_{g+g'+g'',0}+\delta_{\theta(g)+g'+g'',0}
 +\delta_{g+\theta(g')+g'',0}+\delta_{g+g'+\theta(g''),0}.
\end{align*}
\begin{align*}
\lefteqn{N_{g,g',\gamma}} \\
 &=\frac{2a^2b}{a-b}\sum_{v\in U}\inpr{g+g'}{-v}\inpr{\gamma}{-v}-\frac{2a^2b}{a+b}\sum_{v\in U}\inpr{g+g'}{-v}\inpr{\gamma}{-v} \\
 &=\frac{4a^2b^2}{a^2-b^2}\sum_{v\in U}\inpr{g+g'}{v}\inpr{\gamma}{v}\\
 &=\frac{4|U|}{|\Gamma|-|G|}\frac{1}{|U|}\sum_{v\in U}\inpr{g+g'}{v}\inpr{\gamma}{v}.
\end{align*}
\begin{align*}
\lefteqn{N_{g,\gamma,\gamma'}} \\
 &=\frac{2ab^2}{a-b}\sum_{v\in U}\inpr{g}{-v}\inpr{\gamma+\gamma'}{-v}+\frac{2ab^2}{a+b}\sum_{v\in U}\inpr{g}{-v}\inpr{\gamma+\gamma'}{-v} \\
 &=\frac{4a^2b^2}{a^2-b^2}\sum_{v\in U}\inpr{g}{v}\inpr{\gamma+\gamma'}{v}\\
 &=\frac{4|U|}{|\Gamma|-|G|}\frac{1}{|U|}\sum_{v\in U}\inpr{g}{v}\inpr{\gamma+\gamma'}{v}.
\end{align*}
\begin{align*}
\lefteqn{N_{\gamma,\gamma',\gamma''}} \\
 &=\frac{2b^3}{a-b}\sum_{v\in U}\inpr{\gamma+\gamma'+\gamma''}{-v}-\frac{2b^3}{a+b}\sum_{v\in U}\inpr{\gamma+\gamma'+\gamma''}{-v} \\
 &-b^2\sum_{\xi\in \Gamma_*}(\inpr{\gamma}{-\xi}+\inpr{\theta(\gamma)}{-\xi})(\inpr{\gamma'}{-\xi}+\inpr{{\gamma}'^\theta}{-\xi})(\inpr{\gamma}{-\xi}+\inpr{\theta(\gamma'')}{-\xi}) \\
 &=(\frac{4b^4}{a^2-b^2}+4b^2)\sum_{v\in U}\inpr{\gamma+\gamma'+\gamma''}{v}\\
 &-b^2\sum_{\xi\in \Gamma}(\inpr{\gamma+\gamma'+\gamma''}{v}+\inpr{\theta(\gamma)+\gamma'+\gamma''}{v}
 +\inpr{\gamma+\theta(\gamma')+\gamma''}{v}+\inpr{\gamma+\gamma'+\theta(\gamma'')}{v})\\
 &=\frac{4a^2b^2}{a^2-b^2}\sum_{v\in U}\inpr{\gamma+\gamma'+\gamma''}{v}
 -\delta_{\gamma+\gamma'+\gamma'',0}-\delta_{\theta(\gamma)+\gamma'+\gamma'',0}
 -\delta_{\gamma+\theta(\gamma')+\gamma'',0}-\delta_{\gamma+\gamma'+\theta(\gamma''),0}\\
 &=\frac{4|U|}{|\Gamma|-|G|}\frac{1}{|U|}\sum_{v\in U}\inpr{\gamma+\gamma'+\gamma''}{v}\\
 &-\delta_{\gamma+\gamma'+\gamma'',0}-\delta_{\theta(\gamma)+\gamma'+\gamma'',0}
 -\delta_{\gamma+\theta(\gamma')+\gamma'',0}-\delta_{\gamma+\gamma'+\theta(\gamma''),0}.
\end{align*}
\end{proof}

\begin{proof}[Proof of Lemma 3.8] 
\begin{align*}
\lefteqn{\nu_m((u,0))} \\
 &=\sum_{u'\in U}N^{(u,0)}_{(u',0),(u-u',0)}\frac{(a-b)^2}{4}\frac{q_0(u')^m}{q_0(u-u')^m}
 + \sum_{u'\in U}N^{(u,0)}_{(u',\pi),(u-u',\pi)}\frac{(a+b)^2}{4}\frac{q_0(u')^m}{q_0(u-u')^m}\\
 &=\sum_{g\in G_*}N^{(u,0)}_{g,u-g}a^2\frac{q_1(g)^m}{q_1(u-g)^m}
 +\sum_{\gamma\in \Gamma_*}N^{(u,0)}_{\gamma,u-\gamma}b^2\frac{q_2(\gamma)^m}{q_2(u-\gamma)^m} \\
 &=\frac{a^2}{2}\sum_{g\in G}\frac{q_0(g)^m}{q_0(u-g)^m}
 +\frac{b^2}{2}\sum_{\gamma\in \Gamma}\frac{q_1(\gamma)^m}{q_2(u-\gamma)^m}. 
\end{align*}
Since $q_1(g)/q_1(u-g)=\inpr{u}{g}\overline{q_0(u)}$, we have 
$$\frac{a^2}{2}\sum_{g\in G}\frac{q_0(g)^m}{q_0(u-g)^m}=\frac{1}{2|G|}\sum_{g\in G}\inpr{mu}{g}q_0(u)^{-m}=\frac{\delta_{mu,0}q_0(u)^m}{2},$$
and $\nu_m((u,0))=\delta_{mu,0}q_0(u)^m$. 

Note that we have $\cG(q_1,q_2,v,m)=\cG(q_1,q_2,-v,m)$ and $\cG(q_1,q_2,v,-m)=\overline{\cG(q_1,q_2,v,m)}$. 
We set 
$$A(g,m)=2a\sum_{h\in G_*}\inpr{g}{h}q_1(h)^m,\quad B(\gamma,m)=2b\sum_{\xi\in \Gamma_*}\inpr{\gamma}{\xi}q_2(\xi)^m.$$
$$C(v,m)=(a+b)\sum_{u\in U}\inpr{v}{u}q_0(u)^m.$$
Then 
$$\cG(q_1,q_2,v,m)=A(v,m)+B(v,m)+C(v,m).$$

For $\nu_m((u,\pi))$, we have 
\begin{align*}
\lefteqn{\nu_m((u,\pi))} \\
 &=\sum_{u'\in U}N^{(u,\pi)}_{(u',0),(u-u',\pi)}\frac{a^2-b^2}{4}(\frac{q_0(u')^m}{q_0(u-u')^m}+\frac{q_0(u-u')^m}{q_0(u')^m}) 
 +\sum_{u',u''\in U}N^{(u,\pi)}_{(u',\pi),(u'',\pi)}\frac{(a+b)^2}{4}\frac{q_0(u')^m}{q_0(u'')^m} \\
 &+\sum_{u'\in U,\;g\in G_*}N^{(u,\pi)}_{(u',\pi),g}\frac{a(a+b)}{2}(\frac{q_0(u')^m}{q_1(g)^m}+\frac{q_1(g)^m}{q_0(u')^m}) \\
 &+\sum_{u'\in U,\;\gamma\in \Gamma_*}N^{(u,\pi)}_{(u',\pi),\gamma}\frac{b(a+b)}{2}(\frac{q_0(u')^m}{q_2(\gamma)^m}+\frac{q_2(\gamma)^m}{q_0(u')^m})  \\
 &+\sum_{g,h\in G_*}N^{(u,\pi)}_{g,h}a^2\frac{q_1(g)^m}{q_1(h)^m}
 +\sum_{g\in G_*,\;\gamma\in \Gamma_*}N^{(u,\pi)}_{g,\gamma}ab(\frac{q_1(g)^m}{q_2(\gamma)^m}+\frac{q_2(\gamma)^m}{q_1(g)^m})
 +\sum_{\gamma,\xi\in \Gamma_*}N^{(u,\pi)}_{\gamma,\xi}b^2\frac{q_2(\gamma)^m}{q_2(\xi)^m},
\end{align*}
and \begin{align*}
\lefteqn{\sum_{u'\in U}N^{(u,\pi)}_{(u',0),(u-u',\pi)}\frac{a^2-b^2}{4}(\frac{q_0(u')^m}{q_0(u-u')^m}+\frac{q_0(u-u')^m}{q_0(u')^m}) } \\
 &=\frac{a^2-b^2}{4}\sum_{u'\in U}(\frac{q_0(u')^m}{q_0(u-u')^m}+\frac{q_0(u-u')^m}{q_0(u')^m})
 =\frac{a^2-b^2}{2}\sum_{u'\in U}\frac{q_0(u')^m}{q_0(u-u')^m},
\end{align*}
\begin{align*}
\lefteqn{\sum_{u',u''\in U}N^{(u,\pi)}_{(u',\pi),(u'',\pi)}\frac{(a+b)^2}{4}\frac{q_0(u')^m}{q_0(u'')^m} } \\
 &=\frac{(a+b)^2}{4}\sum_{u',u''\in U}\frac{4}{|\Gamma|-|G|}\sum_{v\in U}\inpr{u'+u''-u}{v}\frac{q_0(u')^m}{q_0(u'')^m}\\  
 &=\frac{1}{|\Gamma|-|G|}\sum_{v\in U}\inpr{-u}{v}C(u',m)C(u'',-m), 
\end{align*} 
\begin{align*}
\lefteqn{\sum_{u'\in U,\;g\in G_*}N^{(u,\pi)}_{(u',\pi),g}\frac{a(a+b)}{2}(\frac{q_0(u')^m}{q_1(g)^m}+\frac{q_1(g)^m}{q_0(u')^m})} \\
 &=\frac{a(a+b)}{2}\sum_{u'\in U,\;g\in G_*}\frac{4}{|\Gamma|-|G|}\sum_{v\in U}\inpr{u'-u+g}{v}(\frac{q_0(u')^m}{q_1(g)^m}+\frac{q_1(g)^m}{q_0(u')^m}) \\
 &=\frac{1}{|\Gamma|-|G|}\sum_{v\in U}\inpr{-u}{v}(C(v,m)A(v,-m)+A(v,m)C(v,-m)), 
\end{align*}
\begin{align*}
\lefteqn{\sum_{u'\in U,\;\gamma\in \Gamma_*}N^{(u,\pi)}_{(u',\pi),\gamma}\frac{b(a+b)}{2}(\frac{q_0(u')^m}{q_2(\gamma)^m}+\frac{q_2(\gamma)^m}{q_0(u')^m}) } \\
 &=\frac{b(a+b)}{2}\sum_{u'\in U,\;\gamma\in \Gamma_*}\frac{4}{|\Gamma|-|G|}\sum_{v\in U}\inpr{u'-u+\gamma}{v}(\frac{q_0(u')^m}{q_2(\gamma)^m}+\frac{q_2(\gamma)^m}{q_0(u')^m})  \\
 &=\frac{1}{|\Gamma|-|G|}\sum_{v\in U}\inpr{-u}{v}(C(v,m)B(v,-m)+B(v,m)C(v,-m)),
\end{align*}
\begin{align*}
\lefteqn{\sum_{g,h\in G_*}N^{(u,\pi)}_{g,h}a^2\frac{q_1(g)^m}{q_1(h)^m}} \\
 &=a^2\sum_{g,h\in G_*}(\frac{4}{|\Gamma|-|G|}\sum_{v\in U}\inpr{g+h-u}{v}+\delta_{g+h,u}+\delta_{g+\theta(h),u}) \frac{q_1(g)^m}{q_1(h)^m} \\
 &=\frac{1}{|\Gamma|-|G|}\sum_{v\in U}\inpr{-u}{v}A(v,m)A(v,-m)+a^2\sum_{g,h\in G_*} (\delta_{g+h,u}+\delta_{g+\theta(h),u}) \frac{q_1(g)^m}{q_1(h)^m}\\
 &=\frac{1}{|\Gamma|-|G|}\sum_{v\in U}\inpr{-u}{v}A(v,m)A(v,-m)+a^2\sum_{g\in G_*} \frac{q_1(g)^m}{q_1(u-g)^m} \\
 &=\frac{1}{|\Gamma|-|G|}\sum_{v\in U}\inpr{-u}{v}A(v,m)A(v,-m)+\frac{a^2}{2}\sum_{g\in G} \frac{q_1(g)^m}{q_1(u-g)^m}-\frac{a^2}{2}\sum_{g\in U} \frac{q_1(g)^m}{q_1(u-g)^m}\\
 &=\frac{1}{|\Gamma|-|G|}\sum_{v\in U}\inpr{-u}{v}A(v,m)A(v,-m)+\frac{\delta_{mu,0}q_0(u)^m}{2}-\frac{a^2}{2}\sum_{g\in U} \frac{q_1(g)^m}{q_1(u-g)^m}.
\end{align*}
\begin{align*}
\lefteqn{\sum_{g\in G_*,\;\gamma\in \Gamma_*}N^{(u,\pi)}_{g,\gamma}ab(\frac{q_1(g)^m}{q_2(\gamma)^m}+\frac{q_2(\gamma)^m}{q_1(g)^m})} \\
 &=ab\sum_{g\in G_*,\;\gamma\in \Gamma_*}\frac{4}{|\Gamma|-|G|}\sum_{v\in U}\inpr{-u}{v}\inpr{g}{v}\inpr{\gamma}{v}(\frac{q_1(g)^m}{q_2(\gamma)^m}+\frac{q_2(\gamma)^m}{q_1(g)^m}) \\
 &=\frac{1}{|\Gamma|-|G|}\sum_{v\in U}\inpr{-u}{v}(A(v,m)B(v,-m)+B(v,m)A(v,-m)), \\
\end{align*}
\begin{align*}
\lefteqn{\sum_{\gamma,\xi\in \Gamma_*}N^{(u,\pi)}_{\gamma,\xi}b^2\frac{q_2(\gamma)^m}{q_2(\xi)^m}} \\
 &=b^2\sum_{\gamma,\xi\in \Gamma_*}(\frac{4}{|\Gamma|-|G|}\sum_{v\in U}\inpr{-u+\gamma+\xi}{v}-\delta_{\gamma+\xi,u}-\delta_{\gamma+\theta(\xi),u})
 \frac{q_2(\gamma)^m}{q_2(\xi)^m} \\
 &=\frac{1}{|\Gamma|-|G|}B(v,m)B(v,-m)-b^2\sum_{\gamma\in \Gamma_*}\frac{q_2(\gamma)^m}{q_2(u-\gamma)^m}  \\
 &=\frac{1}{|\Gamma|-|G|}B(v,m)B(v,-m)-\frac{\delta_{mu,0}q_0(u)^m}{2}+\frac{b^2}{2}\sum_{\gamma\in U}\frac{q_2(\gamma)^m}{q_2(u-\gamma)^m}.
\end{align*}
Thus we get 
\begin{align*}
\lefteqn{\nu_m((u,\pi))} \\
 &=\frac{1}{|\Gamma|-|G|}\sum_{v\in U}\inpr{-u}{v}\cG(q_1,q_2,v,m)\cG(q_1,q_2,v,-m) 
 =\frac{1}{|\Gamma|-|G|}\sum_{v\in U}\inpr{u}{v}|\cG(q_1,q_2,v,m)|^2. 
\end{align*}

For $\nu_m(g)$, we have 
\begin{align*}
\lefteqn{\nu_m(g)} \\
 &=\sum_{u\in U}N^g_{(u,0),g-u}\frac{a(a-b)}{2}(\frac{q_0(u)^m}{q_1(g-u)^m}+\frac{q_1(g-u)^m}{q_0(u)^m})+ 
 \sum_{u,u'\in U}N^g_{(u,\pi),(u',\pi)}\frac{(a+b)^2}{4}\frac{q_0(u)^m}{q_0(u')^m}\\
 &+\sum_{u\in U\; h\in G_*}N^g_{(u,\pi),h}\frac{a(a+b)}{2}(\frac{q_0(u)^m}{q_1(h)^m}+\frac{q_1(h)^m}{q_0(u)^m}) \\
 &+\sum_{u\in U\; \gamma\in \Gamma_*}N^g_{(u,\pi),\gamma}\frac{b(a+b)}{2}(\frac{q_0(u)^m}{q_2(\gamma)^m}+\frac{q_2(\gamma)^m}{q_0(u)^m})\\
 &+\sum_{h,h'\in G_*}N^g_{h,h'}a^2\frac{q_1(h)^m}{q_1(h')^m}
 +\sum_{h\in G_*,\;\gamma\in \Gamma_*}N^g_{h,\gamma}ab(\frac{q_1(h)^m}{q_2(\gamma)^m}+\frac{q_2(\gamma)^m}{q_1(h)^m})
 +\sum_{\gamma,\xi\in \Gamma_*}N^g_{\gamma,\xi}b^2\frac{q_2(\gamma)^m}{q_2(\xi)^m},
\end{align*}
and 
$$\sum_{u\in U}N^g_{(u,0),g-u}\frac{a(a-b)}{2}(\frac{q_0(u)^m}{q_1(g-u)^m}+\frac{q_1(g-u)^m}{q_0(u)^m})
=\frac{a(a-b)}{2}\sum_{u\in U}(\frac{q_0(u)^m}{q_1(g-u)^m}+\frac{q_1(g-u)^m}{q_0(u)^m}),$$
\begin{align*}
\lefteqn{\sum_{u,u'\in U}N^g_{(u,\pi),(u',\pi)}\frac{(a+b)^2}{4}\frac{q_0(u)^m}{q_1(u')^m}} \\
 &=\frac{(a+b)^2}{4}\sum_{u,u'\in U}\frac{4}{|\Gamma|-|G|}\sum_{v\in U}\inpr{u+u'-g}{v}\frac{q_0(u)^m}{q_0(u')^m} \\
 &=\frac{1}{|\Gamma|-|G|}\sum_{v\in U}\inpr{-g}{v}C(v,m)C(v,-m).
\end{align*}
\begin{align*}
\lefteqn{\sum_{u\in U\; h\in G_*}N^g_{(u,\pi),h}\frac{a(a+b)}{2}(\frac{q_0(u)^m}{q_1(h)^m}+\frac{q_1(h)^m}{q_0(u)^m}) } \\
 &=\frac{a(a+b)}{2}\sum_{u\in U\; h\in G_*}(\frac{4}{|\Gamma|-|G|}\sum_{v\in U}\inpr{u+h-g}{v}+\delta_{g-h,u}+\delta_{g-\theta(h),u})(\frac{q_0(u)^m}{q_1(h)^m}+\frac{q_1(h)^m}{q_0(u)^m})  \\
 &=\frac{1}{|\Gamma|-|G|}\sum_{v\in U}(C(v,m)A(v,-m)+A(v,m)C(v,-m))\\
 &+\frac{a(a+b)}{2}\sum_{u\in U}(\frac{q_0(u)^m}{q_1(g-u)^m}+\frac{q_1(g-u)^m}{q_0(u)^m})
\end{align*}
\begin{align*}
\lefteqn{\sum_{u\in U\; \gamma\in \Gamma_*}N^g_{(u,\pi),\gamma}\frac{b(a+b)}{2}(\frac{q_0(u)^m}{q_2(\gamma)^m}+\frac{q_2(\gamma)^m}{q_0(u)^m})} \\
 &=\frac{b(a+b)}{2}\sum_{u\in U\; \gamma\in \Gamma_*}\frac{4}{|\Gamma|-|G|}\sum_{v\in U}\inpr{u}{v}\inpr{\gamma}{v}\inpr{-g}{v}
 (\frac{q_0(u)^m}{q_2(\gamma)^m}+\frac{q_2(\gamma)^m}{q_0(u)^m}) \\
 &=\frac{1}{|\Gamma|-|G|}\sum_{v\in U}(C(v,m)B(v,-m)+B(v,m)C(v,-m).
\end{align*}
\begin{align*}
\lefteqn{\sum_{h,h'\in G_*}N^g_{h,h'}a^2\frac{q_1(h)^m}{q_1(h')^m}} \\
 &=a^2\sum_{h,h'\in G_*}\frac{4}{|\Gamma|-|G|}\sum_{v\in U}\inpr{h+h'-g}{v})
 \frac{q_1(h)^m}{q_1(h')^m} \\
 &+a^2\sum_{h,h'\in G_*}(\delta_{h+h'-g,0}+\delta_{\theta(h)+h'-g,0}+\delta_{h+\theta(h')-g,0}+\delta_{h+h'-\theta(g),0})\frac{q_1(h)^m}{q_1(h')^m}\\
 &=\frac{1}{|\Gamma|-|G|}\sum_{v\in U}A(v,m)A(v,-m)\\
 &+\frac{a^2}{4}\sum_{h,h'\in G\setminus U}(\delta_{h+h'-g,0}+\delta_{\theta(h)+h'-g,0}+\delta_{h+\theta(h')-g,0}+\delta_{h+h'-\theta(g),0})\frac{q_1(h)^m}{q_1(h')^m}\\
 &=\frac{1}{|\Gamma|-|G|}\sum_{v\in U}A(v,m)A(v,-m)\\
 &+\frac{a^2}{4}\sum_{h,h'\in G}(\delta_{h+h'-g,0}+\delta_{\theta(h)+h'-g,0}+\delta_{h+\theta(h')-g,0}+\delta_{h+h'-\theta(g),0})\frac{q_1(h)^m}{q_1(h')^m}\\
 &-\frac{a^2}{4}\sum_{h\in U; h'\in G}(\delta_{h+h'-g,0}+\delta_{\theta(h)+h'-g,0}+\delta_{h+\theta(h')-g,0}+\delta_{h+h'-\theta(g),0})(\frac{q_1(h)^m}{q_1(h')^m}+\frac{q_1(h')^m}{q_1(h)^m})\\
 &+\frac{a^2}{4}\sum_{h,h'\in U}(\delta_{h+h'-g,0}+\delta_{\theta(h)+h'-g,0}+\delta_{h+\theta(h')-g,0}+\delta_{h+h'-\theta(g),0})\frac{q_1(h)^m}{q_1(h')^m}\\
 &=\frac{1}{|\Gamma|-|G|}\sum_{v\in U}A(v,m)A(v,-m)\\
 &+\frac{a^2}{4}\sum_{h,h'\in G\setminus U}(\delta_{h+h'-g,0}+\delta_{\theta(h)+h'-g,0}+\delta_{h+\theta(h')-g,0}+\delta_{h+h'-\theta(g),0})\frac{q_1(h)^m}{q_1(h')^m}\\
 &=\frac{1}{|\Gamma|-|G|}\sum_{v\in U}A(v,m)A(v,-m)
 +a^2\sum_{h\in G}\frac{q_1(h)^m}{q_1(g-h)^m}
 -a^2\sum_{h\in U}(\frac{q_1(h)^m}{q_1(g-h)^m}+\frac{q_1(g-h)^m}{q_1(h)^m})\\
 &=\frac{1}{|\Gamma|-|G|}\sum_{v\in U}A(v,m)A(v,-m)
 +\delta_{mg,0}q_1(g)^m-a^2\sum_{h\in U}(\frac{q_1(h)^m}{q_1(g-h)^m}+\frac{q_1(g-h)^m}{q_1(h)^m}),
\end{align*}
\begin{align*}
\lefteqn{\sum_{h\in G_*,\;\gamma\in \Gamma_*}N^g_{h,\gamma}ab(\frac{q_1(h)^m}{q_2(\gamma)^m}+\frac{q_2(\gamma)^m}{q_1(h)^m})} \\
 &=ab\sum_{h\in G_*,\;\gamma\in \Gamma_*}\frac{4}{|\Gamma|-|G|}\sum_{v\in U}\inpr{h-g}{v}\inpr{\gamma}{v}(\frac{q_1(h)^m}{q_2(\gamma)^m}+\frac{q_2(\gamma)^m}{q_1(h)^m}) \\
 &=\frac{1}{|\Gamma|-|G|}\sum_{v\in U}\inpr{-g}{v}(A(v,m)B(v,-m)+B(v,m)A(v,-m)), 
\end{align*}
\begin{align*}
\lefteqn{\sum_{\gamma,\xi\in \Gamma_*}N^g_{\gamma,\xi}b^2\frac{q_2(\gamma)^m}{q_2(\xi)^m}} \\
 &=b^2\sum_{\gamma,\xi\in \Gamma_*}\frac{4}{|\Gamma|-|G|}\sum_{v\in U}\inpr{-g}{v}\inpr{\gamma+\xi}{v}\frac{q_2(\gamma)^m}{q_2(\xi)^m} 
 =\frac{1}{|\Gamma|-|G|}\sum_{v\in U}\inpr{-g}{v}B(v,m)B(v,-m).
\end{align*}
Thus we get \begin{align*}
\lefteqn{\nu_m(g)} \\
 &=\frac{1}{|\Gamma|-|G|}\sum_{v\in U}\inpr{-g}{v}\cG(q_1,q_2,v,m)\cG(q_1,q_2,v,-m)+\delta_{mg,0}q_1(g)^m \\
 &=\frac{1}{|\Gamma|-|G|}\sum_{v\in U}\inpr{g}{v}|\cG(q_1,q_2,v,m)|^2+\delta_{mg,0}q_1(g)^m. 
\end{align*}

In a similar way we can compute $\nu_m(\gamma)$. 

\end{proof}

\section{Calculations for Section 4}

\begin{proof}[Proof of Lemma 4.2]
$$\sum_{x}S_{0,x}\overline{S_{0,x}}
=\sum_{x}S_{\pi,x}\overline{S_{\pi,x}} \\
=\frac{a^2+b^2}{2}+\frac{a^2}{2}+a^2|G_*|+\frac{b^2}{2} +b^2|\Gamma_*|=1. $$
$$\sum_{x}S_{0,x}\overline{S_{\pi,x}}
=\frac{a^2-b^2}{2}+\frac{a^2}{2}+a^2|G_*|
 -\frac{b^2}{2}-b^2|\Gamma_*|=0. $$
\begin{align*}
\lefteqn{\sum_{x}S_{0,x}\overline{S_{(g_0,\varepsilon),x}}=\sum_{x}S_{\pi,x}\overline{S_{(g_0,\varepsilon),x}}} \\
 &=\frac{a^2}{2}+\frac{a}{2}\sum_{\delta}(\frac{a}{2}+\frac{\varepsilon \delta s}{2\sqrt{2}})\inpr{-g_0}{l}
 +a^2\sum_{g\in G_*}\inpr{-g_0}{g}=\frac{\delta_{g_0,0}}{2}=0. 
\end{align*}
\begin{align*}
\lefteqn{\sum_{x}S_{0,x}\overline{S_{g,x}}=\sum_{x}S_{\pi,x}\overline{S_{g,x}}} \\
 &=a^2+a^2\inpr{-g}{g_0}+ a^2\sum_{h\in G_*}(\inpr{-g}{h}+\inpr{g}{h})
 =a^2\sum_{h\in G}\inpr{-g}{h}=\delta_{0,g}=0. \\
\end{align*}
\begin{align*}
\lefteqn{\sum_{x}S_{0,x}\overline{S_{(\gamma_0,\varepsilon),x}}=-\sum_{x}S_{\pi,x}\overline{S_{(\gamma_0,\varepsilon),x}}} \\
 &=-\frac{b^2}{2}
 -\frac{b}{2}\sum_{\delta}(\frac{b}{2}-\frac{\varepsilon\delta s}{2\sqrt{2}})\inpr{-\gamma_0}{\gamma_0}
 -b^2\sum_{\gamma\in \Gamma_*}\inpr{-\gamma_0}{\gamma} 
 =-\frac{b^2}{2}\sum_{\gamma\in \Gamma_*}\inpr{u-\gamma_0}{\gamma}=-\frac{\delta_{u,\gamma_0}}{2}=0. 
\end{align*}
\begin{align*}
\lefteqn{\sum_{x}S_{0,x}\overline{S_{\gamma,x}}=-\sum_{x}S_{\pi,x}\overline{S_{\gamma,x}}} \\
 &=-b^2-b^2\sum_{\delta}\inpr{-\gamma}{\gamma_0}
 -b^2\sum_{\xi\in \Gamma_*}(\inpr{-\gamma}{\xi}+\inpr{\gamma}{\xi})
 =-b^2\sum_{\xi\in \Gamma}\inpr{-\gamma}{\xi}=-\delta_{\gamma,0}=0. \\
\end{align*}
\begin{align*}
\lefteqn{\sum_{x}S_{(g_0,\varepsilon),x}\overline{S_{(g_0,\varepsilon'),x}}} \\
 &=\frac{a^2}{2}
 +\sum_{\delta}
 (\frac{a}{2}+\frac{\varepsilon \delta s}{2\sqrt{2}}) 
 (\frac{a}{2}+\frac{\varepsilon'\delta\overline{s}}{2\sqrt{2}})+a^2|G_*|+\frac{\varepsilon\varepsilon'}{4} 
 =\frac{1}{2}+\frac{\varepsilon\varepsilon'}{2}=\delta_{\varepsilon,\varepsilon}.
\end{align*}
$\sum_{x}S_{(\gamma_0,\varepsilon),x}\overline{S_{(\gamma_0,\varepsilon'),x}}=\delta_{\gamma_0,\gamma_0}\delta_{\varepsilon,\varepsilon'}$ 
can be shown in the same way. 
\begin{align*}
\lefteqn{\sum_{x}S_{(g_0,\varepsilon),x}\overline{S_{g,x}}} \\
 &=a^2+a\sum_{\delta}(\frac{a}{2}+\frac{\varepsilon\delta}{2\sqrt{2}})\inpr{g_0-g}{g_0}
 +a^2\sum_{h\in G_*}(\inpr{g_0-g}{h}+\inpr{g_0+g}{h})\\
 &=a^2\sum_{h\in G}\inpr{g_0-g}{h}=\delta_{g_0,g}=0. 
\end{align*}
$\sum_{x}S_{(\gamma_0,\varepsilon),x}\overline{S_{\gamma_0,x}}=0$ can be shown in the same way. 
\begin{align*}
\lefteqn{\sum_{x} S_{(g_0,\varepsilon),x} \overline{S_{(\gamma_0,\epsilon'),x}}   } \\
 &=\sum_{\delta}\inpr{g_0}{g_0}(\frac{a}{2}+\frac{\varepsilon\delta s}{2\sqrt{2}})\frac{\varepsilon'\delta \overline{s}}{4}
 -\sum_{\delta}\frac{\varepsilon \delta s}{4}
 \inpr{-\gamma_0}{\gamma_0}(\frac{b}{2}-\frac{\varepsilon'\delta\overline{s}}{2\sqrt{2}}) \\
 &=\frac{\varepsilon\varepsilon'}{4\sqrt{2}}(\inpr{g_0}{g_0}+\inpr{\gamma_0}{\gamma_0})=0. 
\end{align*}
$\sum_{x} S_{(g_0,\varepsilon),x} \overline{S_{\gamma,x}}=0$, 
$\sum_{x}S_{g,x}\overline{S_{(\gamma_0,\varepsilon),x}}=0$, and 
$\sum_{x}S_{g,x}\overline{S_{\gamma,x}}=0$ are obvious.   
\begin{align*}
\lefteqn{\sum_{x}S_{g,x}\overline{S_{g',x}}} \\
 &=2a^2+2a^2\inpr{g-g'}{g_0}
 +a^2\sum_{h\in G_*}(\inpr{g}{h}+\inpr{-g}{h})(\inpr{g'}{h}+\inpr{-g'}{h}) \\
 &=a^2(\sum_{h\in G}\inpr{g+g'}{h}+\sum_{h\in G}\inpr{g-g'}{h})\\
 &=\delta_{g+g',0}+\delta_{g,g'}=\delta_{g,g'}. 
\end{align*}
$\sum_{x}S_{\gamma,x}\overline{S_{\gamma',x}}=\delta_{\gamma,\gamma'}$ can be shown in the same way. 
Thus $S$ is unitary. 

It is obvious that $T$ and $C$ are unitary.  
Since $\overline{S_{x,y}}=S_{\overline{x},y}=S_{x,\overline{y}}$ and $S$ is symmetric, we have $S^2=C$. 
Since $T_{x,x}=T_{\overline{x},\overline{x}}$, we have $CT=T C$. 

The only remaining condition is $(ST)^3=cC$, and it suffices to verify 
$$\sum_{x}S_{j,x}S_{j',x}T_{x,x}=cS_{j,j'}\overline{T_{j,j}T_{j',j'}}$$
for all $j,j'\in J$. 
\begin{align*}
\lefteqn{\sum_{x}S_{0,x}S_{0,x}T_{x,x}=\sum_{x}S_{\pi,x}S_{\pi,x}T_{x,x}} \\
&=\frac{a^2+b^2}{2}+\frac{a^2}{2}q_1(g_0)
 +a^2\sum_{g\in G_*}q_1(g)
 +\frac{b^2}{2}q_2(\gamma_0)+b^2\sum_{\gamma\in \Gamma_*}q_2(\gamma) \\
 &=\frac{a}{2}\cG(q_1)+\frac{b}{2}\cG(q_2)
 =cS_{0,0}\overline{T_{0,0}T_{0,0}}
 =cS_{\pi,\pi}\overline{T_{\pi,\pi}T_{\pi,\pi}}.
\end{align*}
\begin{align*}
\lefteqn{\sum_{x}S_{0,x}S_{\pi,x}T_{x,x}} \\
 &=\frac{a^2-b^2}{2}+\frac{a^2}{2}q_1(g_0)
 +a^2\sum_{g\in G_*}q_1(g)
 -\frac{b^2}{2}q_2(\gamma_0)-b^2\sum_{\gamma\in \Gamma_*}q_2(\gamma) \\
 &=(\frac{a}{2}\cG(q_1)-\frac{b}{2}\cG(q_2))
 =cS_{0,\pi}\overline{T_{0,0}T_{\pi,\pi}}.
\end{align*}
\begin{align*}
\lefteqn{\sum_{x}S_{0,x}S_{(g_0,\varepsilon),x}T_{x,x}=\sum_{x}S_{\pi,x}S_{(g_0,\varepsilon),x}T_{x,x}} \\
 &=\frac{a^2}{2}+\frac{a^2}{2} \inpr{g_0}{g_0} q_1(g_0)
 +a^2\sum_{g\in G_*}\inpr{g_0}{g}q_1(g) 
 =\frac{a^2}{2}\sum_{g\in G}\inpr{g_0}{g}q_1(g)
 =\frac{a^2}{2}\sum_{g\in G}q_1(g_0+g)\overline{q_1(g_0)} \\
 &=\frac{a}{2}\cG(q_1)\overline{q_1(g_0)}
 =cS_{0,(g_0,\varepsilon)}\overline{T_{0,0}T_{(g_0,\varepsilon),(g_0,\varepsilon)}}
 =cS_{\pi,(g_0,\varepsilon)}\overline{T_{\pi,\pi}T_{(g_0,\varepsilon),(g_0,\varepsilon)}}.
\end{align*}
In a similar way, we can show
\begin{align*}
\lefteqn{\sum_{x}S_{0,x}S_{(\gamma_0,\varepsilon),x}T_{x,x}=-\sum_{x}S_{\pi,x}S_{(\gamma_0,\varepsilon),x}T_{x,x}} \\
 &=cS_{0,(\gamma_0,\varepsilon)}\overline{T_{0,0}T_{(\gamma_0,\varepsilon),(\gamma_0,\varepsilon)}}
 =-cS_{\pi,(\gamma_0,\varepsilon)}\overline{T_{\pi,\pi}T_{(\gamma_0,\varepsilon),(\gamma_0,\varepsilon)}}.
\end{align*}
\begin{align*}
\lefteqn{\sum_{x}S_{0,x}S_{g,x}T_{x,x}=S_{\pi,x}S_{g,x}T_{x,x}} \\
 &=a^2+a^2\inpr{g}{g_0}q_1(g_0)
 +a^2\sum_{h\in G_*}(\inpr{g}{h}+\inpr{-g}{h})q_1(h)
 =a^2\sum_{h\in G}q_1(g+h)\overline{q_1(g)}\\
 &=a\cG(q_1)\overline{q_1(g)}
 =cS_{0,g}\overline{T_{0,0}T_{g,g}}
 =cS_{\pi,g}\overline{T_{\pi,\pi}T_{g,g}}.
\end{align*}
In a similar way, we can show 
$$\sum_{x}S_{0,x}S_{\gamma,x}T_{x,x}=-\sum_{x}S_{\pi,x}S_{\gamma,x}T_{x,x}
=cS_{0,\gamma}\overline{T_{0,0}T_{\gamma,\gamma}}=-cS_{\pi,\gamma}\overline{T_{\pi,\pi}T_{\gamma,\gamma}}.$$
\begin{align*}
\lefteqn{\sum_{x}S_{(g_0,\varepsilon),x}S_{(g_0,\varepsilon'),x}T_{x,x}} \\
 &=\frac{a^2}{2}+\sum_{\delta}(\frac{a}{2}+\frac{\varepsilon\delta s}{2\sqrt{2}}) 
 (\frac{a}{2}+\frac{\varepsilon'\delta s}{2\sqrt{2}})q_1(g_0)+a^2\sum_{g\in G_*}q_1(g)
 +\frac{\varepsilon\varepsilon's^2}{4}q_2(\gamma_0)\\
 &=\frac{a^2}{2}\sum_{g\in G}q_1(g)
 +\frac{\varepsilon\varepsilon' s^2}{4}(q_1(g_0)+q_2(\gamma_0))
 =c\frac{a}{2}+\frac{\varepsilon\varepsilon's^2}{4}\sqrt{2}\overline{s}c\\
 &=cS_{(g_0,\varepsilon),(g_0,\varepsilon')}\overline{T_{(g_0,\varepsilon),(g_0,\varepsilon)}T_{(g_0,\varepsilon'),(g_0,\varepsilon')}}. 
\end{align*}
\begin{align*}
\lefteqn{\sum_{x}S_{(g_0,\varepsilon),x}S_{g,x}T_{x,x}} \\
 &=a^2+2a^2\inpr{g_0+g}{g_0}q_1(g_0)
 +a^2\sum_{h\in G_*}(\inpr{g_0+g}{h}+\inpr{g_0-g}{h})q_1(h)\\
 &=a^2\sum_{h\in G}\inpr{g_0+g}{h}q_1(h)=a\cG(q_1)\overline{q_1}(g_0+g)=ca\overline{\inpr{g_0}{g}q_1(g_0)q_1(g)}\\
 &=cS_{(g_0,\varepsilon),g}\overline{T_{(g_0,\varepsilon),(g_0,\varepsilon)}T_{g,g}}.
\end{align*}
In a similar way, we can show
$\sum_{x}S_{(\gamma_0,\varepsilon),x}S_{\gamma,x}T_{x,x}
=cS_{(\gamma_0,\varepsilon),\gamma}\overline{T_{(\gamma_0,\varepsilon),(\gamma_0,\varepsilon)}T_{\gamma,\gamma}}$.
\begin{align*}
\lefteqn{\sum_{x}S_{(g_0,\varepsilon),x}S_{(\gamma_0,\varepsilon'),x}T_{x,x}} \\
 &=\sum_{\delta}(\frac{a}{2}+\frac{\varepsilon\delta s}{2\sqrt{2}})\inpr{g_0}{g_0}
 \frac{\varepsilon'\delta s}{2\sqrt{2}}q_1(g_0)
 -\sum_{\delta}\frac{\varepsilon\delta s}{2\sqrt{2}}
 (\frac{b}{2}-\frac{\varepsilon'\delta s}{2\sqrt{2}})\inpr{\gamma_0}{\gamma_0}q_2(\gamma_0)\\
 &=\frac{\varepsilon\varepsilon's^2}{4}(\inpr{g_0}{g_0}q_1(g_0)
 +\inpr{\gamma_0}{\gamma_0}q_2(\gamma_0))
 =\frac{\varepsilon\varepsilon's^2}{2\sqrt{2}}\frac{\overline{q_1(g_0)}+\overline{q_2(\gamma_0)}}{\sqrt{2}}\\
 &=\frac{\varepsilon\varepsilon's^2}{2\sqrt{2}}\frac{q_1(g_0)+q_2(\gamma_0)}{\sqrt{2}}\overline{q_1(g_0)q_2(\gamma_0)}
 =cS_{(g_0,\varepsilon),(\gamma_0,\varepsilon')}\overline{T_{(g_0,\varepsilon),(g_0,\varepsilon)}T_{(\gamma_0,\varepsilon'),(\gamma_0,\varepsilon')}}.
\end{align*}
It is easy to show 
$$\sum_{x}S_{(g_0,\varepsilon),x}S_{\gamma,x}T_{x,x}=0
=cS_{(g_0,\varepsilon),\gamma}\overline{T_{(g_0,\varepsilon),(g_0,\varepsilon)}T_{\gamma,\gamma}},$$
$$\sum_{x}S_{g,x}S_{(\gamma_0,\varepsilon),x}T_{x,x}=0
=cS_{(g,(\gamma_0,\varepsilon)}\overline{T_{g,g}T_{\gamma_0,\varepsilon),(\gamma_0,\varepsilon)}}.$$
$$\sum_{x}S_{g,x}S_{\gamma,x}T_{x,x}=0
=cS_{g,\gamma}\overline{T_{g,g}T_{\gamma,\gamma}}.$$
\begin{align*}
\lefteqn{\sum_{x}S_{g,x}S_{g',x}T_{x,x}} \\
 &=2a^2+2a^2\inpr{g+g'}{g_0}_1q_1(g_0)
 +a^2\sum_{h\in G_*}(\inpr{g}{h}+\inpr{-g}{h})(\inpr{g'}{h}+\inpr{-g'}{h})q_1(h)\\
 &=a^2\sum_{h\in G}(q_1(g+g'+h)\overline{q_1(g+g')}+q_1(g-g'+h)\overline{q_1(g-g')})\\
 &=ca\overline{(\inpr{g}{g'}q_1(g)q_1(g')+\inpr{g}{-g'}q_1(g)q_1(g')}\\
 &=cS_{g,g'}\overline{T_{g,g}T_{g',g'}}.
\end{align*}
In a similar way, we can show 
$\sum_{x}S_{\gamma,x}S_{\gamma',x}T_{x,x}
=cS_{\gamma,\gamma'}\overline{T_{\gamma,\gamma}T_{\gamma',\gamma'}}.$
\begin{align*}
\lefteqn{\sum_{x}S_{(\gamma_0,\varepsilon),x}S_{(\gamma_0,\varepsilon'),x}T_{x,x}} \\
 &=\frac{b^2}{2}+\frac{\varepsilon\varepsilon's^2}{4}q_1(g_0)
 +\sum_{\delta}(\frac{b}{2}-\frac{\varepsilon\delta s}{2\sqrt{2}})(\frac{b}{2}-\frac{\varepsilon'\delta s}{2\sqrt{2}})q_2(\gamma_0)
 +b^2\sum_{\gamma\in \Gamma_*}q_2(\gamma)\\
 &=\frac{b^2}{2}\sum_{\gamma\in \Gamma}q_2(\gamma)
 +\frac{\varepsilon\varepsilon's^2}{4}(q_1(g_0)+q_2(\gamma_0))\\
 &=-c\frac{b}{2}+c\frac{\varepsilon\varepsilon's}{2\sqrt{2}}
 =cS_{(\gamma_0,\varepsilon),(\gamma_0,\varepsilon')}
 \overline{T_{(\gamma_0,\varepsilon),(\gamma_0,\varepsilon)}T_{(\gamma_0,\varepsilon'),(\gamma_0,\varepsilon')}}. 
\end{align*}
\end{proof}

\begin{proof}[Proof of Lemma 4.3]
\begin{align*}
\lefteqn{N_{\pi,\pi,\pi}} \\
 &=(\frac{(a+b)^3}{4(a-b)}+\frac{(a-b)^3}{4(a+b)})+\frac{a^2}{2} 
 +a^2|G_*|-\frac{b^2}{2}-b^2|\Gamma_*|\\
 &=(\frac{(a+b)^4+(a-b)^4}{4(a^2-b^2)}-\frac{a^2-b^2}{2})+\frac{a^2}{2}|G|-\frac{b^2}{2}|\Gamma|
 =\frac{4a^2b^2}{a^2-b^2}=\frac{4}{|\Gamma|-|G|}. 
\end{align*}
\begin{align*}
\lefteqn{N_{\pi,\pi,(g_0,\varepsilon)}} \\
 &=(\frac{a(a+b)^2}{4(a-b)}+\frac{a(a-b)^2}{4(a+b)}) +\frac{a^2}{2}\inpr{g_0}{g_0}+a^2\sum_{h\in G_*}\inpr{g_0}{h}\\
 &=(a\frac{(a+b)^3+(a-b)^3}{4(a^2-b^2)}-\frac{a^2}{2})
 +\frac{a^2}{2}\sum_{h\in G}\inpr{g_0}{h} 
 =\frac{2a^2b^2}{a^2-b^2}+\frac{1}{2}\delta_{g_0,0} 
 =\frac{2}{|\Gamma|-|G|}.
\end{align*}
\begin{align*}
\lefteqn{N_{\pi,\pi,g}} \\
 &=(\frac{a(a+b)^2}{2(a-b)}+\frac{a(a-b)^2}{2(a+b)})+a^2\inpr{g}{g_0}+a^2\sum_{h\in G_*}(\inpr{g}{h}+\inpr{g}{-h})\\
 &=(a\frac{(a+b)^3+(a-b)^3}{2(a^2-b^2)}-a^2)+a^2\sum_{h\in G}\inpr{g}{h}
 =\frac{4a^2b^2}{a^2-b^2}+\delta_{g,0}=\frac{4}{|\Gamma|-|G|}.
\end{align*}
\begin{align*}
\lefteqn{N_{\pi,\pi,(\gamma_0,\varepsilon)}} \\
 &=(\frac{b(a+b)^2}{4(a-b)}-\frac{b(a-b)^2}{4(a+b)})
 -\frac{b^2}{2}\inpr{\gamma_0}{\gamma_0}-b^2\sum_{\gamma\in \Gamma_*}\inpr{\gamma_0}{\gamma}\\
 &=(\frac{b^2(3a^2+b^2)}{2(a^2-b^2)}+\frac{b^2}{2})-\frac{b^2}{2}\sum_{\gamma\in \Gamma}\inpr{\gamma_0}{\gamma}
 =\frac{2a^2b^2}{a^2-b^2}-\frac{\delta_{\gamma,0}}{2}=\frac{2}{|\Gamma|-|G|}.
\end{align*}
\begin{align*}
\lefteqn{N_{\pi,\pi,\gamma}} \\
 &=(\frac{b(a+b)^2}{2(a-b)}-\frac{b(a-b)^2}{2(a+b)}) -b^2\sum_{\gamma_0\in \Sigma_*}\inpr{\gamma}{\gamma_0} 
 -b^2\sum_{\xi\in \Gamma_*}(\inpr{\gamma}{\xi}+\inpr{\gamma}{-\xi})\\
 &=(b\frac{(a+b)^3-(a-b)^3}{2(a^2-b^2)}+b^2)-b^2\sum_{\xi\in \Gamma}\inpr{\gamma}{\xi} 
 =\frac{4a^2b^2}{a^2-b^2}-\delta_{\gamma,0} 
 =\frac{4}{|\Gamma|-|G|}.\\
\end{align*}
\begin{align*}
\lefteqn{N_{\pi,(g_0,\varepsilon),(g_0,\varepsilon')}} \\
 &=(\frac{a^2(a+b)}{4(a-b)}+\frac{a^2(a-b)}{4(a+b)})
 +\sum_{\delta}(\frac{a}{2}+\frac{\varepsilon\delta s}{2\sqrt{2}})
 (\frac{a}{2}+\frac{\varepsilon'\delta s}{2\sqrt{2}})
 +a^2|G_*|-\sum_{\delta}\frac{\varepsilon\delta s}{2\sqrt{2}}\frac{\varepsilon'\delta s}{2\sqrt{2}}\\
 &=(a^2\frac{(a+b)^2+(a-b)^2}{4(a^2-b^2)}-\frac{a^2}{2})+\frac{1}{2}
 =\frac{1}{|\Gamma|-|G|}+\frac{1}{2}\\
\end{align*}
\begin{align*}
\lefteqn{N_{\pi,(g_0,\varepsilon),g}} \\
 &=(\frac{a^2(a+b)}{2(a-b)}+\frac{a^2(a-b)}{2(a+b)})
 +a\sum_{\delta}\inpr{g_0+g}{g_0}(\frac{a}{2}+\frac{\varepsilon \delta s}{2\sqrt{2}})
 +a^2\sum_{h\in G_*}\inpr{g_0}{h}(\inpr{g}{h}+\inpr{-g}{h})\\
 &=a^2(\frac{(a+b)^2+(a-b)^2}{2(a^2-b^2)}-1)+a^2\sum_{h\in G}\inpr{g_0+g}{h}=\frac{2}{|\Gamma|-|G|}.
\end{align*}
\begin{align*}
\lefteqn{N_{\pi,(g_0,\varepsilon),(\gamma_0,\varepsilon')}} \\
 &=(\frac{ab(a+b)}{4(a-b)}-\frac{ab(a-b)}{4(a+b)})
 +\sum_{\delta}(\frac{a}{2}+\frac{\varepsilon\delta s}{2\sqrt{2}})\inpr{g_0}{g_0}\frac{\varepsilon'\delta s}{2\sqrt{2}} 
 +\sum_{\delta}\frac{\varepsilon\delta s}{2\sqrt{2}}(\frac{b}{2}-\frac{\varepsilon'\delta s}{2\sqrt{2}})\inpr{\gamma_0}{\gamma_0} \\
 &=\frac{a^2b^2}{a^2-b^2}
 +\frac{\varepsilon\varepsilon's^2}{4}(\inpr{g_0}{g_0}-\inpr{\gamma_0}{\gamma_0})
 =\frac{1}{|\Gamma|-|G|}+\frac{\varepsilon\varepsilon's^2}{4} (\inpr{g_0}{g_0}-\inpr{\gamma_0}{\gamma_0})\\
\end{align*}
$$N_{\pi,(g_0,\varepsilon),\gamma}=\frac{ab}{2}(\frac{a+b}{a-b}-\frac{a-b}{a+b})
 +b\sum_{\delta}\frac{\varepsilon \delta s}{4}\inpr{\gamma}{\gamma_0} 
 =\frac{2}{|\Gamma|-|G|}.$$
\begin{align*}
\lefteqn{N_{\pi,g,g'}} \\
 &=(a^2\frac{a+b}{a-b}+a^2\frac{a-b}{a+b})+2a^2
 +a^2\sum_{h\in G_*}(\inpr{g}{h}+\inpr{g}{-h})(\inpr{g'}{h}+\inpr{g'}{-h}) \\
 &=a^2(\frac{(a+b)^2+(a-b)^2}{a^2-b^2}-2)+a^2\sum_{h\in G}(\inpr{g+g'}{h}+\inpr{g-g'}{h})\\
 &=\frac{4a^2b^2}{a^2-b^2} +\delta_{g+g',0}+\delta_{g-g',0}
 =\frac{4}{|\Gamma|-|G|} +\delta_{g,g'}.
\end{align*}
$$N_{\pi,g,(\gamma_0,\varepsilon)}=\frac{ab}{2}(\frac{a+b}{a-b}-\frac{a-b}{a+b})
 +a\sum_{\delta}\inpr{g}{g_0}\frac{\varepsilon \delta s}{2\sqrt{2}} 
 =\frac{2}{|\Gamma|-|G|}. $$
$$N_{\pi,g,\gamma}=(\frac{ab(a+b)}{a-b}-\frac{ab(a-b)}{a+b})=ab\frac{(a+b)^2-(a-b)^2}{a^2-b^2}
 =\frac{4a^2b^2}{a^2-b^2}=\frac{4}{|\Gamma|-|G|}.$$
\begin{align*}
\lefteqn{N_{\pi,(\gamma_0,\varepsilon),(\gamma_0,\varepsilon')}} \\
 &=\frac{b^2}{4}(\frac{a+b}{a-b}+\frac{a-b}{a+b})
 +\sum_{\delta}\frac{\varepsilon \delta s}{2\sqrt{2}}\frac{\varepsilon'\delta s}{2\sqrt{2}} 
 -\sum_{\delta}(\frac{b}{2}-\frac{\varepsilon\delta s}{2\sqrt{2}}) (\frac{b}{2}-\frac{\varepsilon\delta s}{2\sqrt{2}}) 
 -b^2|\Gamma_*|\\
 &=\frac{b^2}{2}(\frac{a^2+b^2}{a^2-b^2}+1)
 -\frac{1}{2}
 =\frac{1}{|\Gamma|-|G|}-\frac{1}{2}.
\end{align*}
\begin{align*}
\lefteqn{N_{\pi,(\gamma_0,\varepsilon),\gamma}} \\
 &=\frac{b^2}{2}(\frac{a+b}{a-b}+\frac{a-b}{a+b})
 -b\sum_{\delta}(\frac{b}{2}-\frac{\varepsilon \delta s}{4})\inpr{\gamma_0+\gamma}{\gamma_0} 
 -b^2\sum_{\xi\in \Gamma_*}\inpr{\gamma_0}{\xi}(\inpr{\gamma}{\xi}+\inpr{\gamma}{-\xi}) \\
 &=(\frac{b^2(a^2+b^2)}{a^2-b^2}+b^2)-b^2\sum_{\xi\in \Gamma}\inpr{\gamma_0+\gamma}{\xi}=\frac{2}{|\Gamma|-|G|}. 
\end{align*}
\begin{align*}
\lefteqn{N_{\pi,\gamma,\gamma'}} \\
 &=(b^2\frac{a+b}{a-b}+b^2\frac{a-b}{a+b})-2b^2\inpr{\gamma+\gamma'}{\gamma_0}
 -b^2\sum_{\xi\in \Gamma_*}(\inpr{\gamma}{\xi}+\inpr{\gamma}{-\xi})(\inpr{\gamma'}{\xi}+\inpr{\gamma'}{-\xi}) \\
 &=b^2(\frac{(a+b)^2+(a-b)^2}{a^2-b^2}+2)
 -b^2\sum_{\xi\in \Gamma}(\inpr{\gamma+\gamma'}{\xi}+\inpr{\gamma-\gamma'}{\xi})\\
 &=\frac{4a^2b^2}{a^2-b^2}-\delta_{\gamma+\gamma',0}-\delta_{\gamma-\gamma',0}=\frac{4}{|\Gamma|-|G|}-\delta_{\gamma,\gamma'}.
\end{align*}
\begin{align*}
\lefteqn{N_{(g_0,\varepsilon),(g_0,\varepsilon'),(g_0,\varepsilon'')}} \\
 &=(\frac{a^3}{4(a-b)}+\frac{a^3}{4(a+b)})
 +\frac{2}{a}\sum_{\delta}(\frac{a}{2}+\frac{\varepsilon \delta s}{2\sqrt{2}})(\frac{a}{2}+\frac{\varepsilon' \delta s}{2\sqrt{2}}) 
 (\frac{a}{2}+\frac{\varepsilon'' \delta s}{2\sqrt{2}})\inpr{g_0}{g_0}\\
 &+a^2\sum_{g\in G_*}\inpr{g_0}{g} 
 +\frac{2}{b}\sum_{\delta}\frac{\varepsilon\delta s}{2\sqrt{2}}\frac{\varepsilon'\delta s}{2\sqrt{2}}\frac{\varepsilon''\delta s}{2\sqrt{2}}\\
 &=(\frac{a^4}{2(a^2-b^2)}-\frac{a^2}{2})+\frac{a^2}{2}\sum_{g\in G}\inpr{g_0}{g} 
 +\frac{(\varepsilon\varepsilon'+\varepsilon'\varepsilon''+\varepsilon''\varepsilon)s^2}{4}\inpr{g_0}{g_0} \\
 &=\frac{a^2b^2}{2(a^2-b^2)}
 +\frac{\varepsilon\varepsilon'+\varepsilon'\varepsilon''+\varepsilon''\varepsilon}{4}
 =\frac{1}{2(|\Gamma|-|G|)}
 +\frac{\varepsilon\varepsilon'+\varepsilon'\varepsilon''+\varepsilon''\varepsilon}{4}.
\end{align*}
\begin{align*}
\lefteqn{N_{(g_0,\varepsilon),(g_0,\varepsilon'),g}} \\
 &=(\frac{a^3}{2(a-b)}+\frac{a^3}{2(a+b)})
 +2\sum_{\delta}(\frac{a}{2}+\frac{\varepsilon\delta s}{2\sqrt{2}})(\frac{a}{2}+\frac{\varepsilon'\delta s}{2\sqrt{2}})\inpr{g}{g_0}
 +a^2\sum_{h\in G_*}(\inpr{g}{h}+\inpr{g}{-h}) \\
 &=(\frac{a^4}{a^2-b^2}-a^2)+a^2\sum_{h\in G}\inpr{g}{h}
 +\frac{\varepsilon\varepsilon's^2}{2}\inpr{g}{g_0} 
 =\frac{1}{|\Gamma|-|G|}+\frac{\varepsilon\varepsilon'\inpr{g_0+g}{g_0}}{2}.
\end{align*}
\begin{align*}
\lefteqn{N_{(g_0,\varepsilon),(g_0,\varepsilon'),(\gamma_0,\varepsilon'')}} \\
 &=(\frac{a^2b}{4(a-b)}-\frac{a^2b}{4(a+b)})
 +\frac{2}{a}\sum_{\delta}(\frac{a}{2}+\frac{\varepsilon\delta s}{2\sqrt{2}})(\frac{a}{2}+\frac{\varepsilon'\delta s}{2\sqrt{2}})
 \frac{\varepsilon''\delta s}{2\sqrt{2}} \\
 &-\frac{2}{b}\sum_{\delta}\frac{\varepsilon\delta s}{2\sqrt{2}}\frac{\varepsilon'\delta s}{2\sqrt{2}}
 (\frac{b}{2}-\frac{\varepsilon''\delta s}{2\sqrt{2}})\inpr{\gamma_0}{\gamma_0} \\
 &=\frac{a^2b^2}{2(a^2-b^2)} +\frac{\varepsilon''(\varepsilon+\varepsilon')s^2}{4}
 -\frac{\varepsilon \varepsilon's^2}{4}\inpr{\gamma_0}{\gamma_0}
 =\frac{1}{2(|\Gamma|-|G|)}
 +\frac{\varepsilon''(\varepsilon+\varepsilon')\inpr{g_0}{g_0}+\varepsilon \varepsilon'}{4}.
\end{align*}
\begin{align*}
\lefteqn{N_{(g_0,\varepsilon),(g_0,\varepsilon),\gamma}} \\
 &=\frac{a^2b}{2(a-b)}-\frac{a^2b}{2(a+b)}
 -2\sum_{\delta}\frac{\varepsilon\delta s}{2\sqrt{2}}\frac{\varepsilon'\delta s}{2\sqrt{2}}\inpr{\gamma}{\gamma_0}\\
 &=\frac{a^2b^2}{a^2-b^2}-\frac{\varepsilon\varepsilon' s^2}{2}\inpr{\gamma_0}{\gamma} 
 =\frac{1}{|\Gamma|-|G|}-\frac{\varepsilon\varepsilon' \inpr{g_0}{g_0}\inpr{\gamma_0}{\gamma}}{2}.
\end{align*}
\begin{align*}
\lefteqn{N_{(g_0,\varepsilon),g,g'}} \\
 &=(\frac{a^3}{a-b}+\frac{a^3}{a+b})+2a\sum_{\delta}(\frac{a}{2}+\frac{\varepsilon \delta s}{2\sqrt{2}})\inpr{g_0+g+g'}{g_0} \\
 &+a^2\sum_{h\in G_*}\inpr{g_0}{h}(\inpr{g}{h}+\inpr{g}{-h})(\inpr{g'}{h}+\inpr{g'}{-h}) \\
 &=a^2(\frac{2a^2}{a^2-b^2}-2)+a^2\sum_{h\in G}(\inpr{g_0+g+g'}{h}+\inpr{g_0+g-g'}{h}) \\
 &=\frac{2}{|\Gamma|-|G|}+\delta_{g_0+g+g',0}+\delta_{g_0+g-g',0}. 
\end{align*}
\begin{align*}
\lefteqn{N_{(g_0,\varepsilon),g,(\gamma_0,\varepsilon')}} \\
 &=(\frac{a^2b}{2(a-b)}-\frac{a^2b}{2(a+b)})
 +2\sum_{\delta}(\frac{a}{2}+\frac{\varepsilon \delta s}{2\sqrt{2}})\frac{\varepsilon'\delta s}{2\sqrt{2}}\inpr{g_0+g}{g_0} \\
 &=\frac{1}{|\Gamma|-|G|}+\frac{\varepsilon \varepsilon' s^2}{2}\inpr{g_0+g}{g_0} 
 =\frac{1}{|\Gamma|-|G|}+\frac{\varepsilon\varepsilon'\inpr{g}{g_0}}{2}.
\end{align*}
$$N_{(g_0,\varepsilon),g,\gamma}=(\frac{a^2b}{a-b}-\frac{a^2b}{a+b})=\frac{2}{|\Gamma|-|G|}.$$
\begin{align*}
\lefteqn{N_{(g_0,\varepsilon''),(\gamma_0,\varepsilon),(\gamma_0,\varepsilon')}} \\
 &=(\frac{ab^2}{4(a-b)}+\frac{ab^2}{4(a+b)})
 +\frac{2}{a}\sum_{\delta}(\frac{a}{2}+\frac{\varepsilon''\delta s}{2\sqrt{2}})\inpr{g_0}{g_0}
 \frac{\varepsilon \delta s}{2\sqrt{2}}\frac{\varepsilon' \delta s}{2\sqrt{2}}\\
 &+\frac{2}{b}\sum_{\delta}\frac{\varepsilon''\delta s}{2\sqrt{2}}(\frac{b}{2}-\frac{\varepsilon\delta s}{2\sqrt{2}}) (\frac{b}{2}-\frac{\varepsilon'\delta s}{2\sqrt{2}})\\
 &=\frac{a^2b^2}{2(a^2-b^2)}
 +\frac{\varepsilon\varepsilon's^2}{4}\inpr{g_0}{g_0} -\frac{\varepsilon''(\varepsilon+\varepsilon')s^2}{4} 
 =\frac{1}{2(|\Gamma|-|G|)}+\frac{\varepsilon\varepsilon'-\varepsilon''(\varepsilon+\varepsilon')\inpr{g_0}{g_0}}{4}.
\end{align*}
\begin{align*}
\lefteqn{N_{(g_0,\varepsilon),(\gamma_0,\varepsilon'),\gamma}} \\
 &=(\frac{ab^2}{2(a-b)}+\frac{ab^2}{2(a+b)})
 +2\sum_{\delta}\frac{\varepsilon\delta s}{2\sqrt{2}}(\frac{b}{2}-\frac{\varepsilon'\delta s}{2\sqrt{2}})
 \inpr{\gamma_0+\gamma}{\gamma_0} \\
 &=\frac{1}{|\Gamma|-|G|}-\frac{\varepsilon\varepsilon's^2}{2}\inpr{\gamma_0+\gamma}{\gamma_0} 
  =\frac{1}{|\Gamma|-|G|}+\frac{\varepsilon\varepsilon'\inpr{\gamma}{\gamma_0}}{2}. 
\end{align*}
\begin{align*}
\lefteqn{N_{(g_0,\varepsilon),\gamma,\gamma'}} \\
 &=\frac{ab^2}{a-b}+\frac{ab^2}{a+b}+2b\inpr{\gamma_0}{\gamma+\gamma'}\sum_{\delta}\frac{s\varepsilon\delta}{2\sqrt{2}} \\
 &=\frac{2}{|\Gamma|-|G|}. 
\end{align*}
\begin{align*}
\lefteqn{N_{g,g',g''}} \\
 &=(\frac{2a^3}{a-b}+\frac{2a^3}{a+b})+4a^2\inpr{g+g'+g''}{g_0}\\
 &+a^2\sum_{h\in G_*}(\inpr{g}{h}+\inpr{-g}{h})(\inpr{g'}{h}+\inpr{-g'}{h})(\inpr{g''}{h}+\inpr{-g''}{h}) \\
 &=(\frac{4a^4}{a^2-b^2}-4a^2)\\
 &+a^2\sum_{h\in G}(\inpr{g+g'+g''}{h}+\inpr{-g+g'+g''}{h}+\inpr{g-g'+g''}{h}+\inpr{g+g'-g''}{h})\\
 &=\frac{4}{|\Gamma|-|G|}
 +\delta_{g+g'+g'',0}+\delta_{-g+g'+g'',0}+\delta_{g-g'+g'',0}+\delta_{g+g'-g'',0}.
\end{align*}
$$N_{g,g',(\gamma_0,\varepsilon)}
=(\frac{a^2b}{a-b}-\frac{a^2b}{a+b})+2a\sum_{\delta}\inpr{g+g'}{g_0}\frac{\delta\varepsilon s}{2\sqrt{2}} 
 =\frac{2}{|\Gamma|-|G|}.$$
$$N_{g,g',\gamma}=(\frac{2a^2b}{a-b}-\frac{2a^2b}{a+b})=\frac{4a^2b^2}{a^2-b^2}=\frac{4}{|\Gamma|-|G|}.$$
\begin{align*}
\lefteqn{N_{g,(\gamma_0,\varepsilon),(\gamma_0,\varepsilon')}} \\
 &=(\frac{ab^2}{2(a-b)}+\frac{ab^2}{2(a+b)})
 +2\sum_{\delta}\inpr{g}{g_0}\frac{\varepsilon\delta s}{2\sqrt{2}}\frac{\varepsilon' \delta s}{2\sqrt{2}} \\
 &=\frac{1}{|\Gamma|-|G|}+\frac{\varepsilon\varepsilon's^2}{2}\inpr{g}{g_0} 
 =\frac{1}{|\Gamma|-|G|}+\frac{\varepsilon\varepsilon'\inpr{g_0+g}{g_0} }{2}. 
\end{align*}
$$N_{g,(\gamma_0,\varepsilon),\gamma}=(\frac{ab^2}{a-b}+\frac{ab^2}{a+b})=\frac{2}{|\Gamma|-|G|}.$$
$$N_{g,\gamma,\gamma'}=(\frac{2ab^2}{a-b}+\frac{2ab^2}{a+b})=\frac{4a^2b^2}{a^2-b^2}=\frac{4}{|\Gamma|-|G|}.$$
\begin{align*}
\lefteqn{N_{(\gamma_0,\varepsilon),(\gamma_0,\varepsilon'),(\gamma_0,\varepsilon'')}} \\
 &=(\frac{b^3}{4(a-b)}-\frac{b^3}{4(a+b)})
 +\frac{2}{a}\sum_{(l,\delta)\in J_1}\frac{\varepsilon\delta s}{2\sqrt{2}}\frac{\varepsilon'\delta s}{2\sqrt{2}}\frac{\varepsilon''\delta s}{2\sqrt{2}} \\
 &-\frac{2}{b}\sum_{\delta}(\frac{b}{2}-\frac{\varepsilon\delta s}{2\sqrt{2}})(\frac{b}{2}-\frac{\varepsilon'\delta s}{2\sqrt{2}})(\frac{b}{2}-\frac{\varepsilon''\delta s}{2\sqrt{2}})
 \inpr{\gamma_0}{\gamma_0} 
 -b^2\sum_{\gamma\in \Gamma_*} \inpr{\gamma_0}{\gamma} \\
 &=(\frac{b^4}{2(a^2-b^2)}+\frac{b^2}{2})
 -\frac{b^2}{2}\sum_{\gamma\in \Gamma}\inpr{\gamma_0}{\gamma}
 -\frac{(\varepsilon\varepsilon'+\varepsilon'\varepsilon''+\varepsilon''\varepsilon)s^2}{4}\inpr{\gamma_0}{\gamma_0}\\
 &=\frac{1}{2(|\Gamma|-|G|)}+\frac{\varepsilon\varepsilon'+\varepsilon'\varepsilon''+\varepsilon''\varepsilon}{4}.
\end{align*}
\begin{align*}
\lefteqn{N_{(\gamma_0,\varepsilon),(\gamma_0,\varepsilon'),\gamma}} \\
 &=\frac{b^3}{2(a-b)}-\frac{b^3}{2(a+b)}
 -2\sum_{\delta}(\frac{b}{2}-\frac{\varepsilon \delta s}{2\sqrt{2}})(\frac{b}{2}-\frac{\varepsilon'\delta s}{2\sqrt{2}})\inpr{\gamma}{\gamma_0} 
 -b^2\sum_{\xi\in \Gamma_*}(\inpr{\gamma}{\xi}+\inpr{\gamma}{-\xi}) \\
 &=(\frac{b^4}{a^2-b^2}+b^2)-b^2\sum_{\xi\in \Gamma}\inpr{\gamma}{\xi}
 -\frac{\varepsilon\varepsilon's^2}{2}\inpr{\gamma}{\gamma_0}
 =\frac{1}{|\Gamma|-|G|}+\frac{\varepsilon\varepsilon'\inpr{\gamma_0+\gamma}{\gamma_0}}{2}.
\end{align*}
\begin{align*}
\lefteqn{N_{(\gamma_0,\varepsilon),\gamma,\gamma'}} \\
 &=\frac{b^3}{a-b}-\frac{b^3}{a+b}
 -2b\sum_{\delta}(\frac{b}{2}-\frac{\varepsilon\delta s}{2\sqrt{2}})\inpr{\gamma_0+\gamma+\gamma'}{\gamma_0} \\
 &-b^2\sum_{\xi\in \Gamma_*}\inpr{\gamma_0}{\xi}(\inpr{\gamma}{\xi}+\inpr{\gamma}{-\xi})(\inpr{\gamma'}{\xi}+\inpr{\gamma'}{-\xi}) \\
 &=(\frac{2b^4}{a^2-b^2}+2b^2)-b^2\sum_{\xi\in \Gamma}(\inpr{\gamma_0+\gamma+\gamma'}{\xi}+\inpr{\gamma_0+\gamma-\gamma'}{\xi})\\
 &=\frac{2}{|\Gamma|-|G|}-\delta_{\gamma_0+\gamma+\gamma',0}-\delta_{\gamma_0+\gamma-\gamma',0}.
\end{align*}
\begin{align*}
\lefteqn{N_{\gamma,\gamma',\gamma''}} \\
 &=(\frac{2b^3}{a-b}-\frac{2b^3}{a+b})-4b^2\inpr{\gamma+\gamma'+\gamma''}{\gamma_0} \\
 &-b^2\sum_{\xi\in \Gamma_*}(\inpr{\gamma}{\xi}+\inpr{-\gamma}{\xi})(\inpr{\gamma'}{\xi}+\inpr{-\gamma'}{\xi})(\inpr{\gamma}{\xi}+\inpr{-\gamma''}{\xi}) \\
 &=(\frac{4b^4}{a^2-b^2}+4b^2)\\
 &-b^2\sum_{\xi\in \Gamma}(\inpr{\gamma+\gamma'+\gamma''}{\xi}+\inpr{-\gamma+\gamma'+\gamma''}{\xi}
 +\inpr{\gamma-\gamma'+\gamma''}{\xi}+\inpr{\gamma+\gamma'-\gamma''}{\xi})\\
 &=\frac{4}{|\Gamma|-|G|}
 -\delta_{\gamma+\gamma'+\gamma'',0}-\delta_{-\gamma+\gamma'+\gamma'',0}
 -\delta_{\gamma-\gamma'+\gamma'',0}-\delta_{\gamma+\gamma'-\gamma'',0}.
\end{align*}
\end{proof}

\begin{proof}[Proof of Lemma 4.7]
We set 
$$A_m=2a\sum_{g\in G_*}q_1(g)^m,\quad B_m=2b\sum_{\gamma\in \Gamma_*}q_2(\gamma)^m.$$
Then 
$$\cG(q_1,m)+\cG(q_2,m)=a(1+q_1(g_0)^m)+b(1+q_2(\gamma_0)^m)+A_m+B_m.$$

We have 
\begin{align*}
\lefteqn{\nu_m(\pi)} \\
 &=N^{\pi}_{0,\pi}\frac{a^2-b^2}{2} 
 +N^{\pi}_{\pi,\pi}\frac{(a+b)^2}{4}+\sum_{\varepsilon}N^\pi_{\pi,(g_0,\varepsilon)}\frac{a(a+b)}{4}(q_1(g_0)^m+q_1(g_0)^{-m}) \\
 &+\sum_{g\in G_*}N^{\pi}_{\pi,g}\frac{a(a+b)}{2}(q_1(g)^m+q_1(g)^{-m})
 +\sum_{\varepsilon}N^\pi_{\pi,(\gamma_0,\varepsilon)}\frac{b(a+b)}{4}(q_2(\gamma_0)^m+q_2(\gamma_0)^{-m})\\
 &+\sum_{\gamma\in \Gamma_*}N^{\pi}_{\pi,\gamma}\frac{b(a+b)}{2}(q_2(\gamma)^m+q_2(\gamma)^{-m})  
 +\sum_{\varepsilon,\varepsilon'}N^\pi_{(g_0,\varepsilon),(g_0,\varepsilon')}\frac{a^2}{4} +\sum_{\varepsilon,\varepsilon'}N^\pi_{(\gamma_0,\varepsilon),(\gamma_0,\varepsilon')}\frac{b^2}{4}\\
 &+\sum_{g\in G_*,\varepsilon}(N^\pi_{(g_0,\varepsilon),g}\frac{a^2}{2}(\frac{q_1(g_0)^m}{q_1(g)^m}+\frac{q_1(g)^m}{q_1(g_0)})
 +N^\pi_{(\gamma_0,\varepsilon),g}\frac{ab}{2}(\frac{q_2(\gamma_0)^m}{q_1(g)^m}+\frac{q_1(g)^m}{q_2(\gamma_0)^m}))\\
 &+\sum_{\gamma\in \Gamma_*,\varepsilon}(N^\pi_{(g_0,\varepsilon),\gamma}\frac{ab}{2}(\frac{q_1(g_0)^m}{q_2(\gamma)^m}+\frac{q_2(\gamma)^m}{q_1(g_0)^m})
 +N^\pi_{(\gamma_0,\varepsilon),\gamma}\frac{ab}{2}(\frac{q_2(\gamma_0)^m}{q_2(\gamma)^m}+\frac{q_2(\gamma)^m}{q_2(\gamma_0)^m}))\\
 &+\sum_{\varepsilon,\varepsilon'}N^\pi_{(g_0,\varepsilon),(\gamma_0,\varepsilon')}\frac{ab}{4}(\frac{q_1(g_0)^m}{q_2(\gamma_0)^m}+\frac{q_2(\gamma_0)^m}{q_1(g_0)^m})\\  
 &+\sum_{g,h\in G_*}N^{\pi}_{g,h}a^2\frac{q_1(g)^m}{q_1(h)^m}
 +\sum_{g\in G_*,\;\gamma\in \Gamma_*}N^{\pi}_{g,\gamma}ab(\frac{q_1(g)^m}{q_2(\gamma)^m}+\frac{q_2(\gamma)^m}{q_1(g)^m})
 +\sum_{\gamma,\xi\in \Gamma_*}N^{\pi}_{\gamma,\xi}b^2\frac{q_2(\gamma)^m}{q_2(\xi)^m}\\
 &=\frac{a^2-b^2}{2} 
 +\frac{(a+b)^2}{|\Gamma|-|G|}+\frac{a(a+b)}{|\Gamma|-|G|}(q_1(g_0)^m+q_1(g_0)^{-m})\\ 
 &+\frac{2a(a+b)}{|\Gamma|-|G|}\sum_{g\in G_*}(q_1(g)^m+q_1(g)^{-m})
 +\frac{b(a+b)}{|\Gamma|-|G|}(q_2(\gamma_0)^m+q_2(\gamma_0)^{-m})\\
 &+\frac{2b(a+b)}{|\Gamma|-|G|}\sum_{\gamma\in \Gamma_*}(q_2(\gamma)^m+q_2(\gamma)^{-m})
 +a^2(\frac{1}{|\Gamma|-|G|}+\frac{1}{2})+b^2(\frac{1}{|\Gamma|-|G|}-\frac{1}{2})\\
 &+\sum_{g\in G_*}(\frac{2a^2}{|\Gamma|-|G|}(\frac{q_1(g_0)^m}{q_1(g)^m}+\frac{q_1(g)^m}{q_1(g_0)})
 +\frac{2ab}{|\Gamma|-|G|}(\frac{q_2(\gamma_0)^m}{q_1(g)^m}+\frac{q_1(g)^m}{q_2(\gamma_0)^m}))\\
 &+\sum_{\gamma\in \Gamma_*,\varepsilon}(\frac{2ab}{|\Gamma|-|G|}(\frac{q_1(g_0)^m}{q_2(\gamma)^m}+\frac{q_2(\gamma)^m}{q_1(g_0)^m})
 +\frac{2ab}{|\Gamma|-|G|}(\frac{q_2(\gamma_0)^m}{q_2(\gamma)^m}+\frac{q_2(\gamma)^m}{q_2(\gamma_0)^m}))\\
 &+\frac{ab}{|\Gamma|-|G|}(\frac{q_1(g_0)^m}{q_2(\gamma_0)^m}+\frac{q_2(\gamma_0)^m}{q_1(g_0)^m})  
 +\frac{4a^2}{|\Gamma|-|G|}\sum_{g,h\in G_*}\frac{q_1(g)^m}{q_1(h)^m}+a^2|G_*|\\
 &+\frac{4ab}{|\Gamma|-|G|}\sum_{g\in G_*,\;\gamma\in \Gamma_*}(\frac{q_1(g)^m}{q_2(\gamma)^m}+\frac{q_2(\gamma)^m}{q_1(g)^m})
 +\frac{4b^2}{|\Gamma|-|G|}\sum_{\gamma,\xi\in \Gamma_*}\frac{q_2(\gamma)^m}{q_2(\xi)^m}-b^2|\Gamma_*|\\
 &=\frac{(a+b)^2+(a+b)(q_1(g_0)^m+q_1(g_0)^{-m}+q_2(\gamma_0)^m+q_2(\gamma_0)^{-m}+A_m+A_{-m}+B_m+B_{-m})}{|\Gamma|-|G|}\\
 &+\frac{a^2+b^2+A_mA_{-m}+A_mB_{-m}+B_mA_{-m}+B_mB_{-m}}{|\Gamma|-|G|}\\
 &=\frac{|\cG(q_1,m)+\cG(q_2,m)|^2}{|\Gamma|-|G|}.
\end{align*}

\begin{align*}
\lefteqn{2\nu_m((g_0,\varepsilon))-\nu_m(\pi)} \\
 &=2N^{(g_0,\varepsilon)}_{0,(g_0,\varepsilon)}\frac{a(a-b)}{4}(q_1(g_0)^m+q_1(g_0)^{-m})-N^\pi_{0,\pi}\frac{a^2-b^2}{2}\\
 &+\sum_{\delta}(2N^{(g_0,\varepsilon)}_{\pi,(g_0,\delta)}-N^\pi_{\pi,(g_0,\delta)})\frac{a(a+b)}{4}(q_1(g_0)^m+q_1(g_0)^{-m}) \\
 &+\sum_{\delta}(2N^{(g_0,\varepsilon)}_{\pi,(\gamma_0,\delta)}-N^\pi_{\pi,(\gamma_0,\delta)})\frac{b(a+b)}{4}(q_1(\gamma_0)^m+q_1(\gamma_0)^{-m})  \\
 &+\sum_{\delta,\delta'}(2N^{(g_0,\varepsilon)}_{(g_0,\delta),(g_0,\delta')}- N^\pi_{(g_0,\delta),(g_0,\delta')})\frac{a^2}{4}
  +\sum_{\delta,\delta'}(2N^{(g_0,\varepsilon)}_{(\gamma_0,\delta),(\gamma_0,\delta')}- N^\pi_{(\gamma_0,\delta),(\gamma_0,\delta')})\frac{b^2}{4}\\
 &+\sum_{\delta,\delta'}(2N^{(g_0,\varepsilon)}_{(g_0,\delta),(\gamma_0,\delta')}- N^\pi_{(g_0,\delta),(\gamma_0,\delta')})\frac{ab}{4}(\frac{q_1(g_0)^m}{q_2(\gamma_0)^m}+\frac{q_2(\gamma_0)^m}{q_1(g_0)^m})\\
 &+\sum_{h\in G_*,\delta}(2N^{(g_0,\varepsilon)}_{(g_0,\delta),h}- N^\pi_{(g_0,\delta),h})\frac{a^2}{2}(\frac{q_1(g_0)^m}{q_1(h)^m}+\frac{q_1(h)^m}{q_1(g_0)^m})\\
 &+\sum_{\xi\in \Gamma_*,\delta}(2N^{(g_0,\varepsilon)}_{(g_0,\delta),\xi}- N^\pi_{(g_0,\delta),\xi})\frac{ab}{2}(\frac{q_1(g_0)^m}{q_2(\xi)^m}+\frac{q_2(\xi)^m}{q_1(g_0)^m})\\
 &+\sum_{h\in G_*,\delta}(2N^{(g_0,\varepsilon)}_{h,(\gamma_0,\delta)}- N^\pi_{h,(\gamma_0,\delta)})\frac{ab}{2}(\frac{q_1(h)^m}{q_2(\gamma_0)^m}+\frac{q_2(\gamma_0)^m}{q_1(h)^m})\\
 &+\sum_{\xi\in \Gamma_*,\delta}(2N^{(g_0,\varepsilon)}_{(\gamma_0,\delta),\gamma}- N^\pi_{(\gamma_0,\delta),\gamma})\frac{b^2}{2}(\frac{q_2(\gamma_0)^m}{q_2(\gamma)^m}+\frac{q_2(\gamma)^m}{q_2(\gamma_0)^m})\\
 &+\sum_{h,h'\in G_*}(2N^{(g_0,\varepsilon)}_{h,h'}-N^\pi_{h,h'})a^2\frac{q_1(h)^m}{q_1(h')^{m}}
 +\sum_{\xi,\xi'\in \Gamma_*}(2N^{(g_0,\varepsilon)}_{\xi,\xi'}-N^\pi_{\xi,\xi'})b^2\frac{q_2(\xi)^m}{q_2(\xi')^{m}}\\
 &=\frac{a(a-b)}{2}(q_1(g_0)^m+q_1(g_0)^{-m})-\frac{a^2-b^2}{2}+\frac{a(a+b)}{2}(q_1(g_0)^m+q_1(g_0)^{-m})-\frac{a^2-b^2}{2} \\
 &+\sum_{h,h'\in G_*}(2\delta_{g_0+h+h',0}+2\delta_{g_0+h-h',0}-\delta_{h,h'})a^2\frac{q_1(h)^m}{q_1(h')^{m}}
 +\sum_{\xi,\xi'\in \Gamma_*}\delta_{\xi,\xi'}b^2\frac{q_2(\xi)^m}{q_2(\xi')^{m}}\\
 &=a^2(q_1(g_0)^m+q_1(g_0)^{-m})-a^2+b^2+2a^2\sum_{h\in G_*}\frac{q_1(h)^m}{q_1(g_0+h)^{m}}-a^2|G_*|+b^2|\Gamma_*|\\
 &=\delta_{mg_0,0}q_1(g_0)^m. 
\end{align*}

\begin{align*}
\lefteqn{\nu_m(g)-\nu_m(\pi)} \\
 &=N^g_{0,g}\frac{a(a-b)}{2}(q_1(g)^m+q_1(g)^{-m})-N^\pi_{0,\pi}\frac{a^2-b^2}{2}\\
 &+\sum_{\delta}(N^g_{\pi,(g_0,\delta)}-N^\pi_{\pi,(g_0,\delta)})\frac{a(a+b)}{4}(q_1(g_0)^m+q_1(g_0)^{-m}) \\
 &+\sum_{h\in G_*}(N^g_{\pi,h}-N^\pi_{\pi,h})\frac{a(a+b)}{2}(q_1(h)^m+q_1(h)^{-m}) \\
 &+\sum_{\delta}(N^g_{\pi,(\gamma_0,\delta)}-N^\pi_{\pi,(\gamma_0,\delta)})\frac{b(a+b)}{4}(q_1(\gamma_0)^m+q_1(\gamma_0)^{-m})  \\
 &+\sum_{\delta,\delta'}(N^g_{(g_0,\delta),(g_0,\delta')}- N^\pi_{(g_0,\delta),(g_0,\delta')})\frac{a^2}{4}
  +\sum_{\delta,\delta'}(N^g_{(\gamma_0,\delta),(\gamma_0,\delta')}- N^\pi_{(\gamma_0,\delta),(\gamma_0,\delta')})\frac{b^2}{4}\\
 &+\sum_{\delta,\delta'}(N^g_{(g_0,\delta),(\gamma_0,\delta')}- N^\pi_{(g_0,\delta),(\gamma_0,\delta')})\frac{ab}{4}(\frac{q_1(g_0)^m}{q_2(\gamma_0)^m}+\frac{q_2(\gamma_0)^m}{q_1(g_0)^m})\\
 &+\sum_{h\in G_*,\delta}(N^g_{(g_0,\delta),h}- N^\pi_{(g_0,\delta),h})\frac{a^2}{2}(\frac{q_1(g_0)^m}{q_1(h)^m}+\frac{q_1(h)^m}{q_1(g_0)^m})\\
 &+\sum_{\xi\in \Gamma_*,\delta}(N^g_{(g_0,\delta),\xi}- N^\pi_{(g_0,\delta),\xi})\frac{ab}{2}(\frac{q_1(g_0)^m}{q_2(\xi)^m}+\frac{q_2(\xi)^m}{q_1(g_0)^m})\\
 &+\sum_{h\in G_*,\delta}(N^g_{h,(\gamma_0,\delta)}- N^\pi_{h,(\gamma_0,\delta)})\frac{ab}{2}(\frac{q_1(h)^m}{q_2(\gamma_0)^m}+\frac{q_2(\gamma_0)^m}{q_1(h)^m})\\
 &+\sum_{\xi\in \Gamma_*,\delta}(N^g_{(\gamma_0,\delta),\xi}- N^\pi_{(\gamma_0,\delta),\xi})\frac{b^2}{2}(\frac{q_2(\gamma_0)^m}{q_2(\xi)^m}+\frac{q_2(\xi)^m}{q_2(\gamma_0)^m})\\
 &+\sum_{h,h'\in G_*}(N^g_{h,h'}-N^\pi_{h,h'})a^2\frac{q_1(h)^m}{q_1(h')^{m}}
 +\sum_{\xi,\xi'\in \Gamma_*}(N^g_{\xi,\xi'}-N^\pi_{\xi,\xi'})b^2\frac{q_2(\xi)^m}{q_2(\xi')^{m}}\\
 &=\frac{a(a-b)}{2}(q_1(g)^m+q_1(g)^{-m})-\frac{a^2-b^2}{2}
 +\frac{a(a+b)}{2}(q_1(g)^m+q_1(g)^{-m}) -\frac{a^2}{2}+\frac{b^2}{2}\\
 &+a^2\sum_{h\in G_*}(\delta_{g_0+g+h,0}+\delta_{g_0+g-h,0})(\frac{q_1(g_0)^m}{q_1(h)^m}+\frac{q_1(h)^m}{q_1(g_0)^m})\\
 &+a^2\sum_{h,h'\in G_*}(\delta_{g+h+h',0}+\delta_{-g+h+h',0}+\delta_{g-h+h',0}+\delta_{g+h-h',0}-\delta_{h,h'})\frac{q_1(h)^m}{q_1(h')^{m}}\\
 &-b^2\sum_{\xi,\xi'\in \Gamma_*}\delta_{\xi,\xi'}\frac{q_2(\xi)^m}{q_2(\xi')^{m}}\\
  &=a^2(q_1(g)^m+q_1(g)^{-m})-a^2+b^2-a^2|G_*|+b^2|\Gamma_*|
 +a^2(\frac{q_1(g_0)^m}{q_1(g_0+g)^m}+\frac{q_1(g_0+g)^m}{q_1(g_0)^m})\\
 &+a^2\sum_{h,h'\in G_*}(\delta_{g+h+h',0}+\delta_{-g+h+h',0}+\delta_{g-h+h',0}+\delta_{g+h-h',0})\frac{q_1(h)^m}{q_1(h')^{m}}\\
 &=\delta_{mg,0}q_1(g)^m.
\end{align*}

\begin{align*}
\lefteqn{2\nu_m((\gamma_0,\varepsilon))-\nu_m(\pi)} \\
 &=2N^{(\gamma_0,\varepsilon)}_{0,(\gamma_0,\varepsilon)}\frac{b(a-b)}{4}(q_1(\gamma_0)^m+q_1(\gamma_0)^{-m})-N^\pi_{0,\pi}\frac{a^2-b^2}{2}\\
 &+\sum_{\delta}(2N^{(\gamma_0,\varepsilon)}_{\pi,(g_0,\delta)}-N^\pi_{\pi,(g_0,\delta)})\frac{a(a+b)}{4}(q_1(g_0)^m+q_1(g_0)^{-m}) \\
 &+\sum_{\delta}(2N^{(\gamma_0,\varepsilon)}_{\pi,(\gamma_0,\delta)}-N^\pi_{\pi,(\gamma_0,\delta)})\frac{b(a+b)}{4}(q_1(\gamma_0)^m+q_1(\gamma_0)^{-m})  \\
 &+\sum_{\delta,\delta'}(2N^{(\gamma_0,\varepsilon)}_{(g_0,\delta),(g_0,\delta')}- N^\pi_{(g_0,\delta),(g_0,\delta')})\frac{a^2}{4}
  +\sum_{\delta,\delta'}(2N^{(\gamma_0,\varepsilon)}_{(\gamma_0,\delta),(\gamma_0,\delta')}- N^\pi_{(\gamma_0,\delta),(\gamma_0,\delta')})\frac{b^2}{4}\\
 &+\sum_{\delta,\delta'}(2N^{(\gamma_0,\varepsilon)}_{(g_0,\delta),(\gamma_0,\delta')}- N^\pi_{(g_0,\delta),(\gamma_0,\delta')})\frac{ab}{4}(\frac{q_1(g_0)^m}{q_2(\gamma_0)^m}+\frac{q_2(\gamma_0)^m}{q_1(g_0)^m})\\
 &+\sum_{h\in G_*,\delta}(2N^{(\gamma_0,\varepsilon)}_{(g_0,\delta),h}- N^\pi_{(g_0,\delta),h})\frac{a^2}{2}(\frac{q_1(g_0)^m}{q_1(h)^m}+\frac{q_1(h)^m}{q_1(g_0)^m})\\
 &+\sum_{\xi\in \Gamma_*,\delta}(2N^{(\gamma_0,\varepsilon)}_{(g_0,\delta),\xi}- N^\pi_{(g_0,\delta),\xi})\frac{ab}{2}(\frac{q_1(g_0)^m}{q_2(\xi)^m}+\frac{q_2(\xi)^m}{q_1(g_0)^m})\\
 &+\sum_{h\in G_*,\delta}(2N^{(\gamma_0,\varepsilon)}_{h,(\gamma_0,\delta)}- N^\pi_{h,(\gamma_0,\delta)})\frac{ab}{2}(\frac{q_1(h)^m}{q_2(\gamma_0)^m}+\frac{q_2(\gamma_0)^m}{q_1(h)^m})\\
 &+\sum_{\xi\in \Gamma_*,\delta}(2N^{(\gamma_0,\varepsilon)}_{(\gamma_0,\delta),\xi}- N^\pi_{(\gamma_0,\delta),\xi})\frac{b^2}{2}(\frac{q_2(\gamma_0)^m}{q_2(\xi)^m}+\frac{q_2(\xi)^m}{q_2(\gamma_0)^m})\\
 &+\sum_{h,h'\in G_*}(2N^{(\gamma_0,\varepsilon)}_{h,h'}-N^\pi_{h,h'})a^2\frac{q_1(h)^m}{q_1(h')^{m}}
 +\sum_{\xi,\xi'\in \Gamma_*}(2N^{(\gamma_0,\varepsilon)}_{\xi,\xi'}-N^\pi_{\xi,\xi'})b^2\frac{q_2(\xi)^m}{q_2(\xi')^{m}}\\
 &=\frac{b(a-b)}{2}(q_1(\gamma_0)^m+q_1(\gamma_0)^{-m})-\frac{a^2-b^2}{2}\\
 &-\frac{b(a+b)}{2}(q_1(\gamma_0)^m+q_1(\gamma_0)^{-m})-\frac{a^2}{2}+\frac{b^2}{2}\\
 &-a^2|G_*|
 -2b^2\sum_{\xi,\xi'\in \Gamma_*}(\delta_{\gamma_0+\xi+\xi',0}+\delta_{\gamma_0+\xi-\xi',0})\frac{q_2(\xi)^m}{q_2(\xi')^{m}}+b^2|\Gamma_*| \\
 &=-b^2(q_1(\gamma_0)^m+q_1(\gamma_0)^{-m})-a^2+b^2-a^2|G_*|+b^2|\Gamma_*|-2b^2\sum_{\xi\in \Gamma_*}\frac{q_2(\xi)^m}{q_2(\gamma_0+\xi)^{m}} \\
 &=-\delta_{m\gamma_0,0}q_2(\gamma_0)^m. 
\end{align*}

\begin{align*}
\lefteqn{\nu_m(\gamma)-\nu_m(\pi)} \\
 &=N^\gamma_{0,\gamma}\frac{b(a-b)}{2}(q_2(\gamma)^m+q_2(\gamma)^{-m})-N^\pi_{0,\pi}\frac{a^2-b^2}{2}\\
 &+\sum_{\delta}(N^\gamma_{\pi,(g_0,\delta)}-N^\pi_{\pi,(g_0,\delta)})\frac{a(a+b)}{4}(q_1(g_0)^m+q_1(g_0)^{-m}) \\
 &+\sum_{h\in G_*}(N^\gamma_{\pi,h}-N^\pi_{\pi,h})\frac{a(a+b)}{2}(q_1(h)^m+q_1(h)^{-m}) \\
 &+\sum_{\delta}(N^\gamma_{\pi,(\gamma_0,\delta)}-N^\pi_{\pi,(\gamma_0,\delta)})\frac{b(a+b)}{4}(q_1(\gamma_0)^m+q_1(\gamma_0)^{-m})  \\
 &+\sum_{\xi\in \Gamma_*}(N^\gamma_{\pi,\xi}-N^\pi_{\pi,\xi})\frac{b(a+b)}{2}(q_1(\xi)^m+q_1(\xi)^{-m})  \\
 &+\sum_{\delta,\delta'}(N^\gamma_{(g_0,\delta),(g_0,\delta')}- N^\pi_{(g_0,\delta),(g_0,\delta')})\frac{a^2}{4}
  +\sum_{\delta,\delta'}(N^\gamma_{(\gamma_0,\delta),(\gamma_0,\delta')}- N^\pi_{(\gamma_0,\delta),(\gamma_0,\delta')})\frac{b^2}{4}\\
 &+\sum_{\delta,\delta'}(N^\gamma_{(g_0,\delta),(\gamma_0,\delta')}- N^\pi_{(g_0,\delta),(\gamma_0,\delta')})\frac{ab}{4}(\frac{q_1(g_0)^m}{q_2(\gamma_0)^m}+\frac{q_2(\gamma_0)^m}{q_1(g_0)^m})\\
 &+\sum_{h\in G_*,\delta}(N^\gamma_{(g_0,\delta),h}- N^\pi_{(g_0,\delta),h})\frac{a^2}{2}(\frac{q_1(g_0)^m}{q_1(h)^m}+\frac{q_1(h)^m}{q_1(g_0)^m})\\
 &+\sum_{\xi\in \Gamma_*,\delta}(N^\gamma_{(g_0,\delta),\xi}- N^\pi_{(g_0,\delta),\xi})\frac{ab}{2}(\frac{q_1(g_0)^m}{q_2(\xi)^m}+\frac{q_2(\xi)^m}{q_1(g_0)^m})\\
 &+\sum_{h\in G_*,\delta}(N^\gamma_{h,(\gamma_0,\delta)}- N^\pi_{h,(\gamma_0,\delta)})\frac{ab}{2}(\frac{q_1(h)^m}{q_2(\gamma_0)^m}+\frac{q_2(\gamma_0)^m}{q_1(h)^m})\\
 &+\sum_{\xi\in \Gamma_*,\delta}(N^\gamma_{(\gamma_0,\delta),\gamma}- N^\pi_{(\gamma_0,\delta),\gamma})\frac{b^2}{2}(\frac{q_2(\gamma_0)^m}{q_2(\gamma)^m}+\frac{q_2(\gamma)^m}{q_2(\gamma_0)^m})\\
 &+\sum_{h,h'\in G_*}(N^\gamma_{h,h'}-N^\pi_{h,h'})a^2\frac{q_1(h)^m}{q_1(h')^{m}}
 +\sum_{\xi,\xi'\in \Gamma_*}(N^\gamma_{\xi,\xi'}-N^\pi_{\xi,\xi'})b^2\frac{q_2(\xi)^m}{q_2(\xi')^{m}}\\
 &=\frac{b(a-b)}{2}(q_2(\gamma)^m+q_2(\gamma)^{-m})-\frac{a^2-b^2}{2}-\frac{a^2}{2}+\frac{b^2}{4}-\frac{b(a+b)}{2}(q_2(\gamma)^m+q_2(\gamma)^{-m})\\
 &-\sum_{\xi\in \Gamma_*,\delta}(\delta_{\gamma+\gamma_0+\xi,0}+\delta_{\gamma+\gamma_0-\xi})\frac{b^2}{2}(\frac{q_2(\gamma_0)^m}{q_2(\xi)^m}+\frac{q_2(\xi)^m}{q_2(\gamma_0)^m})\\
 &-a^2|G_*|+b^2|\Gamma_*|
 -b^2\sum_{\xi,\xi'\in \Gamma_*}(\delta_{\gamma,\xi,\xi',0}+\delta_{-\gamma+\xi+\xi',0}+\delta_{-\gamma-\xi+\xi',0}+\delta_{\gamma+\xi-\xi',0})\frac{q_2(\xi)^m}{q_2(\xi')^{m}}\\
 &=-b^2(q_2(\gamma)^m+q_2(\gamma)^{-m})
 -b^2(\frac{q_2(\gamma_0)^m}{q_2(\gamma+\gamma_0)^m}+\frac{q_2(\gamma+\gamma_0)^m}{q_2(\gamma_0)^m})\\
 &-b^2\sum_{\xi,\xi'\in \Gamma_*}(\delta_{\gamma,\xi,\xi',0}+\delta_{-\gamma+\xi+\xi',0}+\delta_{-\gamma-\xi+\xi',0}+\delta_{\gamma+\xi-\xi',0})\frac{q_2(\xi)^m}{q_2(\xi')^{m}}\\
 &=-\delta_{m\gamma,0}q_2(\gamma_0)^m.
\end{align*}
\end{proof}

\section{Calculations for Section 5}
\begin{proof}[Proof of Lemma 5.4]
\begin{align*}
\lefteqn{\sum_{x}S_{(u,0),x}\overline{S_{(u',0),x}}=\sum_{x}S_{(u,\pi),x}\overline{S_{(u',\pi),x}}} \\
 &=\frac{a^2+b^2}{2}\sum_{v\in U}\inpr{u'-u}{v}+\frac{a^2}{2}\sum_{k\in K_*}\inpr{u'-u}{k}+
 a^2\sum_{g\in G_*}\inpr{u'-u}{g}\\
 &+\frac{b^2}{2}\sum_{\sigma\in \Sigma_*}\inpr{u'-u}{\sigma}
 +b^2\sum_{\gamma\in \Gamma_*}\inpr{u'-u}{\gamma} \\
 &=\frac{a^2}{2}\sum_{g\in G}\inpr{u'-u}{g}+ \frac{b^2}{2}\sum_{\gamma\in \Gamma}\inpr{u'-u}{\gamma} \\
 &=\delta_{u,u'}. 
\end{align*}
\begin{align*}
\lefteqn{\sum_{x}S_{(u,0),x}\overline{S_{(u',\pi),x}}} \\
 &=\frac{a^2-b^2}{2}\sum_{v\in U}\inpr{u'-u}{v}+\frac{a^2}{2}\sum_{k\in K_*}\inpr{u'-u}{k}+
 a^2\sum_{g\in G_*}\inpr{u'-u}{g}\\
 &-\frac{b^2}{2}\sum_{\sigma\in \Sigma_*}\inpr{u'-u}{\sigma}
 -b^2\sum_{\gamma\in \Gamma_*}\inpr{u'-u}{\gamma} \\
 &=\frac{a^2}{2}\sum_{g\in G}\inpr{u'-u}{g}- \frac{b^2}{2}\sum_{\gamma\in \Gamma}\inpr{u'-u}{\gamma} \\
 &=0. 
\end{align*}
\begin{align*}
\lefteqn{\sum_{x}S_{(u,0),x}\overline{S_{(k,\varepsilon),x}}=\sum_{x}S_{(u,\pi),x}\overline{S_{(k,\varepsilon),x}}} \\
 &=\frac{a^2}{2}\sum_{v\in U}\inpr{k-u}{v}+\frac{a^2}{2}\sum_{l\in K_*}\inpr{k-u}{l}+a^2\sum_{g\in G_*}\inpr{k-u}{g} \\
 &=\frac{a^2}{2}\sum_{g\in G_*}\inpr{k-u}{g}=\frac{\delta_{u,k}}{2}=0. 
\end{align*}
\begin{align*}
\lefteqn{\sum_{x}S_{(u,0),x}\overline{S_{g,x}}=\sum_{x}S_{(u,\pi),x}\overline{S_{g,x}}} \\
 &=a^2\sum_{v\in U}\inpr{g-u}{v}+a^2\sum_{k\in K_*}\inpr{g-u}{k}+
 a^2\sum_{h\in G_*}(\inpr{g-u}{h}+\inpr{g-u}{\theta(h)})\\
 &=a^2\sum_{h\in G}\inpr{g-u}{h}=\delta_{u,g}=0. \\
\end{align*}
\begin{align*}
\lefteqn{\sum_{x}S_{(u,0),x}\overline{S_{(\sigma,\varepsilon),x}}=-\sum_{x}S_{(u,\pi),x}\overline{S_{(\sigma,\varepsilon),x}}} \\
 &=-\frac{b^2}{2}\sum_{v\in U}\inpr{\sigma-u}{v}
 -\frac{b^2}{2}\sum_{\tau\in \Sigma_*}\inpr{\sigma-u}{\tau}-b^2\sum_{\gamma\in \Gamma_*}\inpr{\sigma-u}{\gamma} \\
 &=-\frac{b^2}{2}\sum_{\gamma\in \Gamma_*}\inpr{\sigma-u}{\gamma}=-\frac{\delta_{u,\sigma}}{2}=0. 
\end{align*}
\begin{align*}
\lefteqn{\sum_{x}S_{(u,0),x}\overline{S_{\gamma,x}}=-\sum_{x}S_{(u,\pi),x}\overline{S_{\gamma,x}}} \\
 &=-b^2\sum_{v\in U}\inpr{\gamma-u}{v}-b^2\sum_{\sigma\in \Sigma_*}\inpr{\gamma-u}{\sigma}
 -b^2\sum_{\xi\in \Gamma_*}(\inpr{\gamma-u}{\xi}+\inpr{\gamma-u}{\xi})\\
 &=-b^2\sum_{\xi\in \Gamma}\inpr{\gamma-u}{\xi}=-\delta_{u,\gamma}=0. \\
\end{align*}
\begin{align*}
\lefteqn{\sum_{x}S_{(k,\varepsilon),x}\overline{S_{(k',\varepsilon'),x}}} \\
 &=\frac{a^2}{2}\sum_{u\in U}\inpr{k'-k}{u}
 +\sum_{(l,\delta)\in J_1}\inpr{k'-k}{l}
 (\frac{a}{2}+\frac{\varepsilon \delta s}{4}) 
 (\frac{a}{2}+\frac{\varepsilon'\delta\overline{s}}{4})\\
&+a^2\sum_{g\in G_*}\inpr{k'-k}{g}
+\sum_{\sigma\in \Sigma_*}
\frac{\varepsilon\varepsilon' f(k,\sigma)\overline{f(k',\sigma)} }{8} 
\\
 &=\frac{a^2}{2}\sum_{g\in G}\inpr{k'-k}{g}+\frac{\varepsilon\varepsilon'}{8}(\sum_{l\in K_*}\inpr{k'-k}{l}
 +\sum_{\sigma\in \Sigma_*}f(k,\sigma)\overline{f(k',\sigma)})\\
 &=\frac{\delta_{k,k'}}{2}+\frac{\varepsilon\varepsilon'}{8}(\sum_{l\in K_*}\inpr{k'-k}{l}
 +\sum_{\sigma\in \Sigma_*}f(k,\sigma)\overline{f(k',\sigma)})\\
 &=\delta_{k,k'}\delta_{\varepsilon,\varepsilon'}. 
\end{align*}
$\sum_{x}S_{(\sigma,\varepsilon),x}\overline{S_{(\sigma',\varepsilon'),x}}=\delta_{\sigma,\sigma'}\delta_{\varepsilon,\varepsilon'}$ 
can be shown in the same way. 
\begin{align*}
\lefteqn{\sum_{x}S_{(k,\varepsilon),x}\overline{S_{g,x}}} \\
 &=a^2\sum_{u\in U}\inpr{k-g}{u}+a^2\sum_{l\in K_*}\inpr{k-g}{l}+a^2\sum_{h\in G_*}(\inpr{k-g}{h}+\inpr{k+g}{h})\\
 &=a^2\sum_{h\in G}\inpr{k-g}{h}=\delta_{k,g}=0. 
\end{align*}
$\sum_{x}S_{(\sigma,\varepsilon),x}\overline{S_{\gamma,x}}=0$ can be shown in the same way. 
\begin{align*}
\lefteqn{\sum_{x} S_{(k,\varepsilon),x} \overline{S_{(\sigma,\epsilon'),x}}   } \\
 &=\sum_{(l,\delta)\in J_1}\overline{\inpr{k}{l}}(\frac{a}{2}+\frac{\varepsilon s\delta}{4})\frac{\varepsilon'\delta \overline{f(l,\sigma)}}{4}
 -\sum_{(\tau,\delta)\in J_2}\frac{\varepsilon \delta f(k,\tau)}{4}
 \inpr{\sigma}{\tau}(\frac{b}{2}+\frac{-\varepsilon'\varepsilon''\overline{s}}{4}) \\
 &=\frac{\varepsilon\varepsilon'}{8}s\sum_{l\in K_*}\overline{\inpr{k}{l}}\overline{f(l,\sigma)}
 +\frac{\varepsilon\varepsilon'}{8}\overline{s}\sum_{\tau\in \Sigma_*}\inpr{\sigma}{\tau}f(k,\tau)
 =0. 
\end{align*}
$\sum_{x} S_{(k,\varepsilon),x} \overline{S_{\gamma,x}}=0$, 
$\sum_{x}S_{g,x}\overline{S_{(\sigma,\varepsilon),x}}=0$, and 
$\sum_{x}S_{g,x}\overline{S_{\gamma,x}}=0$ are obvious.   
\begin{align*}
\lefteqn{\sum_{x}S_{g,x}\overline{S_{g',x}}} \\
 &=2a^2\sum_{u\in U}\inpr{g'-g}{u}+2a^2\sum_{l\in K_*}\inpr{g'-g}{l}
 +a^2\sum_{h\in G_*}(\inpr{-g}{h}+\inpr{-g}{\theta(h)})(\inpr{g'}{h}+\inpr{g'}{\theta(h)}) \\
 &=a^2(\sum_{h\in G}\inpr{g'-g}{h}+\sum_{h\in G}\inpr{g'-\theta(g)}{h})\\
 &=\delta_{g',g}+\delta_{g',\theta(g)}=\delta_{g,g'}. 
\end{align*}
$\sum_{x}S_{\gamma,x}\overline{S_{\gamma',x}}=\delta_{\gamma,\gamma'}$ can be shown in the same way. 
Thus $S$ is unitary. 

It is obvious that $T$ and $C$ are unitary.  
Since $\overline{S_{x,y}}=S_{\overline{x},y}=S_{x,\overline{y}}$ and $S$ is symmetric, we have $S^2=C$. 
Since $T_{x,x}=T_{\overline{x},\overline{x}}$, we have $CT=T C$. 

The only remaining condition is $(ST)^3=cC$, and it suffices to verify 
$$\sum_{x}S_{j,x}S_{j',x}T_{x,x}=cS_{j,j'}\overline{T_{j,j}T_{j',j'}}$$
for all $j,j'\in J$. 
\begin{align*}
\lefteqn{\sum_{x}S_{(u,0),x}S_{(u',0,),x}T_{x,x}=\sum_{x}S_{(u,\pi),x}S_{(u',\pi),x}T_{x,x}} \\
 &=\frac{a^2+b^2}{2}\sum_{v\in U}\overline{\inpr{u+u'}{v}}q_0(v)
 +\frac{a^2}{2}\sum_{k\in K_*}\overline{\inpr{u+u'}{k}}q_1(k)
 +a^2\sum_{g\in G_*}\overline{\inpr{u+u'}{g}}q_1(g)\\
 &+\frac{b^2}{2}\sum_{\sigma\in \Sigma_*}\overline{\inpr{u+u'}{\sigma}}q_2(\sigma)+b^2\sum_{\gamma\in \Gamma_*}\overline{\inpr{u+u'}{\gamma}}q_2(\gamma) \\
 &=\frac{a^2}{2}\sum_{g\in G}\inpr{-(u+v)}{g}q_1(g)+\frac{b^2}{2}\sum_{\gamma\in G}\inpr{-(u+v)}{\gamma}q_2(\gamma)\\
 &=\frac{a^2}{2}\sum_{g\in G} q_1(g-u-u')\overline{q_0(u+u')}
  +\frac{b^2}{2}\sum_{\gamma\in \Gamma}q_2(\gamma-u-u')\overline{q_0(u+u')} \\
 &=(\frac{a}{2}\cG(q_1)+\frac{b}{2}\cG(q_2))\overline{\inpr{u}{u'}q_0(u)q_0(u')}\\
 &=cS_{(u,0),(u',0)}\overline{T_{(u,0),(u,0)}T_{(u',0),(u',0)}}
 =cS_{(u,\pi),(u',\pi)}\overline{T_{(u,\pi),(u,\pi)}T_{(u',\pi),(u'\pi)}}.
\end{align*}
\begin{align*}
\lefteqn{\sum_{x}S_{(u,0),x}S_{(u',\pi,),x}T_{x,x}} \\
 &=\frac{a^2-b^2}{2}\sum_{v\in U}\overline{\inpr{u+u'}{v}}q_0(v)
 +\frac{a^2}{2}\sum_{k\in K_*}\overline{\inpr{u+u'}{k}}q_1(k)
 +a^2\sum_{g\in G_*}\overline{\inpr{u+u'}{g}}q_1(g)\\
 &-\frac{b^2}{2}\sum_{\sigma\in \Sigma_*}\overline{\inpr{u+u'}{\sigma}}q_2(\sigma)-b^2\sum_{\gamma\in \Gamma_*}\overline{\inpr{u+u'}{\gamma}}q_2(\gamma) \\
 &=\frac{a^2}{2}\sum_{g\in G} q_1(g-u-u')\overline{q_0(u+u')}
  -\frac{b^2}{2}\sum_{\gamma\in \Gamma}q_2(\gamma-u-u')\overline{q_0(u+u')} \\
 &=(\frac{a}{2}\cG(q_1)-\frac{b}{2}\cG(q_2))\overline{\inpr{u}{u'}q_0(u)q_0(u')}\\
 &=cS_{(u,0),(u',\pi)}\overline{T_{(u,0),(u,0)}T_{(u',\pi),(u',\pi)}}.
\end{align*}
\begin{align*}
\lefteqn{\sum_{x}S_{(u,0),x}S_{(k,\varepsilon),x}T_{x,x}=\sum_{x}S_{(u,\pi),x}S_{(k,\varepsilon),x}T_{x,x}} \\
 &=\frac{a^2}{2}\sum_{v\in U}\overline{\inpr{u+k}{v}}q_0(v)+\frac{a^2}{2}\sum_{l\in K_*}\overline{\inpr{u+k}{l}}q_1(l)
 +a^2\sum_{g\in G_*}\overline{\inpr{u+k}{g}}q_1(g) \\
 &=\frac{a^2}{2}\sum_{g\in G}q_1(g-u-k)\overline{q_1(u+k)} \\
 &=\frac{a}{2}\cG(q_1)\overline{\inpr{u}{k}q_0(u)q_1(k)}\\
 &=cS_{(u,0),(k,\varepsilon)}\overline{T_{(u,0),(u,0)}T_{(k,\varepsilon),(k,\varepsilon)}}
 =cS_{(u,\pi),(k,\varepsilon)}\overline{T_{(u,\pi),(u,\pi)}T_{(k,\varepsilon),(k,\varepsilon)}}.
\end{align*}
In a similar way, we can show
\begin{align*}
\lefteqn{\sum_{x}S_{(u,0),x}S_{(\sigma,\varepsilon),x}T_{x,x}=-\sum_{x}S_{(u,\pi),x}S_{(\sigma,\varepsilon),x}T_{x,x}} \\
 &=cS_{(u,0),(\sigma,\varepsilon)}\overline{T_{(u,0),(u,0)}T_{(\sigma,\varepsilon),(\sigma,\varepsilon)}}
 =-cS_{(u,\pi),(\sigma,\varepsilon)}\overline{T_{(u,\pi),(u,\pi)}T_{(\sigma,\varepsilon),(\sigma,\varepsilon)}}.
\end{align*}
\begin{align*}
\lefteqn{\sum_{x}S_{(u,0),x}S_{g,x}T_{x,x}=S_{(u,\pi),x}S_{g,x}T_{x,x}} \\
 &=a^2\sum_{v\in U}\overline{\inpr{u+g}{v}}q_0(v)+a^2\sum_{l\in K_*}\overline{\inpr{u+g}{l}}q_1(l)
 +a^2\sum_{h\in G_*}(\overline{\inpr{u+g}{h}}+\overline{\inpr{u+g}{\theta(h)}})q_1(h)\\
 &=a^2\sum_{h\in G}q_1(h-u-g)\overline{q_1(u+g)}
 =a\cG(q_1)\overline{\inpr{u}{g}q_0(u)q_1(g)}\\
 &=cS_{(u,0),g}\overline{T_{(u,0),(u,0)}T_{g,g}}
 =cS_{(u,\pi),g}\overline{T_{(u,\pi),(u,\pi)}T_{g,g}}.
\end{align*}
In a similar way, we can show 
\begin{align*}
\lefteqn{\sum_{x}S_{(u,0),x}S_{\gamma,x}T_{x,x}=-S_{(u,\pi),x}S_{\gamma,x}T_{x,x}} \\
 &=cS_{(u,0),\gamma}\overline{T_{(u,0),(u,0)}T_{\gamma,\gamma}}
 =-cS_{(u,\pi),\gamma}\overline{T_{(u,\pi),(u,\pi)}T_{\gamma,\gamma}}.
\end{align*}
\begin{align*}
\lefteqn{\sum_{x}S_{(k,\varepsilon),x}S_{(k',\varepsilon'),x}T_{x,x}} \\
 &=\frac{a^2}{2}\sum_{u\in U}\overline{\inpr{k+k'}{u}}q_0(u)
 +\sum_{(l,\delta)\in J_1}\overline{\inpr{k+k'}{l}}(\frac{a}{2}
 +\frac{\varepsilon\delta s}{4}) (\frac{a}{2}+\frac{\varepsilon'\delta s}{4})q_1(l)\\
 &+a^2\sum_{g\in G_*}\overline{\inpr{k+k'}{g}}q_1(g)
 +\frac{\varepsilon\varepsilon'}{8}\sum_{\sigma\Sigma_*}f(k,\sigma)f(k',\sigma)q_2(\sigma)\\
 &=\frac{a^2}{2}\sum_{g\in G}\overline{\inpr{k+k'}{g}}q_1(g)
 +\frac{\varepsilon\varepsilon' s^2}{8}\sum_{l\in K_*}\overline{\inpr{k+k'}{l}}q_1(l)
 +\frac{\varepsilon\varepsilon'}{8}\sum_{\sigma\in \Sigma_*}f(k,\sigma)f(k',\sigma) q_2(\sigma)\\
 &=c\frac{a}{2}\overline{\inpr{k}{k'}q_1(k)q_1(k')}
 +\frac{\varepsilon\varepsilon'}{4}(\frac{s^2}{2}\sum_{l\in K_*}\overline{\inpr{k+k'}{l}}q_1(l)
 +\frac{1}{2}\sum_{\sigma\in \Sigma_*}f(k,\sigma)f(k',\sigma) q_2(\sigma))\\
 &=cS_{(k,\varepsilon),(k',\varepsilon')}\overline{T_{(k,\varepsilon),(k,\varepsilon)}T_{(k',\varepsilon'),(k',\varepsilon')}}
\end{align*}
\begin{align*}
\lefteqn{\sum_{x}S_{(k,\varepsilon),x}S_{g,x}T_{x,x}} \\
 &=a^2\sum_{u\in U}\overline{\inpr{k+g}{u}}q_0(u)+a^2\sum_{l\in K_*}\overline{\inpr{k+g}{l}}q_1(l)
 +a^2\sum_{h\in G_*}(\overline{\inpr{k+g}{h}}+\overline{\inpr{k+g}{\theta(h)}})q_1(h)\\
 &=a^2\sum_{h\in G}\overline{\inpr{k+g}{h}}q_1(h)=a\cG(q_1)\overline{q_1}(k+g)=ca\overline{\inpr{k}{g}q_1(k)q_1(g)}\\
 &=cS_{(k,\varepsilon),g}\overline{T_{(k,\varepsilon),(k,\varepsilon)}T_{g,g}}.
\end{align*}
In a similar way, we can show
$\sum_{x}S_{(\sigma,\varepsilon),x}S_{\gamma,x}T_{x,x}
=cS_{(\sigma,\varepsilon),\gamma}\overline{T_{(\sigma,\varepsilon),(\sigma,\varepsilon)}T_{\gamma,\gamma}}$.
\begin{align*}
\lefteqn{\sum_{x}S_{(k,\varepsilon),x}S_{(\sigma,\varepsilon'),x}T_{x,x}} \\
 &=\sum_{(l,\delta)\in J_1}(\frac{a}{2}+\frac{\varepsilon\delta s}{4})\overline{\inpr{k}{l}}
 \frac{\varepsilon'\delta f(l,\sigma)}{4}q_1(l)
 -\sum_{(\tau,\delta)\in J_2}\frac{\varepsilon\delta f(k,\tau)}{4}
 (\frac{b}{2}+\frac{-\varepsilon'\delta s}{4})\overline{\inpr{\sigma}{\tau}}q_2(\tau)\\
 &=\frac{\varepsilon\varepsilon'}{8}(s\sum_{l\in K_*}f(l,\sigma)\overline{\inpr{k}{l}}q_1(l)
 +s\sum_{\tau\in \Sigma_*}f(k,\tau)\overline{\inpr{\sigma}{\tau}}q_2(\tau))\\
 &=cS_{(k,\varepsilon),(\sigma,\varepsilon')}\overline{T_{(k,\varepsilon),(k,\varepsilon)}T_{(\sigma,\varepsilon'),(\sigma,\varepsilon')}}.
\end{align*}
It is easy to show 
$$\sum_{x}S_{(k,\varepsilon),x}S_{\gamma,x}T_{x,x}=0
=cS_{(k,\varepsilon),\gamma}\overline{T_{(k,\varepsilon),(k,\varepsilon)}T_{\gamma,\gamma}},$$
$$\sum_{x}S_{g,x}S_{(\sigma,\varepsilon),x}T_{x,x}=0
=cS_{(g,(\sigma,\varepsilon)}\overline{T_{g,g}T_{\sigma,\varepsilon),(\sigma,\varepsilon)}}.$$
$$\sum_{x}S_{g,x}S_{\gamma,x}T_{x,x}=0
=cS_{g,\gamma}\overline{T_{g,g}T_{\gamma,\gamma}}.$$
\begin{align*}
\lefteqn{\sum_{x}S_{g,x}S_{g',x}T_{x,x}} \\
 &=2a^2\sum_{u\in U}\overline{\inpr{g+g'}{u}}q_0(u)
 +2a^2\sum_{l\in K_*}\overline{\inpr{g+g'}{l}}q_1(l)\\
 &+a^2\sum_{h\in G_*}(\overline{\inpr{g}{h}}+\overline{\inpr{g}{\theta(h)}})(\overline{\inpr{g'}{h}}+\overline{\inpr{g'}{\theta(h)}})q_1(h)\\
 &=a^2\sum_{h\in G}(\overline{\inpr{g+g'}{h}}+\overline{\inpr{g+\theta(g')}{h}})q_1(h)\\
 &=a^2\sum_{h\in G}(q_1(g+g'-h)\overline{q_1(g+g')}+q_1(g+\theta(g')-h)\overline{q_1(g+\theta(g'))})\\
 &=ca\overline{(\inpr{g}{g'}q_1(g)q_1(g')+\inpr{g}{\theta(g')}q_1(g)q_1(g')}\\
 &=cS_{g,g'}\overline{T_{g,g}T_{g',g'}}.
\end{align*}
In a similar way, we can show 
$\sum_{x}S_{\gamma,x}S_{\gamma',x}T_{x,x}
=cS_{\gamma,\gamma'}\overline{T_{\gamma,\gamma}T_{\gamma',\gamma'}}.$
\begin{align*}
\lefteqn{\sum_{x}S_{(\sigma,\varepsilon),x}S_{(\sigma',\varepsilon'),x}T_{x,x}} \\
 &=\frac{b^2}{2}\sum_{u\in U}\overline{\inpr{\sigma+\sigma'}{u}}q_0(u)
 +\frac{\varepsilon\varepsilon'}{8}\sum_{l\in K_*}f(l,\sigma)f(l,\sigma')q_1(l)\\
 &+\sum_{(\tau,\delta)\in J_2}\overline{\inpr{\sigma+\sigma'}{\tau}}(\frac{b}{2}
 -\frac{\varepsilon\delta s}{4}) 
 (\frac{b}{2}-\frac{\varepsilon'\delta s}{4})q_2(\tau)\\
 &+b^2\sum_{\gamma\in \Gamma_*}\overline{\inpr{\sigma+\sigma'}{\gamma}}q_2(\gamma)\\
 &=\frac{b^2}{2}\sum_{\gamma\in \Gamma}\overline{\inpr{\sigma+\sigma'}{\gamma}}q_2(\gamma)
 +\frac{\varepsilon\varepsilon'}{8}\sum_{l\in K_*}f(l,\sigma)f(l,\sigma')q_1(l)
 +\frac{\varepsilon\varepsilon'}{8}s^2\sum_{\tau \in \Sigma_*}\overline{\inpr{\sigma+\sigma'}{\tau}}q_2(\tau)\\
 &=-c\frac{b}{2}\overline{\inpr{\sigma}{\sigma'}}\overline{q_2(\sigma)q_2(\sigma')}
 +\frac{\varepsilon\varepsilon'}{4}(\frac{1}{2}\sum_{l\in K_*}f(l,\sigma)f(l,\sigma')q_1(l)
 +\frac{1}{2}s^2\sum_{\tau \in \Sigma_*}\overline{\inpr{\sigma+\sigma'}{\tau}}q_2(\tau))\\
 &=cS_{(\sigma,\varepsilon),(\sigma',\varepsilon')}
 \overline{T_{(\sigma,\varepsilon),(\sigma,\varepsilon)}T_{(\sigma',\varepsilon'),(\sigma',\varepsilon')}}. 
\end{align*}
\end{proof}

\begin{proof}[Proof of Lemma 5.5] Since $S_{(u,0),x}=\inpr{u}{x}S_{(0,0),x}$ in an obvious sense, 
those equations involving $(u,0)$ can be shown as in the proof of the unitarity of $S$.  
Recall $\inpr{u}{v}=1$ for all $u,v\in U$. 

\begin{align*}
\lefteqn{N_{(u,\pi),(u',\pi),(u'',\pi)}} \\
 &=(\frac{(a+b)^3}{4(a-b)}+\frac{(a-b)^3}{4(a+b)})\sum_{v\in U}\overline{\inpr{u+u'+u''}{v}} 
 +\frac{a^2}{2}\sum_{k\in K_*}\overline{\inpr{u+u'+u''}{k}} \\
 &+a^2\sum_{g\in G_*}\overline{\inpr{u+u'+u''}{g}}-\frac{b^2}{2}\sum_{\sigma\in \Sigma_*}\overline{\inpr{u+u'+u''}{\sigma}}
 -b^2\sum_{\gamma\in \Gamma_*}\overline{\inpr{u+u'+u''}{\gamma}}\\
 &=(\frac{(a+b)^4+(a-b)^4}{4(a^2-b^2)}-\frac{a^2-b^2}{2})\sum_{v\in U}\overline{\inpr{u+u'+u''}{v}}\\
 &+\frac{a^2}{2}\sum_{g\in G}\overline{\inpr{u+u'+u''}{g}}-\frac{b^2}{2}\sum_{\gamma\in \Gamma}\overline{\inpr{u+u'+u''}{\gamma}}\\
 &=\frac{4a^2b^2}{a^2-b^2}\sum_{v\in U}\overline{\inpr{u+u'+u''}{v}}+\frac{\delta_{u+u'+u'',0}}{2}-\frac{\delta_{u+u'+u'',0}}{2}
 =\frac{8}{|\Gamma|-|G|}. 
\end{align*}
\begin{align*}
\lefteqn{N_{(u,\pi),(u',\pi),(k,\varepsilon)}} \\
 &=(\frac{a(a+b)^2}{4(a-b)}+\frac{a(a-b)^2}{4(a+b)})\sum_{v\in U}\overline{\inpr{u+u'+k}{v}}
 +\frac{a^2}{2}\sum_{l\in K_*}\overline{\inpr{u+u'+k}{l}} \\
 &+a^2\sum_{h\in G_*}\overline{\inpr{u+u'+k}{h}}\\
 &=(a\frac{(a+b)^3+(a-b)^3}{4(a^2-b^2)}-\frac{a^2}{2})\sum_{v\in U}\overline{\inpr{u+u'+k}{v}}
 +\frac{a^2}{2}\sum_{h\in G}\overline{\inpr{u+u'+k}{h}} \\
 &=\frac{2a^2b^2}{a^2-b^2}\sum_{v\in U}\overline{\inpr{u+u'+k}{v}}+\frac{1}{2}\delta_{u+u'+k,0} 
 =\frac{2}{|\Gamma|-|G|}(1+\overline{\inpr{k}{u_0}}).
\end{align*}
\begin{align*}
\lefteqn{N_{(u,\pi),(u',\pi),g}} \\
 &(\frac{a(a+b)^2}{2(a-b)}+\frac{a(a-b)^2}{2(a+b)})\sum_{v\in U}\overline{\inpr{u+u'+g}{v}}+a^2\sum_{l\in K_*}\overline{\inpr{u+u'+g}{l}} \\
 &+a^2\sum_{h\in G_*}(\overline{\inpr{u+u'+g}{h}}+\overline{\inpr{u+u'+g}{\theta(h)}})\\
 &=(a\frac{(a+b)^3+(a-b)^3}{2(a^2-b^2)}-a^2)\sum_{v\in U}\overline{\inpr{u+u'+g}{v}}+a^2\sum_{h\in G}\overline{\inpr{u+u'+g}{h}} \\
 &=\frac{4a^2b^2}{a^2-b^2}\sum_{v\in U}\overline{\inpr{u+u'+g}{v}}+\delta_{u+u'+g,0} \\
 &=\frac{4}{|\Gamma|-|G|}(1+\overline{\inpr{g}{u_0}}).
\end{align*}
\begin{align*}
\lefteqn{N_{(u,\pi),(u',\pi),(\sigma,\varepsilon)}} \\
 &=(\frac{b(a+b)^2}{4(a-b)}-\frac{b(a-b)^2}{4(a+b)})\sum_{v\in U}\overline{\inpr{u+u'+\sigma}{v}}-\frac{b^2}{2}\sum_{\tau \in \Sigma_*}\overline{\inpr{u+u'+\sigma}{\tau}} \\
 &-b^2\sum_{\gamma\in \Gamma_*}\overline{\inpr{u+u'+\sigma}{\gamma}}\\
 &=(\frac{b^2(3a^2+b^2)}{2(a^2-b^2)}+\frac{b^2}{2})\sum_{v\in U}\overline{\inpr{u+u'+\sigma}{v}}-\frac{b^2}{2}\sum_{\gamma\in \Gamma}
 \overline{\inpr{u+u'+\sigma}{\gamma}}\\
 &=\frac{2a^2b^2}{a^2-b^2}\sum_{v\in U}\overline{\inpr{u+u'+\sigma}{v}}-\frac{\delta_{u+u'+\gamma,0}}{2}\\
 &=\frac{2}{|\Gamma|-|G|}(1+\overline{\inpr{\sigma}{u_0}}).
\end{align*}
\begin{align*}
\lefteqn{N_{(u,\pi),(u',\pi),\gamma}} \\
 &=(\frac{b(a+b)^2}{2(a-b)}-\frac{b(a-b)^2}{2(a+b)})\sum_{v\in U}\overline{\inpr{u+u'+\gamma}{v}}
 -b^2\sum_{\tau\in \Sigma_*}\overline{\inpr{u+u'+\gamma}{\tau}} \\
 &-b^2\sum_{\xi\in \Gamma_*}(\overline{\inpr{u+u'+\gamma}{\xi}}+\overline{\inpr{u+u'+\gamma}{\theta(\xi)}})\\
 &=(b\frac{(a+b)^3-(a-b)^3}{2(a^2-b^2)}+b^2)\sum_{v\in U}\overline{\inpr{u+u'+\gamma}{v}}-b^2\sum_{\xi\in \Gamma}\overline{\inpr{u+u'+\gamma}{\xi}} \\
 &=\frac{4a^2b^2}{a^2-b^2}\sum_{v\in U}\overline{\inpr{u+u'+\gamma}{v}}-\delta_{u+u'+\gamma,0} \\
 &=\frac{4}{|\Gamma|-|G|}(1+\overline{\inpr{\gamma}{u_0}}).\\
\end{align*}
\begin{align*}
\lefteqn{N_{(u,\pi),(k,\varepsilon),(k,\varepsilon')}} \\
 &=(\frac{a^2(a+b)}{4(a-b)}+\frac{a^2(a-b)}{4(a+b)})\sum_{v\in U}\overline{\inpr{u+k+k'}{v}}
 +\sum_{(l,\delta)\in J_1}\overline{\inpr{u+k+k'}{l}}(\frac{a}{2}+\frac{\varepsilon\delta s}{4}) (\frac{a}{2}+\frac{\varepsilon'\delta s}{4})\\
 &+a^2\sum_{h\in G_*}\overline{\inpr{u+k+k'}{h}}-\sum_{(\sigma,\delta)\in J_2}\overline{\inpr{u}{\sigma}}\frac{\varepsilon\delta f(k,\sigma)}{4} \frac{\varepsilon'\delta f(k',\sigma)}{4}\\
 &=(a^2\frac{(a+b)^2+(a-b)^2}{4(a^2-b^2)}-\frac{a^2}{2})\sum_{v\in U}\overline{\inpr{u+k+k'}{v}}+\frac{a^2}{2}\sum_{h\in G}\overline{\inpr{u+k+k'}{h}}\\
 &+\frac{\varepsilon\varepsilon'}{8}(s^2\sum_{l\in K_*}\overline{\inpr{u+k+k'}{l}}-\sum_{\sigma\in \Sigma_*}f(k,\sigma)f(k',\sigma)\overline{\inpr{u}{\sigma}})\\
 &=\frac{1}{|\Gamma|-|G|}(1+\overline{\inpr{k+k'}{u_0}})+\frac{\delta_{u+k+k',0}}{2}\\
 &+\frac{\varepsilon\varepsilon's^2}{8}(\sum_{l\in K_*}\overline{\inpr{u+k+k'}{l}}-\sum_{\sigma\in \Sigma_*}\overline{\inpr{k'-k}{\sigma}\inpr{u}{\sigma}})\\
 &=\frac{2}{|\Gamma|-|G|}+\frac{\delta_{u+k+k',0}}{2}+
 \frac{\varepsilon\varepsilon's^2}{4}(\overline{\inpr{u+k+k'}{k_1}}-\overline{\inpr{u+k+k'}{\sigma_1}}\inpr{2k}{\sigma_1}).
\end{align*}
\begin{align*}
\lefteqn{N_{(u,\pi),(k,\varepsilon),g}} \\
 &=(\frac{a^2(a+b)}{2(a-b)}+\frac{a^2(a-b)}{2(a+b)})\sum_{v\in U}\overline{\inpr{u+k+g}{v}}
 +a\sum_{(l,\delta)\in J_1}\overline{\inpr{u+k+g}{l}}(\frac{a}{2}+\frac{\varepsilon \delta s}{4})\\
 &+a^2\sum_{h\in G_*}\overline{\inpr{u+k}{h}}(\overline{\inpr{g}{h}}+\overline{\inpr{g}{\theta(h)}})\\
 &=a^2(\frac{(a+b)^2+(a-b)^2}{2(a^2-b^2)}-1)\sum_{v\in U}\overline{\inpr{u+k+g}{v}}+a^2\sum_{h\in G}\overline{\inpr{u+k+g}{h}}\\
 &=\frac{2}{|\Gamma|-|G|}(1+\overline{\inpr{k+g}{u_0}})+\delta_{u+k+g,0}\\
 &=\frac{2}{|\Gamma|-|G|}(1+\overline{\inpr{k+g}{u_0}}).
\end{align*}
\begin{align*}
\lefteqn{N_{(u,\pi),(k,\varepsilon),(\sigma,\varepsilon')}} \\
 &=(\frac{ab(a+b)}{4(a-b)}-\frac{ab(a-b)}{4(a+b)})\sum_{v\in U}\overline{\inpr{u}{v}\inpr{k}{v}\inpr{\sigma}{v}}
 +\sum_{(l,\delta)\in J_1}(\frac{a}{2}+\frac{\varepsilon\delta s}{4})\overline{\inpr{u+k}{l}}\frac{\varepsilon'\delta f(l,\sigma)}{4} \\
 &+\sum_{(\tau,\delta)\in J_2}\frac{\varepsilon\delta f(k,\tau)}{4}(\frac{b}{2}-\frac{\varepsilon'\delta s}{4})\overline{\inpr{u+\sigma}{\tau}} \\
 &=\frac{a^2b^2}{a^2-b^2}\sum_{v\in U}\overline{\inpr{k}{v}\inpr{\sigma}{v}}
 +\frac{s\varepsilon\varepsilon'}{8}(\sum_{l\in K_*}f(l,\sigma)\overline{\inpr{u+k}{l}}-\sum_{\tau\in \Sigma_*}f(k,\tau)\overline{\inpr{u+\sigma}{\tau}})\\
 &=\frac{1}{|\Gamma|-|G|}\sum_{v\in U}\overline{\inpr{k}{v}\inpr{\sigma}{v}}
  +\frac{s\varepsilon\varepsilon'}{8}\sum_{v\in U}(f(k+v,\sigma)\overline{\inpr{u+k}{k+v}}-f(k,\sigma+v)\overline{\inpr{u+\sigma}{\sigma+v}})\\
  &=\frac{1}{|\Gamma|-|G|}\sum_{v\in U}\overline{\inpr{k}{v}\inpr{\sigma}{v}}
  +\frac{sf(k,\sigma)\varepsilon\varepsilon'}{8}(\overline{\inpr{k}{k}\inpr{u}{k}}-\overline{\inpr{\sigma}{\sigma}\inpr{u}{\sigma}})
  \sum_{v\in U}\overline{\inpr{v}{k}\inpr{v}{\sigma}}\\
   &=(1+\overline{\inpr{u_0}{k}\inpr{u_0}{\sigma}})(\frac{1}{|\Gamma|-|G|}+\frac{sf(k,\sigma)\varepsilon\varepsilon'}{8}(\overline{\inpr{k}{k}\inpr{u}{k}}-\overline{\inpr{\sigma}{\sigma}\inpr{u}{\sigma}}))\\
   &=0.
\end{align*}
\begin{align*}
\lefteqn{N_{(u,\pi),(k,\varepsilon),\gamma}} \\
 &=\frac{ab}{2}(\frac{a+b}{a-b}-\frac{a-b}{a+b})\sum_{v\in U}\overline{\inpr{u+k}{v}\inpr{\gamma}{v}}
 +b\sum_{(\sigma,\delta)\in J_2}\frac{\varepsilon \delta f(k,\sigma)}{4}\overline{\inpr{\gamma}{\sigma}} \\
 &=\frac{2}{|\Gamma|-|G|}(1+\overline{\inpr{k}{u_0}\inpr{\gamma}{u_0}}).
\end{align*}
\begin{align*}
\lefteqn{N_{(u,\pi),g,g'}} \\
 &=(a^2\frac{a+b}{a-b}+a^2\frac{a-b}{a+b})\sum_{v\in U}\overline{\inpr{u+g+g'}{v}}+2a^2\sum_{l\in K_*}\overline{\inpr{u+g+g'}{l}}\\
 &+a^2\sum_{h\in G_*}\overline{\inpr{u}{h}}(\overline{\inpr{g}{h}}+\overline{\inpr{g}{\theta(h)}})(\overline{\inpr{g'}{h}}+\overline{\inpr{g'}{\theta(h)}}) \\
 &=a^2(\frac{(a+b)^2+(a-b)^2}{a^2-b^2}-2)\sum_{v\in U}\overline{\inpr{u+g+g'}{v}}
 +a^2\sum_{h\in G}(\overline{\inpr{u+g+g'}{h}}+\overline{\inpr{u+g+\theta(g')}{h}})\\
 &=\frac{4a^2b^2}{a^2-b^2}\sum_{v\in U}\overline{\inpr{u+g+g'}{v}} +\delta_{u+g+g',0}+\delta_{u+g+\theta(g'),0}\\
 &=\frac{4}{|\Gamma|-|G|}\sum_{v\in U}\overline{\inpr{u+g+g'}{v}} +\delta_{u+g+g',0}+\delta_{u+g+\theta(g'),0}.
\end{align*}
\begin{align*}
\lefteqn{N_{(u,\pi),g,(\sigma,\varepsilon)}} \\
 &=\frac{ab}{2}(\frac{a+b}{a-b}-\frac{a-b}{a+b})\sum_{v\in U}\overline{\inpr{g}{v}\inpr{u+\sigma}{v}}+a\sum_{(l,\delta)\in J_1}\overline{\inpr{g}{l}\inpr{u}{l}}\frac{\varepsilon \delta f(l,\sigma)}{4} \\
 &=\frac{2}{|\Gamma|-|G|}(1+\overline{\inpr{g}{u_0}\inpr{\sigma}{u_0}}). 
\end{align*}
\begin{align*}
\lefteqn{N_{(u,\pi),g,\gamma}} \\
 &=(\frac{ab(a+b)}{a-b}-\frac{ab(a-b)}{a+b})\sum_{v\in U}\overline{\inpr{u}{v}\inpr{g}{v}\inpr{\gamma}{v}}\\
 &=ab\frac{(a+b)^2-(a-b)^2}{a^2-b^2}\sum_{v\in U}\overline{\inpr{u}{v}\inpr{g}{v}\inpr{\gamma}{v}}\\
 &=\frac{4a^2b^2}{a^2-b^2}\sum_{v\in U}\overline{\inpr{u}{v}\inpr{g}{v}\inpr{\gamma}{v}}\\
 &=\frac{4}{|\Gamma|-|G|}\sum_{v\in U}\overline{\inpr{u}{v}\inpr{g}{v}\inpr{\gamma}{v}}.
\end{align*}
\begin{align*}
\lefteqn{N_{(u,\pi),(\sigma,\varepsilon),(\sigma',\varepsilon')}} \\
 &=\frac{b^2}{4}(\frac{a+b}{a-b}+\frac{a-b}{a+b})\sum_{v\in U}\overline{\inpr{u+\sigma+\sigma'}{v}}
 +\sum_{(l,\delta)\in J_1}\frac{\varepsilon \delta f(l,\sigma)}{4}\frac{\varepsilon'\delta f(l,\sigma')}{4}\overline{\inpr{u}{l}} \\
 &-\sum_{(\tau,\delta)\in J_2}(\frac{b}{2}-\frac{\varepsilon\delta s}{4}) (\frac{b}{2}-\frac{\varepsilon\delta s}{4}) \overline{\inpr{u+\sigma+\sigma'}{\tau}}
 -b^2\sum_{\xi\in \Gamma_*}\overline{\inpr{u+\sigma+\sigma'}{\xi}}\\
 &=\frac{b^2}{2}(\frac{a^2+b^2}{a^2-b^2}+1)\sum_{v\in U}\overline{\inpr{u+\sigma+\sigma'}{v}}
 -\frac{b^2}{2}\sum_{\xi\in \Gamma}\overline{\inpr{u+\sigma+\sigma'}{\xi}}\\
 &+\frac{\varepsilon\varepsilon'}{8}\sum_{l\in K_*}f(l,\sigma)f(l,\sigma')\overline{\inpr{u}{l}}
 -\frac{\varepsilon\varepsilon s^2}{8}\sum_{\tau\in \Sigma_*}\overline{\inpr{u+\sigma+\sigma'}{\tau}}\\
 &=\frac{1}{|\Gamma|-|G|}(1+\overline{\inpr{u+\sigma+\sigma'}{u_0}})-\frac{\delta_{u+\sigma+\sigma',0}}{2}\\
 &+\frac{\varepsilon\varepsilon' s^2}{8}(\sum_{l\in K_*}\overline{\inpr{u+\sigma+\sigma'}{l}}\inpr{2\sigma}{l}-\sum_{\tau\in \Sigma_*}\overline{\inpr{u+\sigma+\sigma'}{\tau}})\\
 &=\frac{2}{|\Gamma|-|G|}-\frac{\delta_{u+\sigma+\sigma',0}}{2}
 +\frac{\varepsilon\varepsilon' s^2}{4}(\overline{\inpr{u+\sigma+\sigma'}{k_1}}\inpr{2\sigma}{k_1}-\overline{\inpr{u+\sigma+\sigma'}{\sigma}})\\
\end{align*}
\begin{align*}
\lefteqn{N_{(u,\pi),(\sigma,\varepsilon),\gamma}} \\
 &=\frac{b^2}{2}(\frac{a+b}{a-b}+\frac{a-b}{a+b})\sum_{v\in U}\overline{\inpr{u+\sigma+\gamma}{v}}
 -b\sum_{(\tau,\delta)\in J_2}(\frac{b}{2}-\frac{\varepsilon \delta s}{4})\overline{\inpr{u+\sigma+\gamma}{\tau}} \\
 &-b^2\sum_{\xi\in \Gamma_*}\overline{\inpr{u+\sigma}{\xi}}(\overline{\inpr{\gamma}{\xi}}+\overline{\inpr{\gamma}{\theta(\xi)}}) \\
 &=\frac{2}{|\Gamma|-|G|}(1+\overline{\inpr{u+\sigma+\gamma}{u_0}})-\delta_{u+\sigma+\gamma,0}. 
\end{align*}
\begin{align*}
\lefteqn{N_{(u,\pi),\gamma,\gamma'}} \\
 &=(b^2\frac{a+b}{a-b}++b^2\frac{a-b}{a+b})\sum_{v\in U}\overline{\inpr{u+\gamma+\gamma'}{v}}-2b^2 \sum_{\tau\in \Sigma_*}\overline{\inpr{u+\gamma+\gamma'}{\tau}}\\
 &-b^2\sum_{\xi\in \Gamma_*}\overline{\inpr{u}{\xi}}(\overline{\inpr{\gamma}{\xi}}+\overline{\inpr{\gamma}{\theta(\xi)}})(\overline{\inpr{\gamma'}{\xi}}+\overline{\inpr{\gamma'}{\theta(\xi)}}) \\
 &=b^2(\frac{(a+b)^2+(a-b)^2}{a^2-b^2}+2)\sum_{v\in U}\overline{\inpr{u+\gamma+\gamma'}{v}}
 -b^2\sum_{\xi\in \Gamma}(\overline{\inpr{u+\gamma+\gamma'}{\xi}}+\overline{\inpr{u+\gamma+\theta(\gamma')}{\xi}})\\
 &=\frac{4a^2b^2}{a^2-b^2}\sum_{v\in U}\overline{\inpr{u+\gamma+\gamma'}{v}}-\delta_{u+\gamma+\gamma',0}-\delta_{u+\gamma+\theta(\gamma'),0}\\
 &=\frac{4}{|\Gamma|-|G|}\sum_{v\in U}\overline{\inpr{u+\gamma+\gamma'}{v}}
 -\delta_{u+\gamma+\gamma',0}-\delta_{u+\gamma+\theta(\gamma'),0}.
\end{align*}
\begin{align*}
\lefteqn{N_{(k,\varepsilon),(k',\varepsilon'),(k'',\varepsilon'')}} \\
 &=(\frac{a^3}{4(a-b)}+\frac{a^3}{4(a+b)})\sum_{v\in U}\overline{\inpr{k+k'+k''}{v}}\\
 &+\frac{2}{a}\sum_{(l,\delta)\in J_1}(\frac{a}{2}+\frac{\varepsilon \delta s}{4})(\frac{a}{2}+\frac{\varepsilon' \delta s}{4}) 
 (\frac{a}{2}+\frac{\varepsilon'' \delta s}{4})\overline{\inpr{k+k'+k''}{l}}\\
 &+a^2\sum_{g\in G_*}\overline{\inpr{k+k'+k''}{g}} 
 +\frac{2}{b}\sum_{(\tau,\delta)\in J_2}\frac{\varepsilon\delta f(k,\tau)}{4}\frac{\varepsilon'\delta f(k',\tau)}{4}\frac{\varepsilon''\delta f(k'',\tau)}{4}\\
 &=(\frac{a^4}{2(a^2-b^2)}-\frac{a^2}{2})\sum_{v\in U}\overline{\inpr{k+k'+k''}{v}}+\frac{a^2}{2}\sum_{g\in G}\overline{\inpr{k+k'+k''}{g}} \\
 &+\frac{s^2(\varepsilon\varepsilon'+\varepsilon'\varepsilon''+\varepsilon''\varepsilon)}{8}\sum_{l\in K_*}\overline{\inpr{k+k'+k''}{l}} \\
 &=\frac{a^2b^2}{2(a^2-b^2)}\sum_{v\in U}\overline{\inpr{k+k'+k''}{v}}+\frac{\delta_{k+k'+k'',0}}{2}
 +\frac{s^2(\varepsilon\varepsilon'+\varepsilon'\varepsilon''+\varepsilon''\varepsilon)}{8}\sum_{l\in K_*}\overline{\inpr{k+k'+k''}{l}}\\
 &=(1+\overline{\inpr{u_0}{k+k'+k''}})(\frac{1}{2(|\Gamma|-|G|)}
 +\frac{s^2(\varepsilon\varepsilon'+\varepsilon'\varepsilon''+\varepsilon''\varepsilon)}{8}\overline{\inpr{k+k'+k''}{k+k'+k''}})\\
\end{align*}
\begin{align*}
\lefteqn{N_{(k,\varepsilon),(k',\varepsilon'),g}} \\
 &=(\frac{a^3}{2(a-b)}+\frac{a^3}{2(a+b)})\sum_{v\in U}\overline{\inpr{v}{k+k'+g}}\\
 &+2\sum_{(l,\delta)\in J_1}(\frac{a}{2}+\frac{\varepsilon\delta s}{4})(\frac{a}{2}
 +\frac{\varepsilon'\delta s}{4})\overline{\inpr{l}{k+k'+g}}
 +a^2\sum_{h\in G_*}\overline{\inpr{h}{k+k'}}(\overline{\inpr{h}{g}}+\overline{\inpr{h}{\theta(g)}}) \\
 &=(\frac{a^4}{a^2-b^2}-a^2) \sum_{v\in U}\overline{\inpr{v}{k+k'+g}}+a^2\sum_{h\in G}\overline{\inpr{h}{k+k'+g}}\\
 &+\frac{s^2\varepsilon\varepsilon'}{4}\sum_{l\in K_*}\overline{\inpr{l}{k+k'+g}} \\
 &=\frac{1}{|\Gamma|-|G|}\sum_{v\in U}\overline{\inpr{v}{k+k'+g}}+\delta_{k+k'+g,0}+\frac{s^2\varepsilon\varepsilon'}{4}\sum_{l\in K_*}\overline{\inpr{l}{k+k'+g}}\\
 &=\frac{1}{|\Gamma|-|G|}\sum_{v\in U}\overline{\inpr{v}{k+k'+g}}+\frac{s^2\varepsilon\varepsilon'}{4}\sum_{v\in U}\overline{\inpr{k+v}{k+k'+g}}\\
 &=(\frac{1}{|\Gamma|-|G|}+\frac{s^2\varepsilon\varepsilon'\overline{\inpr{k}{k+k'+g}}}{4})(1+\inpr{g}{u_0}).
\end{align*}
\begin{align*}
\lefteqn{N_{(k,\varepsilon),(k',\varepsilon'),(\sigma,\varepsilon'')}} \\
 &=(\frac{a^2b}{4(a-b)}-\frac{a^2b}{4(a+b)})\sum_{v\in U}\overline{\inpr{k+k'}{v}\inpr{\sigma}{v}}\\
 &+\frac{2}{a}\sum_{(l,\delta)\in J_1}\overline{\inpr{k+k'}{l}}(\frac{a}{2}+\frac{\varepsilon\delta s}{4})(\frac{a}{2}+\frac{\varepsilon'\delta s}{4})
 \frac{\varepsilon''\delta f(l,\sigma)}{4} \\
 &-\frac{2}{b}\sum_{(\tau,\delta)\in J_2}\frac{\varepsilon\delta f(k,\tau)}{4}\frac{\varepsilon'\delta f(k',\tau)}{4}
 (\frac{b}{2}-\frac{\varepsilon''\delta s}{4})\overline{\inpr{\sigma}{\tau}} \\
 &=\frac{a^2b^2}{2(a^2-b^2)}\sum_{v\in U}\overline{\inpr{k+k'}{v}\inpr{\sigma}{v}}
 +\frac{\varepsilon''(\varepsilon+\varepsilon')s}{8}\sum_{l\in K_*}\overline{\inpr{k+k'}{l}}f(l,\sigma) \\
 &-\frac{\varepsilon \varepsilon'}{8}\sum_{\tau \in \Sigma_*}f(k,\tau)f(k',\tau)\overline{\inpr{\sigma}{\tau}}\\
 &=\frac{1}{2(|\Gamma|-|G|)}(1+\overline{\inpr{\sigma}{u_0}})
 +\frac{\varepsilon''(\varepsilon+\varepsilon')s}{8}\sum_{v\in U}\overline{\inpr{k+k'}{k+v}}f(k+v,\sigma) \\
 &-\frac{\varepsilon \varepsilon'}{8}\sum_{v \in U}f(k,\sigma+v)f(k',\sigma+v)\overline{\inpr{\sigma}{\sigma+v}}.\\
 &=(1+\overline{\inpr{\sigma}{u_0}})(\frac{1}{2(|\Gamma|-|G|)}+\frac{\varepsilon''(\varepsilon+ \varepsilon')s}{8}\overline{\inpr{k+k'}{k}}f(k,\sigma)
 -\frac{\varepsilon\varepsilon'}{8}f(k,\sigma)f(k',\sigma)\overline{\inpr{\sigma}{\sigma}})\\
  &=(1+\overline{\inpr{\sigma}{u_0}})(\frac{1}{2(|\Gamma|-|G|)}+\frac{\varepsilon''(\varepsilon+ \varepsilon')s}{8}\overline{\inpr{k+k'}{k}}f(k,\sigma)
 -\frac{s^2\varepsilon\varepsilon'}{8}\overline{\inpr{k'-k}{\sigma}\inpr{\sigma}{\sigma}}).
\end{align*}
\begin{align*}
\lefteqn{N_{(k,\varepsilon),(k',\varepsilon'),\gamma}} \\
 &=(\frac{a^2b}{2(a-b)}-\frac{a^2b}{2(a+b)})\sum_{v\in U}\overline{\inpr{v}{k+k'}\inpr{v}{\gamma}}
 -2\sum_{(\tau,\delta)\in J_2}\frac{\varepsilon\delta f(k,\tau)}{4}\frac{\varepsilon'\delta f(k',\tau)}{4}\overline{\inpr{\tau}{\gamma}}
  \\
 &=\frac{a^2b^2}{a^2-b^2}\sum_{v\in U}\overline{\inpr{v}{k+k'}\inpr{v}{\gamma}}-\frac{\varepsilon\varepsilon' }{4}\sum_{\tau\in \Sigma_*}
 f(k,\tau)f(k',\tau)\overline{\inpr{\tau}{\gamma}} \\
 &=\frac{1}{|\Gamma|-|G|}\sum_{v\in U}\overline{\inpr{v}{\gamma}}-\frac{\varepsilon\varepsilon' s^2}{4}\sum_{\tau\in \Sigma_*}
 \overline{\inpr{k'-k}{\tau}\inpr{\tau}{\gamma}}  \\
 &=(1+\overline{\inpr{\gamma}{u_0}})(\frac{1}{|\Gamma|-|G|}-\frac{\varepsilon\varepsilon's^2}{4}\overline{\inpr{k'-k}{\sigma_1}\inpr{\gamma}{\sigma_1}}).
\end{align*}
\begin{align*}
\lefteqn{N_{(k,\varepsilon),g,g'}} \\
 &=(\frac{a^3}{a-b}+\frac{a^3}{a+b})\sum_{v\in U}\overline{\inpr{k+g+g'}{v}}+2a\sum_{(l,\delta)\in J_1}(\frac{a}{2}+\frac{\varepsilon \delta s}{4})\overline{\inpr{k+g+g'}{l}} \\
 &+a^2\sum_{h\in G_*}\overline{\inpr{k}{h}}(\overline{\inpr{g}{h}}+\overline{\inpr{g}{\theta(h)}})(\overline{\inpr{g'}{h}}+\overline{\inpr{g'}{\theta(h)}}) \\
 &=a^2(\frac{2a^2}{a^2-b^2}-2)\sum_{v\in U}\overline{\inpr{k+g+g'}{v}}+a^2\sum_{h\in G}(\overline{\inpr{k+g+g'}{h}}+\overline{\inpr{k+g+\theta(g')}{h}}) \\
 &=\frac{2}{|\Gamma|-|G|}(1+\overline{\inpr{k+g+g'}{u_0}})+\delta_{k+g+g',0}+\delta_{k+g-g',0}. 
\end{align*}
\begin{align*}
\lefteqn{N_{(k,\varepsilon),g,(\sigma,\varepsilon')}} \\
 &=(\frac{a^2b}{2(a-b)}-\frac{a^2b}{2(a+b)})\sum_{v\in U}\overline{\inpr{k+g}{v}\inpr{\sigma}{v}}
 +2\sum_{(\delta,l)\in J_1}(\frac{a}{2}+\frac{\varepsilon \delta s}{4})\frac{\varepsilon'\delta f(l,\sigma)}{4}\overline{\inpr{k+g}{l}} \\
 &=\frac{1}{|\Gamma|-|G|}\sum_{v\in U}\overline{\inpr{k+g}{v}\inpr{\sigma}{v}}+\frac{\varepsilon \varepsilon' s}{4}\sum_{l\in K_*}f(l,\sigma)\overline{\inpr{k+g}{l}} \\
 &=(\frac{1}{|\Gamma|-|G|}+\frac{\varepsilon\varepsilon' s f(k,\sigma)\overline{\inpr{k+g}{k}}}{4})(1+\overline{\inpr{k+g}{u_0}\inpr{\sigma}{u_0}}).
\end{align*}
$$N_{(k,\varepsilon),g,\gamma}
=(\frac{a^2b}{a-b}-\frac{a^2b}{a+b})\sum_{v\in U}\overline{\inpr{k+g}{v}\inpr{\gamma}{v}}
=\frac{2}{|\Gamma|-|G|}(1+\overline{\inpr{k+g}{u_0}\inpr{\gamma}{u_0}}).$$
\begin{align*}
\lefteqn{N_{(k,\varepsilon''),(\sigma,\varepsilon),(\sigma',\varepsilon')}} \\
 &=(\frac{ab^2}{4(a-b)}+\frac{ab^2}{4(a+b)})\sum_{v\in U}\overline{\inpr{k}{u}\inpr{\sigma+\sigma'}{u}}\\
 &+\frac{2}{a}\sum_{(l,\delta)\in J_1}(\frac{a}{2}+\frac{\varepsilon''\delta s}{4})\overline{\inpr{k}{l}}
 \frac{\varepsilon \delta f(l,\sigma)}{4}\frac{\varepsilon' \delta f(l,\sigma')}{4}\\
 &+\frac{2}{b}\sum_{(\tau,\delta)J_2}\frac{\varepsilon''\delta f(k,\tau)}{4}(\frac{b}{2}-\frac{\varepsilon\delta s}{4}) (\frac{b}{2}-\frac{\varepsilon'\delta s}{4})\overline{\inpr{\sigma+\sigma'}{\tau}}\\
 &=\frac{a^2b^2}{2(a^2-b^2)}\sum_{v\in U}\overline{\inpr{k}{u}\inpr{\sigma+\sigma'}{u}}\\
 &+\frac{\varepsilon\varepsilon'}{8}\sum_{l\in K_*}f(l,\sigma)f(l,\sigma')\overline{\inpr{k}{l}}
 -\frac{s\varepsilon''(\varepsilon+\varepsilon')}{8}\sum_{\tau\in \Sigma_*}f(k,\tau)\overline{\inpr{\sigma+\sigma'}{\tau}} \\
 &=\frac{1}{2(|\Gamma|-|G|)}\sum_{v\in U}\overline{\inpr{k}{u}\inpr{\sigma+\sigma'}{u}}\\
 &+\frac{\varepsilon\varepsilon'}{8}\sum_{v\in U}f(k,\sigma)f(k,\sigma')\overline{\inpr{\sigma+\sigma'}{v}\inpr{k}{k+v}}
 -\frac{s\varepsilon''(\varepsilon+\varepsilon')}{8}\sum_{v\in U}f(k,\sigma)\overline{\inpr{k}{v}\inpr{\sigma+\sigma'}{\sigma+v}}\\
 &=(\frac{1}{2(|\Gamma|-|G|)}+\frac{\varepsilon\varepsilon's^2\overline{\inpr{\sigma'-\sigma}{k}\inpr{k}{k}}}{8}  
 -\frac{s\varepsilon''(\varepsilon+\varepsilon')f(k,\sigma)\overline{\inpr{\sigma+\sigma'}{\sigma}}}{8}) (1+\overline{\inpr{k}{u_0}}).
\end{align*}
\begin{align*}
\lefteqn{N_{(k,\varepsilon),(\sigma,\varepsilon'),\gamma}} \\
 &=(\frac{ab^2}{2(a-b)}+\frac{ab^2}{2(a+b)})\sum_{v\in U}\overline{\inpr{k}{v}\inpr{\sigma+\gamma}{v}}
 +2\sum_{(\tau,\delta)\in J_2}\frac{\varepsilon\delta f(k,\tau)}{4}(\frac{b}{2}-\frac{\varepsilon'\delta s}{4})\overline{\inpr{\sigma+\gamma}{\tau}} \\
 &=\frac{1}{|\Gamma|-|G|}\sum_{v\in U}\overline{\inpr{k}{v}\inpr{\sigma+\gamma}{v}}-\frac{\varepsilon\varepsilon's}{4}\sum_{\tau \in \Sigma_*}f(k,\tau)\overline{\inpr{\sigma+\gamma}{\tau}} \\
 &=(\frac{1}{|\Gamma|-|G|}-\frac{\varepsilon \varepsilon' s f(k,\sigma) \overline{\inpr{\sigma+\gamma}{\sigma}}}{4})(1+\overline{\inpr{k}{u_0}\inpr{\sigma+\gamma}{u_0}}).
\end{align*}
\begin{align*}
\lefteqn{N_{(k,\varepsilon),\gamma,\gamma'}} \\
 &=(\frac{ab^2}{a-b}+\frac{ab^2}{a+b})\sum_{v\in U}\overline{\inpr{k}{v}\inpr{\gamma+\gamma'}{v}}+\frac{2b^2}{a}\sum_{(\sigma,\delta)\in J_2}\frac{\varepsilon\delta f(k,\sigma)}{4}\overline{\inpr{\gamma+\gamma'}{\sigma}} \\
 &=\frac{2}{|\Gamma|-|G|}(1+\overline{\inpr{k}{u_0}\inpr{\gamma+\gamma'}{u_0}}).
\end{align*}
\begin{align*}
\lefteqn{N_{g,g',g''}} \\
 &=(\frac{2a^3}{a-b}+\frac{2a^3}{a+b})\sum_{v\in U}\overline{\inpr{g+g'+g''}{v}}+4a^2\sum_{l\in K_*}\overline{\inpr{g+g'+g''}{k}}\\
 &+a^2\sum_{h\in G_*}\overline{(\inpr{g}{h}+\inpr{g}{\theta(h)})(\inpr{g'}{h}+\inpr{g'}{\theta(h)})(\inpr{g''}{h}+\inpr{g''}{\theta(h)})} \\
 &=(\frac{4a^4}{a^2-b^2}-4a^2)\sum_{v\in U}\overline{\inpr{g+g'+g''}{v}}\\
 &+a^2\sum_{h\in G}\overline{(\inpr{g+g'+g''}{h}+\inpr{\theta(g)+g'+g''}{h}+\inpr{g+\theta(g')+g''}{h}+\inpr{g+g'+\theta(g'')}{h})}\\
 &=\frac{4a^2b^2}{a^2-b^2}\sum_{v\in U}\overline{\inpr{g+g'+g''}{v}}+\delta_{g+g'+g'',0}+\delta_{\theta(g)+g'+g'',0}
 +\delta_{g+\theta(g')+g'',0}+\delta_{g+g'+\theta(g''),0}\\
 &=\frac{4}{|\Gamma|-|G|}(1+\overline{\inpr{g+g'+g''}{u_0}})
 +\delta_{g+g'+g'',0}+\delta_{\theta(g)+g'+g'',0}+\delta_{g+\theta(g')+g'',0}+\delta_{g+g'+\theta(g''),0}.
\end{align*}
\begin{align*}
\lefteqn{N_{g,g',(\sigma,\varepsilon)}} \\
 &=(\frac{a^2b}{a-b}-\frac{a^2b}{a+b})\sum_{v\in U}\overline{\inpr{g+g'}{v}\inpr{\sigma}{v}}
 +2a\sum_{(l,\delta)\in J_1}\overline{\inpr{g+g'}{l}}\frac{\delta\varepsilon f(l,\sigma)}{4} \\
 &=\frac{2}{|\Gamma|-|G|}(1+\overline{\inpr{g+g'}{u_0}\inpr{\sigma}{u_0}}).
\end{align*}
\begin{align*}
\lefteqn{N_{g,g',\gamma}} \\
 &=(\frac{2a^2b}{a-b}-\frac{2a^2b}{a+b})\sum_{v\in U}\overline{\inpr{g+g'}{v}\inpr{\gamma}{v}}\\
 &=\frac{4a^2b^2}{a^2-b^2}\sum_{v\in U}\overline{\inpr{g+g'}{v}\inpr{\gamma}{v}}\\
 &=\frac{4}{|\Gamma|-|G|}(1+\overline{\inpr{g+g'}{u_0}\inpr{\gamma}{u_0}}).
\end{align*}
\begin{align*}
\lefteqn{N_{g,(\sigma,\varepsilon),(\sigma',\varepsilon')}} \\
 &=(\frac{ab^2}{2(a-b)}+\frac{ab^2}{2(a+b)})\sum_{v\in U}\overline{\inpr{g}{v}\inpr{\sigma+\sigma'}{v}}
 +2\sum_{(l,\delta)\in J_1}\overline{\inpr{g}{l}}\frac{\varepsilon\delta f(l,\sigma)}{4}\frac{\varepsilon' \delta f(l,\sigma')}{4} \\
 &=\frac{1}{|\Gamma|-|G|}\sum_{v\in U}\overline{\inpr{g}{v}}
 +\frac{\varepsilon\varepsilon'}{4}\sum_{l\in K_*}\overline{\inpr{g}{l}}f(l,\sigma)f(l,\sigma') \\
 &=\frac{1}{|\Gamma|-|G|}\sum_{v\in U}\overline{\inpr{g}{v}}
 +\frac{\varepsilon\varepsilon's^2}{4}\sum_{v\in U}\overline{\inpr{g}{k_1+v}\inpr{k_1+v}{\sigma'-\sigma}} \\
 &=(\frac{1}{|\Gamma|-|G|}+\frac{\varepsilon\varepsilon's^2\overline{\inpr{g}{k_1}\inpr{k_1}{\sigma'-\sigma}} }{4})
 (1+\overline{\inpr{g}{u_0}}).
\end{align*}
$$N_{g,(\sigma,\varepsilon),\gamma}=(\frac{ab^2}{a-b}+\frac{ab^2}{a+b})\sum_{v\in U}\overline{\inpr{g}{v}\inpr{\sigma+\gamma}{v}}
=\frac{2}{|\Gamma|-|G|}(1+\inpr{g}{u_0}\overline{\inpr{\sigma+\gamma}{u_0}}).$$
\begin{align*}
\lefteqn{N_{g,\gamma,\gamma'}} \\
 &=(\frac{2ab^2}{a-b}+\frac{2ab^2}{a+b})\sum_{v\in U}\overline{\inpr{g}{v}\inpr{\gamma+\gamma'}{v}} \\
 &=\frac{4a^2b^2}{a^2-b^2}\sum_{v\in U}\overline{\inpr{g}{v}\inpr{\gamma+\gamma'}{v}}\\
 &=\frac{4}{|\Gamma|-|G|}(1+\overline{\inpr{g}{u_0}\inpr{\gamma+\gamma'}{u_0}}).
\end{align*}
\begin{align*}
\lefteqn{N_{(\sigma,\varepsilon),(\sigma',\varepsilon'),(\sigma'',\varepsilon'')}} \\
 &=(\frac{b^3}{4(a-b)}-\frac{b^3}{4(a+b)})\sum_{v\in U}\overline{\inpr{\sigma+\sigma'+\sigma''}{v}}\\
 &+\frac{2}{a}\sum_{(l,\delta)\in J_1}\frac{\varepsilon\delta f(l,\sigma)}{4}\frac{\varepsilon'\delta f(l,\sigma')}{4}\frac{\varepsilon''\delta f(l,\sigma'')}{4} \\
 &-\frac{2}{b}\sum_{(\tau,\delta)\in J_2}(\frac{b}{2}-\frac{\varepsilon\delta s}{4})(\frac{b}{2}-\frac{\varepsilon'\delta s}{4})(\frac{b}{2}-\frac{\varepsilon''\delta s}{4})
 \overline{\inpr{\sigma+\sigma'+\sigma''}{\tau}} 
 -b^2\sum_{\gamma\in \Gamma_*} \overline{\inpr{\sigma+\sigma'+\sigma''}{\gamma}} \\
 &=(\frac{b^4}{2(a^2-b^2)}+\frac{b^2}{2})\sum_{v\in U}\overline{\inpr{\sigma+\sigma'+\sigma''}{v}}
 -\frac{b^2}{2}\sum_{\gamma\in \Gamma}\overline{\inpr{\sigma+\sigma'+\sigma''}{\gamma}}\\
 &-\frac{s^2(\varepsilon\varepsilon'+\varepsilon'\varepsilon''+\varepsilon''\varepsilon)}{8}\sum_{\tau\in \Sigma_*}\overline{\inpr{\sigma+\sigma'+\sigma''}{\tau}}\\
 &=(1+\overline{\inpr{\sigma+\sigma'+\sigma''}{u_0}})(\frac{1}{2(|\Gamma|-|G|)}- \frac{s^2(\varepsilon\varepsilon'+\varepsilon'\varepsilon''+\varepsilon''\varepsilon)}{8}\overline{\inpr{\sigma+\sigma'+\sigma''}{\sigma+\sigma'+\sigma''}}).
\end{align*}
\begin{align*}
\lefteqn{N_{(\sigma,\varepsilon),(\sigma',\varepsilon'),\gamma}} \\
 &=(\frac{b^3}{2(a-b)}-\frac{b^3}{2(a+b)})\sum_{v\in U}\overline{\inpr{\sigma+\sigma'+\gamma}{v}}
 -2\sum_{(\tau,\delta)\in J_2}(\frac{b}{2}-\frac{\varepsilon \delta s}{4})(\frac{b}{2}-\frac{\varepsilon'\delta s}{4})\overline{\inpr{\sigma+\sigma'+ \gamma}{\tau}} \\
 &-b^2\sum_{\xi\in \Gamma_*}\overline{\inpr{\sigma+\sigma'}{\xi}}(\overline{\inpr{\gamma}{\xi}}+\overline{\inpr{\gamma}{\theta(\xi)}}) \\
 &=(\frac{b^4}{a^2-b^2}+b^2)\sum_{v\in U}\overline{\inpr{\sigma+\sigma'+\gamma}{v}}-b^2\sum_{\xi\in \Gamma}\overline{\inpr{\sigma+\sigma'+\gamma}{\xi}}
 -\frac{\varepsilon\varepsilon's^2}{4}\sum_{\tau\in \Sigma_*}\overline{\inpr{\sigma+\sigma'+\gamma}{\tau}}\\
 &=(\frac{1}{|\Gamma|-|G|}-\frac{\varepsilon\varepsilon's^2\overline{\inpr{\sigma+\sigma'+\gamma}{\sigma}}}{4})(1+\overline{\inpr{\sigma+\sigma'+\gamma}{u_0}}).
\end{align*}
\begin{align*}
\lefteqn{N_{(\sigma,\varepsilon),\gamma,\gamma'}} \\
 &=(\frac{b^3}{a-b}-\frac{b^3}{a+b})\sum_{v\in U}\overline{\inpr{\sigma+\gamma+\gamma'}{v}}
 -2b\sum_{(\tau,\delta)\in J_2}(\frac{b}{2}-\frac{\varepsilon\delta s}{4})\overline{\inpr{\sigma+\gamma+\gamma'}{\tau}} \\
 &-b^2\sum_{\xi\in \Gamma_*}\overline{\inpr{\sigma}{\xi}}(\overline{\inpr{\gamma}{\xi}}+\overline{\inpr{\gamma}{\theta(\xi)}})(\overline{\inpr{\gamma'}{\xi}}+\overline{\inpr{\gamma'}{\theta(\xi)}}) \\
 &=(\frac{2b^4}{a^2-b^2}+2b^2)\sum_{v\in U}\overline{\inpr{\sigma+\gamma+\gamma'}{v}}-b^2\sum_{\xi\in \Gamma}(\overline{\inpr{\sigma+\gamma+\gamma'}{\xi}}+\overline{\inpr{\sigma+\gamma+\theta(\gamma')}{\xi}})\\
 &=\frac{2}{|\Gamma|-|G|}(1+\overline{\inpr{\sigma+\gamma+\gamma'}{u_0}})-\delta_{\sigma+\gamma+\gamma',0}-\delta_{\sigma+\gamma+\theta(\gamma'),0}.
\end{align*}
\begin{align*}
\lefteqn{N_{\gamma,\gamma',\gamma''}} \\
 &=(\frac{2b^3}{a-b}-\frac{2b^3}{a+b})\sum_{v\in U}\overline{\inpr{\gamma+\gamma'+\gamma''}{v}}-4b^2\sum_{\tau \in \Sigma_*}\overline{\inpr{\gamma+\gamma'+\gamma''}{\tau}} \\
 &-b^2\sum_{\xi\in \Gamma_*}\overline{(\inpr{\gamma}{\xi}+\inpr{\gamma}{\theta(\xi)})(\inpr{\gamma'}{\xi}+\inpr{\gamma'}{\theta(\xi)})(\inpr{\gamma}{\xi}+\inpr{\gamma''}{\theta(\xi)})} \\
 &=(\frac{4b^4}{a^2-b^2}+4b^2)\sum_{v\in U}\overline{\inpr{\gamma+\gamma'+\gamma''}{v}}\\
 &-b^2\sum_{\xi\in \Gamma}(\inpr{\gamma+\gamma'+\gamma''}{\xi}+\inpr{\theta(\gamma)+\gamma'+\gamma''}{\xi}
 +\inpr{\gamma+\theta(\gamma')+\gamma''}{\xi}+\inpr{\gamma+\gamma'+\theta(\gamma'')}{\xi})\\
 &=\frac{4a^2b^2}{a^2-b^2}\sum_{v\in U}\overline{\inpr{\gamma+\gamma'+\gamma''}{v}}
 -\delta_{\gamma+\gamma'+\gamma'',0}-\delta_{\theta(\gamma)+\gamma'+\gamma'',0}
 -\delta_{\gamma+\theta(\gamma')+\gamma'',0}-\delta_{\gamma+\gamma'+\theta(\gamma''),0}\\
 &=\frac{4}{|\Gamma|-|G|}(1+\overline{\inpr{\gamma+\gamma'+\gamma''}{u_0}})
 -\delta_{\gamma+\gamma'+\gamma'',0}-\delta_{\theta(\gamma)+\gamma'+\gamma'',0}
 -\delta_{\gamma+\theta(\gamma')+\gamma'',0}-\delta_{\gamma+\gamma'+\theta(\gamma''),0}.
\end{align*}
\end{proof}

\begin{proof}[Proof of Lemma 5.7]   
\begin{align*}
\lefteqn{\nu_m((u,0))} \\
 &=\sum_{u'\in U}N^{(u,0)}_{(u',0),(u-u',0)}\frac{(a-b)^2}{4}\frac{q_0(u')^m}{q_0(u-u')^m}
  +\sum_{u'\in U}N^{(u,0)}_{(u',\pi),(u-u',\pi)}\frac{(a+b)^2}{4}\frac{q_0(u')^m}{q_0(u-u')^m}\\
 &+\sum_{g\in G_*}N^{(u,0)}_{g,u-g}a^2\frac{q_1(g)^m}{q_1(u-g)^m}
 +\sum_{\gamma\in \Gamma_*}N^{(u,0)}_{\gamma,u-\gamma}b^2\frac{q_2(\gamma)^m}{q_2(u-\gamma)^m} \\
 &+\sum_{(k,\varepsilon)\in J_1}N^{(u,0)}_{(k,\varepsilon),(k-u,\varepsilon)}\frac{a^2}{4}\frac{q_1(k)^m}{q_1(k+u)^m}
 +\sum_{(\sigma,\varepsilon)\in J_2}N^{(u,0)}_{(\sigma,\varepsilon),(\sigma-u,\varepsilon)}\frac{b^2}{4}\frac{q_2(\sigma)^m}{q_1(\sigma+u)^m}\\
 &=\frac{a^2}{2}\sum_{g\in G}\frac{q_0(g)^m}{q_0(u-g)^m}
 +\frac{b^2}{2}\sum_{\gamma\in \Gamma}\frac{q_1(\gamma)^m}{q_2(u-\gamma)^m}. 
\end{align*}
Since $q_1(g)/q_1(u-g)=\inpr{u}{g}\overline{q_0(u)}$, we have 
$$\frac{a^2}{2}\sum_{g\in G}\frac{q_0(g)^m}{q_0(u-g)^m}=\frac{1}{2|G|}\sum_{g\in G}\inpr{mu}{g}q_0(u)^{-m}=\frac{\delta_{mu,0}q_0(u)^m}{2},$$
and $\nu_m((u,0))=\delta_{mu,0}q_0(u)^m$. 

Note that we have $\cG(q_1,q_2,v,m)=\cG(q_1,q_2,-v,m)$ and $\cG(q_1,q_2,v,-m)=\overline{\cG(q_1,q_2,v,m)}$. 
We set 
$$A(g,m)=2a\sum_{h\in G_*}\inpr{g}{h}q_1(h)^m,\quad B(\gamma,m)=2b\sum_{\xi\in \Gamma_*}\inpr{\gamma}{\xi}q_2(\xi)^m.$$
$$C(g,m)=a\sum_{k\in K_*}\inpr{g}{k}q_0(k)^m,\quad D(\gamma,m)=b\sum_{\sigma\in \Sigma_*}\inpr{\gamma}{\sigma}q_2(\sigma)^m.$$
Then 
$$\cG(q_1,q_2,v,m)=(a+b)(1+q_0(u_0)^m)+A(v,m)+B(v,m)+C(v,m)+D(v,m).$$

In the computation below, we often use the following facts: 
\begin{enumerate}
\item $\inpr{u}{v}=1$ for any $u,v\in U$, 
\item $\inpr{k}{u_0}\inpr{\sigma}{u_0}=-1$ for any $k\in K_*$ and $\sigma\in \Sigma_*$, 
\item  $\inpr{u}{k+l}=1$ for any $u\in U$ and $k,l\in K_*$, 
\item $\inpr{u}{\sigma+\tau}=1$ for any $\sigma,\tau\in \Sigma_*$.  
\end{enumerate}
For $\nu_m((u,\pi))$, we have 
\begin{align*}
\lefteqn{\nu_m((u,\pi))} \\
 &=\sum_{v\in U}N^{(u,\pi)}_{(v,0),(u-v,\pi)}\frac{a^2-b^2}{4}(\frac{q_0(v)^m}{q_0(u-v)^m}+\frac{q_0(u-v)^m}{q_0(v)^m}) 
 +\sum_{v,v'\in U}N^{(u,\pi)}_{(v,\pi),(v',\pi)}\frac{(a+b)^2}{4}\frac{q_0(v)^m}{q_0(v')^m} \\
 &+\sum_{v\in U,\;g\in G_*}N^{(u,\pi)}_{(v,\pi),g}\frac{a(a+b)}{2}(\frac{q_0(v)^m}{q_1(g)^m}+\frac{q_1(g)^m}{q_0(v)^m}) 
 +\sum_{v\in U,\;\xi\in \Gamma_*}N^{(u,\pi)}_{(v,\pi),\xi}\frac{b(a+b)}{2}(\frac{q_0(v)^m}{q_2(\xi)^m}+\frac{q_2(\xi)^m}{q_0(v)^m})\\
 &+\sum_{v\in U,\;(l,\varepsilon)\in J_1}N^{(u,\pi)}_{(v,\pi),(l,\varepsilon)}\frac{a(a+b)}{4}(\frac{q_0(v)^m}{q_1(l)^m}+\frac{q_1(l)^m}{q_0(v)^m})+\\ &\sum_{v\in U,\;(\tau,\varepsilon)\in J_2}N^{(u,\pi)}_{(v,\pi),(\tau,\varepsilon)}\frac{b(a+b)}{4}(\frac{q_0(v)^m}{q_2(\tau)^m}+\frac{q_2(\tau)^m}{q_0(v)^m})
 +\sum_{(l,\varepsilon),(l',\varepsilon')\in J_1}N^{(u,\pi)}_{(l,\varepsilon),(l',\varepsilon')}\frac{a^2}{4}\frac{q_1(l)^m}{q_1(l')^m} +\\& \sum_{(\tau,\varepsilon),(\tau',\varepsilon')\in J_2}N^{(u,\pi)}_{(\tau,\varepsilon),(\tau',\varepsilon')}\frac{b^2}{4}\frac{q_2(\tau)^m}{q_2(\tau')^m}+\sum_{(l,\varepsilon)\in J_1,\;(\tau,\varepsilon')\in J_2}N^{(u,\pi)}_{(l,\varepsilon),(\tau,\varepsilon')}
 \frac{ab}{4}(\frac{q_1(l)^m}{q_2(\tau)^m}+\frac{q_2(\tau)^m}{q_1(l)^m})\\
 &+\sum_{(l,\varepsilon)\in J_1,\;\xi\in \Gamma_*}N^{(u,\pi)}_{(l,\varepsilon),\xi}\frac{ab}{2}(\frac{q_1(l)^m}{q_2(\xi)^m}+\frac{q_2(\xi)^m}{q_1(l)^m})
 +\sum_{h\in G_*,\;(\tau,\varepsilon)\in J_2}N^{(u.\pi)}_{h,(\tau,\varepsilon)}\frac{ab}{2}(\frac{q_1(h)^m}{q_2(\tau)^m}+\frac{q_2(\tau)}{q_1(h)^m})\\
 &+\sum_{(l,\varepsilon)\in J_1,\;h\in G_*}N^{(u,\pi)}_{(l,\varepsilon),h}\frac{a^2}{2}(\frac{q_1(l)^m}{q_1(h)^m}+\frac{q_1(h)^m}{q_1(l)^m})
 +\sum_{(\tau,\varepsilon)\in J_2,\;\xi\in \Gamma_*}N^{(u,\pi)}_{(\tau,\varepsilon),\xi}\frac{b^2}{2}(\frac{q_2(\tau)^m}{q_2(\xi)^m}+\frac{q_2(\xi)^m}{q_2(\tau)^m})\\
 &+\sum_{h,h'\in G_*}N^{(u,\pi)}_{h,h'}a^2\frac{q_1(h)^m}{q_1(h')^m}
 +\sum_{h\in G_*,\;\xi\in \Gamma_*}N^{(u,\pi)}_{h,\xi}ab(\frac{q_1(h)^m}{q_2(\xi)^m}+\frac{q_2(\xi)^m}{q_1(h)^m})
 +\sum_{\xi,\xi'\in \Gamma_*}N^{(u,\pi)}_{\xi,\xi'}b^2\frac{q_2(\xi)^m}{q_2(\xi')^m}\\
 &=\frac{a^2-b^2}{4}\sum_{v\in U}(\frac{q_0(v)^m}{q_0(u-v)^m}+\frac{q_0(u-v)^m}{q_0(v)^m}) 
 +\frac{2(a+b)^2}{|\Gamma|-|G|}\sum_{v,v'\in U}\frac{q_0(v)^m}{q_0(v')^m} \\
 &+\frac{2a(a+b)}{|\Gamma|-|G|}\sum_{v\in U,\;h\in G_*}(1+\inpr{u_0}{h})(\frac{q_0(v)^m}{q_1(h)^m}+\frac{q_1(h)^m}{q_0(v)^m}) 
 \\&+\frac{2b(a+b)}{|\Gamma|-|G|}\sum_{v\in U,\;\xi\in \Gamma_*}(1+\inpr{u_0}{\xi})(\frac{q_0(v)^m}{q_2(\xi)^m}+\frac{q_2(\xi)^m}{q_0(v)^m})\\
 &+\frac{a(a+b)}{|\Gamma|-|G|}\sum_{v\in U,\;l\in K_*}(1+\inpr{u_0}{l})(\frac{q_0(v)^m}{q_1(l)^m}+\frac{q_1(l)^m}{q_0(v)^m})
 \\&+\frac{b(a+b)}{|\Gamma|-|G|}\sum_{v\in U,\;\tau\in \Sigma_*}(1+\inpr{u_0}{\tau})(\frac{q_0(v)^m}{q_2(\tau)^m}+\frac{q_2(\tau)^m}{q_0(v)^m})\\
 &+a^2\sum_{l,l'\in K_*}(\frac{2}{|\Gamma|-|G|}+\frac{\delta_{u+l+l',0}}{2})\frac{q_1(l)^m}{q_1(l')^m}
 +b^2\sum_{\tau,\tau'\in \Sigma_*}(\frac{2}{|\Gamma|-|G|}-\frac{\delta_{u+\tau+\tau',0}}{2})\frac{q_2(\tau)^m}{q_2(\tau')^m}\\
 &+\frac{2ab}{|\Gamma|-|G|}\sum_{l\in K_*,\;\xi\in \Gamma_*}(1+\inpr{u_0}{l}\inpr{u_0}{\xi})(\frac{q_1(l)^m}{q_2(\xi)^m}+\frac{q_2(\xi)^m}{q_1(l)^m})\\
 &+\frac{2ab}{|\Gamma|-|G|}\sum_{h\in G_*,\;\tau\in \Sigma_*}(1+\inpr{u_0}{h}\inpr{u_0}{\tau})(\frac{q_1(h)^m}{q_2(\tau)^m}+\frac{q_2(\tau)}{q_1(h)^m})\\
 &+\frac{2a^2}{|\Gamma|-|G|}\sum_{l\in K_*,\;h\in G_*}(1+\inpr{u_0}{l+h})(\frac{q_1(l)^m}{q_1(h)^m}+\frac{q_1(h)^m}{q_1(l)^m})\\
 &+\frac{2b^2}{|\Gamma|-|G|}\sum_{\tau\in \Sigma_*,\;\xi\in \Gamma_*}(1+\inpr{u_0}{\tau+\xi})(\frac{q_2(\tau)^m}{q_2(\xi)^m}+\frac{q_2(\xi)^m}{q_2(\tau)^m})\\
 &+a^2\sum_{h,h'\in G_*}(\frac{4(1+\inpr{u_0}{h+h'})}{|\Gamma|-|G|}+\delta_{u+h+h',0}+\delta_{u+h+\theta(h'),0})\frac{q_1(h)^m}{q_1(h')^m}\\
 &+\frac{4ab}{|\Gamma|-|G|}\sum_{h\in G_*,\;\xi\in \Gamma_*}(1+\inpr{u_0}{h}\inpr{u_0}{\xi})(\frac{q_1(h)^m}{q_2(\xi)^m}+\frac{q_2(\xi)^m}{q_1(h)^m})\\
 &+b^2\sum_{\xi,\xi'\in \Gamma_*}(\frac{4(1+\inpr{u_0}{\xi+\xi'})}{|\Gamma|-|G|}-\delta_{u+\xi+\xi',0}-\delta_{u+\xi+\theta(\xi'),0})\frac{q_2(\xi)^m}{q_2(\xi')^m},
\end{align*}
and 
\begin{align*}
\lefteqn{\nu_m((u,\pi))} \\
 &=\frac{a^2-b^2}{4}\sum_{v\in U}(\frac{q_0(v)^m}{q_0(u-v)^m}+\frac{q_0(u-v)^m}{q_0(v)^m})
 +\frac{a^2}{2}\sum_{l \in K_*}\frac{q_1(l)^m}{q_1(u+l)^m}
 -\frac{b^2}{2}\sum_{\tau \in \Sigma_*}\frac{q_2(\tau)^m}{q_2(u+\tau)^m}\\
 &+a^2\sum_{h\in G_*}\frac{q_1(h)^m}{q_1(u+h)^m}-b^2\sum_{\xi\in \Gamma_*}\frac{q_2(\xi)^m}{q_2(u+\xi)^m}\\ 
 &+\frac{a+b}{|\Gamma|-|G|}\sum_{w\in U}\sum_{v\in U}(q_0(v)^m(A(w,-m)+B(w,-m))+(A(w,m)+B(w,m))q_0(v)^{-m})\\
 &+\frac{a+b}{|\Gamma|-|G|}\sum_{w\in U}\sum_{v\in U}(q_0(v)^m(C(w,-m)+D(w,-m))+(C(w,m)+D(w,m))q_0(v)^{-m})\\
 &+\frac{1}{|\Gamma|-|G|}\sum_{w\in U}(C(w,m)C(w,-m)+D(w,m)D(w,-m))\\
 &+\frac{1}{|\Gamma|-|G|}\sum_{w\in U}(A(w,m)C(w,-m)+C(w,m)A(w,-m)+A(w,m)D(w,-m)+D(w,m)A(w,-m))\\
 &+\frac{1}{|\Gamma|-|G|}\sum_{w\in U}(B(w,m)C(w,-m)+C(w,m)B(w,-m)+B(w,m)D(w,-m)+D(w,m)B(w,-m))\\
 &+\frac{1}{|\Gamma|-|G|}\sum_{w\in U}(A(w,m)A(w,-m)+A(w,m)B(w,-m)+B(w,m)A(w,-m)+B(w,m)B(w,-m))\\
 &=\frac{a^2}{2}\sum_{h\in G}\frac{q_1(h)^m}{q_1(u+h)^m}-\frac{b^2}{2}\sum_{\xi\in \Gamma}\frac{q_2(\xi)^m}{q_2(u+\xi)^m}
 +\frac{1}{|\Gamma|-|G|}\sum_{w\in U}|\cG(q_1,q_2,w,m)|^2\\
 &=\frac{1}{|\Gamma|-|G|}\sum_{w\in U}|\cG(q_1,q_2,w,m)|^2.
\end{align*}

\begin{align*}
\lefteqn{\nu_m((k,\varepsilon))} \\
 &=\sum_{v\in U}N^{(k,\varepsilon)}_{(v,0),(k-v,\varepsilon)}\frac{a(a-b)}{4}(\frac{q_0(v)^m}{q_0(k-v)^m}+\frac{q_0(k-v)^m}{q_0(v)^m}) 
 +\sum_{v,v'\in U}N^{(k,\varepsilon)}_{(v,\pi),(v',\pi)}\frac{(a+b)^2}{4}\frac{q_0(v)^m}{q_0(v')^m} \\
 &+\sum_{v\in U,\;g\in G_*}N^{(k,\varepsilon)}_{(v,\pi),g}\frac{a(a+b)}{2}(\frac{q_0(v)^m}{q_1(g)^m}+\frac{q_1(g)^m}{q_0(v)^m}) 
 +\sum_{v\in U,\;\xi\in \Gamma_*}N^{(k,\varepsilon)}_{(v,\pi),\xi}\frac{b(a+b)}{2}(\frac{q_0(v)^m}{q_2(\xi)^m}+\frac{q_2(\xi)^m}{q_0(v)^m})\\
 &+\sum_{v\in U,\;(l,\delta)\in J_1}N^{(k,\varepsilon)}_{(v,\pi),(l,\delta)}\frac{a(a+b)}{4}(\frac{q_0(v)^m}{q_1(l)^m}+\frac{q_1(l)^m}{q_0(v)^m})
 \\& +\sum_{v\in U,\;(\tau,\delta)\in J_2}N^{(k,\varepsilon)}_{(v,\pi),(\tau,\delta)}\frac{b(a+b)}{4}(\frac{q_0(v)^m}{q_2(\tau)^m}+\frac{q_2(\tau)^m}{q_0(v)^m})+\sum_{(l,\delta),(l',\delta')\in J_1}N^{(k,\varepsilon)}_{(l,\delta),(l',\delta')}\frac{a^2}{4}\frac{q_1(l)^m}{q_1(l')^m}
 \\&+\sum_{(\tau,\delta),(\tau',\delta')\in J_2}N^{(k,\varepsilon)}_{(\tau,\delta),(\tau',\delta')}\frac{b^2}{4}\frac{q_2(\tau)^m}{q_2(\tau')^m}+\sum_{(l,\delta)\in J_1,\;(\tau,\delta')\in J_2}N^{(k,\varepsilon)}_{(l,\delta),(\tau,\delta')}
 \frac{ab}{4}(\frac{q_1(l)^m}{q_2(\tau)^m}+\frac{q_2(\tau)^m}{q_1(l)^m})\\
 &+\sum_{(l,\delta)\in J_1,\;\xi\in \Gamma_*}N^{(k,\varepsilon)}_{(l,\delta),\xi}\frac{ab}{2}(\frac{q_1(l)^m}{q_2(\xi)^m}+\frac{q_2(\xi)^m}{q_1(l)^m})
 +\sum_{h\in G_*,\;(\tau,\delta)\in J_2}N^{(u,\pi)}_{h,(\tau,\delta)}\frac{ab}{2}(\frac{q_1(h)^m}{q_2(\tau)^m}+\frac{q_2(\tau)^m}{q_1(h)^m})\\
 &+\sum_{(l,\delta)\in J_1,\;h\in G_*}N^{(k,\varepsilon)}_{(l,\delta),h}\frac{a^2}{2}(\frac{q_1(l)^m}{q_1(h)^m}+\frac{q_1(h)^m}{q_1(l)^m})
 +\sum_{(\tau,\delta)\in J_2,\;\xi\in \Gamma_*}N^{(k,\varepsilon)}_{(\tau,\delta),\xi}\frac{b^2}{2}(\frac{q_2(\tau)^m}{q_2(\xi)^m}+\frac{q_2(\xi)^m}{q_2(\tau)^m})\\
 &+\sum_{h,h'\in G_*}N^{(k,\varepsilon)}_{h,h'}a^2\frac{q_1(h)^m}{q_1(h')^m}
 +\sum_{h\in G_*,\;\xi\in \Gamma_*}N^{(k,\varepsilon)}_{h,\xi}ab(\frac{q_1(h)^m}{q_2(\xi)^m}+\frac{q_2(\xi)^m}{q_1(h)^m})
 +\sum_{\xi,\xi'\in \Gamma_*}N^{(k,\varepsilon)}_{\xi,\xi'}b^2\frac{q_2(\xi)^m}{q_2(\xi')^m}\\
  &=\frac{a(a-b)}{4}\sum_{v\in U}(\frac{q_0(v)^m}{q_1(k-v)^m}+\frac{q_1(k-v)^m}{q_0(v)^m}) 
 +\frac{(a+b)^2}{2(|\Gamma|-|G|)}\sum_{v,v'\in U}(1+\inpr{k}{u_0})\frac{q_0(v)^m}{q_0(v')^m} \\
 &+\frac{a(a+b)}{|\Gamma|-|G|}\sum_{v\in U,\;g\in G_*}(1+\inpr{k+g}{u_0})(\frac{q_0(v)^m}{q_1(g)^m}+\frac{q_1(g)^m}{q_0(v)^m}) 
 \\&+\frac{b(a+b)}{|\Gamma|-|G|}\sum_{v\in U,\;\xi\in \Gamma_*}(1+\inpr{k+\xi}{u_0})(\frac{q_0(v)^m}{q_2(\xi)^m}+\frac{q_2(\xi)^m}{q_0(v)^m})\\
 &+\frac{a(a+b)}{4}\sum_{v\in U,\;l\in K_*}(\frac{2(1+\inpr{u_0}{k+l})}{|\Gamma|-|G|}+\delta_{v-k+l,0})(\frac{q_0(v)^m}{q_1(l)^m}+\frac{q_1(l)^m}{q_0(v)^m})\\
 &+\frac{a^2}{2(|\Gamma|-|G|)}\sum_{l,l'\in K_*}(1+\inpr{k+l+l'}{u_0})  \frac{q_1(l)^m}{q_1(l')^m}
 \\&+\frac{b^2}{2(|\Gamma|-|G|)}\sum_{\tau,\tau'\in \Sigma_*}(1+\inpr{k}{u_0}\inpr{\tau+\tau'}{u_0})\frac{q_2(\tau)^m}{q_2(\tau')^m}\\
 &+\frac{ab}{2(|\Gamma|-|G|)}\sum_{l\in K_*,\tau\in \Sigma_*}(1+\inpr{k+l}{u_0}\inpr{\tau}{u_0}) (\frac{q_1(l)^m}{q_2(\tau)^m}+\frac{q_2(\tau)^m}{q_1(l)^m})\\
 &+\frac{ab}{|\Gamma|-|G|}\sum_{l\in K_*,\;\xi\in \Gamma_*}(1+\inpr{k+l}{u_0}\inpr{\xi}{u_0})(\frac{q_1(l)^m}{q_2(\xi)^m}+\frac{q_2(\xi)^m}{q_1(l)^m})\\
 &+\frac{ab}{|\Gamma|-|G|}\sum_{h\in G_*,\;\tau\in \Sigma_*}(1+\inpr{k+h}{u_0}\inpr{\tau}{u_0})(\frac{q_1(h)^m}{q_2(\tau)^m}+\frac{q_2(\tau)^m}{q_1(h)^m})\\
 &+\frac{a^2}{|\Gamma|-|G|}\sum_{l\in K_*,\;h\in G_*}(1+\inpr{k+l+h}{u_0})(\frac{q_1(l)^m}{q_1(h)^m}+\frac{q_1(h)^m}{q_1(l)^m})\\
 &+\frac{b^2}{|\Gamma|-|G|}\sum_{\tau\in \Sigma_*,\;\xi\in \Gamma_*}(1+\inpr{k}{u_0}\inpr{\tau+\xi}{u_0})(\frac{q_2(\tau)^m}{q_2(\xi)^m}+\frac{q_2(\xi)^m}{q_2(\tau)^m})\\
 &+a^2\sum_{h,h'\in G_*}(\frac{2(1+\inpr{k+h+h'}{u_0})}{|\Gamma|-|G|}+\delta_{-k+h+h',0}+\delta_{-k+h+\theta(h'),0})\frac{q_1(h)^m}{q_1(h')^m}\\
 &+\frac{2ab}{|\Gamma|-|G|}\sum_{h\in G_*,\;\xi\in \Gamma_*}(1+\inpr{k+h}{u_0}\inpr{\xi}{u_0})(\frac{q_1(h)^m}{q_2(\xi)^m}+\frac{q_2(\xi)^m}{q_1(h)^m})\\
 &+\frac{2b^2}{|\Gamma|-|G|}\sum_{\xi,\xi'\in \Gamma_*}(1+\inpr{k}{u_0}\inpr{\xi+\xi'}{u_0})\frac{q_2(\xi)^m}{q_2(\xi')^m},
 \end{align*}
 and
 \begin{align*}
\lefteqn{\nu_m((k,\varepsilon))} \\
 &=\frac{a^2}{2}\sum_{v\in U}(\frac{q_0(v)^m}{q_1(k-v)^m}+\frac{q_1(k-v)^m}{q_0(v)^m})
 +a^2\sum_{h\in G_*}\frac{q_1(h)^m}{q_1(k+h)^m} \\
 &+\frac{(a+b)^2}{2(|\Gamma|-|G|)}\sum_{v,v'\in U}(1+\inpr{k}{u_0})\frac{q_0(v)^m}{q_0(v')^m} \\
 &+\frac{1}{2(|\Gamma|-|G|)}\sum_{w\in U}\inpr{k}{w}((A(w,m)+B(w,m)+C(w,m)+D(w,m))q_0(w)^{-m})\\
 &+\frac{1}{2(|\Gamma|-|G|)}\sum_{w\in U}\inpr{k}{w}((A(w,-m)+B(w,-m)+C(w,-m)+D(w,-m))q_0(w)^m)\\
 &+\frac{1}{2(|\Gamma|-|G|)}\sum_{w\in U}\inpr{k}{w}\\&(A(w,m)C(w,-m)+C(w,m)A(w,-m)+A(w,b)D(w,-m)+D(w,m)A(w,-m))\\
 &+\frac{1}{2(|\Gamma|-|G|)}\sum_{w\in U}\inpr{k}{w}\\&(B(w,m)C(w,-m)+C(w,m)B(w,-m)+B(w,b)D(w,-m)+D(w,m)B(w,-m))\\
 &+\frac{1}{2(|\Gamma|-|G|)}\sum_{w\in U}\inpr{k}{w}\\&(A(w,m)A(w,-m)+A(w,m)B(w,-m)+B(w,b)B(w,-m)+B(w,m)B(w,-m))\\
 &=\frac{a^2}{2}\sum_{v\in U}(\frac{q_0(0)^m}{q_1(k)^m}+\frac{q_1(k)^m}{q_0(0)^m}+\frac{q_0(u_0)^m}{q_1(k-u_0)^m}+\frac{q_1(k-u_0)^m}{q_0(u_0)^m})
 +\frac{a^2}{2}\sum_{h\in G\setminus G_2}\frac{q_1(h)^m}{q_1(k+h)^m}\\
 &+\frac{1}{2(|\Gamma|-|G|)}\sum_{w\in U}\inpr{k}{w}|\cG(q_1,q_2,w,m)|^2\\
 &=\frac{a^2}{2}\sum_{h\in G}\frac{q_1(h)^m}{q_1(k+h)^m}\\
 &+\frac{1}{2(|\Gamma|-|G|)}\sum_{w\in U}\inpr{k}{w}|\cG(q_1,q_2,w,m)|^2\\
 &=\frac{1}{2(|\Gamma|-|G|)}\sum_{w\in U}\inpr{k}{w}|\cG(q_1,q_2,w,m)|^2+\frac{\delta_{mk,0}q_1(k)^m}{2}.\\
 \end{align*}

\begin{align*}
\lefteqn{\nu_m(g)}\\
 &=\sum_{v\in U}N^g_{(v,0),g-v}\frac{a(a-b)}{2}(\frac{q_0(v)^m}{q_1(g-v)^m}+\frac{q_1(g-v)^m}{q_0(v)^m}) 
 +\sum_{v,v'\in U}N^g_{(v,\pi),(v',\pi)}\frac{(a+b)^2}{4}\frac{q_0(v)^m}{q_0(v')^m} \\
 &+\sum_{v\in U,\;h\in G_*}N^g_{(v,\pi),h}\frac{a(a+b)}{2}(\frac{q_0(v)^m}{q_1(h)^m}+\frac{q_1(h)^m}{q_0(v)^m}) 
 +\sum_{v\in U,\;\xi\in \Gamma_*}N^g_{(v,\pi),\xi}\frac{b(a+b)}{2}(\frac{q_0(v)^m}{q_2(\xi)^m}+\frac{q_2(\xi)^m}{q_0(v)^m})\\
 &+\sum_{v\in U,\;(l,\delta)\in J_1}N^g_{(v,\pi),(l,\delta)}\frac{a(a+b)}{4}(\frac{q_0(v)^m}{q_1(l)^m}+\frac{q_1(l)^m}{q_0(v)^m})
 \\&+\sum_{v\in U,\;(\tau,\delta)\in J_2}N^g_{(v,\pi),(\tau,\delta)}\frac{b(a+b)}{4}(\frac{q_0(v)^m}{q_2(\tau)^m}+\frac{q_2(\tau)^m}{q_0(v)^m})+\sum_{(l,\delta),(l',\delta')\in J_1}N^g_{(l,\delta),(l',\delta')}\frac{a^2}{4}\frac{q_1(l)^m}{q_1(l')^m}
 \\&+\sum_{(\tau,\delta),(\tau',\delta')\in J_2}N^g_{(\tau,\delta),(\tau',\delta')}\frac{b^2}{4}\frac{q_2(\tau)^m}{q_2(\tau')^m}+\sum_{(l,\delta)\in J_1,\;(\tau,\delta)\in J_2}N^g_{(l,\delta),(\tau,\delta)}
 \frac{ab}{4}(\frac{q_1(l)^m}{q_2(\tau)^m}+\frac{q_2(\tau)^m}{q_1(l)^m})\\
 &+\sum_{(l,\delta)\in J_1,\;\xi\in \Gamma_*}N^g_{(l,\delta),\xi}\frac{ab}{2}(\frac{q_1(l)^m}{q_2(\xi)^m}+\frac{q_2(\xi)^m}{q_1(l)^m})
 +\sum_{h\in G_*,\;(\tau,\delta)\in J_2}N^g_{h,(\tau,\delta)}\frac{ab}{2}(\frac{q_1(h)^m}{q_2(\tau)^m}+\frac{q_2(\tau)^m}{q_1(h)^m})\\
 &+\sum_{(l,\delta)\in J_1,\;h\in G_*}N^g_{(l,\delta),h}\frac{a^2}{2}(\frac{q_1(l)^m}{q_1(h)^m}+\frac{q_1(h)^m}{q_1(l)^m})
 +\sum_{(\tau,\delta)\in J_2,\;\xi\in \Gamma_*}N^g_{(\tau,\delta),\xi}\frac{b^2}{2}(\frac{q_2(\tau)^m}{q_2(\xi)^m}+\frac{q_2(\xi)^m}{q_2(\tau)^m})\\
 &+\sum_{h,h'\in G_*}N^g_{h,h'}a^2\frac{q_1(h)^m}{q_1(h')^m}
 +\sum_{h\in G_*,\;\xi\in \Gamma_*}N^g_{h,\xi}ab(\frac{q_1(h)^m}{q_2(\xi)^m}+\frac{q_2(\xi)^m}{q_1(h)^m})
 +\sum_{\xi,\xi'\in \Gamma_*}N^g_{\xi,\xi'}b^2\frac{q_2(\xi)^m}{q_2(\xi')^m}\\
 \end{align*}
 \begin{align*}
 &=\frac{a(a-b)}{2}\sum_{v\in U}(\frac{q_0(v)^m}{q_1(g-v)^m}+\frac{q_1(g-v)^m}{q_0(v)^m}) 
 +\frac{(a+b)^2}{|\Gamma|-|G|}\sum_{v,v'\in U}(1+\inpr{u_0}{g}) \frac{q_0(v)^m}{q_0(v')^m} \\
 &+\frac{a(a+b)}{2}\sum_{v\in U,\;h\in G_*}(\frac{4(1+\inpr{u_0}{g+h})}{|\Gamma|-|G|}+\delta_{v-g+h,0}+\delta_{v-g+\theta(h),0})(\frac{q_0(v)^m}{q_1(h)^m}+\frac{q_1(h)^m}{q_0(v)^m})\\ 
 &+\frac{2b(a+b)}{|\Gamma|-|G|}\sum_{v\in U,\;\xi\in \Gamma_*}(1+\inpr{u_0}{g}\inpr{u_0}{\xi})(\frac{q_0(v)^m}{q_2(\xi)^m}+\frac{q_2(\xi)^m}{q_0(v)^m})\\
 &+\frac{a(a+b)}{|\Gamma|-|G|}\sum_{v\in U,\;l\in K_*}(1+\inpr{u_0}{l+g})(\frac{q_0(v)^m}{q_1(l)^m}+\frac{q_1(l)^m}{q_0(v)^m})\\
 &+\frac{b(a+b)}{|\Gamma|-|G|}\sum_{v\in U,\;\tau\in \Sigma_*}(1+\inpr{u_0}{g}\inpr{u_0}{\tau}) (\frac{q_0(v)^m}{q_2(\tau)^m}+\frac{q_2(\tau)^m}{q_0(v)^m})
\\& +\frac{a^2}{|\Gamma|-|G|}\sum_{l,l'\in K_*}(1+\inpr{u_0}{g})\frac{q_1(l)^m}{q_1(l')^m}\\
 &+\frac{b^2}{|\Gamma|-|G|}\sum_{\tau,\tau'\in \Sigma_*}(1+\inpr{u_0}{g})\frac{q_2(\tau)^m}{q_2(\tau')^m}
\\& + \frac{ab}{|\Gamma|-|G|}\sum_{l\in K_*,\;\tau\in \Sigma_*}(1+\inpr{u_0}{l+g}\inpr{u_0}{\tau})
 (\frac{q_1(l)^m}{q_2(\tau)^m}+\frac{q_2(\tau)^m}{q_1(l)^m})\\
 &+\frac{2ab}{|\Gamma|-|G|}\sum_{l\in K_*,\;\xi\in \Gamma_*}(1+\inpr{u_0}{l+g}\inpr{u_0}{\xi})(\frac{q_1(l)^m}{q_2(\xi)^m}+\frac{q_2(\xi)^m}{q_1(l)^m})\\
 &+\frac{2ab}{|\Gamma|-|G|}\sum_{h\in G_*,\;\tau\in \Sigma}(1+\inpr{u_0}{g+h}\inpr{u_0}{\tau})(\frac{q_1(h)^m}{q_2(\tau)^m}+\frac{q_2(\tau)^m}{q_1(h)^m})\\
 &+a^2\sum_{l\in K_*,\;h\in G_*}(\frac{2(1+\inpr{u_0}{k+g+h})}{|\Gamma|-|G|}+\delta_{l-g+h,0}+\delta_{l-g+\theta(h),0})
 (\frac{q_1(l)^m}{q_1(h)^m}+\frac{q_1(h)^m}{q_1(l)^m})\\
 &+\frac{2b^2}{|\Gamma|-|G|}\sum_{\tau\in \Sigma_*,\;\xi\in \Gamma_*}(1+\inpr{u_0}{g}\inpr{u_0}{\tau+\xi})(\frac{q_2(\tau)^m}{q_2(\xi)^m}+\frac{q_2(\xi)^m}{q_2(\tau)^m})\\
 &+a^2\sum_{h,h'\in G_*}(\frac{4(1+\inpr{u_0}{g+h+h'})}{|\Gamma|-|G|})+\delta_{-g+h+h',0}+\delta_{-\theta(g)+h+h',0}+\delta_{-g+\theta(h)+h',0}+\delta_{-g+h+\theta(h')})\frac{q_1(h)^m}{q_1(h')^m}\\
 &+\frac{4ab}{|\Gamma|-|G|}\sum_{h\in G_*,\;\xi\in \Gamma_*}(1+\inpr{u_0}{g+h}\inpr{u_0}{\xi})(\frac{q_1(h)^m}{q_2(\xi)^m}+\frac{q_2(\xi)^m}{q_1(h)^m})\\
 &+\frac{4b^2}{|\Gamma|-|G|}\sum_{\xi,\xi'\in \Gamma_*}(1+\inpr{u_0}{g}\inpr{u_0}{\xi+\xi'})\frac{q_2(\xi)^m}{q_2(\xi')^m},\\
 \end{align*} 
 and 
 \begin{align*}
\lefteqn{\nu_m(g)}\\
 &=a^2\sum_{v\in U}(\frac{q_0(v)^m}{q_1(g-v)^m}+\frac{q_1(g-v)^m}{q_0(v)^m})+a^2\sum_{l\in K_*}(\frac{q_1(l)^m}{q_1(g-l)^m}+\frac{q_1(g-l)^m}{q_1(l)^m})\\
 &+a^2\sum_{h,h'\in G_*}(\delta_{-g+h+h',0}+\delta_{-\theta(g)+h+h',0}+\delta_{-g+\theta(h)+h',0}+\delta_{-g+h+\theta(h'),0})\frac{q_1(h)^m}{q_1(h')^m}\\
 &+\frac{(a+b)^2}{|\Gamma|-|G|}\sum_{v,v'\in U}(1+\inpr{u_0}{g}) \frac{q_0(v)^m}{q_0(v')^m} \\
 &+\frac{2a(a+b)}{|\Gamma|-|G|}\sum_{v\in U,\;h\in G_*}(1+\inpr{u_0}{g+h})(\frac{q_0(v)^m}{q_1(h)^m}+\frac{q_1(h)^m}{q_0(v)^m})\\ 
 &+\frac{2b(a+b)}{|\Gamma|-|G|}\sum_{v\in U,\;\xi\in \Gamma_*}(1+\inpr{u_0}{g}\inpr{u_0}{\xi})(\frac{q_0(v)^m}{q_2(\xi)^m}+\frac{q_2(\xi)^m}{q_0(v)^m})\\
 &+\frac{a(a+b)}{|\Gamma|-|G|}\sum_{v\in U,\;l\in K_*}(1+\inpr{u_0}{l+g})(\frac{q_0(v)^m}{q_1(l)^m}+\frac{q_1(l)^m}{q_0(v)^m})\\
 &+\frac{b(a+b)}{|\Gamma|-|G|}\sum_{v\in U,\;\tau\in \Sigma_*}(1+\inpr{u_0}{g}\inpr{u_0}{\tau}) (\frac{q_0(v)^m}{q_2(\tau)^m}+\frac{q_2(\tau)^m}{q_0(v)^m})
 \\&+\frac{a^2}{|\Gamma|-|G|}\sum_{l,l'\in K_*}(1+\inpr{u_0}{g})\frac{q_1(l)^m}{q_1(l')^m}
 \\&+\frac{b^2}{|\Gamma|-|G|}\sum_{\tau,\tau'\in \Sigma_*}(1+\inpr{u_0}{g})\frac{q_2(\tau)^m}{q_2(\tau')^m}
 \\&+ \frac{ab}{|\Gamma|-|G|}\sum_{l\in K_*,\;\tau\in \Sigma_*}(1+\inpr{u_0}{l+g}\inpr{u_0}{\tau})
 (\frac{q_1(l)^m}{q_2(\tau)^m}+\frac{q_2(\tau)^m}{q_1(l)^m})\\
 &+\frac{2ab}{|\Gamma|-|G|}\sum_{l\in K_*,\;\xi\in \Gamma_*}(1+\inpr{u_0}{l+g}\inpr{u_0}{\xi})(\frac{q_1(l)^m}{q_2(\xi)^m}+\frac{q_2(\xi)^m}{q_1(l)^m})\\
 &+\frac{2ab}{|\Gamma|-|G|}\sum_{h\in G_*,\;\tau\in \Sigma}(1+\inpr{u_0}{g+h}\inpr{u_0}{\tau})(\frac{q_1(h)^m}{q_2(\tau)^m}+\frac{q_2(\tau)^m}{q_1(h)^m})\\
 &+\frac{2a^2}{|\Gamma|-|G|}\sum_{l\in K_*,\;h\in G_*}(1+\inpr{u_0}{k+g+h})(\frac{q_1(l)^m}{q_1(h)^m}+\frac{q_1(h)^m}{q_1(l)^m})\\
 &+\frac{2b^2}{|\Gamma|-|G|}\sum_{\tau\in \Sigma_*,\;\xi\in \Gamma_*}(1+\inpr{u_0}{g}\inpr{u_0}{\tau+\xi})(\frac{q_2(\tau)^m}{q_2(\xi)^m}+\frac{q_2(\xi)^m}{q_2(\tau)^m})\\
 &+\frac{4a^2}{|\Gamma|-|G|}\sum_{h,h'\in G_*}(1+\inpr{u_0}{g+h+h'})\frac{q_1(h)^m}{q_1(h')^m}\\
 &+\frac{4ab}{|\Gamma|-|G|}\sum_{h\in G_*,\;\xi\in \Gamma_*}(1+\inpr{u_0}{g+h}\inpr{u_0}{\xi})(\frac{q_1(h)^m}{q_2(\xi)^m}+\frac{q_2(\xi)^m}{q_1(h)^m})\\
 &+\frac{4b^2}{|\Gamma|-|G|}\sum_{\xi,\xi'\in \Gamma_*}(1+\inpr{u_0}{g}\inpr{u_0}{\xi+\xi'})\frac{q_2(\xi)^m}{q_2(\xi')^m}
 \end{align*}
 \begin{align*}
 &=a^2\sum_{l\in K}(\frac{q_1(l)^m}{q_1(g-l)^m}+\frac{q_1(g-l)^m}{q_1(l)^m})+\frac{1}{|\Gamma|-|G|}\sum_{w\in U}\inpr{g}{w}|\cG(q_1,q_2,w,m)|^2\\
 &+a^2\sum_{h,h'\in G_*}(\delta_{-g+h+h',0}+\delta_{-\theta(g)+h+h',0}+\delta_{-g+\theta(h)+h',0}+\delta_{-g+h+\theta(h'),0})\frac{q_1(h)^m}{q_1(h')^m}\\
 &=\delta_{mg,0}q_1(g)^m+\frac{1}{|\Gamma|-|G|}\sum_{w\in U}\inpr{g}{w}|\cG(q_1,q_2,w,m)|^2.
 \end{align*} 
 
\begin{align*}
\lefteqn{\nu_m((\sigma,\varepsilon))}\\
 &=\sum_{v\in U}N^{(\sigma,\varepsilon)}_{(v,0),(\sigma-v,\varepsilon)}\frac{b(a-b)}{4}(\frac{q_0(v)^m}{q_2(\sigma-v)^m}+\frac{q_2(\sigma-v)^m}{q_0(v)^m}) 
 +\sum_{v,v'\in U}N^{(\sigma,\varepsilon)}_{(v,\pi),(v',\pi)}\frac{(a+b)^2}{4}\frac{q_0(v)^m}{q_0(v')^m} \\
 &+\sum_{v\in U,\;h\in G_*}N^{(\sigma,\varepsilon)}_{(v,\pi),h}\frac{a(a+b)}{2}(\frac{q_0(v)^m}{q_1(h)^m}+\frac{q_1(h)^m}{q_0(v)^m}) 
 +\sum_{v\in U,\;\xi\in \Gamma_*}N^{(\sigma,\varepsilon)}_{(v,\pi),\xi}\frac{b(a+b)}{2}(\frac{q_0(v)^m}{q_2(\xi)^m}+\frac{q_2(\xi)^m}{q_0(v)^m})\\
 &+\sum_{v\in U,\;(l,\delta)\in J_1}N^{(\sigma,\varepsilon)}_{(v,\pi),(l,\delta)}\frac{a(a+b)}{4}(\frac{q_0(v)^m}{q_1(l)^m}+\frac{q_1(l)^m}{q_0(v)^m})
 \\&+\sum_{v\in U,\;(\tau,\delta)\in J_2}N^{(\sigma,\varepsilon)}_{(v,\pi),(\tau,\delta)}\frac{b(a+b)}{4}(\frac{q_0(v)^m}{q_2(\tau)^m}+\frac{q_2(\tau)^m}{q_0(v)^m})+\sum_{(l,\delta),(l',\delta')\in J_1}N^{(\sigma,\varepsilon)}_{(l,\delta),(l',\delta')}\frac{a^2}{4}\frac{q_1(l)^m}{q_1(l')^m}
 \\&+\sum_{(\tau,\delta),(\tau',\delta')\in J_2}N^{(\sigma,\varepsilon)}_{(\tau,\delta),(\tau',\delta')}\frac{b^2}{4}\frac{q_2(\tau)^m}{q_2(\tau')^m}+\sum_{(l,\delta)\in J_1,\;(\tau,\delta)\in J_2}N^{(\sigma,\varepsilon)}_{(l,\delta),(\tau,\delta)}
 \frac{ab}{4}(\frac{q_1(l)^m}{q_2(\tau)^m}+\frac{q_2(\tau)^m}{q_1(l)^m})\\
 &+\sum_{(l,\delta)\in J_1,\;\xi\in \Gamma_*}N^{(\sigma,\varepsilon)}_{(l,\delta),\xi}\frac{ab}{2}(\frac{q_1(l)^m}{q_2(\xi)^m}+\frac{q_2(\xi)^m}{q_1(l)^m})
 +\sum_{h\in G_*,\;(\tau,\delta)\in J_2}N^{(\sigma,\varepsilon)}_{h,(\tau,\delta)}\frac{ab}{2}(\frac{q_1(h)^m}{q_2(\tau)^m}+\frac{q_2(\tau)^m}{q_1(h)^m})\\
 &+\sum_{(l,\delta)\in J_1,\;h\in G_*}N^{(\sigma,\varepsilon)}_{(l,\delta),h}\frac{a^2}{2}(\frac{q_1(l)^m}{q_1(h)^m}+\frac{q_1(h)^m}{q_1(l)^m})
 +\sum_{(\tau,\delta)\in J_2,\;\xi\in \Gamma_*}N^{(\sigma,\varepsilon)}_{(\tau,\delta),\xi}\frac{b^2}{2}(\frac{q_2(\tau)^m}{q_2(\xi)^m}+\frac{q_2(\xi)^m}{q_2(\tau)^m})\\
 &+\sum_{h,h'\in G_*}N^{(\sigma,\varepsilon)}_{h,h'}a^2\frac{q_1(h)^m}{q_1(h')^m}
 +\sum_{h\in G_*,\;\xi\in \Gamma_*}N^{(\sigma,\varepsilon)}_{h,\xi}ab(\frac{q_1(h)^m}{q_2(\xi)^m}+\frac{q_2(\xi)^m}{q_1(h)^m})
 +\sum_{\xi,\xi'\in \Gamma_*}N^{(\sigma,\varepsilon)}_{\xi,\xi'}b^2\frac{q_2(\xi)^m}{q_2(\xi')^m}\\
 \end{align*}
 \begin{align*}
 &=\frac{b(a-b)}{4}\sum_{v\in U}(\frac{q_0(v)^m}{q_2(\sigma-v)^m}+\frac{q_2(\sigma-v)^m}{q_0(v)^m}) 
\\& +\frac{(a+b)^2}{2(|\Gamma|-|G|)}\sum_{v,v'\in U}(1+\inpr{\sigma}{u_0}) \frac{q_0(v)^m}{q_0(v')^m} \\
 &+\frac{a(a+b)}{|\Gamma|-|G|}\sum_{v\in U,\;h\in G_*}(1+\inpr{h}{u_0}\inpr{\sigma}{u_0})(\frac{q_0(v)^m}{q_1(h)^m}+\frac{q_1(h)^m}{q_0(v)^m}) \\
 &+\frac{b(a+b)}{|\Gamma|-|G|}\sum_{v\in U,\;\xi\in \Gamma_*}(1+\inpr{\sigma+\xi}{u_0})(\frac{q_0(v)^m}{q_2(\xi)^m}+\frac{q_2(\xi)^m}{q_0(v)^m})\\
 &+\frac{b(a+b)}{4}\sum_{v\in U,\;\tau\in \Sigma_*}(\frac{4}{|\Gamma|-|G|}-\delta_{v-\sigma+\tau})(\frac{q_0(v)^m}{q_2(\tau)^m}+\frac{q_2(\tau)^m}{q_0(v)^m})\\
 &+\frac{a^2}{2(|\Gamma|-|G|)}\sum_{l,l'\in K_*}(1+\inpr{u_0}{\sigma})\frac{q_1(l)^m}{q_1(l')^m}
 \\&+\frac{b^2}{2(|\Gamma|-|G|)}\sum_{\tau,\tau'\in \Sigma_*}(1+\inpr{u_0}{\sigma+\tau+\tau'})\frac{q_2(\tau)^m}{q_2(\tau')^m}\\
 &+\frac{ab}{2(|\Gamma|-|G|)}\sum_{l\in K_*,\;\tau\in \Sigma_*}(1+\inpr{u_0}{l})(\frac{q_1(l)^m}{q_2(\tau)^m}+\frac{q_2(\tau)^m}{q_1(l)^m})\\
 &+\frac{ab}{|\Gamma|-|G|}\sum_{l,\in K_*,\;\xi\in \Gamma_*}(1+\inpr{u_0}{l}\inpr{u_0}{\sigma+\xi})(\frac{q_1(l)^m}{q_2(\xi)^m}+\frac{q_2(\xi)^m}{q_1(l)^m})\\
 &+\frac{ab}{|\Gamma|-|G|}\sum_{h\in G_*,\;\tau\in \Sigma_*}(1+\inpr{u_0}{h})(\frac{q_1(h)^m}{q_2(\tau)^m}+\frac{q_2(\tau)^m}{q_1(h)^m})\\
 &+\frac{a^2}{|\Gamma|-|G|}\sum_{l\in K_*,\;h\in G_*}(1+\inpr{u_0}{l+h}\inpr{u_0}{\sigma})(\frac{q_1(l)^m}{q_1(h)^m}+\frac{q_1(h)^m}{q_1(l)^m})\\
 &+\frac{b^2}{|\Gamma|-|G|}\sum_{\tau\in \Sigma_*,\;\xi\in \Gamma_*}(1+\inpr{u_0}{\gamma})(\frac{q_2(\tau)^m}{q_2(\xi)^m}+\frac{q_2(\xi)^m}{q_2(\tau)^m})\\
 &+\frac{2a^2}{|\Gamma|-|G|}\sum_{h,h'\in G_*}(1+\inpr{u_0}{h+h'}\inpr{u_0}{\sigma})\frac{q_1(h)^m}{q_1(h')^m}\\
 &+\frac{2ab}{|\Gamma|-|G|}\sum_{h\in G_*,\;\xi\in \Gamma_*}(1+\inpr{u_0}{h}\inpr{u_0}{\sigma+\xi})(\frac{q_1(h)^m}{q_2(\xi)^m}+\frac{q_2(\xi)^m}{q_1(h)^m})\\
 &+b^2\sum_{\xi,\xi'\in \Gamma_*}(\frac{2}{|\Gamma|-|G|}-\delta_{-\sigma+\xi+\xi',0}-\delta_{-\sigma+\xi+\theta(\xi),0})\frac{q_2(\xi)^m}{q_2(\xi')^m}.
 \end{align*}

\begin{align*}
\lefteqn{\nu_m(\gamma)=\sum_{v\in U}N^\gamma_{(v,0),\gamma-v}\frac{b(a-b)}{2}(\frac{q_0(v)^m}{q_2(\gamma-v)^m}+\frac{q_2(\gamma-v)^m}{q_0(v)^m}) 
 +\sum_{v,v'\in U}N^\gamma_{(v,\pi),(v',\pi)}\frac{(a+b)^2}{4}\frac{q_0(v)^m}{q_0(v')^m} }\\
 &+\sum_{v\in U,\;h\in G_*}N^\gamma_{(v,\pi),h}\frac{a(a+b)}{2}(\frac{q_0(v)^m}{q_1(h)^m}+\frac{q_1(h)^m}{q_0(v)^m}) 
 +\sum_{v\in U,\;\xi\in \Gamma_*}N^\gamma_{(v,\pi),\xi}\frac{b(a+b)}{2}(\frac{q_0(v)^m}{q_2(\xi)^m}+\frac{q_2(\xi)^m}{q_0(v)^m})\\
 &+\sum_{v\in U,\;(l,\delta)\in J_1}N^\gamma_{(v,\pi),(l,\delta)}\frac{a(a+b)}{4}(\frac{q_0(v)^m}{q_1(l)^m}+\frac{q_1(l)^m}{q_0(v)^m})
 \\&+\sum_{v\in U,\;(\tau,\delta)\in J_2}N^\gamma_{(v,\pi),(\tau,\delta)}\frac{b(a+b)}{4}(\frac{q_0(v)^m}{q_2(\tau)^m}+\frac{q_2(\tau)^m}{q_0(v)^m})
 +\sum_{(l,\delta),(l',\delta')\in J_1}N^\gamma_{(l,\delta),(l',\delta')}\frac{a^2}{4}\frac{q_1(l)^m}{q_1(l')^m}
 \\&+\sum_{(\tau,\delta),(\tau',\delta')\in J_2}N^\gamma_{(\tau,\delta),(\tau',\delta')}\frac{b^2}{4}\frac{q_2(\tau)^m}{q_2(\tau')^m}+\sum_{(l,\delta)\in J_1,\;(\tau,\delta)\in J_2}N^\gamma_{(l,\delta),(\tau,\delta)}
 \frac{ab}{4}(\frac{q_1(l)^m}{q_2(\tau)^m}+\frac{q_2(\tau)^m}{q_1(l)^m})\\
 &+\sum_{(l,\delta)\in J_1,\;\xi\in \Gamma_*}N^\gamma_{(l,\delta),\xi}\frac{ab}{2}(\frac{q_1(l)^m}{q_2(\xi)^m}+\frac{q_2(\xi)^m}{q_1(l)^m})
 +\sum_{h\in G_*,\;(\tau,\delta)\in J_2}N^\gamma_{h,(\tau,\delta)}\frac{ab}{2}(\frac{q_1(h)^m}{q_2(\tau)^m}+\frac{q_2(\tau)^m}{q_1(h)^m})\\
 &+\sum_{(l,\delta)\in J_1,\;h\in G_*}N^\gamma_{(l,\delta),h}\frac{a^2}{2}(\frac{q_1(l)^m}{q_1(h)^m}+\frac{q_1(h)^m}{q_1(l)^m})
 +\sum_{(\tau,\delta)\in J_2,\;\xi\in \Gamma_*}N^\gamma_{(\tau,\delta),\gamma}\frac{b^2}{2}(\frac{q_2(\tau)^m}{q_2(\xi)^m}+\frac{q_2(\xi)^m}{q_2(\tau)^m})\\
 &+\sum_{h,h'\in G_*}N^\gamma_{h,h'}a^2\frac{q_1(h)^m}{q_1(h')^m}
 +\sum_{h\in G_*,\;\xi\in \Gamma_*}N^\gamma_{h,\xi}ab(\frac{q_1(h)^m}{q_2(\xi)^m}+\frac{q_2(\xi)^m}{q_1(h)^m})
 +\sum_{\xi,\xi'\in \Gamma_*}N^\gamma_{\xi,\xi'}b^2\frac{q_2(\xi)^m}{q_2(\xi')^m}\\
\end{align*}
\begin{align*}
&=\frac{b(a-b)}{2}\sum_{v\in U}(\frac{q_0(v)^m}{q_2(\gamma-v)^m}+\frac{q_2(\gamma-v)^m}{q_0(v)^m}) 
 +\frac{(a+b)^2}{|\Gamma|-|G|}\sum_{v,v'\in U}(1+\inpr{u_0}{\gamma})\frac{q_0(v)^m}{q_0(v')^m} \\
 &+\frac{2a(a+b)}{|\Gamma|-|G|}\sum_{v\in U,\;h\in G_*}(1+\inpr{u_0}{h}\inpr{u_0}{\gamma})(\frac{q_0(v)^m}{q_1(h)^m}+\frac{q_1(h)^m}{q_0(v)^m})\\ 
 &+\frac{b(a+b)}{2}\sum_{v\in U,\;\xi\in \Gamma_*}(\frac{4(1+\inpr{u_0}{\gamma+\xi})}{|\Gamma|-|G|}-\delta_{v-\gamma+\xi,0}-\delta_{v-\gamma+\theta(\xi)})
 (\frac{q_0(v)^m}{q_2(\xi)^m}+\frac{q_2(\xi)^m}{q_0(v)^m})\\
 &+\frac{a(a+b)}{|\Gamma|-|G|}\sum_{v\in U,l\in K_*}(1+\inpr{u_0}{l}\inpr{u_0}{\gamma})(\frac{q_0(v)^m}{q_1(l)^m}+\frac{q_1(l)^m}{q_0(v)^m})\\
 &+\frac{b(a+b)}{|\Gamma|-|G|}\sum_{v\in U,\tau\in \Sigma_*}(1+\inpr{u_0}{\tau+\gamma})(\frac{q_0(v)^m}{q_2(\tau)^m}+\frac{q_2(\tau)^m}{q_0(v)^m})\\
 &+\frac{a^2}{|\Gamma|-|G|}\sum_{l,l'\in K_*}(1+\inpr{u_0}{\gamma})\frac{q_1(l)^m}{q_1(l')^m}
 +\frac{b^2}{|\Gamma|-|G|}\sum_{\tau,\tau'\in \Sigma_*}(1+\inpr{u_0}{\gamma})\frac{q_2(\tau)^m}{q_2(\tau')^m}\\
 &+\frac{ab}{|\Gamma|-|G|}\sum_{l\in K_*,\tau\in \Sigma_*}(1+\inpr{u_0}{l}\inpr{u_0}{\tau+\gamma})
 (\frac{q_1(l)^m}{q_2(\tau)^m}+\frac{q_2(\tau)^m}{q_1(l)^m})\\
 &+\frac{2ab}{|\Gamma|-|G|}\sum_{l\in K_*,\xi\in \Gamma_*}(1+\inpr{u_0}{l}\inpr{u_0}{\gamma+\xi})(\frac{q_1(l)^m}{q_2(\xi)^m}+\frac{q_2(\xi)^m}{q_1(l)^m})\\
 &+\frac{2ab}{|\Gamma|-|G|}\sum_{h\in G_*,\tau\in \Sigma_*}(1+\inpr{u_0}{h}\inpr{u_0}{\gamma+\tau})(\frac{q_1(h)^m}{q_2(\tau)^m}+\frac{q_2(\tau)^m}{q_1(h)^m})\\
 &+\frac{2a^2}{|\Gamma|-|G|}\sum_{l\in K_*,h\in G_*}(1+\inpr{u_0}{l+h}\inpr{u_0}{\gamma})(\frac{q_1(l)^m}{q_1(h)^m}+\frac{q_1(h)^m}{q_1(l)^m})\\
 &+b^2\sum_{\tau\in \Sigma_*,\xi\in \Gamma_*}(\frac{2(1+\inpr{u_0}{\gamma+\tau+\xi})}{|\Gamma|-|G|}-\delta_{-\gamma+\tau+\xi,0}-\delta_{-\gamma+\tau+\theta(\xi),0})(\frac{q_2(\tau)^m}{q_2(\xi)^m}+\frac{q_2(\xi)^m}{q_2(\tau)^m})\\
 &+\frac{4a^2}{|\Gamma|-|G|}\sum_{h,h'\in G_*}(1+\inpr{u_0}{h+h'}\inpr{u_0}{\gamma})\frac{q_1(h)^m}{q_1(h')^m}\\
 &+\frac{4ab}{|\Gamma|-|G|}\sum_{h\in G_*,\;\xi\in \Gamma_*}(1+\inpr{u_0}{h}\inpr{u_0}{\gamma+\xi})(\frac{q_1(h)^m}{q_2(\xi)^m}+\frac{q_2(\xi)^m}{q_1(h)^m})\\
 &+b^2\sum_{\xi,\xi'\in \Gamma_*}(\frac{4(1+\inpr{u_0}{\gamma+\xi+\xi'})}{|\Gamma|-|G|}-\delta_{-\gamma+\xi+\xi'}-\delta_{-\theta(\gamma)+\xi+\xi'}-\delta_{-\gamma+\theta(\xi)+\xi'}-\delta_{\gamma+\xi+\theta(\xi')})\frac{q_2(\xi)^m}{q_2(\xi')^m}.
 \end{align*}
 The rest of computation is the same as in the previous cases.  
\end{proof}
\section{Calculations for Section 6}

\begin{proof}[Proof of Lemma 6.2]
We first check the unitarity of $S$. 
$$\sum_{x}|S_{0,x}|^2=\sum_{x}|S_{\pi,x}|^2=\frac{a^2+b^2}{2}+\frac{a^2}{2}(|K|-1)+a^2\frac{|G|-|K|}{2}+b^2\frac{|\Gamma|-1}{2}
=1.$$
$$\sum_{x}S_{0,x}\overline{S_{\pi,x}}=\frac{a^2-b^2}{2}+\frac{a^2}{2}(|K|-1)+a^2\frac{|G|-|K|}{2}-b^2\frac{|\Gamma|-1}{2}
=0.$$
\begin{align*}
\lefteqn{\sum_{x}S_{0,x}\overline{S_{(k,\varepsilon),x}}=\sum_{x}S_{\pi,x}\overline{S_{(k,\varepsilon),x}}} \\
 &=\frac{a^2}{2}+\frac{a^2}{2}\sum_{l\in K\setminus \{0\}}\inpr{k}{l}_1+a^2\sum_{g\in G_*}\inpr{k}{g}_1
 =\frac{a^2|G|\delta_{k,0}}{2}=0. \\
\end{align*}
\begin{align*}
\lefteqn{\sum_{x}S_{0,x}\overline{S_{g,x}}=\sum_{x}S_{\pi,x}\overline{S_{g,x}}} \\
 &=a^2+a^2\sum_{l\neq 0}\inpr{k}{l}_1+a^2\sum_{h\in G_*}(\inpr{g}{h}_1+\overline{\inpr{g}{h}_1})
 =a^2|G|\delta_{g,0}=0. \\
\end{align*}
$$\sum_{x}S_{0,x}\overline{S_{\gamma,x}}=-\sum_{x}S_{\pi,x}\overline{S_{\gamma,x}}
=-b^2-b^2\sum_{\gamma\in \Gamma_*}(\inpr{\gamma}{\xi}_2+\overline{\inpr{\gamma}{\xi}}_2)=-b^2|\Gamma|\delta_{\gamma,0}=0.$$
\begin{align*}
\lefteqn{\sum_{x}S_{(k,\varepsilon),x}\overline{S_{(k',\varepsilon'),x}}} \\
 &=\frac{a^2}{2}+\frac{1}{4}\sum_{(l,t)\in J_2}(a^2\inpr{k-k'}{l}+\varepsilon\varepsilon'q_1(k)\overline{q_1(k')}\delta_{k,l}\delta_{k',l})
 +a^2\sum_{g\in G_*}\inpr{k-k'}{g} \\
 &=\frac{a^2|G|\delta_{k,k'}}{2}+\frac{\varepsilon\varepsilon'\delta_{k,k'}}{2}=\delta_{k,k'}\delta_{\varepsilon,\varepsilon'}. 
\end{align*}
\begin{align*}
\lefteqn{\sum_{x}S_{(k,\varepsilon),x}\overline{S_{g,x}}} \\
 &=a^2+a^2\sum_{l\in K\setminus \{0\}}\inpr{k-g}{l}_1+a^2\sum_{h\in G_*}(\inpr{k+g}{h}_1+\overline{\inpr{-k+g}{h}_1}) \\
 &=a^2|G|\delta_{k,g}=0. 
\end{align*}
$\sum_{x}S_{(k,\varepsilon),x}\overline{S_{\gamma,x}}=0$ is obvious. 
\begin{align*}
\lefteqn{\sum_{x}S_{g,x}\overline{S_{g',x}}=2a^2+2a^2\sum_{l\in K\setminus \{0\}}\inpr{g-g'}{l}_1} \\
 &+a^2\sum_{h\in G_*}(\inpr{g}{h}_1+\overline{\inpr{g}{h}}_1)(\inpr{g'}{h}_1+\overline{\inpr{g'}{h}}_1) \\
 &=a^2(\sum_{l\in K}\inpr{g-g'}{l}_1+\sum_{l\in K}\inpr{g+g'}{l}_1\\
 &+\sum_{h\in G_*}(\inpr{g+g'}{h}_1+\overline{\inpr{g+g'}{h}}_1+\inpr{g-g'}{h}_1+\overline{\inpr{g-g'}{h}}_1)\\
 &=\delta_{g+g',0}+\delta_{g,g'}=\delta_{g,g'}. 
\end{align*}
$\sum_{x}S_{g,x}\overline{S_{\gamma,x}}=0$ is obvious. 
\begin{align*}
\lefteqn{\sum_{x}S_{\gamma,x}\overline{S_{\gamma',x}}} \\
 &=2b^2+b^2\sum_{\xi\in \Gamma_*}(\inpr{\gamma}{\xi}_2+\overline{\inpr{\gamma}{\xi}_2}) (\inpr{\gamma'}{\xi}_2+\overline{\inpr{\gamma'}{\xi}_2})\\
 &=\delta_{\gamma+\gamma',0}+\delta_{\gamma,\gamma'}=\delta_{\gamma,\gamma'}. 
\end{align*}
Thus $S$ is unitary. 
It is obvious that $T$ and $C$ are unitary.  

Since $T_{x,x}=T_{\overline{x},\overline{x}}$, we have $CT=T C$. 
We claim $\overline{S_{x,y}}=S_{\overline{x},y}=S_{x,\overline{y}}$, which implies $S^2=C$. 
Indeed, the only block of $S$ that may not be real is the $J_1$-$J_1$ block, and for other blocks, we have  
$$\overline{S_{x,y}}=S_{x,y}=S_{\overline{x},y}=S_{x,\overline{y}}.$$
Using $c^2, q_1(k)^2=\inpr{k}{k}_1\in \{1,-1\}$, we get   
$$\overline{S_{(k,\varepsilon),(k',\varepsilon')}}=\frac{a\inpr{k}{k'}_1
+\overline{c}\varepsilon\varepsilon'\overline{q_1(k)}\delta_{k,k'}}{2}
=\frac{a\inpr{k}{k'}_1
+c^2\inpr{k}{k}_1c\varepsilon\varepsilon'q_1(k)\delta_{k,k'}}{2},$$
and we get 
$$\overline{S_{(k,\varepsilon),(k',\varepsilon')}}=S_{(k,c^2\inpr{k}{k}_1\varepsilon),(k',\varepsilon')}
=S_{(k,\varepsilon),(k',c^2\inpr{k'}{k'}_1\varepsilon')}.$$ 
The claim is shown. 

The only remaining condition is $(ST)^3=cC$, or equivalently  
$STS=cC(TST)^*$. 
Since $S$ is symmetric and $C$ commutes with $S$ and $\overline{T}$, we always have 
$$C(TST)^*=C\overline{TST}=\overline{T}S\overline{T}.$$
Thus it suffices to verify 
$$\sum_{x}S_{j,x}S_{j',x}T_{x,x}=cS_{j,j'}\overline{T_{j,j}T_{j',j'}}$$
for all $j,j'\in J$. 
\begin{align*}
\lefteqn{\sum_{x}S_{0,x}^2T_{x,x}} \\
 &=\frac{a^2+b^2}{2}+\frac{a^2}{2}\sum_{k\in K\setminus \{0\}}q_1(k)+a^2\sum_{g\in G_*}q_1(g)
 +b^2\sum_{\gamma\in \Gamma_*}q_2(\gamma) \\
 &=\frac{a\cG(q_1)}{2}+\frac{b\cG(q_2)}{2}
 =c\frac{a-b}{2}=cS_{0,0}\overline{T_{0,0}^2}.
\end{align*}
\begin{align*}
\lefteqn{\sum_{x}S_{0,x}S_{\pi,x}T_{x,x}} \\
 &=\frac{a^2-b^2}{2}+\frac{a^2}{2}\sum_{k\in K\setminus\{0\}}q_1(k)+a^2\sum_{g\in G_*}q_1(g)
 -b^2\sum_{\gamma\in \Gamma_*}q_2(\gamma) \\
 &=\frac{a\cG(q_1)}{2}-\frac{b\cG(q_2)}{2}
 =c\frac{a+b}{2}=cS_{0,\pi}\overline{T_{0,0}T_{\pi,\pi}}.
\end{align*}
\begin{align*}
\lefteqn{\sum_{x}S_{0,x}S_{(k,\varepsilon),x}T_{x,x}=S_{\pi,x}S_{(k,\varepsilon),x}T_{x,x}} \\
 &=\frac{a^2}{2}+\frac{a^2}{2}\sum_{l\in K\setminus\{0\}}\inpr{k}{l}q_1(l)+a^2\sum_{h\in G_*}\inpr{k}{h}_1q_1(h)
 =\frac{a^2}{2}\sum_{h\in G}q_1(k+h)\overline{q_1(k)}\\
 & =\frac{a\cG(q_1)\overline{q_1(k)}}{2}=c\frac{a}{2}\overline{q_1(k)}\\
 &=cS_{0,(k,\varepsilon)}\overline{T_{0,0}T_{(k,\varepsilon),(k,\varepsilon)}}
 =cS_{\pi,(k,\varepsilon)}\overline{T_{\pi,\pi}T_{(k,\varepsilon),(k,\varepsilon)}}.
\end{align*}
\begin{align*}
\lefteqn{\sum_{x}S_{0,x}S_{g,x}T_{x,x}=S_{\pi,x}S_{g,x}T_{x,x}} \\
 &=a^2+a^2\sum_{l\in K\setminus\{0\}}\inpr{g}{l}q_1(l)+a^2\sum_{h\in G_*}(\inpr{g}{h}_1+\overline{\inpr{g}{h}_1})q_1(h)\\
 &=a^2\sum_{h\in G}q_1(g+h)\overline{q_1(g)}
 =a\cG(q_1)\overline{q_1(g)}\\
 &=cS_{0,g}\overline{T_{0,0}T_{g,g}}
 =cS_{\pi,g}\overline{T_{\pi,\pi}T_{g,g}}.
\end{align*}
\begin{align*}
\lefteqn{\sum_{x}S_{0,x}S_{\gamma,x}T_{x,x}=-\sum_{x}S_{\pi,x}S_{\gamma,x}T_{x,x}} \\
 &=-b^2-b^2\sum_{\xi\in \Gamma_*}(\inpr{\gamma}{\xi}_2+\overline{\inpr{\gamma}{\xi}_2})q_2(\xi)
 =-b^2\sum_{\xi\in \Gamma}\inpr{\xi}{\gamma}_2q_2(\xi)=-b\cG(q_2)\overline{q_2(\gamma)}\\
 &=cS_{0,\gamma}\overline{T_{0,0}T_{\gamma,\gamma}}=-cS_{\pi,\gamma}\overline{T_{0,0}T_{\gamma,\gamma}}.
\end{align*}
\begin{align*}
\lefteqn{\sum_{x}S_{(k,\varepsilon),x}S_{(k',\varepsilon'),x}T_{x,x}} \\
 &\frac{a^2}{2}+\frac{1}{2}\sum_{l\in K\setminus\{0\}}(a^2\inpr{k+k'}{l}_1
 +c^2\varepsilon\varepsilon'q_1(k)q_1(k')\delta_{k,l}\delta_{k',l})q_1(l)
 +a^2\sum_{h\in G_*}\inpr{k+k'}{h}_1q_1(h)\\
 &=\frac{a\cG(q_1)\overline{q_1(k+k')}}{2}+\frac{c^2\varepsilon\varepsilon'q_1(k)^3\delta_{k,k'}}{2}\\
 &=c\frac{a\overline{\inpr{k}{k'}_1q_1(k)q_1(k')}+c\varepsilon\varepsilon'q_1(k)^3\delta_{k,k'}}{2}
 =c\frac{a\inpr{k}{k'}_1+c\varepsilon\varepsilon'q_1(k)\delta_{k,k'}}{2}\overline{q_1(k)q_1(k')}\\
 &=cS_{(k,\varepsilon),(k',\varepsilon')}\overline{T_{(k,\varepsilon),(k,\varepsilon)}T_{(k',\varepsilon'),(k',\varepsilon')}},
\end{align*}
where we used $\inpr{k}{k'}_1\in \R$ and $q_1(k)^4=\inpr{k}{k}_1^2=1$. 
\begin{align*}
\lefteqn{\sum_{x}S_{(k,\varepsilon),x}S_{g,x}T_{x,x}} \\
 &=a^2+a^2\sum_{l\in K\setminus\{0\}}\inpr{k+g}{l}_1q_1(l)
 +a^2\sum_{h\in G_*}\inpr{k}{h}_1(\inpr{g}{h}+\overline{\inpr{g}{h}})q_1(h)\\
 &=a^2\sum_{h\in G}q_1(k+g+h)\overline{q_1(k+g)}=a\cG(q_1)\overline{q_1(k+g)}\\
 &=ca\overline{\inpr{k}{g}_1q_1(k)q_1(g)}\\
 &=cS_{(k,\varepsilon),g}\overline{T_{(k,\varepsilon),(k,\varepsilon)}T_{g,g}}.
\end{align*}
\begin{align*}
\lefteqn{\sum_{x}S_{g,x}S_{g',x}T_{x,x}=2a^2+2a^2\sum_{l\in K\setminus\{0\}}\inpr{g+g'}{l}_1q_1(l)} \\
 &+a^2\sum_{h\in G_*}(\inpr{g}{h}+\overline{\inpr{g}{h}})(\inpr{g'}{h}+\overline{\inpr{g'}{h}})q_1(h)\\
 &=a^2\sum_{h\in G}(q_1(g+g'+h)\overline{q_1(g+g')}+q_1(g-g'+h)\overline{q_1(g-g')})\\
 &=a\cG(q_1)\overline{(q_1(g+g')+q_1(g-g')}
 =ca\overline{(\inpr{g}{g'}_1+\overline{\inpr{g}{g'}_1})  q_1(g)q_1(g')}\\
 &=cS_{g,g'}\overline{T_{g,g}T_{g',g'}}.
\end{align*}
\begin{align*}
\lefteqn{\sum_{x}S_{\gamma,x}S_{\gamma',x}T_{x,x}} \\
 &=2b^2+b^2\sum_{\xi\in \Gamma_*}(\inpr{\gamma}{\xi}_2+\overline{\inpr{\gamma}{\xi}_2})(\inpr{\gamma'}{\xi}_2+\overline{\inpr{\gamma'}{\xi}_2})q_2(\xi)\\
 &=b\cG(q_2)(\overline{q_2(\gamma+\gamma')+q_2(\gamma-\gamma')})
 =-cb(\inpr{\gamma}{\xi}_2+\overline{\inpr{\gamma}{\xi}_2})\overline{q_2(\gamma)q_2(\gamma')}\\
 &=cS_{\gamma,\gamma'}\overline{T_{\gamma,\gamma}T_{\gamma',\gamma'}}.
\end{align*}
\begin{align*}
\lefteqn{\sum_{x}S_{(k,\varepsilon),x}S_{\gamma,x}T_{x,x}=\sum_{x}S_{g,x}S_{\gamma,x}T_{x,x}} \\
 &=S_{(k,\varepsilon),\gamma}\overline{T_{(k,\varepsilon),(k,\varepsilon)}T_{\gamma,\gamma}}
=S_{g,\gamma}\overline{T_{g,g}T_{\gamma,\gamma}}=0.
\end{align*}

\end{proof}
\begin{proof}[Proof of Lemma 6.3] 
Since $S_{0,x}=S_{\pi,x}$ for $x\neq 0,\pi,\gamma$, we have  
\begin{align*}
\lefteqn{N_{\pi,y,z}-N_{0,y,z}=\sum_{x}\frac{(S_{\pi,x}-S_{0,x})S_{y,x}S_{z,x}}{S_{0,x}}} \\
 &=\frac{(S_{\pi,0}-S_{0,0})S_{y,0}S_{z,0}}{S_{0,0}}
 +\frac{(S_{\pi,\pi}-S_{0,\pi})S_{y,\pi}S_{z,\pi}}{S_{0,\pi}}
 +\sum_{\gamma\in \Gamma_*}\frac{(S_{\pi,\gamma}-S_{0,\gamma})S_{y,\gamma}S_{z,\gamma}}{S_{0,\gamma}}\\
 &=b(\frac{S_{y,0}S_{z,0}}{S_{0,0}}-\frac{S_{y,\pi}S_{z,\pi}}{S_{0,\pi}})
 -2\sum_{\gamma\in \Gamma_*}S_{y,\gamma}S_{z,\gamma}.
\end{align*}
Thus
\begin{align*}
\lefteqn{N_{\pi,\pi,\pi}=1+b(\frac{S_{0,\pi}^2}{S_{0,0}}-\frac{S_{0,0}^2}{S_{0,\pi}})-b^2(|\Gamma|-1)}\\
&=b^2\frac{S_{0,0}^2+S_{0,0}S_{0,\pi}+S_{0,\pi}^2}{S_{0,0}S_{0,\pi}}+b^2
=b^2\frac{(S_{0,0}+S_{0,\pi})^2}{S_{0,0}S_{0,\pi}}=\frac{4a^2b^2}{a^2-b^2}=\frac{4}{|\Gamma|-|G|},
\end{align*}
$$N_{\pi,\pi,\gamma}
=b^2(\frac{S_{\pi,0}}{S_{0,0}}+\frac{S_{\pi,\pi}}{S_{0,\pi}})
 -2b^2\sum_{\xi\in \Gamma_*}(\inpr{\gamma}{\xi}+\overline{\inpr{\gamma}{\xi}})=\frac{4a^2b^2}{a^2-b^2}=\frac{4}{|\Gamma|-|G|}.$$
\begin{align*}\lefteqn{
N_{\pi,\gamma,\gamma'}
=\delta_{\gamma,\gamma'}+b^3(\frac{1}{S_{0,0}}-\frac{1}{S_{0,\pi}})
-2\sum_{\xi\in \Gamma_*}S_{\gamma,\xi}S_{\gamma',\xi}}\\
 &=\delta_{\gamma,\gamma'}+\frac{b^4}{S_{0,0}S_{0,\pi}}-2(\delta_{\gamma,\gamma'}-S_{\gamma,0}S_{\gamma',0}-S_{\gamma,\pi}S_{\gamma',\pi}) \\
 &=\frac{4a^2b^2}{a^2-b^2}-\delta_{\gamma,\gamma'}=\frac{4}{|\Gamma|-|G|}-\delta_{\gamma,\gamma'}.
\end{align*}

For $z\neq 0,\pi,\gamma$, we get 
$$N_{\pi,\pi,z}=b(\frac{S_{\pi,0}}{S_{0,0}}-\frac{S_{0,0}}{S_{0,\pi}})S_{z,0}
=b^2\frac{S_{0,0}+S_{0,\pi}}{S_{0,0}S_{0,\pi}}S_{z,0}=\frac{4ab^2}{a^2-b^2}S_{z,0},$$
$$N_{\pi,\gamma,z}=N_{\gamma,\gamma',z}=b^2(\frac{1}{S_{0,0}}+\frac{1}{S_{0,\pi}})S_{z,0}=\frac{4ab^2}{a^2-b^2}S_{z,0},$$
and in particular, we get 
$$N_{\pi,\pi,(k,\varepsilon)}=N_{\pi,\gamma,(k,\varepsilon)}=N_{\gamma,\gamma',(k,\varepsilon)}=\frac{2a^2b^2}{a^2-b^2}=\frac{2}{|\Gamma|-|G|},$$
$$N_{\pi,\pi,g}=N_{\pi,\gamma,g}=N_{\gamma,\gamma',g}=\frac{4a^2b^2}{a^2-b^2}=\frac{4}{|\Gamma|-|G|}.$$

For $y,z\neq 0,\pi,\gamma$, 
$$N_{\pi,y,z}=\delta_{y,\overline{z}}+b(\frac{1}{S_{0,0}}-\frac{1}{S_{0,\pi}})S_{y,0}S_{z,0}=
\delta_{y,\overline{z}}+\frac{4b^2}{b^2-a^2}S_{y,0}S_{z,0},$$
$$N_{\gamma,y,z}=b(\frac{1}{S_{0,0}}-\frac{1}{S_{0,\pi}})S_{y,0}S_{z,0}=\frac{4b^2}{b^2-a^2}S_{y,0}S_{z,0},$$
and in particular, 
$$N_{\pi,(k,\varepsilon),(k',\varepsilon')}=\delta_{(k,\varepsilon),\overline{(k',\varepsilon')}}+\frac{1}{|\Gamma|-|G|},
\quad N_{\gamma,(k,\varepsilon),(k',\varepsilon')}=\frac{1}{|\Gamma|-|G|},$$
$$N_{\pi,(k,\varepsilon),g}=N_{\gamma,(k,\varepsilon),g}=\frac{2}{|\Gamma|-|G|},$$
$$N_{\pi,g,g'}=\delta_{g,g'}+\frac{4}{|\Gamma|-|G|},
\quad N_{\gamma,g,g'}=\frac{4}{|\Gamma|-|G|}.$$
\begin{align*}
\lefteqn{N_{\gamma,\gamma',\gamma''}=b^3 (\frac{1}{S_{0,0}}-\frac{1}{S_{0,\pi}})} \\
 &-b^2\sum_{\xi\in \Gamma_*}(\inpr{\gamma}{\xi}_2+\overline{\inpr{\gamma}{\xi}_2}) 
 (\inpr{\gamma'}{\xi}_2+\overline{\inpr{\gamma'}{\xi}_2})(\inpr{\gamma''}{\xi}_2+\overline{\inpr{\gamma''}{\xi}_2})\\
 &=\frac{4b^4}{a^2-b^2}-b^2(|\Gamma|(\delta_{\gamma+\gamma'+\gamma'',0}+\delta_{\gamma+\gamma',\gamma''}
 +\delta_{\gamma'+\gamma'',\gamma}+\delta_{\gamma''+\gamma,\gamma'})-4)\\
 &=\frac{4a^2b^2}{a^2-b^2}-(\delta_{\gamma+\gamma'+\gamma'',0}+\delta_{\gamma+\gamma',\gamma''}
 +\delta_{\gamma'+\gamma'',\gamma}+\delta_{\gamma''+\gamma,\gamma'})\\
 &=\frac{4}{|\Gamma|-|G|}-(\delta_{\gamma+\gamma'+\gamma'',0}+\delta_{\gamma+\gamma',\gamma''}
 +\delta_{\gamma'+\gamma'',\gamma}+\delta_{\gamma''+\gamma,\gamma'}).
\end{align*}

For $y,z,w\neq 0,\pi,\gamma$, we have 
\begin{align*}
\lefteqn{N_{y,z,w}=\frac{4a}{a^2-b^2}S_{y,0}S_{z,0}S_{w,0} } \\
 &+\frac{2}{a}\sum_{(l,t)\in J_1}S_{y,(l,t)}S_{z,(l,t)}S_{w,(l,t)}
 +\frac{1}{a}\sum_{h\in G_*}S_{y,h}S_{z,h}S_{w,h},
\end{align*}
and 
\begin{align*}
\lefteqn{
N_{(k,\varepsilon),(k',\varepsilon'),(k'',\varepsilon'')}
=\frac{a^4}{2(a^2-b^2)}}\\
&+\frac{1}{4a}\sum_{(l,\varepsilon)\in J_1} (a\inpr{k}{l}_1+c\varepsilon tq_1(k)\delta_{k,l})
(a\inpr{k'}{l}_1+c\varepsilon' tq_1(k')\delta_{k',l})\\
&\times (a\inpr{k''}{l}_1+c\varepsilon'' tq_1(k'')\delta_{k'',l})
+a^2\sum_{h\in G_*}\inpr{k+k'+k''}{h}_1\\
&=\frac{a^2b^2}{2(a^2-b^2)}+\frac{a^2}{2}\sum_{l\in K}\inpr{k+k'+k''}{l}_1\\
&+c^2\frac{\varepsilon\varepsilon'\inpr{k}{k+k''}_1\delta_{k,k'}+\varepsilon'\varepsilon''\inpr{k'}{k'+k}_1\delta_{k',k''}+
\varepsilon''\varepsilon\inpr{k''}{k''+k'}_1\delta_{k'',k}}{2}\\ 
&+a^2\sum_{h\in G_*}\inpr{k+k'+k''}{h}\\
 &=\frac{1}{2}(\frac{1}{|\Gamma|-|G|}+\delta_{k+k'+k'',0})\\
 &+c^2\frac{\varepsilon\varepsilon'\inpr{k}{k+k''}_1\delta_{k,k'}+\varepsilon'\varepsilon''\inpr{k'}{k'+k}_1\delta_{k',k''}+
\varepsilon''\varepsilon\inpr{k''}{k''+k'}_1\delta_{k'',k}}{2},
\end{align*}
\begin{align*}
\lefteqn{N_{(k,\varepsilon),(k',\varepsilon'),g}=\frac{a^4}{a^2-b^2} } \\
 &+\frac{1}{2}\sum_{(l,t)\in J_1}  (a\inpr{k}{l}_1+c\varepsilon tq_1(k)\delta_{k,l})
(a\inpr{k'}{l}_1+c\varepsilon' tq_1(k')\delta_{k',l})    \inpr{g}{l}_1\\
 &+a^2\sum_{h\in G_*}\inpr{k+k'}{h}_1(\inpr{g}{h}_1+\overline{\inpr{g}{h}_1})\\
 &=\frac{a^2b^2}{a^2-b^2}+a^2\sum_{l\in K}\inpr{k+k'+g}{l}_1+
 a^2\sum_{h\in G_*}(\inpr{k+k'+g}{h}_1+\inpr{k+k'+g}{-h}_1)\\
 &+c^2\varepsilon\varepsilon'\inpr{k}{k+g}_1\delta_{k,k'}\\
 &=\frac{1}{|\Gamma|-|G|}+c^2\varepsilon\varepsilon'\inpr{k}{k+g}_1\delta_{k,k'},\\
\end{align*}
\begin{align*}
\lefteqn{N_{(k,\varepsilon),g,g'}=\frac{2a^4}{a^2-b^2}} \\
 &+a\sum_{(l,t)\in J_1} (a\inpr{k}{l}_1+c\varepsilon tq_1(k)\delta_{k,l})\inpr{g+g'}{l}_1 \\ 
 &+a^2\sum_{h\in G_*}\inpr{k}{h}_1(\inpr{g}{h}_1+\overline{\inpr{g}{h}_1})(\inpr{g'}{h}_1+\overline{\inpr{g'}{h}_1})\\
 &=\frac{2a^2b^2}{a^2-b^2}+
 2a^2\sum_{l\in K}\inpr{k+g+g'}{l}_1
 +a^2\sum_{h\in G_*}\inpr{k}{h}_1(\inpr{g}{h}_1+\overline{\inpr{g}{h}_1})(\inpr{g'}{h}_1+\overline{\inpr{g'}{h}_1}))\\
 &=\frac{2}{|\Gamma|-|G|}+\delta_{k+g+g',0}+\delta_{k+g-g'}. 
\end{align*}
\begin{align*}
\lefteqn{N_{g,g',g''}=\frac{4a^4}{a^2-b^2} } \\
 &+4a^2\sum_{l\in K\setminus\{0\}}\inpr{g+g'+g''}{l}_1\\
 &+a^2\sum_{h\in G_*}
 (\inpr{g}{h}_1+\overline{\inpr{g}{h}_1})(\inpr{g'}{h}_1+\overline{\inpr{g'}{h}_1}))(\inpr{g'}{h}_1+\overline{\inpr{g'}{h}_1})) \\
 &=\frac{4a^2b^2}{a^2-b^2}+4a^2\sum_{l\in K\setminus\{0\}}\inpr{g+g'+g''}{l}_1
 +a^2\sum_{h\in G_*}(\inpr{g+g'+g''}{h}+\overline{\inpr{g+g'+g''}{h}}) \\
 &+a^2\sum_{h\in G_*}(\inpr{g+g'-g''}{h}+\overline{\inpr{g+g'-g''}{h}})
 +a^2\sum_{h\in G_*}(\inpr{g-g'+g''}{h}+\overline{\inpr{g-g'+g''}{h}})\\
 &+a^2\sum_{h\in G_*}(\inpr{-g+g'+g''}{h}+\overline{\inpr{-g+g'+g''}{h}})\\
 &=\frac{4}{|\Gamma|-|G|}+\delta_{g+g'+g'',0}+\delta_{g+g',g''}+\delta_{g'+g'',g}+\delta_{g''+g,g'}.
\end{align*}
\end{proof}

\begin{proof}[Proof of Lemma 6.8]
We set 
$$A_m=2a\sum_{g\in G_*}q_1(g)^m,\quad B_m=2b\sum_{\gamma\in \Gamma_*}q_2(\gamma)^m.$$
$$C_m=a\sum_{k\in K\setminus\{0\}}q_0(k)^m.$$
Then 
$$\cG(q_1,m)+\cG(q_2,m)=a+b+A_m+B_m+C_m.$$

For $\nu_m((k,\pi))$, we have 
We have 
\begin{align*}
\lefteqn{\nu_m(\pi)} \\
 &=N^{\pi}_{0,\pi}\frac{a^2-b^2}{2} 
 +N^{\pi}_{\pi,\pi}\frac{(a+b)^2}{4}+\sum_{(k,\varepsilon)\in J_1}N^\pi_{\pi,(k,\varepsilon)}\frac{a(a+b)}{4}(q_1(k)^m+q_1(k)^{-m}) \\
 &+\sum_{g\in G_*}N^{\pi}_{\pi,g}\frac{a(a+b)}{2}(q_1(g)^m+q_1(g)^{-m})+\sum_{\gamma\in \Gamma_*}N^{\pi}_{\pi,\gamma}\frac{b(a+b)}{2}(q_2(\gamma)^m+q_2(\gamma)^{-m})  \\
 &+\sum_{(k,\varepsilon),(k',\varepsilon')\in J_1}N^\pi_{(k,\varepsilon),(k',\varepsilon')}\frac{a^2}{4}\frac{q_1(k)^m}{q_1(k')^m}
 +\sum_{(k,\varepsilon)\in J_1,\; g\in G_*}N^\pi_{(k,\varepsilon),g}\frac{a^2}{2}(\frac{q_1(k)^m}{q_1(g)^m}+\frac{q_1(g)^m}{q_1(k)})\\
 &+\sum_{(k,\varepsilon)\in J_1,\;\gamma\in \Gamma_*}N^\pi_{(k,\varepsilon),\gamma}\frac{ab}{2}(\frac{q_1(k)^m}{q_2(\gamma)^m}+\frac{q_2(\gamma)^m}{q_1(k)^m})\\
 &+\sum_{g,h\in G_*}N^{\pi}_{g,h}a^2\frac{q_1(g)^m}{q_1(h)^m}
 +\sum_{g\in G_*,\;\gamma\in \Gamma_*}N^{\pi}_{g,\gamma}ab(\frac{q_1(g)^m}{q_2(\gamma)^m}+\frac{q_2(\gamma)^m}{q_1(g)^m})
 +\sum_{\gamma,\xi\in \Gamma_*}N^{\pi}_{\gamma,\xi}b^2\frac{q_2(\gamma)^m}{q_2(\xi)^m}\\
 &=\frac{a^2-b^2}{2}+ \frac{(a+b)^2}{|\Gamma|-|G|}
 +\frac{a(a+b)}{|\Gamma|-|G|}\sum_{k\in K\setminus\{0\}}(q_1(k)^m+q_1(k)^{-m}) \\
 &+\frac{2a(a+b)}{|\Gamma|-|G|}\sum_{g\in G_*}(q_1(g)^m+q_1(g)^{-m})
 +\frac{2b(a+b)}{|\Gamma|-|G|}\sum_{\gamma\in \Gamma_*}(q_2(\gamma)^m+q_2(\gamma)^{-m})  \\
 &+\frac{a^2}{4}\sum_{(k,\varepsilon),(k',\varepsilon')\in J_1}(\frac{1}{|\Gamma|-|G|}+\delta_{(k,\varepsilon),\overline{(k,\varepsilon)}})
 \frac{q_1(k)^m}{q_1(k')^m}\\
 &+\frac{2a^2}{|\Gamma|-|G|}\sum_{k\in K\setminus|\{0\},\; g\in G_*}(\frac{q_1(k)^m}{q_1(g)^m}+\frac{q_1(m)}{q_1(k)})
 +\frac{2ab}{|\Gamma|-|G|}\sum_{k\in K\setminus\{0\},\;\gamma\in \Gamma_*}(\frac{q_1(k)^m}{q_2(\gamma)^m}+\frac{q_2(\gamma)^m}{q_1(k)^m})\\
 &+a^2\sum_{g,h\in G_*}(\frac{4}{|\Gamma|-|G|}+\delta_{g,h})\frac{q_1(g)^m}{q_1(h)^m}
 +b^2\sum_{\gamma,\xi\in \Gamma_*}(\frac{4}{|\Gamma|-|G|}-\delta_{\gamma,\xi})\frac{q_2(\gamma)^m}{q_2(\xi)^m}\\
 &+\frac{4ab}{|\Gamma|-|G|}\sum_{g\in G_*,\;\gamma\in \Gamma_*}(\frac{q_1(g)^m}{q_2(\gamma)^m}+\frac{q_2(\gamma)^m}{q_1(g)^m})\\
 &=\frac{a^2-b^2}{2}+ \frac{(a+b)^2}{|\Gamma|-|G|} +\frac{(a+b)(C_m+C_{-m})}{|\Gamma|-|G|} 
 +\frac{(a+b)(A_m+A_{-m})}{|\Gamma|-|G|}
 +\frac{(a+b)(B_{m}+B_{-m})}{|\Gamma|-|G|}  \\
 &+\frac{C_mC_{-m}}{|\Gamma|-|G|}+\frac{3a^2}{2}
 +\frac{C_mA_{-m}+A_mC_{-m}}{|\Gamma|-|G|}
 +\frac{C_mB_{-m}+B_mC_{-m}}{|\Gamma|-|G|}\\
 &+ \frac{A_mA_{-m}}{|\Gamma|-|G|}+a^2|G_*|+ \frac{B_mB_{-m}}{|\Gamma|-|G|}-b^2|\Gamma_*|
 +\frac{A_mB_{-m}+B_mA_{-m}}{|\Gamma|-|G|}\\
 &=\frac{|a+b+A_m+B_m+C_m|^2}{|\Gamma|-|G|}.
\end{align*}

\begin{align*}
\lefteqn{\nu_m((k,\varepsilon))} \\
 &=N^{(k,\varepsilon)}_{0,(k,\varepsilon)}\frac{a(a-b)}{4}(q_1(k)^m+q_1(k)^{-m})+N^{(k,\epsilon)}_{\pi,\pi}\frac{(a+b)^2}{4}\\
 & +\sum_{(l,\delta)\in J_1}N^{(k,\varepsilon)}_{\pi,(l,\delta)}\frac{a(a+b)}{4}(q_1(l)^m+q_1(l)^{-m})
 +\sum_{g\in G_*}N^{(k,\varepsilon)}_{\pi,g}\frac{a(a+b)}{2}(q_1(g)^m+q_1(g)^{-m})\\
 &+\sum_{\gamma\in \Gamma_*}N^{(k,\varepsilon)}_{\pi,\gamma}\frac{b(a+b)}{2}(q_2(\gamma)^m+q_2(\gamma)^{-m})
 +\sum_{(l,\delta),(l',\delta')\in J_1}N^{(k,\varepsilon)}_{(l,\delta),(l',\delta')}\frac{a^2}{4}\frac{q_1(l)^m}{q_1(l')^m}\\
 &+\sum_{(l,\delta)\in J_1,\;g\in G_*}N^{(k,\varepsilon)}_{(l,\delta),g}\frac{a^2}{2}
 (\frac{q_1(l)^m}{q_1(g)^m}+\frac{q_1(g)^m}{q_1(l)^m})
 +\sum_{(l,\delta)\in J_1,\; \gamma\in \Gamma_*}N^{(k,\varepsilon)}_{(l,\delta),\gamma}\frac{ab}{2}(\frac{q_1(l)^m}{q_2(\gamma)^m}+\frac{q_2(\gamma)^m}{q_1(l)^m}) \\
 &+\sum_{g,h\in G_*}N^{(k,\varepsilon)}_{g,h}a^2\frac{q_1(g)^m}{q_1(h)^m}
 +\sum_{g\in G_*,\;\gamma\in \Gamma}N^{(k,\varepsilon)}_{g,\gamma}ab(\frac{q_1(g)^m}{q_2(\gamma)^m}
 +\frac{q_2(\gamma)^m}{q_1(g)^m})
 +\sum_{\gamma,\xi\in \Gamma_*}N^{(k,\varepsilon)}_{\gamma,\xi}b^2\frac{q_2(\gamma)^m}{q_2(\xi)^m}\\
 &=\frac{a(a-b)}{4}(q_1(k)^m+q_1(k)^{-m})+\frac{(a+b)^2}{2(|\Gamma|-|G|)}\\
 & +\frac{a(a+b)}{4}\sum_{(l,\delta)\in J_1}(\frac{1}{|\Gamma|-|G|}+\delta_{(k,\varepsilon),(l,\delta)}) (q_1(l)^m+q_1(l)^{-m})
 +\frac{a(a+b)}{|\Gamma|-|G|}\sum_{g\in G_*}(q_1(g)^m+q_1(g)^{-m})\\
 &+\frac{b(a+b)}{|\Gamma|-|G|}\sum_{\gamma\in \Gamma_*}(q_2(\gamma)^m+q_2(\gamma)^{-m})
 +\frac{a^2}{8}\sum_{(l,\delta),(l',\delta')\in J_1}(\frac{1}{|\Gamma|-|G|}+\delta_{l+l',k}) \frac{q_1(l)^m}{q_1(l')^m}\\
 &+\frac{a^2c^2}{8}\sum_{(l,\delta),(l',\delta')\in J_1}(\delta\delta'\inpr{l}{l+k}\delta_{l,l'}+\delta'\overline{\varepsilon}\inpr{l'}{l'+l}\delta_{l',k}
 +\overline{\varepsilon}\delta\inpr{k}{k+l'}\delta_{k,l}) \frac{q_1(l)^m}{q_1(l')^m}\\
 &+\frac{a^2}{2}\sum_{(l,\delta)\in J_1,\;g\in G_*}(\frac{1}{|\Gamma|-|G|}+c^2\overline{\varepsilon}\delta\inpr{k}{l+g}\delta_{l,k})
 (\frac{q_1(l)^m}{q_1(g)^m}+\frac{q_1(g)^m}{q_1(l)^m})\\
 &+\frac{ab}{2(|\Gamma|-|G|)}\sum_{(l,\delta)\in J_1,\; \gamma\in \Gamma_*}(\frac{q_1(l)^m}{q_2(\gamma)^m}+\frac{q_2(\gamma)^m}{q_1(l)^m})
 +a^2\sum_{g,h\in G_*}(\frac{2}{|\Gamma|-|G|}+\delta_{g+h,k}+\delta_{g-h,k})\frac{q_1(g)^m}{q_1(h)^m}\\
 &+\frac{2ab}{|\Gamma|-|G|}\sum_{g\in G_*,\;\gamma\in \Gamma}(\frac{q_1(g)^m}{q_2(\gamma)^m}+\frac{q_2(\gamma)^m}{q_1(g)^m})
 +\frac{2b^2}{|\Gamma|-|G|}\sum_{\gamma,\xi\in \Gamma_*}\frac{q_2(\gamma)^m}{q_2(\xi)^m}\\
 &=\frac{a(a-b)}{4}(q_1(k)^m+q_1(k)^{-m})+\frac{(a+b)^2}{2(|\Gamma|-|G|)}
  + \frac{(a+b)(C_m+C_{-m}+A_m+A_{-m}+B_m+B_{-m} )}{2(|\Gamma|-|G|)}\\
  &+ \frac{a(a+b)}{4}(q_1(k)^m+q_1(k)^{-m})
 +\frac{C_mC_{-m}}{2(|\Gamma|-|G|)}+\frac{a^2}{2}\frac{q_1(l)^m}{q_1(k-l)^m}-\frac{a^2}{2}(q_1(k)^m+q_1(k)^{-m})\\
 &+\frac{A_mC_{-m}+C_mA_{-m}+B_mC_{-m}+C_mB_{-m}}{2(|\Gamma|-|G|)}\\
 &+\frac{A_mA_{-m}+A_mB_{-m}+B_mA_{-m}+B_mB_{-m}}{2(|\Gamma|-|G|)}+a^2\sum_{g\in G_*}\frac{q_1(g)^m}{q_1(g-k)^m}\\
 \end{align*}
 \begin{align*}
 &=\frac{|\cG(q_1,m)+\cG(q_2,m)|^2}{2(|\Gamma|-|G|)}+\frac{a^2}{2}\sum_{g\in G}\frac{q_1(g)^m}{q_{1}(k-g)^m} \\
 &=\frac{|\cG(q_1,m)+\cG(q_2,m)|^2}{2(|\Gamma|-|G|)}+\frac{\delta_{mk,0}q_1(k)^m}{2}.\\
\end{align*}

\begin{align*}
\lefteqn{\nu_m(g)=N^g_{0,g}\frac{a(a-b)}{2}(q_1(g)^m+q_1(g)^{-m})+N^g_{\pi,\pi}\frac{(a+b)^2}{4}} \\
 &+\sum_{(k,\varepsilon)\in J_1}N^g_{\pi,(k,\varepsilon)}\frac{a(a+b)}{4}(q_1(k)^m+q_1(k)^{-m}) 
 +\sum_{h\in G_*}N^g_{\pi,h}\frac{a(a+b)}{2}(q_1(h)^m+q_1(h)^{-m})\\
 &+\sum_{\gamma\in \Gamma_*}N^g_{\pi,\gamma}\frac{b(a+b)}{2}(q_2(\gamma)^m+q_2(\gamma)^{-m})
 +\sum_{(k,\varepsilon),(k',\varepsilon')}N^g_{(k,\varepsilon),(k',\varepsilon')}\frac{a^2}{4}\frac{q_1(k)^m}{q_1(k')^m}\\
 &+\sum_{(k,\varepsilon)\in J_1;h\in G_*}N^g_{(k,\varepsilon),h}\frac{a^2}{2}(\frac{q_1(k)^m}{q_1(h)^m}+\frac{q_1(h)^m}{q_1(k)^m})
 +\sum_{(k,\varepsilon)\in J_1;\gamma\in \Gamma_*}N^g_{(k,\varepsilon),\gamma}\frac{ab}{2}(\frac{q_1(k)^m}{q_2(\gamma)^m}+\frac{q_2(\gamma)^m}{q_1(k)^m})\\
 &+\sum_{h,h'\in G_*}N^g_{h,h'}a^2\frac{q_1(h)^m}{q_1(h')^m}
 +\sum_{h\in G_*,\;\gamma\in \Gamma_*}N^g_{h,\gamma}ab(\frac{q_1(h)^m}{q_2(\gamma)^m}+\frac{q_2(\gamma)^m}{q_1(h)^m})
 +\sum_{\gamma,\xi\in \Gamma_*}N^g_{\gamma,\xi}b^2\frac{q_2(\gamma)^m}{q_2(\xi)^m}\\
 &=\frac{a(a-b)}{2}(q_1(g)^m+q_1(g)^{-m})+\frac{(a+b)^2}{|\Gamma|-|G|}
 +\frac{a(a+b)}{|\Gamma|-|G|}\sum_{k\in K\setminus\{0\}}(q_1(k)^m+q_1(k)^{-m})\\ 
 &+\frac{a(a+b)}{2}\sum_{h\in G_*}(\frac{4}{|\Gamma|-|G|}+\delta_{g,h})(q_1(h)^m+q_1(h)^{-m})
 +\frac{2b(a+b)}{|\Gamma|-|G|}\sum_{\gamma\in \Gamma_*}(q_2(\gamma)^m+q_2(\gamma)^{-m})\\
 &+\frac{a^2}{4}\sum_{(k,\varepsilon),(k',\varepsilon')}(\frac{1}{|\Gamma|-|G|}+c^2\varepsilon\varepsilon'\inpr{k}{k+g}\delta_{k,k'})\frac{q_1(k)^m}{q_1(k')^m}\\
 &+\frac{a^2}{2}\sum_{(k,\varepsilon)\in J_1;h\in G_*}(\frac{2}{|\Gamma|-|G|}+\delta_{k+g+h,0}+\delta_{k+g-h,0})(\frac{q_1(k)^m}{q_1(h)^m}+\frac{q_1(h)^m}{q_1(k)^m})\\
 &+\frac{2ab}{|\Gamma|-|G|}\sum_{k\in K\setminus\{0\};\gamma\in \Gamma_*}(\frac{q_1(k)^m}{q_2(\gamma)^m}+\frac{q_2(\gamma)^m}{q_1(k)^m})\\
 &+a^2\sum_{h,h'\in G_*}(\frac{4}{|\Gamma|-|G|}+\delta_{g+h+h',0}+\delta_{-g+h+h',0}+\delta_{g-h+h',0}+\delta_{g+h-h',0})\frac{q_1(h)^m}{q_1(h')^m}\\
 &+\frac{4ab}{|\Gamma|-|G|}\sum_{h\in G_*,\;\gamma\in \Gamma_*}(\frac{q_1(h)^m}{q_2(\gamma)^m}+\frac{q_2(\gamma)^m}{q_1(h)^m})+\frac{4b^2}{|\Gamma|-|G|}\sum_{\gamma,\xi\in \Gamma_*}\frac{q_2(\gamma)^m}{q_2(\xi)^m}\\
 \end{align*}
 \begin{align*}
 &=\frac{a(a-b)}{2}(q_1(g)^m+q_1(g)^{-m})+\frac{(a+b)^2}{|\Gamma|-|G|}
 +\frac{(a+b)(C_m+C_{-m}+A_m+A_{-m}+B_m+B_{-m})}{|\Gamma|-|G|}\\ 
 &+\frac{a(a+b)}{2}(q_1(g)^m+q_1(h)^{-g})+\frac{C_mC_{-m}+C_mA_{-m}+A_mC_{-m}+C_mA_{-m}+A_mC_{-m}}{|\Gamma|-|G|}\\
 &+\frac{a^2}{2}\sum_{(k,\varepsilon)\in J_1;h\in G_*}(\delta_{k+g+h,0}+\delta_{k+g,h})(\frac{q_1(k)^m}{q_1(h)^m}+\frac{q_1(h)^m}{q_1(k)^m})\\
 &+a^2\sum_{h,h'\in G_*}(\delta_{g+h+h',0}+\delta_{h+h',g}+\delta_{g+h',h}+\delta_{g+h,h'})\frac{q_1(h)^m}{q_1(h')^m}\\
 &+\frac{A_mA_{-m}+A_mB_{-m}+A_{-m}B_m+B_mB_{-m}}{|\Gamma|-|G|}\\
 &=\frac{|\cG(q_1)+\cG(q_2)|^2}{|\Gamma|-|G|}
 +a^2\sum_{k\in K}(\frac{q_1(k)^m}{q_1(g+k)^m}+\frac{q_1(g+k)^m}{q_1(k)^m})\\
 &+a^2\sum_{h,h'\in G_*}(\delta_{g+h+h',0}+\delta_{h+h',g}+\delta_{g+h',h}+\delta_{g+h,h'})\frac{q_1(h)^m}{q_1(h')^m}\\
 &=\frac{|\cG(q_1,m)+\cG(q_2,m)|^2}{|\Gamma|-|G|}+\delta_{mg,0}q_1(g)^m.
\end{align*}

\begin{align*}
\lefteqn{\nu_m(\gamma)=N^\gamma_{0,\gamma}\frac{b(a-b)}{2}(q_2(\gamma)^m+q_2(\gamma)^{-m})+N^\gamma_{\pi,\pi}\frac{(a+b)^2}{4}} \\
 &+\sum_{(k,\varepsilon)\in J_2}N^\gamma_{\pi,(k,\varepsilon)}\frac{a(a+b)}{4}(q_1(k)^m+q_1(k)^{-m}) 
 +\sum_{g\in G_*}N^\gamma_{\pi,g}\frac{a(a+b)}{2}(q_1(g)^m+q_1(g)^{-m})\\
 &+\sum_{\xi\in \Gamma_*}N^\gamma_{\pi,\xi}\frac{b(a+b)}{2}(q_2(\xi)^m+q_2(\xi)^{-m}) 
 +\sum_{(k,\varepsilon),(k',\varepsilon')\in J_1}N^\gamma_{(k,\varepsilon),(k',\varepsilon')}\frac{a^2}{4}\frac{q_1(k)^m}{q_1(k')^m}\\
 &+\sum_{(k,\varepsilon)\in J_1,\;g\in G_*}N^\gamma_{(k,\varepsilon),g}\frac{a^2}{2}(\frac{q_1(k)^m}{q_1(g)^m}+\frac{q_1(g)^m}{q_1(k)^m})
 +\sum_{(k,\varepsilon)\in J_1,\;\xi\in \Gamma_*}N^\gamma_{(k,\varepsilon),\xi}\frac{ab}{2}(\frac{q_1(k)^m}{q_2(\xi)^m}+\frac{q_2(\xi)^m}{q_1(k)^m})\\
 &+\sum_{g,h\in G_*}N^\gamma_{g,h}a^2\frac{q_1(g)^m}{q_1(h)^m}
 +\sum_{g\in G_*,\;\xi\in \Gamma_*}N^\gamma_{g,\xi}ab(\frac{q_1(g)^m}{q_2(\xi)^m}+\frac{q_2(\xi)^m}{q_1(g)^m})
 +\sum_{\xi,\xi'\in \Gamma_*}N^\gamma_{\xi,\xi'}b^2\frac{q_2(\xi)^m}{q_2(\xi')^m}\\
 &=\frac{b(a-b)}{2}(q_2(\gamma)^m+q_2(\gamma)^{-m})+\frac{(a+b)^2}{|\Gamma|-|G|}
 +\frac{a(a+b)}{|\Gamma|-|G|}\sum_{k\in K\setminus\{0\}}(q_1(k)^m+q_1(k)^{-m})\\ 
 &+\frac{2a(a+b)}{|\Gamma|-|G|}\sum_{g\in G_*}(q_1(g)^m+q_1(g)^{-m})
 +\frac{b(a+b)}{2}\sum_{\xi\in \Gamma_*}(\frac{4}{|\Gamma|-|G|}-\delta_{\xi,\gamma})(q_2(\xi)^m+q_2(\xi)^{-m}) \\
 &+\frac{a^2}{|\Gamma|-|G|}\sum_{k,k'\in K\setminus\{0\}}\frac{q_1(k)^m}{q_1(k')^m}
 +\frac{2a^2}{|\Gamma|-|G|}\sum_{k\in K\setminus\{0\},\;g\in G_*}(\frac{q_1(k)^m}{q_1(g)^m}+\frac{q_1(g)^m}{q_1(k)^m})\\
 &+\frac{2ab}{|\Gamma|-|G|}\sum_{k\in K\setminus\{0\},\;\xi\in \Gamma_*}(\frac{q_1(k)^m}{q_2(\xi)^m}+\frac{q_2(\xi)^m}{q_1(k)^m})\\
 &+\frac{4a^2}{|\Gamma|-|G|}\sum_{g,h\in G_*}\frac{q_1(g)^m}{q_1(h)^m}
 +\frac{4ab}{|\Gamma|-|G|}\sum_{g\in G_*,\;\xi\in \Gamma_*}(\frac{q_1(g)^m}{q_2(\xi)^m}+\frac{q_2(\xi)^m}{q_1(g)^m})\\
 &+b^2\sum_{\xi,\xi'\in \Gamma_*}(\frac{4}{|\Gamma|-|G|}-\delta_{\gamma+\xi+\xi'}-\delta_{-\gamma+\xi+\xi'}-\delta_{\gamma-\xi+\xi'}-\delta_{\gamma+\xi-\xi'})\frac{q_2(\xi)^m}{q_2(\xi')^m}\\
 &=\frac{b(a-b)}{2}(q_2(\gamma)^m+q_2(\gamma)^{-m})
 +\frac{(a+b)^2+(a+b)(C_m+C_{-m}+A_m+A_{-m}+B_m+B_{-m})}{|\Gamma|-|G|}\\ 
 &-\frac{b(a+b)}{2}(q_2(\gamma)^m+q_2(\gamma)^{-m}) 
 +\frac{C_mC_{-m}+C_mA_{-m}+A_mC_{-m}+C_{m}B_{-m}+B_mC_{-m}}{|\Gamma|-|G|}\\
 &+\frac{A_mA_{-m}+A_mB_{-m}+B_mA_{-m}+B_mB_{-m}}{|\Gamma|-|G|}\\
 &-b^2\sum_{\xi,\xi'\in \Gamma_*}(\delta_{\gamma+\xi+\xi'}+\delta_{-\gamma+\xi+\xi'}+\delta_{\gamma-\xi+\xi'}+\delta_{\gamma+\xi-\xi'})\frac{q_2(\xi)^m}{q_2(\xi')^m}\\
 &=\frac{|\cG(q_1,m)+\cG(q_2,m)|^2}{|\Gamma|-|G|}-b^2(q_2(\gamma)^m+q_2(\gamma)^{-m})\\
 &-b^2\sum_{\xi,\xi'\in \Gamma_*}(\delta_{\gamma+\xi+\xi'}+\delta_{-\gamma+\xi+\xi'}+\delta_{\gamma-\xi+\xi'}+\delta_{\gamma+\xi-\xi'})\frac{q_2(\xi)^m}{q_2(\xi')^m}\\
 &=\frac{|\cG(q_1,m)+\cG(q_2,m)|^2}{|\Gamma|-|G|}-\delta_{m\gamma,0}q_2(\gamma)^m. 
\end{align*}
\end{proof}

\section{Verification of Conjecture 4.5 for $G=\Z_4$.}
\subsection{Realization in $\cO_3\rtimes \Z_2$.}
Assume that  a fusion category $\cC$ with the following fusion rules 
is realized in $\End(M)$ for a type III factor $M$. 
$$[\alpha^2]=[\id],$$
$$[\alpha][\rho]=[\rho][\alpha],$$
$$[\rho^2]=[\alpha]+[\alpha][\rho].$$
We also assume that $\alpha$ has a non-trivial associator. 
Such a fusion category certainly exists (see \cite{MR3635673}, \cite{MR3306607}). 

We may assume that there exists a unitary $U\in M$ satisfying $\alpha^2=\Ad U$, 
$\alpha(U)=-U$. 
Moreover, we may assume $U^2=1$, and $\alpha$ as a $\Z_4$-action is stable (see \cite{MR0448101}). 
From $[\alpha][\rho]=[\rho][\alpha]$, there exists a unitary $V\in M$ satisfying 
$\Ad V\circ \alpha \circ \rho=\rho\circ \alpha$. 
Then multiplying $V$ by a complex number of modulus 1 if necessary, we may assume that $V$ is a $\alpha$-cocycle 
as a $\Z_4$-action, and stability of $\alpha$ implies that it is a cobounday. 
Thus there exists a unitary $W\in M$ satisfying $V=W^*\alpha(W)$, 
which means that $\Ad W\circ \rho$ commutes with $\alpha$. 
Therefore from the beginning we may assume $\alpha\circ \rho=\rho\circ \alpha$. 

We choose isometries $s_0\in (\alpha,\rho^2)$, $s_1\in (\rho,\rho^2)$, $s_2\in (\alpha\circ \rho,\rho^2)$. 
Since $\alpha$ commutes with $\rho$, we have $\alpha(s_0)\in (\alpha,\rho^2)$, $\alpha(s_1)\in (\rho,\rho^2)$, 
$\alpha(s_2)\in (\alpha\circ \rho,\rho^2)$, and there exist 4-th roots of unity $\xi_i$, $i=0,1,2$, 
satisfying 
$\alpha(s_0)=\xi_0 s_0$, $\alpha(s_1)=\xi_1 s_1$, $\alpha(s_2)=\xi_2 s_2$. 
Choosing a unitary $Z\in M$ satisfying $\alpha(Z)=\xi_2^{-1}Z$, and replacing $\rho$ with $\Ad Z\circ \rho$ 
if necessary, we may and do always assume $\xi_2=1$, and hence $\alpha(s_2)=s_2$. 

Since $U\rho(U)\in (\rho,\rho)$, there exists $\epsilon\in \{1,-1\}$ satisfying $\rho(U)=\epsilon U$. 
On one hand, we have $\rho^2(U)=\epsilon^2U=U$, and on the other hand, 
$$\rho^2(U)=s_0\alpha(U)s_0^*+s_1\rho(U)s_1^*+s_2\alpha(\rho(U))s_2^*
=-s_0Us_0^*+\epsilon s_1Us_1^*-\epsilon s_2Us_2^*,$$
which implies $\xi_0^2=-1$, $\xi_1^2=\epsilon$, $\xi_2^2=-\epsilon$, 
and so we get $\epsilon=-1$ and $\xi_0^2=\xi_1^2=-1$. 

Let $\brho=\alpha^{-1}\circ \rho$. 
Then $\brho$ is the dual of $\rho$, and we may choose $R_\rho=s_0$. 
Since $\brho\circ\rho=\rho\circ\brho$, we have $R_{\brho}=\zeta s_0$ with $|\zeta|=1$. 
Let $d=1+\sqrt{2}$. 
Then since $R_{\brho}^*\rho(R_\rho)=1/d$, we get $s_0^*\rho(s_0)=\zeta/d$, 
which implies 
$$\frac{\zeta}{d}s_0=\rho(s_0)^*\rho^2(s_0)s_0=\rho(s_0)^*s_0\alpha(s_0)=\frac{\overline{\zeta}\xi}{d}s_0.$$
Thus we get $\xi_1=\zeta^2$, and $\zeta^4=-1$. 

Note that Frobenius reciprocity implies 
$\sqrt{d}s_1^*\rho(s_0)\in (\rho,\rho\brho)=(\alpha\rho,\rho^2)$ and 
$\sqrt{d}s_2^*\rho(s_0)\in (\alpha\rho,\rho\brho)=(\rho,\rho^2)$ are isometries.  
We may assume $\sqrt{d}s_1^*\rho(s_0)=s_2$ by redefining $s_2$, and 
there exists $\nu\in \T$ satisfying $\sqrt{d}s_2^*\rho(s_0)=\nu s_1$. 
Applying $\alpha$ to these, we get $\xi_0=\xi_1$, and $\xi_1=\zeta^2$.

So far we have shown that we can choose $\alpha$, $\rho$, $U$, $s_i$, $i=0,1,2$, satisfying 
$\alpha(s_0)=\zeta^2 s_0$, $\alpha(s_1)=\zeta^2 s_1$, $\alpha(s_2)=s_2$, $\alpha(U)=\rho(U)=-U$, 
$\alpha^2=\Ad U$, $Us_0=-s_0U$, $Us_1=-s_1U$, $Us_2=s_2U$, and 
$$\rho(s_0)=\frac{\zeta}{d}s_0+\frac{1}{\sqrt{d}}(s_1s_2+\nu s_2s_1),$$
where $\zeta$ is a number satisfying $\zeta^4=-1$. 

Frobenius reciprocity implies that $\sqrt{d}\rho(s_1)^*s_0\in (\rho,\brho\rho)=(\alpha\rho,\rho^2)$ 
is an isometry, and there exists $\eta_1\in \T$ satisfying $\sqrt{d}s_0^*\rho(s_1)=\eta_1s_2^*$.  
Since $\rho(s_1)s_0$ is an isometry in $(\id,\rho^3\alpha^{-1})=\C s_1s_0$, 
there exists $\eta_2\in \T$ satisfying $\rho(s_1)s_0=\eta_2 s_1s_0$. 
Since $s_1^*\rho(s_1)\in (\rho^2,\rho^2)$ and 
$s_2^*\rho(s_1)\in (\rho^2,\alpha\rho^2)=\C s_1s_2^*+ \C s_2Us_1^*$, 
there exist $A,B,C,D\in \C$ satisfying 
$$\rho(s_1)=\frac{\eta_1}{\sqrt{d}}s_0s_2^*+\eta_2 s_1s_0s_0^*+
A s_1s_1s_1^*+B s_1s_2s_2^*+C s_2s_1s_2^*+Ds_2s_2Us_1^*.$$

Since $s_2=\sqrt{d}s_1^*\rho(s_0)$, we get 
\begin{align*}
\lefteqn{ \rho(s_2)=\sqrt{d}\rho(s_1^*)\rho^2(s_0)              } \\
 &=\sqrt{d}\rho(s_1)^*(\zeta^2 s_0s_0s_0^*+s_1\rho(s_0)s_1^*+\zeta^2 s_2\rho(s_0)s_2^*) \\
 &=\zeta^2\overline{\eta_1}s_2s_0s_0^*+\sqrt{d}(\overline{\eta_2}s_0s_0^*+\overline{A}s_1s_1^*+\overline{B}s_2s_2^*)\rho(s_0)s_1^*
 +\zeta^2\sqrt{d}(\overline{C}s_2s_1^*+\overline{D}s_1Us_2^*)\rho(s_0)s_2^* \\
 &=\frac{\zeta\overline{\eta_2}}{\sqrt{d}}s_0s_1^*+\zeta^2\overline{\eta_1}s_2s_0s_0^*+\zeta^2\overline{C}s_2s_2s_2^*+\nu\overline{B}s_2s_1s_1^*
 +\overline{A}s_1s_2s_1^*-\zeta^2\nu\overline{D}s_1s_1Us_2^*.
\end{align*}

On the other hand, since $s_1=\overline{\nu}\sqrt{d}s_2^*\rho(s_0)$, 
\begin{align*}
\lefteqn{\rho(s_1)=\overline{\nu}\sqrt{d}\rho(s_2)^*\rho^2(s_0)} \\
 &=\overline{\nu}\sqrt{d}\rho(s_2)^* (\zeta^2 s_0s_0s_0^*+s_1\rho(s_0)s_1^*+\zeta^2 s_2\rho(s_0)s_2^*) \\
 &=\overline{\nu}\zeta \eta_2 s_1s_0s_0^*+ \overline{\nu}\sqrt{d}(As_1s_2^*-\overline{\zeta}^2\overline{\nu}Ds_2Us_1^*)\rho(s_0)s_1^*\\
 &+\zeta^2\overline{\nu}\sqrt{d}(\overline{\zeta}^2\eta_1 s_0s_0^*
 +\overline{\nu}Bs_1s_1^*+\overline{\zeta}^2C s_2s_2^*)\rho(s_0)s_2^*\\
 &=\frac{\zeta\overline{\nu}\eta_1}{\sqrt{d}}s_0s_2^*+\zeta\overline{\nu}\eta_2s_1s_0s_0^*
 +\zeta\overline{\nu}As_1s_1s_1^*+\zeta^2\overline{\nu}^2B s_1s_2s_2^*+Cs_2s_1s_2^*-\overline{\zeta}^2\overline{\nu}^2 s_2Us_2s_1^*,
\end{align*}
which shows $\nu=\zeta$. 

We summarize what we have gotten so far: 
$$\rho(s_0)=\frac{\zeta}{d}s_0+\frac{1}{\sqrt{d}}(s_1s_2+\zeta s_2s_1),$$
$$\rho(s_1)=\frac{\eta_1}{\sqrt{d}}s_0s_2^*+\eta_2 s_1s_0s_0^*+
A s_1s_1s_1^*+B s_1s_2s_2^*+C s_2s_1s_2^*+Ds_2s_2Us_1^*,$$
$$\rho(s_2)=\frac{\zeta\overline{\eta_2}}{\sqrt{d}}s_0s_1^*+\zeta^2\overline{\eta_1}s_2s_0s_0^*+\zeta^2\overline{C}s_2s_2s_2^*+\zeta\overline{B}s_2s_1s_1^*
 +\overline{A}s_1s_2s_1^*+\overline{\zeta}\overline{D}s_1s_1Us_2^*.$$

\begin{theorem} Let the notation be as above. 
Then $\zeta=\zeta_8^{\pm 1}$, $A=\frac{1-\sqrt{2}\mp i}{2}=\frac{\zeta_{16}^{\pm 11}}{\sqrt{2+\sqrt{2}}}$, $|D|=2^{-\frac{1}{4}}$, 
$Us_0=-s_0U$, $Us_1=-s_1U$, $Us_2=s_2U$, $U^2=1$, and 
$\alpha(s_0)=\zeta^2s_0$, $\alpha(s_1)=\zeta^2s_1$, $\alpha(s_2)=s_2$, $\alpha(U)=-U$, 
$\rho(U)=-U$, 
$$\rho(s_0)=\frac{\zeta}{d}s_0+\frac{1}{\sqrt{d}}(s_1s_2+\zeta s_2s_1),$$
$$\rho(s_1)=\frac{\overline{\zeta}}{\sqrt{d}}s_0s_2^*+\zeta^3 s_1s_0s_0^*+
A s_1s_1s_1^*+\zeta^3 A s_1s_2s_2^*+\overline{\zeta}A s_2s_1s_2^*+Ds_2s_2Us_1^*,$$
$$\rho(s_2)=\frac{\overline{\zeta}^2}{\sqrt{d}}s_0s_1^*+\zeta^3s_2s_0s_0^*
+\zeta^3\overline{A}s_2s_2s_2^*+\overline{\zeta}^2\overline{A}s_2s_1s_1^*
 +\overline{A}s_1s_2s_1^*+\overline{\zeta}\overline{D}s_1s_1Us_2^*,$$
$\alpha^2=\Ad U$, $U^2=1$. 
\end{theorem} 

\begin{proof} 
Since $s_1=\zeta\overline{\eta_2}\sqrt{d}\rho(s_2)^*s_0$, 
\begin{align*}
\lefteqn{\rho(s_1)=\zeta\overline{\eta_2}\sqrt{d}\rho^2(s_2^*)\rho(s_0)} \\
 &=\zeta\overline{\eta_2}\sqrt{d}\rho^2(s_2^*)(\frac{\zeta}{d}s_0+\frac{1}{\sqrt{d}}s_1s_2+\frac{\zeta}{\sqrt{d}}s_2s_1) \\
 &=\frac{\zeta^2\overline{\eta_2}}{\sqrt{d}}s_0\alpha(s_2^*)+\zeta\overline{\eta_2}s_1\rho(s_2)^*s_2+\zeta^2\overline{\eta_2}s_2\rho(\alpha(s_2))^*s_1\\
 &=\frac{\zeta^2\overline{\eta_2}}{\sqrt{d}}s_0s_2^*+\zeta\overline{\eta_2}s_1
 (\overline{\zeta}^2\eta_1s_0s_0^*+\overline{\zeta}Bs_1s_1^*+\overline{\zeta}^2Cs_2s_2^*)
 +\zeta^2\overline{\eta_2}s_2(As_1s_2^*+\zeta D s_2Us_1^*).
\end{align*} 
Thus we get $\eta_1=\zeta^2\overline{\eta_2}$, $\eta_2=\overline{\zeta}\overline{\eta_2}\eta_1$, 
$A=\overline{\eta_2}B$, $B=\overline{\zeta}\overline{\eta_2}C$, $C=\zeta^2\overline{\eta_2}A$, $D=\zeta^3\overline{\eta_2}D$. 
This implies $|A|=|B|=|C|$, and $\rho(s_1)^*\rho(s_1)=1$ implies 
$$1=\frac{1}{d}+2|A|^2,\quad 1=|A|^2+|D|^2.$$
Since neither $A$ nor $D$ is 0, and we get  $\eta_1=\overline{\zeta}$, 
$\eta_2=\zeta^3$, $B=\zeta^3A$, and $C=\overline{\zeta}A$.  

Now the Cuntz algebra relation of $\rho(s_0)$, $\rho(s_1)$, $\rho(s_2)$ is equivalent to 
$$\frac{1}{d}+(\zeta^5+1)A=0,$$
$$\frac{1}{d}+(\zeta^2+1)A^2=0,$$
$$|A|^2=\frac{1}{2+\sqrt{2}},$$
$$|A|^2+|D|^2=1.$$
Solving these, we get the statement. 
\end{proof}

\begin{remark} The phase of $D$ cannot be determined from $\cC$ because replacing $s_1$ with $cs_1$ and 
$s_2$ with $\overline{c}s_2$ changes $D$ into $c^4D$ though other coefficients do not change. 
In other words, we may always assume $D=2^{-1/4}$. 
\end{remark}
\subsection{$\Tube \cC$}
Let 
$$\Lambda=\dim \cC=2+2d^2=2(2+2d)=4(2+\sqrt{2}).$$

It is easy to show that $\alpha$ has a unique half-braiding $\cE_\alpha(\alpha)=\zeta^2$, $\cE_\alpha(\rho)=\zeta^2$, 
$\cE_\alpha(\alpha\rho)=-1$, and there is a unique extension $\tilde{\alpha}\in \cZ(\cC)$ of $\alpha$. 
For simplicity, we denote $\alpha=\tilde{\alpha}$ if there there is no possibility of confusion. 
Since  
$$\left(
\begin{array}{cc}
S_{0,0} &S_{0,\alpha}  \\
S_{\alpha,0} &S_{\alpha,\alpha} 
\end{array}
\right)=\frac{1}{\Lambda}\left(
\begin{array}{cc}
1 &1  \\
1 &-1 
\end{array}
\right),
$$
the category generated by $\alpha$ is non-degenerate. 
We set $\cD=\cZ(\cC)\cap \{\alpha\}'$, which is the modular tensor category with 
$$\cZ(\cC)=\{0,\alpha\}\boxtimes \cD.$$
Thus the modular data $(S,T)$ of $\cZ(\cC)$ has the tensor product decomposition 
$$(S,T)=(\left(
\begin{array}{cc}
\frac{1}{\sqrt{2}} &\frac{1}{\sqrt{2}}  \\
\frac{1}{\sqrt{2}} &-\frac{1}{\sqrt{2}} 
\end{array}
\right)\otimes S',
\left(
\begin{array}{cc}
1 &0  \\
0 &\zeta^2 
\end{array}
\right)\otimes T'),$$
where $(S',T')$ are the modular data of $\cD$. 
For $X,Y\in \Irr(\cD)$, we have $S'_{X,Y}=\sqrt{2}S_{X,Y}$. 
We set $s(\alpha,X)=\frac{S_{\alpha,X}}{|S_{\alpha,X}|}$. 
Then by definition, 
$$\Irr (\cD)=\{X\in \Irr(\cZ(\cC));\;s(\alpha,X)=1\}.$$

We determine the structure of $\cA_0$, which is an abelian algebra of dimension 4. 
Let $p(0,0)=\frac{1}{2}(1_0+(0\;\alpha|1|\alpha\;0))$, $p(0,1)=\frac{1}{2}(1_0-(0\;\alpha|1|\alpha\;0))$, 
which are projections. 
There exists a unique half-braiding for $\id$, and  
$$z(0)=\frac{1}{4(2+\sqrt{2})}(1_0+(0\;\alpha|1|\alpha\;0)+(1+\sqrt{2})(0\;\rho|1|\rho\;0)+(1+\sqrt{2})(0\;\alpha\rho|1|\alpha\rho)\;0)$$
is a minimal central projection of $\Tube \cC$. 
Let 
$$E(0,0)=\frac{1}{4(2-\sqrt{2})}(1_0+(0\;\alpha|1|\alpha\;0)+(1-\sqrt{2})(0\;\rho|1|\rho\;0)+(1-\sqrt{2})(0\;\alpha\rho|1|\alpha\rho)\;0),$$
$$E(0,1)_\epsilon=\frac{1}{2}p(0,1)(1+\varepsilon\zeta^2 (0\;\rho|1|\rho\;0)),\quad \varepsilon \in \{1,-1\},$$
which are projections. 
Then we have 
$$\cA_0=\C z(0)\oplus \C E(0,0)\oplus \C E(0,1)_+\oplus \C E(0,1)_-$$
as an algebra, and $E(0,0)(0\;\rho|1|\rho\;0)=(1-\sqrt{2})E(0,0)$, 
$E(0,1)_\epsilon (0;\rho|1|\rho\;0)=-\epsilon\zeta^2E(0,1)_\epsilon$. 

The left $\cA_0$ modules $\cA_{0,\rho}$ and $\cA_{0,\alpha\rho}$ have bases 
$\{p(0,0)(0\;\rho|s_1|\rho\;\rho),p(0,1)(0\;\rho|s_1|\rho\;\rho)\}$ 
and $\{p(0,0)(0\;\rho|s_2U|\rho\;\alpha\rho),p(0,1)(0\;\rho|s_2U|\rho\;\alpha\rho)\}$ 
respectively. 
Since $z(0)$ acts on $\cA_{0,\rho}$ and $\cA_{0,\alpha\rho}$ trivially, 
$E(0,0)$ acts on $p(0,0)\cA_{0,\rho}$ and $p(0,0)\cA_{0,\alpha\rho}$ with multiplicity 1. 

On the other hand, we have 
 \begin{align*}
\lefteqn{(0\;\rho|1|\rho\;0)p(0,1)(0\;\rho|s_1|\rho\;\rho)=p(0,1)(0\;\rho|1|\rho\;0)(0\;\rho|s_1|\rho\;\rho)} \\
 &=p(0,1)((0\;\rho|s_1^*\rho(s_1)s_1|\rho\;\rho)+(0\;\alpha\rho|s_2^*\rho(s_1)s_2|\alpha\rho\;\rho)) \\
 &=Ap(0,1)((0\;\rho|s_1|\rho\;\rho)+\overline{\zeta}(0\;\alpha\rho|s_1|\alpha\rho\;\rho))  \\
 &=Ap(0,1)((0\;\rho|s_1|\rho\;\rho)+\zeta^5(0\;\alpha|1|\alpha\;0)(0\;\rho|s_1|\rho\;\rho))  \\
 &=A(1-\zeta^5)p(0,1)(0\;\rho|s_1|\rho\;\rho)\\
 &=-\zeta^2  p(0,1)(0\;\rho|s_1|\rho\;\rho)
\end{align*}
and
 \begin{align*}
\lefteqn{(0\;\rho|1|\rho\;0)p(0,1)(0\;\rho|s_2U|\rho\;\alpha\rho)
=p(0,1)(0\;\rho|1|\rho\;0)(0\;\rho|s_2U|\rho\;\alpha\rho)} \\
 &=p(0,1)((0\;\rho|s_1^*\rho(s_2U)s_1|\rho\;\alpha\rho)+(0\;\alpha\rho|s_2^*\rho(s_2U)s_2|\alpha\rho\;\alpha\rho)) \\
 &=p(0,1)((0\;\rho|s_1^*\rho(s_2)s_1U|\rho\;\alpha\rho)+(0\;\alpha\rho|-s_2^*\rho(s_2)s_2U|\alpha\rho\;\alpha\rho)) \\
 &=\overline{A}p(0,1)((0\;\rho|s_2U|\rho\;\alpha\rho)-\zeta^3(0\;\alpha\rho|s_2U|\alpha\rho\;\alpha\rho)) \\
 &=\overline{A}p(0,1)((0\;\rho|s_2U|\rho\;\alpha\rho)+\zeta^3(0\;\alpha|1|\alpha\;0)(0\;\rho|s_2U|\rho\;\alpha\rho)) \\
 &=\overline{A}(1-\zeta^3)p(0,1)((0\;\rho|s_2U|\rho\;\alpha\rho)) \\
 &=\zeta^2p(0,1)((0\;\rho|s_2U|\rho\;\alpha\rho)) .\\
\end{align*}

Summing up the above argument, we have the following.

\begin{lemma} Let $\pi=\id\oplus \rho\oplus \alpha\rho$, $\varphi=\id\oplus \rho$, and $\psi=\id\oplus \alpha\rho$.
\begin{itemize}
\item[$(1)$] $\pi$ has a unique half-braiding with $e(\tilde{\pi})_{0,0}=E(0,0)$, and $\cE_\pi(0)_{0,0}=1$, 
$\cE_\pi(\alpha)_{0,0}=1$, $\cE_\pi(\rho)_{0,0}=$ 
$\cE_\pi(\alpha\rho)_{0,0}=-\frac{1}{d^2}$.
\item[$(2)$] $\varphi$ has a unique half-braiding with $e(\tilde{\varphi})_{0,0}=E(0,1)_+$, and 
$\cE_\varphi(\alpha)_{0,0}=-1$, $\cE_\varphi(\rho)_{0,0}=\zeta^2/d$. 
\item[$(3)$] $\psi$ has a unique half-braiding with $e(\tilde{\psi})_{0,0}=E(0,1)_-$, and 
$\cE_\psi(\alpha)_{0,0}=-1$, $\cE_\psi(\rho)_{0,0}=-\zeta^2/d$.
\end{itemize}
The entries of $S$ and $T$ for these half-braidings are 
$$S_{0,0}=S_{\tilde{\pi},\tilde{\pi}}=\frac{2-\sqrt{2}}{8},\quad S_{0,\tilde{\pi}}=\frac{2+\sqrt{2}}{8},\quad 
S_{0,\tvarphi}=S_{0,\tpsi}=\frac{1}{4},$$
$$S_{\tpi,\tvarphi}=S_{\tpi,\tpsi}=\frac{1}{4},$$
$$S_{\tvarphi,\tvarphi}=\frac{1\mp i}{4},\quad S_{\tvarphi,\tpsi}=\frac{1\pm i}{4},\quad S_{\tpsi,\tpsi}=\frac{1\mp i}{4}.$$
$$T_{0,0}=T_{\tpi,\tpi}=T_{\tvarphi,\tvarphi}=T_{\tpsi,\tpsi}=1.$$
\end{lemma}

The half-braiding equations for $\rho$ are 
$$\cE_\rho(\alpha)^2=-1,$$
$$\rho(s_0)\cE_\rho(\alpha)=\cE_\rho(\rho)\rho(\cE_\rho(\rho))s_0,$$
$$\rho(s_1)\cE_\rho(\rho)=\cE_\rho(\rho)\rho(\cE_\rho(\rho))s_1,$$
$$\rho(s_2)\cE_\rho(\alpha\rho)=\cE_\rho(\rho)\rho(\cE_\rho(\rho))s_2,$$
where $\cE_\rho(\alpha)\in \T$, $\cE_\rho(\rho)\in (\rho^2,\rho^2)=\C s_0s_0^*+\C s_1s_1^*+\C s_2s_2^*$, 
and $\cE_{\rho}(\alpha\rho)=\cE_\rho(\alpha)\alpha(\cE_\rho(\rho))=\cE_\rho(\alpha)\cE_\rho(\rho)$.
Direct computation shows the following. 

\begin{lemma}
There are exactly two half-braidings  
$$\cE^0_\rho(\alpha)=-\zeta^2,\quad \cE^0_\rho(\rho)=\zeta^3s_0s_0^*+\zeta^3s_1s_1^*+\zeta^6s_2s_2^*,$$
$$\cE^1_\rho(\alpha)=\zeta^2,\quad \cE^1_\rho(\rho)=\zeta^5s_0s_0^*+s_1s_1^*+\zeta^5s_2s_2^*,$$
for $\rho$. 
The entries of $S$ and $T$ for them are
$$S_{0,\trho^0}=S_{0,\trho^1}=\frac{1}{4\sqrt{2}},\quad s(\alpha,\trho^0)=1,\quad s(\alpha,\trho^1)=-1, \quad S_{\tpi,\trho^0}=S_{\tpi,\trho^1}=-\frac{1}{4\sqrt{2}},$$
$$S_{\tvarphi,\trho^i}=\frac{\mp i}{4},\quad S_{\tpsi,\trho^i}=\frac{\pm i}{4}.$$
$$S_{\trho^0,\trho^0}=\frac{-1\pm\sqrt{2} i}{4\sqrt{2}},\quad 
S_{\trho^0,\trho^1}=\frac{-1\mp\sqrt{2} i}{4\sqrt{2}},\quad S_{\trho^1,\trho^1}=\frac{1\mp\sqrt{2} i}{4\sqrt{2}}$$
$$T_{\trho^0,\trho^0}=-1,\quad T_{\trho^1,\trho^1}=\mp i.$$ 
\end{lemma}

So far we have obtained simple objects 
$$0,\tpi, \alpha\tvarphi, \alpha\tpsi,  \trho^0,  \alpha\trho^1, 
\alpha,\alpha\tpi,\tvarphi,\tpsi,\alpha\trho^0,\trho^2$$
of $\cZ(\cC)$, and the first half belong to $\cD$. 
Since $\dim \cA_{\rho}=\dim_{\alpha\rho}=8$ and $\dim \cA_{\rho,\alpha\rho}=4$, 
there are only two missing, and they are $\tmu^0$ and $\tmu^1$ with $\mu=\rho\oplus \alpha\rho$. 
We may assume $\tmu^0\in \cD$. 

Using $S(\alpha X,Y)=s(\alpha,Y)S_{X,Y}$ and $T_{\alpha X,\alpha X}=T_{\alpha,\alpha}T_{X,X}\overline{s(\alpha,X)}$, 
we can show from our computation so far that $S'$ and $T'$ with respect to 
$$0,\tpi,\alpha\tpsi,  \alpha\tvarphi,  \trho^0,  \alpha\trho^1,\tmu^0 $$
are as follows: 
$$S'=\left(
\begin{array}{ccccccc}
\frac{\sqrt{2}-1}{4} &\frac{\sqrt{2}+1}{4} &\frac{1}{2\sqrt{2}} &\frac{1}{2\sqrt{2}} &\frac{1}{4} &\frac{1}{4}&\frac{1}{2}   \\
\frac{\sqrt{2}+1}{4} &\frac{\sqrt{2}-1}{4} &\frac{1}{2\sqrt{2}} &\frac{1}{2\sqrt{2}} &-\frac{1}{4} &-\frac{1}{4}&-\frac{1}{2}   \\
\frac{1}{2\sqrt{2}} &\frac{1}{2\sqrt{2}} &\frac{-1\mp i}{2\sqrt{2}} &\frac{-1\pm i}{2\sqrt{2}} &\frac{\pm i}{2\sqrt{2}} &\frac{\mp i}{2\sqrt{2}} &0  \\
\frac{1}{2\sqrt{2}} &\frac{1}{2\sqrt{2}} &\frac{-1\pm i}{2\sqrt{2}} &\frac{-1\mp i}{2\sqrt{2}} &\frac{\mp i}{2\sqrt{2}} &\frac{\pm i}{2\sqrt{2}} &0  \\
\frac{1}{4} &-\frac{1}{4} &\frac{\pm i}{2\sqrt{2}} &\frac{\mp i}{2\sqrt{2}} &-\frac{1}{4}+\frac{\pm i}{2\sqrt{2}} &-\frac{1}{4}+\frac{\mp i}{2\sqrt{2}} &* \\
\frac{1}{4} &-\frac{1}{4} &\frac{\mp i}{2\sqrt{2}} &\frac{\pm i}{2\sqrt{2}} &-\frac{1}{4}+\frac{\mp i}{2\sqrt{2}} &-\frac{1}{4}+\frac{\pm i}{2\sqrt{2}} &*  \\
\frac{1}{2} &-\frac{1}{2} &0 &0 &* &* &* \\
\end{array}
\right),
$$
$$T'=\mathrm{Diag}(1,1,\mp i,\mp i,-1,-1,*).$$

\begin{theorem} Let the notation be as above. 
Then
$$S'=\left(
\begin{array}{ccccccc}
\frac{\sqrt{2}-1}{4} &\frac{\sqrt{2}+1}{4} &\frac{1}{2\sqrt{2}} &\frac{1}{2\sqrt{2}} &\frac{1}{4} &\frac{1}{4}&\frac{1}{2}   \\
\frac{\sqrt{2}+1}{4} &\frac{\sqrt{2}-1}{4} &\frac{1}{2\sqrt{2}} &\frac{1}{2\sqrt{2}} &-\frac{1}{4} &-\frac{1}{4}&-\frac{1}{2}   \\
\frac{1}{2\sqrt{2}} &\frac{1}{2\sqrt{2}} &\frac{-1\mp i}{2\sqrt{2}} &\frac{-1\pm i}{2\sqrt{2}} &\frac{\pm i}{2\sqrt{2}} &\frac{\mp i}{2\sqrt{2}} &0  \\
\frac{1}{2\sqrt{2}} &\frac{1}{2\sqrt{2}} &\frac{-1\pm i}{2\sqrt{2}} &\frac{-1\mp i}{2\sqrt{2}} &\frac{\mp i}{2\sqrt{2}} &\frac{\pm i}{2\sqrt{2}} &0  \\
\frac{1}{4} &-\frac{1}{4} &\frac{\pm i}{2\sqrt{2}} &\frac{\mp i}{2\sqrt{2}} &-\frac{1}{4}+\frac{\pm i}{2\sqrt{2}} &-\frac{1}{4}+\frac{\mp i}{2\sqrt{2}} &\frac{1}{2}  \\
\frac{1}{4} &-\frac{1}{4} &\frac{\mp i}{2\sqrt{2}} &\frac{\pm i}{2\sqrt{2}} &-\frac{1}{4}+\frac{\mp i}{2\sqrt{2}} &-\frac{1}{4}+\frac{\pm i}{2\sqrt{2}} &\frac{1}{2}  \\
\frac{1}{2} &-\frac{1}{2} &0 &0 &\frac{1}{2} &\frac{1}{2} &0 \\
\end{array}
\right),
$$
$$T'=\mathrm{Diag}(1,1,\mp i,\mp i,-1,-1,\zeta_8^{\pm 3}).$$
\end{theorem}

\begin{proof} Since $S$ is a symmetric unitary matrix, the remaining entries are determined as above. 
The Gauss sum formula for $T$ determines the last corner of $T'$. 
\end{proof}

\newcommand{\urlprefix}{}
\bibliographystyle{alpha}
\bibliography{newbib}

\end{document}